\documentclass[11pt]{article}

\usepackage[letterpaper, margin=1in]{geometry}

\usepackage{amsmath,amssymb,amsthm,mathtools}
\usepackage{graphicx}
\usepackage{subcaption}
\usepackage[numbers]{natbib}
\usepackage[colorlinks=true,citecolor=blue,linkcolor=blue]{hyperref}
\usepackage{microtype}

\usepackage[utf8]{inputenc} 
\usepackage[T1]{fontenc}    
\usepackage{hyperref}       
\usepackage{url}            
\usepackage{booktabs}       
\usepackage{amsfonts}       
\usepackage{nicefrac}       
\usepackage{microtype}      
\usepackage{xcolor}         

\usepackage{amsmath, amssymb}  
\usepackage{graphicx}          
\usepackage{natbib}            
\usepackage[ruled,vlined]{algorithm2e}  

\hypersetup{
    colorlinks=true,
    linkcolor=ultramarine,
    citecolor=ultramarine,
    filecolor=ultramarine,
    urlcolor=ultramarine,
    pdfview=FitH,
}

\usepackage{mathrsfs}                   
\DeclareSymbolFont{rsfs}{U}{rsfs}{m}{n}
\DeclareSymbolFontAlphabet{\mathrsfs}{rsfs}
\usepackage[mathscr]{euscript}         
\usepackage{dsfont}                     
\usepackage{bm}                         
\usepackage{mathtools}

\DeclareMathOperator{\EE}{\mathbb{E}}
\DeclareMathOperator{\PP}{\mathbb{P}}
\newcommand{\R}{\mathbb{R}}
\DeclareMathOperator{\tr}{\mathrm{Tr}}
\newcommand\abs[1]{\left\lvert#1\right\rvert}
\newcommand\norm[1]{\left\lVert#1\right\rVert}
\DeclarePairedDelimiter\inner{\langle}{\rangle}
\DeclareMathOperator*{\argmin}{arg\,min}

\usepackage{amsthm}

\makeatletter
\@ifundefined{theorem}{\newtheorem{theorem}{Theorem}[section]}{}
\@ifundefined{lemma}{\newtheorem{lemma}[theorem]{Lemma}}{}
\@ifundefined{claim}{}{}
\@ifundefined{proposition}{\newtheorem{proposition}[theorem]{Proposition}}{}
\@ifundefined{corollary}{\newtheorem{corollary}[theorem]{Corollary}}{}

\@ifundefined{definition}{\newtheorem{definition}[theorem]{Definition}}{}
\@ifundefined{example}{\newtheorem{example}[theorem]{Example}}{}
\@ifundefined{assumption}{\newtheorem{assumption}[theorem]{Assumption}}{}
\@ifundefined{remark}{\newtheorem{remark}[theorem]{Remark}}{}
\makeatother

\newcommand{\upright}{\rmfamily\mdseries\upshape}

\usepackage{enumitem}                   
\usepackage{multirow}                   
\usepackage{rotating}                   

\definecolor{babyblue}{rgb}{0.63, 0.79, 0.95}
\definecolor{ultramarine}{rgb}{0.07, 0.04, 0.56}

\usepackage[color=babyblue]{todonotes}
\usepackage{verbatim}
\newcommand{\bgm}[1]{\textcolor{cyan}{#1}}

\title{High-dimensional Limit of SGD\\
for Diagonal Linear Networks}

\author{
\begin{tabular}{cc}
Bego\~na Garc\'ia Malaxechebarr\'ia & Courtney Paquette \\
University of Washington & McGill University \& Google DeepMind \\
\texttt{begogar9@uw.edu} & \texttt{courtney.paquette@mcgill.ca} \\[1.2em]
Maryam Fazel & Dmitriy Drusvyatskiy \\
University of Washington \& Amazon Inc. & University of California, San Diego \\
\texttt{mfazel@uw.edu} & \texttt{ddrusvyatskiy@ucsd.edu}
\end{tabular}
}
\date{}

\begin{document}

\maketitle

\begin{abstract}
  Understanding the behavior of stochastic gradient methods is a central problem in modern machine learning. Recent work has highlighted diagonal linear networks as a simplified yet expressive setting for analyzing the optimization and generalization properties of neural models. In this work, we show that in the high-dimensional regime, stochastic gradient descent on diagonal linear networks is well-approximated by continuous dynamics governed by a stochastic differential equation (SDE), which explicitly decouples the drift from the gradient noise. We further derive a deterministic partial differential equation whose solution propagates the relevant state of the iterates and characterizes the time evolution of a broad class of observable statistics, including the risk, curvature, and other metrics for optimality. 
  Finally, we show that, under a suitable parametrization, the stochastic dynamics are globally well posed and converge exponentially fast to zero risk with high probability, yielding a fully explicit non-asymptotic description of their long-time behavior. Numerical simulations corroborate our theoretical findings.
\end{abstract}

\section{Introduction}
Diagonal linear networks serve as an appealing and analytically tractable model for investigating various phenomena that are observed in deep learning. For example, diagonal linear networks display intriguing forms of implicit bias, whereby stochastic optimization algorithms tend to favor low-complexity solutions even in the absence of explicit regularization \citep{pesme2021implicit,andriushchenko2023sgd,even2023sgd}.
Nonetheless, analyzing SGD for diagonal linear networks is still challenging. In particular, classical analyses of SGD are frequently too coarse, relying on worst-case guarantees that hold only for very small stepsizes. In practice, however, the phenomena of interest emerge only when using relatively large stepsizes. The central difficulty stems from the intrinsic randomness of SGD: the data used at each iteration are random, and the update rule itself is randomized through the sampling of indices that determine the stochastic gradient. 
This motivates the development of sharper analytical frameworks for SGD—ones that capture the evolution of key quantities such as risk and curvature, while remaining mathematically tractable.

One approach is to study gradient flow or diffusion-based approximations  \citep{pesme2021implicit, vivien2022label}. Gradient flow emerges in the limit of a vanishing stepsize, while diffusion-based approximations incorporate stochasticity through an added noise term. Importantly, these methods can often only be rigorously justified as accurate approximations of SGD when the stepsize is extremely small ($\gamma \to 0$)—a regime that is largely unrealistic in practice.

Indeed, SGD for modern applications is typically run with relatively large stepsizes and in high-dimensional settings, where neither the stepsize nor the stochasticity can be treated as infinitesimal. The perspective adopted in this paper is to instead exploit high dimensionality—namely, the large number of parameters—as the primary simplifying mechanism. In the high-dimensional limit, concentration phenomena emerge, yielding exact, deterministic expressions for the risk (and other quantities of interest) at every iteration, even when the stepsize is large.



\begin{figure*}[t!]
    \centering
    \includegraphics[width=0.32\textwidth]{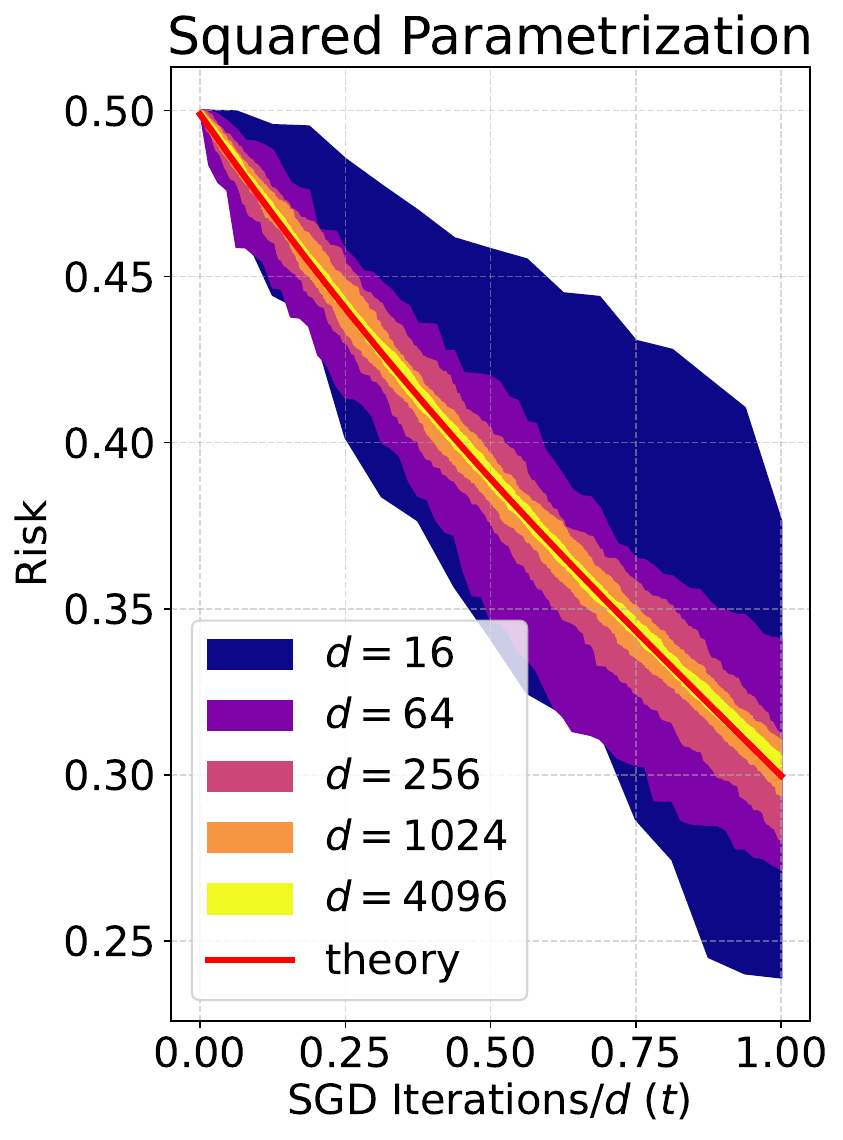} 
    \includegraphics[width=0.32\textwidth]{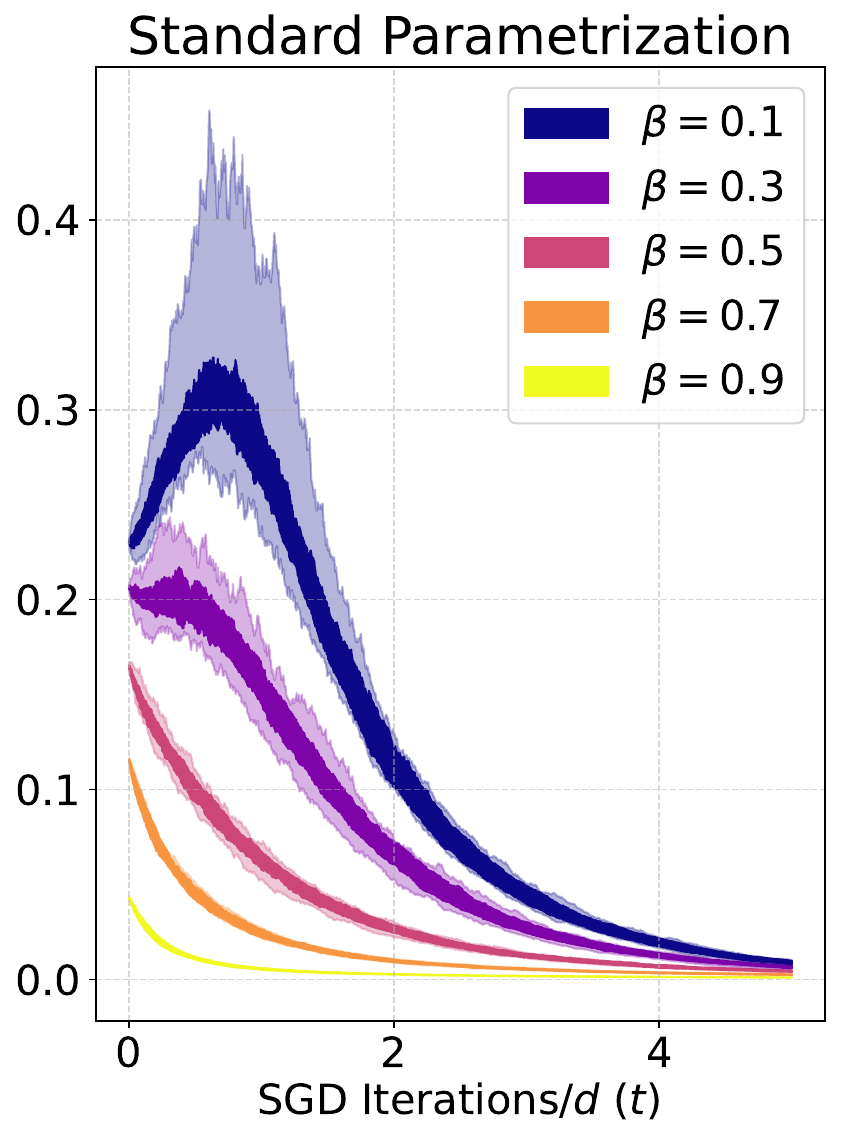} 
    \includegraphics[width=0.32\textwidth]{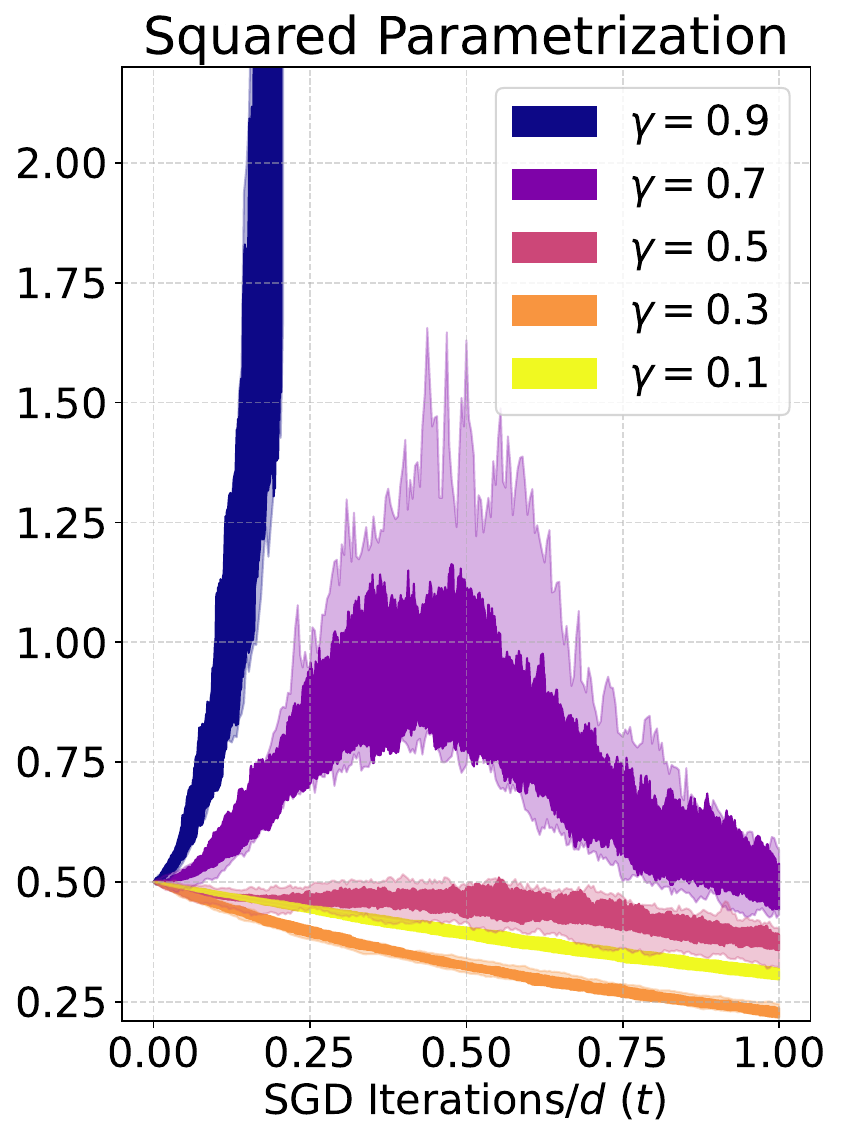} 
    \caption{\textbf{Three views of empirical risk dynamics for SGD on a diagonal linear network.}
\emph{Left:} Covariance $K=I_d$. As $d$ increases, the risk trajectory of SGD concentrates around a deterministic limit (red) described in Theorem~\ref{thm: statistic}.
\emph{Middle:} Power-law covariance spectrum. The homogenized SGD (transparent) from Theorem~\ref{thm: statistic} closely tracks SGD (opaque) over a range of power-law exponents $\beta$ in dimension $d=10^3$.
\emph{Right:} Covariance $K=I_d$ in dimension $d=10^3$. Varying the stepsize $\gamma$ reveals distinct convergence/divergence regimes; the homogenized prediction remains accurate even for stepsizes above the convergence threshold. Left and right panels use the parametrization (\ref{ex: squaredParam}), whereas the middle one uses the parametrization (\ref{ex: hadamardParam}). See Appendix~\ref{app: numerics} for simulation details.} \label{fig: risk} 
\end{figure*}

Specifically, in this work, we derive a continuous-time stochastic differential equation (SDE) that tracks discrete-time SGD in this high-dimensional, large stepsize regime for diagonal linear networks. Using this SDE as a foundation, we obtain a deterministic PDE that propagates the relevant state of the iterates and thereby predicts the time evolution of a broad class of key observable statistics. \emph{To the best of our knowledge, our result provides the first deterministic description of the discrete-time SGD trajectory in high dimensions with fixed, non-vanishing stepsizes for diagonal linear networks, placing them among the first nonlinearly parametrized models admitting such a description beyond gradient-flow or stochastic-gradient-flow limits.} 

This effectively allows for the analysis of SGD \emph{without running it} (See Figure~\ref{fig: risk}): solving the deterministic evolution provides learning curves without the need for expensive Monte Carlo simulations across mini-batch realizations. While the PDE offers theoretical insight, the SDE remains a powerful practical tool—it is significantly cheaper to simulate, and It\^o calculus provides a direct route to closed evolution equations for many statistics of interest, which we validate numerically.

Setting the stage, our work targets a general class of problems where the risk exhibits the form: 
\begin{equation}\label{eqn:func_class}
\mathcal{R}(x) = \EE_a  \left[f \left( \frac{1}{\sqrt{d}} \begin{pmatrix}
\psi(x) \\
\beta^*
\end{pmatrix}^{\!\top} a\right)\right] \qquad\text{ for } a\sim \mathcal{N}(0, K).     
\end{equation}
Here $K \in \R^{d \times d}$ is the covariance of the data,  $\beta^*\in \R^d$ should be thought of as a ground truth vector, 
$\psi(x)$ maps the features $x \in \R^{2d}$ into $\R^d$, and $f$ is a loss such as mean-square error or logistic. 
The primary example of \eqref{eqn:func_class} is the mean-squared error of a {two-layer} \textit{diagonal linear network},  given by
\begin{equation}\label{eqn:dnn1}
    \min_{u,v\in \R^d} \frac{1}{2d}\EE_a  \inner{u\odot v - \beta^*, a}^2,
\end{equation}
where we set $x=(u,v)$ with $u,v\in \R^d$ and $u\odot v$ is the Hadamard product. Thus, even though the network is linear in its prediction, the optimization variables enter the risk nonlinearly. This nonlinear parametrization is precisely what makes the high-dimensional closure problem substantially more delicate than in linear or generalized linear models. It is standard to check that \eqref{eqn:dnn1} can be  re-parametrized by a difference of squares \citep{nacson2022implicit, pesme2021implicit}:
\begin{equation}\label{eq: squaredParam}
    \min_{u,v\in \R^d} \frac{1}{4d}\EE_a  \inner{u^2 - v^2 - \beta^*, a}^2.
\end{equation}
 Diagonal linear networks, in both formulations, have recently garnered  attention due to their ability to recover sparse solutions \citep{pesme2021implicit, vaskevicius2019implicit, kanade2023statistical, haochen2021shape}.

To solve \eqref{eqn:func_class},  we run \textit{streaming} stochastic gradient descent under a deterministic stepsize schedule $\gamma_k$, that is, at each iteration we generate a fresh sample of data points $a_{k+1} \sim \mathcal{N}(0,K)$ and update the iterates according to the rule
\[
x_{k+1} = x_k - \gamma_k \, \nabla_x \Psi (x_k; a_{k+1}),
\]
where $\Psi(x,a)$ is the integrand in equation \eqref{eqn:func_class}.
To analyze SGD, we introduce a stochastic differential equation, called the \textit{homogenized SGD}:
\begin{equation}\label{eq: sdeFormula}
\mathrm{d} \mathscr{X}_t = - \gamma(t) \, d \, \nabla \mathcal{R}(\mathscr{X}_t) \mathrm{d}t + \gamma(t) \sqrt{\EE_a \left[\nabla f\left(\tfrac{1}{\sqrt{d}}
\psi(\mathscr{X}_t)
^{\!\top} a\right)^2\right]} \, \nabla \psi(\mathscr{X}_t)^\top  \sqrt{K}\, \mathrm{d}\mathfrak{B}_t,
\end{equation}
where the risk $\mathcal{R}$ and inner map $\psi$ are defined in \eqref{eqn:func_class}, the initial conditions are given by $\mathscr{X}_0 = x_0$, and $\mathrm{d}\mathfrak{B}_t$ is the differential of a standard Brownian motion in $\R^d$. In this work, we rigorously show that the homogenized SGD behaves like SGD in high dimensions, even with stepsizes at or above the convergence threshold (see Fig.~\ref{fig: risk}).

To enable a comparison between SGD and homogenized SGD, we extend the discrete-time iterate sequence $\{x_k\}$ to continuous time. The 
$k$-th iterate of SGD corresponds to the continuous time parameter 
$t$ in homogenized SGD via $k = \lfloor td \rfloor$. Thus,
when $t = 1$, SGD has completed  $d$ updates.

\paragraph{Main Contributions.} 
Our first main result shows that SGD and homogenized SGD are close in the sense that for a wide class of functions $\varphi$ (statistics) the values $\varphi(x_{\lfloor td \rfloor})$ and $\varphi(\mathscr{X}_t)$
along the two trajectories are uniformly close. Two particularly informative examples are the risk $\mathcal{R}(x)$ and the scaled Hessian-trace statistic $\frac{1}{d}\tr\!\big(\nabla^2 \mathcal{R}(x)\big)$, a scalar measure of curvature (often interpreted as a notion of sharpness) along the optimization path. See Figures~\ref{fig: risk} and~\ref{fig: statistics} for an illustration.

\begin{theorem}[Informal]{
Under mild conditions, formalized in Theorem~\ref{thm: statistic}, for any statistic $\varphi$ satisfying Assumption~\ref{ass: statistic}, any $\varepsilon \in (0, \tfrac{1}{2})$ and any $T>0$, there exists a constant $C$ (independent of $d$) such that \textit{with overwhelming probability}\footnote{We say an event holds \textit{with overwhelming probability (w.o.p.)} if there is a function $\omega \colon \mathbb{N} \to \R$ with $\omega(d)/\log d \to \infty$ so that the event holds with probability at least $1- e^{-\omega(d)}$.} the estimate holds:
\begin{equation}\label{eq: sdeConcentration}
\sup_{0\leq t \leq T} \abs{\varphi(x_{\lfloor td \rfloor}) - \varphi(\mathscr{X}_t)} \leqslant C d^{-\varepsilon}.
\end{equation}}
\end{theorem}

Thus, in the high-dimensional regime, SGD is accurately described by the homogenized SGD dynamics. In particular, the SDE approximation is both accurate and analytically tractable, providing a practical tool to study the behavior of SGD while allowing fixed (and potentially large) stepsizes. This contrasts with the classical small stepsize diffusion approximation, which instead analyzes the regime where the stepsize is infinitesimally small for any fixed dimension.



Our second main result goes a step further. We show that in high dimensions, the randomness becomes negligible at the level of $\varphi$. 
Specifically, the processes $\varphi(x_{\lfloor td \rfloor})$ and $\varphi(\mathscr{X}_t)$ concentrate uniformly over $t\in [0,T]$ around a {\em deterministic function} $\phi(t)$ (see Fig.~\ref{fig: risk}). 
The function $\phi(t)$ is given explicitly in terms of the solution of a deterministic partial integro-differential equation (see \eqref{eq: detEquivVarphi} and \eqref{eq:  PDESystem} for details). 
In other words, $\phi(t)$ yields a deterministic prediction for $\varphi$ for large $d$, allowing one to analyze these statistic curves without averaging over SGD noise. 

\begin{theorem}[Informal]
Under mild conditions, formalized in Theorem~\ref{thm: statistic}, for any statistic $\varphi$ satisfying Assumption~\ref{ass: statistic}, any $\varepsilon \in (0, \tfrac{1}{2})$ and any $T>0$, there exists a constant $C$ (independent of $d$) such that with overwhelming probability the estimate holds:
\begin{equation}\label{eq: deterministicConcentration}
\sup_{0\leq t \leq T} \biggl(  \abs{\varphi(x_{\lfloor td \rfloor}) - \phi(t)} + \abs{\varphi(\mathscr{X}_t) - \phi(t)} \biggr) \leqslant C d^{-\varepsilon}
\end{equation}
where $\phi(t)$ is a deterministic function  defined in \eqref{eq: detEquivVarphi} and \eqref{eq:  PDESystem}.
\end{theorem}

Thus the function $\phi$ provides a tractable, deterministic description of the evolution of a broad class of statistics $\varphi(x)$ along the SGD trajectory in the high-dimensional regime, even for \emph{large} stepsizes. We therefore expect this theorem to be useful for analyzing fine-grained properties of SGD such as instability, saddle point avoidance/escape, progressive sharpening, implicit bias, etc.

\begin{figure*}[t!]
    \centering
    \includegraphics[width=0.64\textwidth]{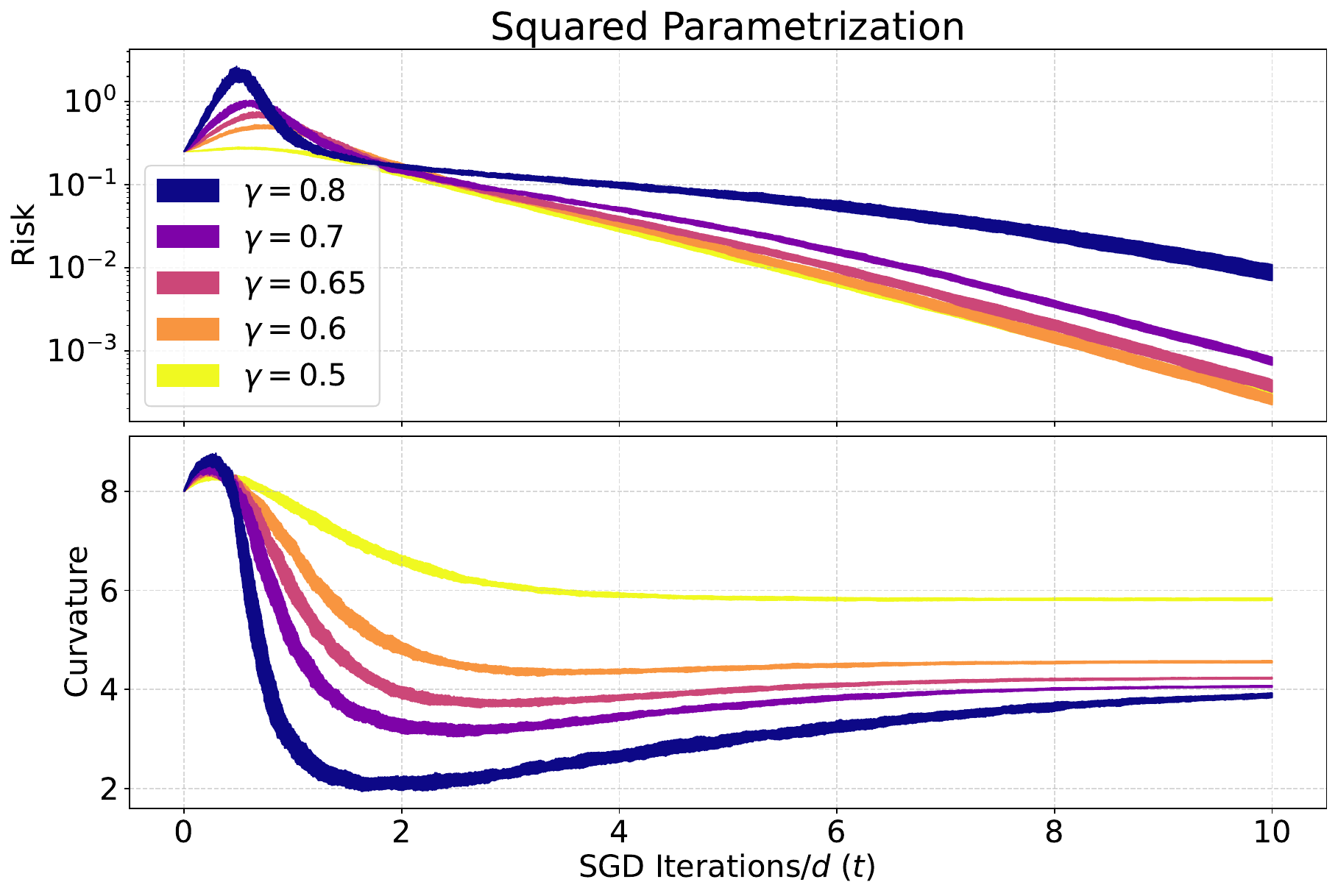} 
    \includegraphics[width=0.32\textwidth]{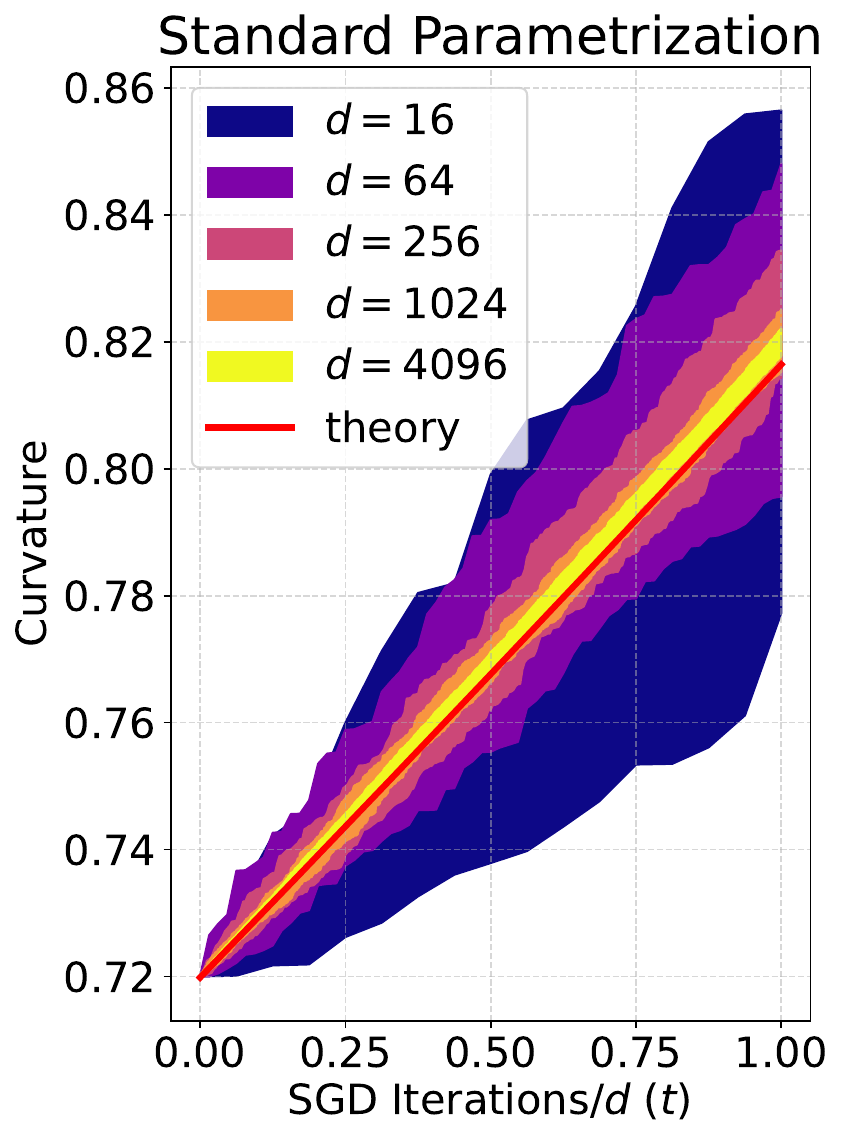} 
    \caption{\textbf{Curvature dynamics for SGD on a diagonal linear network.} \emph{Left:} The evolution of the curvature measured by the scaled trace of the Hessian $\frac{1}{d}\tr(\nabla^2\mathcal{R})$ is shown alongside the empirical risk $\mathcal{R}$, illustrating ``flat'' progress in which the risk increases sharply accompanied by a marked drop in curvature as we vary the stepsize $\gamma$.
    \emph{Right:} As the dimension $d$ increases, the curvature dynamics of SGD concentrate around a deterministic limit (shown in {red}), as proven in Theorem~\ref{thm: statistic}. See Appendix~\ref{app: numerics} for simulation details.}
    \label{fig: statistics}
\end{figure*}

Lastly, our final result complements this picture by establishing global linear convergence of the stochastic dynamics themselves. 

\begin{theorem}[Informal]
Consider homogenized SGD \eqref{eq: sdeFormula} for diagonal linear networks under the squared parametrization \eqref{eq: squaredParam}, initialized at $\mathscr{U}_{0,i}=\mathscr{V}_{0,i}=1$ for $i=1,\dots,d$. Then there exists a numerical constant $c>0$ such that for any stepsize $\gamma\in (0,c)$, the risk $\mathcal{R}(\mathscr{X}_t)$ converges to zero exponentially fast. More precisely, for any $\delta\in(0,1)$, there exist numerical constants $C,\mu>0$ such that with probability at least $1-\delta$ we have
$\mathcal{R}(\mathscr{X}_t) \leqslant C e^{-\mu t}$ for all $t\geqslant 0.$
\end{theorem}

\subsection{Literature Review}

\paragraph{Continuous-time and stochastic-process approximations of SGD.}
A common approach to analyzing SGD is to replace the discrete-time iteration by a continuous-time process. Gradient flow arises in the infinitesimal-stepsize limit, but suppresses both finite-stepsize effects and stochasticity. Stochastic gradient flow and diffusion-type SDE approximations incorporate noise and can capture solution-selection phenomena not explained by gradient flow \citep{pesme2021implicit}. Related viewpoints include studying SGD augmented with label noise as a proxy for isolating stochastic effects during training \citep{vivien2022label}. These diffusion approximations are typically justified in small-stepsize limits, often at fixed dimension. In high-dimensional settings, stochastic-process descriptions can reveal additional regularization mechanisms, including biases induced by parameter-dependent noise \citep{andriushchenko2023sgd} and repulsive interactions in eigenvalue dynamics that promote rank deficiency \citep{varre2024sgd}. Beyond diffusion approximations, dynamical mean field theory (DMFT) provides another route to tracking high-dimensional training dynamics, including for SGD in Gaussian mixture classification and related batching regimes \citep{mignacco2020dynamical,celentano2021highdimensional,gerbelot2024rigorous}.

\paragraph{High-dimensional deterministic limits for SGD dynamics.}
Complementary to continuous-time approximations, a long line of work derives deterministic dynamical descriptions of learning in high-dimensional teacher--student and random-feature models. Early studies of multi-index and soft-committee models established ODE characterizations of the risk dynamics and showed how algorithmic choices, such as the stepsize, affect convergence and generalization \citep{saad1995dynamics,biehl1994online,biehl1995learning,saad1995exact}. Building on this statistical-physics tradition, \citet{goldt2019dynamics}
rigorously justified the resulting finite-dimensional ODE description for online
SGD in two-layer teacher--student networks with large input dimension and
finitely many hidden units, including overparameterized students and the case
where both layers are trained. Parallel mathematical work developed systematic
scaling-limit techniques for high-dimensional online learning dynamics, often
based on empirical-measure and martingale decompositions
\citep{wang2017scaling,wang2019solvable}. More recent work emphasizes that generalization depends on the interaction between algorithm, architecture, and data distribution \citep{goldt2019dynamics}, and develops general techniques (e.g., martingale arguments) for proving high-dimensional limits \citep{wang2017scaling,wang2019solvable}. Comparisons across limiting regimes and refinements of these ODE descriptions are studied in \citep{arnaboldi2023highdimensional}. Other variants show how changes in the geometry or constraints of the dynamics can lead to signal-locking and dimension-robust behavior \citep{arous2022highdimensional,arnaboldi2023escaping,damian2023smoothing}. Streaming and multi-pass SGD have also been analyzed through deterministic descriptions based on integral equations and homogenization-type methods \citep{mousavi-hosseini2023neural,paquette2021sgd,lee2022trajectory,paquette2025homogenization}. Recent high-dimensional analyses have also begun to address simplified
attention/transformer-like architectures, including sequence single-index
models and single-layer attention models trained by SGD or few-step gradient
procedures \citep{arnaboldi2026asymptotics,barnfield2026highdimensional}; these works, however, address low-dimensional order-parameter dynamics,
gradient-flow approximations, or few-step training rather than
trajectory-level deterministic equivalents for discrete-time SGD in diagonal
parametrizations.

\paragraph{Covariance structure and finite-stepsize effects.}
Many classical high-dimensional analyses assume isotropic data, but non-identity covariance can qualitatively change the training dynamics and is important when the population covariance is unknown. Prior work derives equations of motion under Gaussian-equivalence principles for multi-index and random-feature models \citep{goldt2020modeling}, with extensions to deeper architectures \citep{goldt2022gaussian}. Covariance structure has also been linked to plateau phenomena and slow phases in learning \citep{yoshida2019datadependence}. Recent work derives exact high-dimensional limits for linear regression and analyzes how non-identity covariance reshapes the optimization landscape and convergence behavior \citep{collins–woodfin2024highdimensional,collins-woodfin2024hitting}; related extensions to adaptive algorithms are studied in \citep{balasubramanian2025highdimensional,collins-woodfin2024high}. 
These results are especially relevant at finite stepsizes, where stochasticity, data geometry, and algorithmic parameters can interact beyond the small-stepsize limits captured by diffusion approximations.

\paragraph{Implicit bias, flatness, and the role of stepsize.}
A related theme is the implicit bias of gradient-based optimization. In separable logistic regression, gradient descent converges in direction to the max-margin classifier \citep{soudry2018implicit}, and in modern overparameterized settings SGD is often viewed as providing an implicit regularization mechanism that can improve generalization with little or no explicit regularization \citep{zhang2017understanding}. One influential hypothesis links this behavior to SGD's preference for flatter minima \citep{he2019control} and to empirical correlations between flatness and generalization \citep{keskar2017largebatch}. Related evidence appears in structured recovery problems, where flat minima can coincide with exact recovery under suitable conditions \citep{ding2024flat}. Stepsize and batch size are key levers in this picture: larger stepsizes and smaller batch sizes have been argued to encourage exploration of flatter regions \citep{jastrzębski2019relation}, and the batch-size-to-learning-rate ratio can affect generalization \citep{he2019control}. Initialization also plays an important role in determining the implicit bias: the initialization scale can control the transition between kernel/lazy and rich/active regimes \citep{woodworth2020kernel}, while the relative shape of the initialization can further affect the limiting solution selected by gradient methods \citep{azulay2021implicit}. Other proposed mechanisms include landscape-smoothing interpretations \citep{kleinberg2018alternative}, random-walk models on random landscapes \citep{hoffer2017train}, and explanations based on batch variability and ``catapult'' effects \citep{zhu2024catapults}. At the same time, recent evidence cautions that sharpness and flatness can be tightly entangled with optimization hyperparameters such as the stepsize, and need not correlate monotonically with generalization across settings \citep{andriushchenko2023modern}.

\paragraph{Diagonal linear neural networks.}
Diagonal linear networks provide a minimal yet expressive model for studying implicit bias, stochastic optimization, and solution selection in overparameterized settings. They allow precise comparisons between SGD and gradient descent or gradient flow, including regimes where SGD exhibits favorable implicit bias \citep{pesme2021implicit}. In sparse feature learning, empirical work suggests that large stepsizes alone may be insufficient for sparsification without stochasticity \citep{andriushchenko2023sgd}. Complementary theory shows that, in sparse regression with diagonal linear networks, large stepsizes can systematically improve SGD while harming gradient descent \citep{even2023sgd}.

\subsection{High-dimensional Model Structure}
We will consider objective functions $\mathcal{R}$ defined as the expectation of a composition of a simple outer function $f\colon \R^2 \to \R$ and an inner map $\psi\colon \R^{2d}\to \R^d$, given by
\begin{equation}\label{eq: riskSetting}
\mathcal{R}(x) := \EE_a  \Psi (x; a), \quad \text{for} \quad a\sim \mathcal{D} \subset \R^d, \quad \text{where} \quad \Psi (x; a) := f \left( \frac{1}{\sqrt{d}} \begin{pmatrix}
\psi(x) \\
\beta^*
\end{pmatrix}^{\!\top} a\right).
\end{equation}
In the probabilistic analysis, it is important to treat $f$ as a function of both components of its input \( r = (r_1, r_2) \in \R^2 \). However, in the optimization context, we often regard \( f \) as a function of a single real variable \( r_1 \), with the second component \( r_2 \) considered a fixed (but possibly random) parameter. 

We impose the following mild Lipschitzness assumption on $f$. Throughout, the symbol $\|\cdot\|$ will denote the standard $\ell_2$-norm on vectors and the Frobenius norm on matrices. 

\begin{assumption}[Pseudo-Lipschitz continuity of $f$]\upright\label{ass: fPseudoLipschitz} The outer function $f\colon \R^2 \to \R$ is $\alpha$-pseudo-Lipschitz with constant $L(f)$. That is, for all $r, \hat{r}\in \R^2$, it holds:
\begin{equation*}
\abs{f(r)-f(\hat{r})} \leqslant L(f) \norm{r - \hat{r}}(1+\norm{r}^\alpha + \norm{\hat{r}}^\alpha).
\end{equation*}
\end{assumption}

In words, pseudo-Lipschitzness stipulates polynomial growth of Lipschitz constants on balls with growing radii. Next we summarize our assumption on the data vector $a\sim \mathcal{D}$.

\begin{assumption}[Data]\upright\label{ass: data}
We consider data samples \( a \sim \mathcal{D} \) drawn from a Gaussian distribution \( \mathcal{N}(0, K) \), where the covariance  \( K \in \R^{d \times d} \) is diagonal. The entries of \( K \) are uniformly bounded, meaning its operator norm satisfies $\norm{K}_{\mathrm{op}} \leqslant \bar{K}$, for some constant $\bar{K}>0$ independent $d$. In addition, we assume that the trace of \( K \) scales linearly with the dimension $d$, that is, $\tr(K) = \Theta(d)$. 
\end{assumption}
\begin{remark}\upright
The Gaussian assumption is used later both for concentration and for explicit
conditional moment calculations. In particular, for a fixed parameter value
$x$, the proof conditions the data vector $a$ on its projection onto the
low-dimensional subspace generated by $\psi(x)$ and $\beta^*$. Gaussian
conditioning gives closed-form conditional means and covariances, which are used
to identify the limiting drift and diffusion terms. The same assumption also
provides concentration bounds for the linear and quadratic forms appearing in
the martingale estimates. Some of these concentration steps may extend to more general light-tailed data, but the explicit conditional moment calculations used to identify the limiting drift and diffusion explicitly rely on Gaussianity.
\end{remark}
\begin{remark}\upright
Observe that the trace condition $\tr(K) = \Theta(d)$ reflects the assumption that the total variance of the input distribution remains proportional to the ambient dimension. Our main results are calibrated for this high-dimensional regime. In settings where the trace of \( K \) grows more slowly than \( d \), alternative stepsize scalings for SGD may be required to obtain nontrivial limiting behavior.    
\end{remark}

Our main target application is that of a diagonal linear network, in either parametrization \eqref{eqn:dnn1} and \eqref{eq: squaredParam}. In order to treat both parametrizations simultaneously,  we assume that the inner function $\psi$ has a special quadratic form. 

\begin{assumption}[Diagonal Linear Neural Networks]\upright\label{ass: diagonalNeuralNetworks}
The inner function $\psi \colon \R^{2d} \to \R^d$ is \\ component-wise quadratic, meaning that each output coordinate for $i\in \{1, \dots, d\}$ is given by 
\begin{equation}
\psi_i(u,v)
=
\begin{pmatrix}
u_i & v_i
\end{pmatrix}
\mathcal{Q}
\begin{pmatrix}
u_i \\
v_i
\end{pmatrix}
+
l^\top
\begin{pmatrix}
u_i \\
v_i
\end{pmatrix}
+c,  
\end{equation}
where $\mathcal{Q}
=
\begin{pmatrix}
q_{11} & q_{12}\\
q_{12} & q_{22}
\end{pmatrix}
\in \R^{2\times 2}$ is symmetric, $l= \begin{pmatrix}
    l_1 \\
    l_2
\end{pmatrix} \in \R^2$, and $c\in \R$.
\end{assumption}

Concretely, the formulations \eqref{eqn:dnn1} and 
\eqref{eq: squaredParam}
correspond to setting $\psi(u,v) = u\odot v$ for $\mathcal{Q} = \begin{pmatrix}
    0 & \frac{1}{2}\\
    \frac{1}{2} & 0
\end{pmatrix}$, and $\psi(u,v) = u^2 -v^2$ for $\mathcal{Q} = \begin{pmatrix}
    1 & 0\\
    0 & -1
\end{pmatrix}$, respectively; while setting $l$ and $c$ to be zero in both cases.
Next,  we assume that the ground truth vector $\beta^*$ has bounded entries.

\begin{assumption}[Parameter Scaling]\upright\label{ass: parameterScaling}
The entries of the signal $\beta^* \in \R^d$ are uniformly bounded, that is, $\norm{\beta^*}_\infty \leqslant C$ for some $C>0$ independent of $d$. 
\end{assumption}

In what follows, for a parameter vector $x=(u,v)\in \R^{2d}$, we introduce the block matrix 
\begin{equation}\label{eqn:stacked_notation}
W(x) := [\, \psi(x) \, |  \,  \beta^* \, | \, \mathds{1}_d \, ]\in \R^{d\times 3},    
\end{equation}
and let $B(x)$ be the covariance matrix of the random vector $\frac{1}{\sqrt{d}} W(x)^\top a \in \R^3$, or more explicitly
\begin{equation}\label{eq: defB}
B(x)\;:=\;\frac1d\,W(x)^\top K\,W(x)\in\R^{3\times 3}.
\end{equation}
 Since $\mathcal{R}(x)$ depends on 
$x$ only through the Gaussian
$\frac{1}{\sqrt{d}} W(x)^\top a$, the learning dynamics can be summarized by the  matrix $B(x)$: there exists a function $h\colon \R^{3\times 3}\to\R$ such that the risk decomposes as $\mathcal{R}(x) =  h\left(B(x)\right)$. Explicitly, for both formulations \eqref{eqn:dnn1} and 
\eqref{eq: squaredParam}, we obtain the same function
\begin{equation*}
h\left(B\right) = \frac{1}{2}\left(B_{11} - B_{12} - B_{21} + B_{22}\right).
\end{equation*}

The next two assumptions simply require that (i) the risk is a smooth function of this summary (so we may differentiate through the expectation to obtain the gradient) and (ii) the corresponding second-moment quantity controlling the size of the stochastic gradient noise is equally well-behaved.

\begin{assumption}[Risk Representation]\upright\label{ass: riskRepresentation} There 
is an open set $\mathcal{U}\subseteq \R^{3 \times 3}$ such that $B(x_0) \in \mathcal{U}$, and provided that $B(x) \in \mathcal{U}$, the map $x\to \mathcal{R}(x) :=  h\left(B(x)\right)$ is differentiable and satisfies
\begin{equation*}
\nabla \mathcal{R}(x) = \EE_a \nabla_x \Psi(x ; a).  
\end{equation*}
Furthermore, $h$ is continuously differentiable on $\mathcal{U}$ and its derivative $\nabla h$ is $\alpha$-pseudo-Lipschitz, that is, there is a constant $L(h)>0$, so that for all $B, \hat{B}\in \mathcal{U}$, it holds:
\begin{equation}\label{eqn:norm_pseudo_h}
\norm{\nabla h(B) - \nabla h(\hat{B})} \leqslant L \norm{B-\hat{B}}(1 + \norm{B}^\alpha + \norm{\hat{B}}^\alpha ).
\end{equation}
\end{assumption}

We note that the commutation of expectation and gradient holds trivially on $\mathcal{U} = \R^{3 \times 3}$ once $\Psi$ is continuously differentiable. Moreover, the choice of the Frobenius norm in \eqref{eqn:norm_pseudo_h} is essentially arbitrary, and it can be replaced by the operator norm due to equivalence of norms on $\R^{3\times 3}$.


Our arguments will rely on the evolution of the second moment of the gradient $\EE_a \left[\nabla_{r_1} f\left( r\right)^2\right]$, where we set $r:=\frac{1}{\sqrt{d}} \begin{pmatrix}
\psi(x) \\
\beta^*
\end{pmatrix}^{\!\top} a$. Since this quantity depends on $x$ only through the Gaussian $\frac{1}{\sqrt{d}} W(x)^\top a$, analogous to the risk, there exists a function $I\colon \R^{3 \times 3}\to \R$ satisfying $\EE_a [\nabla_{r_1} f(r)^2] = I\left(B(x)\right)$. In particular, for the two formulations \eqref{eqn:dnn1} and 
\eqref{eq: squaredParam}, the function $I$ is simply
\[
I(B) = 2 h\left(B\right).
\]

\begin{assumption}[Pseudo-Lipschitzness of Square Gradients]\upright\label{ass: squaredGradientsPseudoLipschitz}
The function $I$ is $\alpha$-pseudo-Lipschitz with constant $L(I)>0$, that is, for all $B, \hat{B}\in \mathcal{U}$,
\begin{equation*}
\abs{I(B) - I(\hat{B})} \leqslant L \norm{B-\hat{B}}(1 + \norm{B}^\alpha + \norm{\hat{B}}^\alpha ).
\end{equation*}
\end{assumption}

\begin{remark}\upright
All the typical losses satisfy the requisite assumptions, including logistic regression and the mean square error (see Appendix~\ref{app: examples}). We also note that Assumptions~\ref{ass: riskRepresentation} and~\ref{ass: squaredGradientsPseudoLipschitz} are nearly satisfied for $L$-smooth objectives $f$, and a version of the main theorem holds under just this
assumption (albeit with a weaker conclusion) (see [\cite{collins-woodfin2024hitting}, Appendix B]).
\end{remark}

\subsection{Algorithm Formulation}
We consider the widely used \textit{streaming stochastic gradient descent} (SGD) algorithm. At each iteration $k$, the algorithm generates a fresh data point $a_{k+1} \sim \mathcal{N}(0, K)$ and updates the iterates $x_k \in \R^{2d}$ using a stepsize $\gamma_k$ according to the recurrence
\begin{equation}\label{eq:  sgd_rule}
x_{k+1} = x_k - \gamma_k \, \nabla_x \Psi (x_k; a_{k+1}), 
\end{equation}
where $\nabla_x$ is the usual gradient operator with respect to the $x$ variable. We impose the following assumption on the stepsize, which stipulates that $\gamma_k$ is a bounded deterministic function of $k/d$.

\begin{assumption}[Stepsize]\upright\label{ass: learningRate}
There exists a constant $\bar{\gamma} <\infty$ and a deterministic scalar function $\gamma\colon [0,\infty) \to [0,\infty)$ which is bounded by $\bar{\gamma}$ and satisfies the equation $\gamma_k = \gamma\left(\frac{k}{d}\right)$.
\end{assumption}

We always work in a formulation where the entries of the iterates $\{x_k\}$ remain bounded, independent of dimension. Within the class of high-dimensional representations, we note that the
random variable $\frac{1}{\sqrt{d}} \left(\psi(x)\right)^\top a$ should not carry dimension dependence, as otherwise the outer function $f$ (which can very well be non-linear) degenerates to its behavior at infinity.

The functions $h$ and $I$ will allow us to construct closed, deterministic dynamics that describe
the high-dimensional limit of stochastic gradient descent. To condense the notation, we use
\begin{equation*}
r_k := \frac{1}{\sqrt{d}} \begin{pmatrix}
\psi(x_k) \\
\beta^*
\end{pmatrix}^{\!\top} a_{k+1}
\in \R^2,
\end{equation*}
so that the SGD update \eqref{eq: sgd_rule} simplifies as follows: 
\begin{equation}\label{eq:  sgd_rule_update}
x_{k+1} = x_k - \frac{\gamma_k}{\sqrt d}\,
\nabla_{r_1}f(r_k)\,
(\nabla\psi(x_k))^\top a_{k+1}.
\end{equation}
Note that the stepsize $\gamma_k$ is scaled in a way that SGD will behave well across different dimensions; without the factor of $\sqrt{d}$ in \eqref{eq:  sgd_rule_update}, the algorithm would degenerate to pure noise or to gradient flow as dimension increases. For any fixed dimension, note however that the stepsize $\gamma_k$ can be arbitrarily large (as long as it is bounded uniformly across dimensions). 

\section{High-dimensional Diffusion Approximation for SGD}\label{sec: sde}
We begin by introducing our stochastic differential equation (SDE), called \textit{homogenized SGD}. Specifically, we approximate SGD by the diffusion $(\mathscr{X}_t)_{t\geqslant 0}$ solving the SDE:
\begin{equation}\label{eq: homogenizedSGD}
\mathrm{d} \mathscr{X}_t = - \gamma(t) d \nabla \mathcal{R}(\mathscr{X}_t) \mathrm{d}t + \gamma(t) \sqrt{I\left(B(\mathscr{X}_t) \right)} \, \left(\nabla \psi(\mathscr{X}_t)\right)^\top  \sqrt{K} \mathrm{d}\mathfrak{B}_t,
\end{equation}
with initial condition $\mathscr{X}_0 = x_0$ and $\mathrm{d}\mathfrak{B}_t$ the differential of a standard Brownian motion in $\R^d$. In the diffusion coefficient in \eqref{eq: homogenizedSGD}, the quantity \(I(B(x))\) captures the effective scalar magnitude of the stochastic gradient noise, while \((\nabla\psi(x))^\top\sqrt K\) describes how this noise propagates through the nonlinear parametrization and the data covariance. This structure is not imposed by covariance matching; it arises from the high-dimensional concentration argument used to identify the limiting drift and diffusion.
\begin{remark}[Non-diagonal covariance]\upright
The homogenized SDE \eqref{eq: homogenizedSGD} admits a natural extension to general covariance matrices $K$ in the mean-squared error setting introduced in Section~\ref{subsec: meanSquaredSetting}, where the diffusion is determined by the full covariance of the stochastic gradient. The key point is that in the diagonal $K$ setting, the last two terms arising in the conditional covariance decomposition of the Hessian contribution are effectively negligible for the class of statistics considered in this work, whereas for non-diagonal covariance matrices these terms appear to remain non-negligible in the high-dimensional limit, contributing an additional rank-one fluctuation term to the covariance structure of the noise. Indeed, an application of Isserlis/Wick formula yields
\begin{equation*}
\EE_a\!\left[\inner{\psi(\mathscr X_t)-\beta^\star,a}^2 aa^\top\right]
=
((\psi(\mathscr X_t)-\beta^\star)^\top K (\psi(\mathscr X_t)-\beta^\star))K
+
(K(\psi(\mathscr X_t)-\beta^\star))(K(\psi(\mathscr X_t)-\beta^\star))^\top.
\end{equation*}
This leads formally to the SDE
\begin{equation}\label{eq:nonDiagonalSquareLossSDE}
\mathrm{d} \mathscr{X}_t
=
-\gamma(t)\, d\, \nabla \mathcal{R}(\mathscr{X}_t)\,\mathrm{d}t
+
\gamma(t)
\left(
I \left(B(\mathscr X_t)\right)
\left(\nabla \psi(\mathscr{X}_t)\right)^\top
K
\nabla \psi(\mathscr{X}_t)
+
d\,\nabla \mathcal R(\mathscr X_t)
\left(\nabla \mathcal R(\mathscr X_t)\right)^\top
\right)^{1/2}
\mathrm d\mathfrak B_t,
\end{equation}
where in this case $\mathrm d \mathfrak B_t$ denotes the differential of a standard Brownian motion in $\mathbb R^{2d}$.

The main obstruction to extending our deterministic equivalent result to general covariance $K$ is therefore not the SDE approximation itself, but rather the closure of the deterministic equations for the statistics. When $K$ is non-diagonal, the covariance couples the coordinates through non-commuting resolvent terms, and the finite-dimensional PDE closure used in this work no longer applies. Nevertheless, our numerical experiments suggest that the same high-dimensional concentration phenomenon persists beyond the diagonal setting; see Figure~\ref{fig: nondiagK}.
\end{remark}
\begin{figure}[t!]
    \centering
    \includegraphics[width=0.45\textwidth]{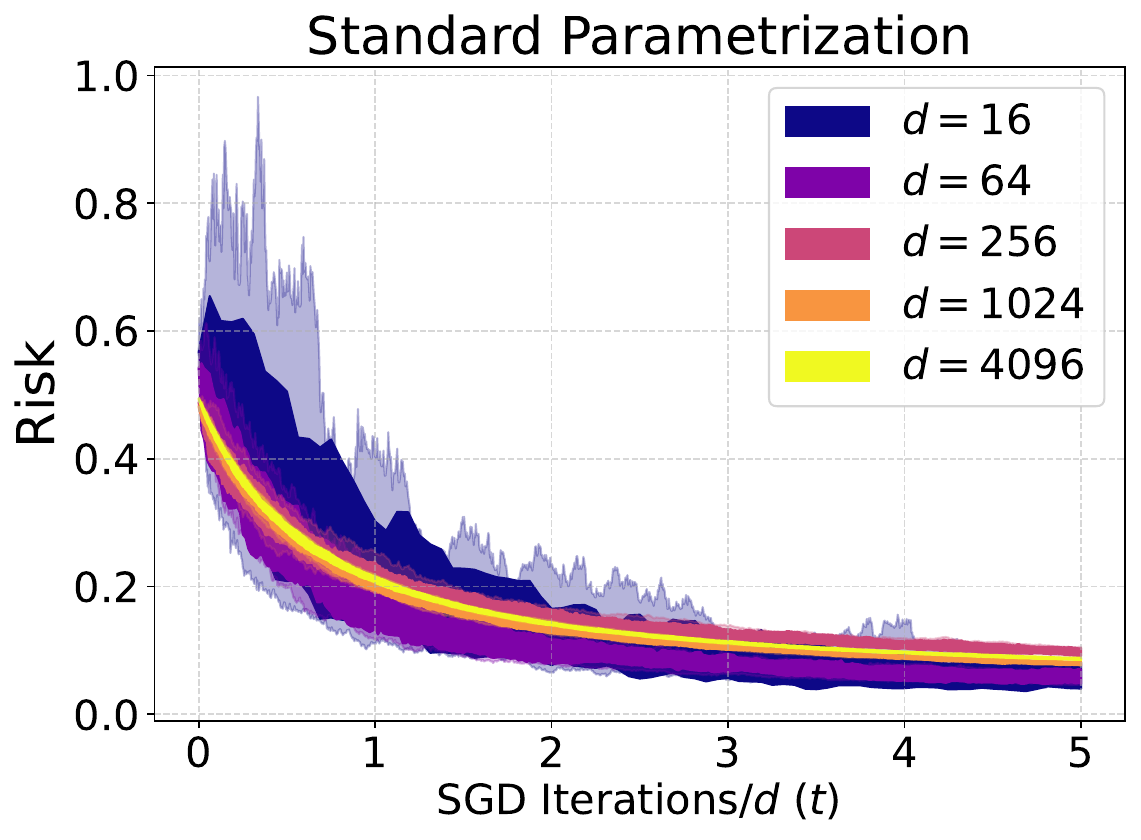} 
    \includegraphics[width=0.45\textwidth]{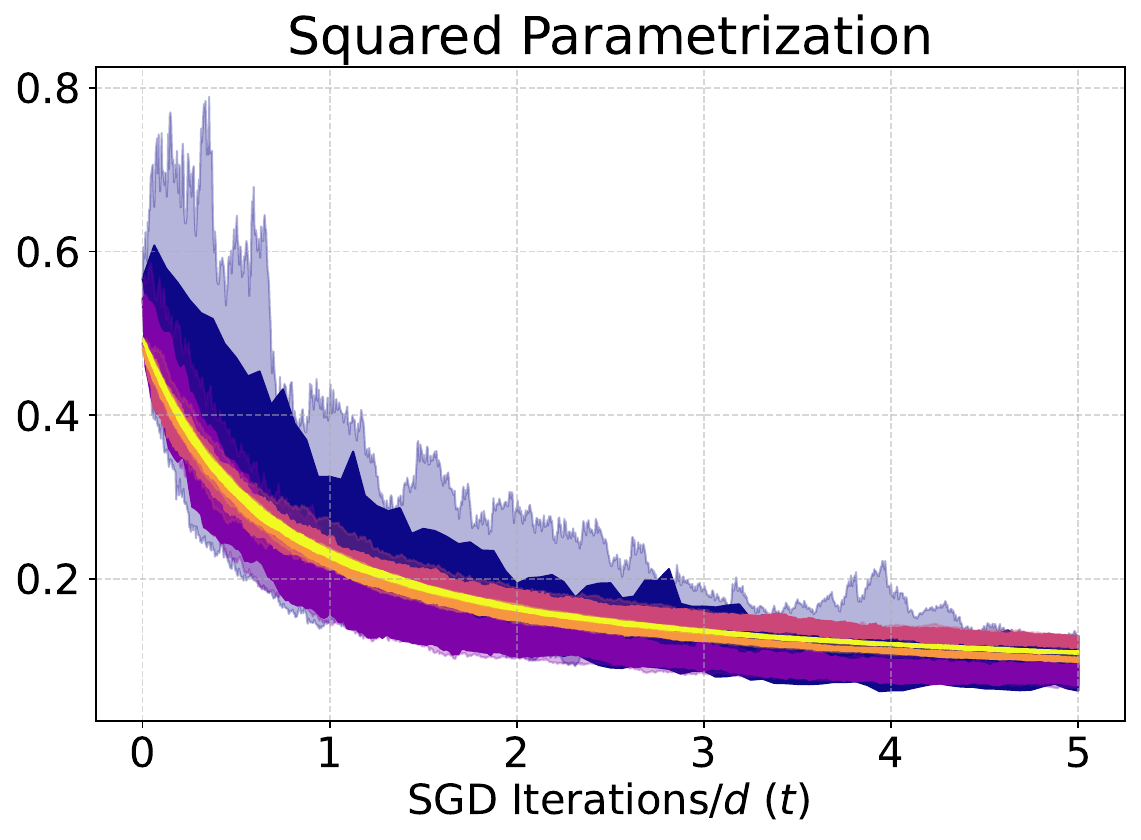} 
    \caption{\textbf{Risk concentration of SGD and the homogenized SDE under non-diagonal covariance on a diagonal linear network.}
    As the dimension $d$ increases, the risk trajectories of SGD (opaque) concentrate around the prediction of the non-diagonal homogenized SDE \eqref{eq:nonDiagonalSquareLossSDE} (transparent), suggesting that the same high-dimensional concentration phenomenon persists beyond the diagonal covariance setting. The covariance matrix $K$ is sampled from a Marchenko--Pastur ensemble. See Appendix~\ref{app: numerics} for simulation details.}
    \label{fig: nondiagK}
\end{figure}
The SDE~\eqref{eq: homogenizedSGD} is useful because it describes the evolution of many quantities of interest by direct It\^o calculus. Concretely, given any sufficiently regular statistic $\varphi$, It\^o's formula gives a decomposition:
\begin{equation}\label{eqn:ito}
    \mathrm{d}\varphi(\mathscr{X}_t)
    = \underbrace{\mathcal{L}\varphi(t,\mathscr{X}_t)}_{\text{drift}}
    \mathrm{d}t
    +
    \underbrace{\nabla\varphi(\mathscr{X}_t)^\top \sigma(t,\mathscr{X}_t)\,\mathrm{d}\mathfrak{B}_t}_{\text{Brownian / martingale term}},
\end{equation}
where $\mathcal{L}$ is the drift–diffusion operator (first‑order drift contribution plus second‑order It\^o correction) and $\sigma$ is the diffusion matrix coefficient of the SDE.

\subsection{Relation to Prior SDE and Homogenization Approximations}
A key conceptual distinction between our framework and classical diffusion (weak-approximation) approaches lies in the regime under which the continuous-time limit becomes accurate. Traditional diffusion approximations  \citep{kushner1978weak, li2019stochastic, malladi2022sdes, compagnoni2025adaptive} operate at fixed dimension with accuracy controlled by the small-stepsize limit $\gamma \to 0$ (after a time rescaling). In that setting, the drift dominates and the stochastic fluctuations vanish, yielding weak convergence of the discrete algorithm to its ODE/SDE formulation.

Our aim is not to provide a black-box homogenization theorem for arbitrary stochastic systems, but to obtain a closed finite-stepsize, high-dimensional description for a nonlinear overparameterized model. The diagonal linear network is a natural minimal target: it is simple enough to admit an exact resolvent closure, but already exhibits the interaction between nonconvex parametrization, stochastic gradients, and finite stepsize effects. Existing small-stepsize diffusion approximations and more general homogenization viewpoints do not directly yield a closed deterministic description for the nonlinear feature map $\psi(x)$ considered here; the closure instead requires tracking a multi-resolvent statistic and leads to the PDE system \eqref{eq: PDESystem}.

In the diagonal linear network setting defined in \eqref{eq: squaredParam}, \cite{pesme2021implicit} follows this diffusion viewpoint and studies SGD via stochastic gradient flow (SGF), a perturbed-gradient-flow SDE whose diffusion term is calibrated so that an Euler discretization matches the covariance of the SGD noise:
\begin{equation}\label{eq: homogenizedSGF}
\mathrm{d} \mathscr{X}_t = - d \nabla \mathcal{R}(\mathscr{X}_t) \mathrm{d}t +  \sqrt{\gamma(t) I\left(B(\mathscr{X}_t) \right)} \, \left(\nabla \psi(\mathscr{X}_t)\right)^\top  \sqrt{K}\, \mathrm{d}\mathfrak{B}_t.
\end{equation}
The two SDEs are closely related in form, but they correspond to different
limiting regimes. Homogenized SGD \eqref{eq: homogenizedSGD} has drift and noise
both proportional to $\gamma(t)$, whereas the covariance-matched diffusion
\eqref{eq: homogenizedSGF} has drift independent of $\gamma(t)$ and noise of
order $\sqrt{\gamma(t)}$. Equivalently, a deterministic time change in
\eqref{eq: homogenizedSGD} produces the algebraic scaling of
\eqref{eq: homogenizedSGF}. However, this does not make the two approximations
equivalent for the discrete-time dynamics studied here: \eqref{eq:
homogenizedSGF} is justified in the small-stepsize regime $\gamma\to0$ at fixed
dimension, while \eqref{eq: homogenizedSGD} is justified in the high-dimensional
regime $d\to\infty$ at fixed stepsize. Thus the two SDEs describe different
approximation limits of SGD. Our approach departs from covariance matching: our diffusion coefficient is derived from a high-dimensional limit theory \citep{collins-woodfin2024hitting}, where approximation accuracy improves as $d \to \infty$ while the stepsize $\gamma$ remains fixed. Here the effective dynamics of the algorithm are controlled by high-dimensional concentration rather than vanishing stepsizes. The resulting SDE model provides an increasingly precise description of the discrete-time risk curves as $d$ grows. Empirically, we find that it tracks SGD more closely than the covariance-matched diffusion model at large stepsizes (see Figure~\ref{fig: sdes}), as our theory predicts. Shrinking $\gamma$ drives the high-dimensional dynamics toward \eqref{eq: homogenizedSGF}, whereas increasing $d$ strengthens the approximation without changing the limiting dynamics.

\begin{figure}[t!]
    \centering
    \includegraphics[width=0.45\textwidth]{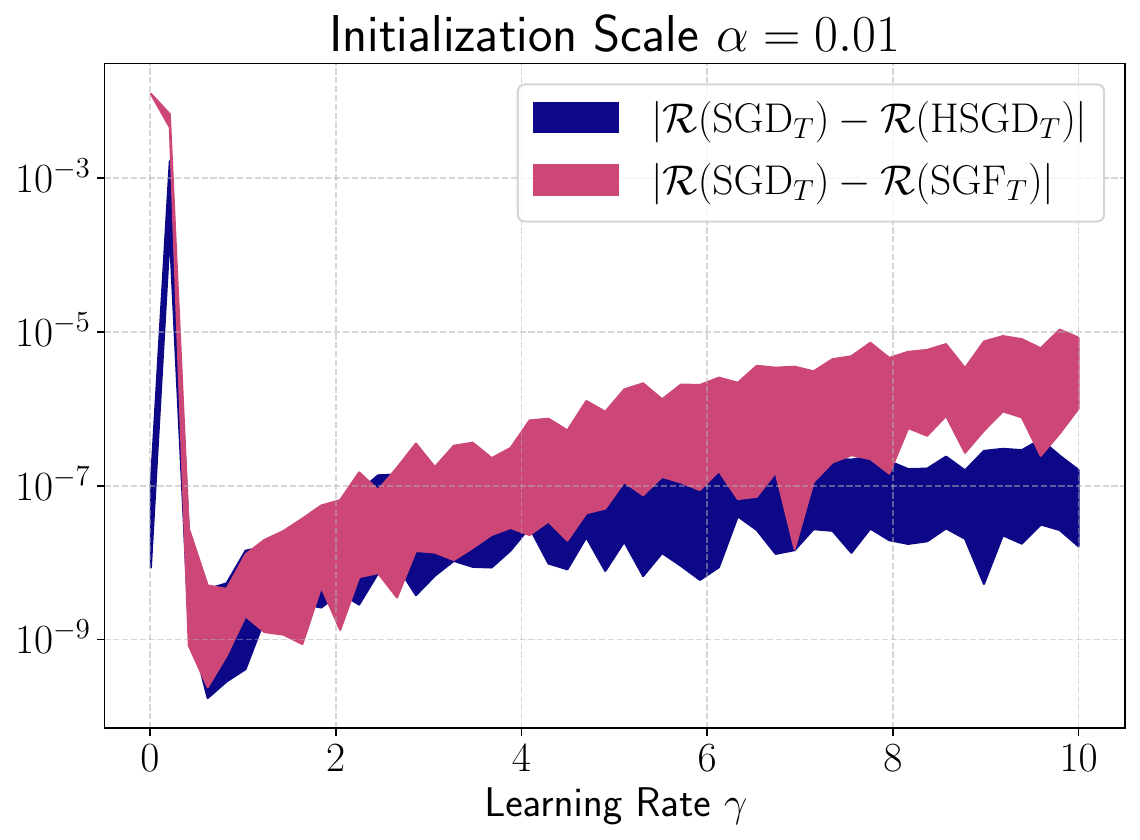} 
    \includegraphics[width=0.45\textwidth]{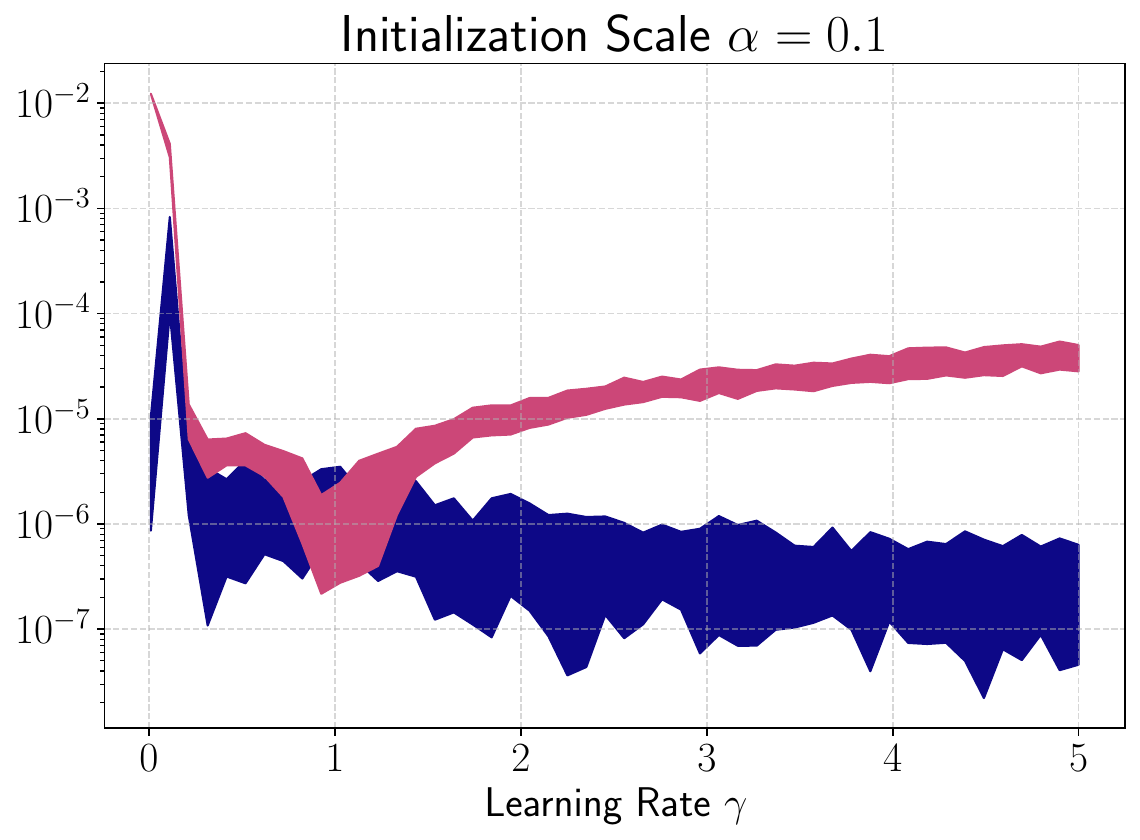} 
    \caption{\textbf{Risk discrepancy between SGD and its continuous-time approximations on a diagonal linear network.}
    For each stepsize $\gamma$, we report the absolute difference between the empirical risk of SGD after $T\!\cdot\! d$ iterations (with $T=20$) and two approximations: (i) homogenized SGD (HSGD) \eqref{eq: homogenizedSGD} (blue), and (ii) stochastic gradient flow (SGF) \eqref{eq: homogenizedSGF} (pink). As $\gamma$ increases, HSGD is a more accurate approximation of SGD, whereas SGF degrades.
    Initialization scale $\alpha$ controls proximity to the saddle point $x=0$: smaller $\alpha$ corresponds to a longer transient before learning accelerates. See Appendix~\ref{app: numerics} for simulation details.}
    \label{fig: sdes}
\end{figure}

\section{High-dimensional Concentration of SGD and its Diffusion Approximation}
\label{sec: concentration}
This section formalizes the main concentration phenomenon: in the high-dimensional regime, both streaming SGD and the homogenized SGD introduced in Section~\ref{sec: sde} concentrate around the same deterministic dynamics. The key point is that, under our structural assumptions, the risk and the quantities that drive the update can be expressed through a small collection of empirical matrix statistics. We encode this information into a $3\times 3$ resolvent-based matrix $S(x,z)$ (indexed by $z$ on a fixed contour $\Gamma$). Knowing the map $z\mapsto S(x,z)$ determines $B(x)$ via a contour integral, and hence determines the learning curves and all other statistics of interest.

Let $z=(z_1,z_2,z_3,z_4)\in\mathbb{C}^4$ be such that
\[
z_1\notin \sigma(\operatorname{diag}(u)),\qquad
z_2\notin \sigma(\operatorname{diag}(v)),\qquad
z_3\notin \sigma(\operatorname{diag}(\beta^*)),\qquad
z_4\notin \sigma(K).
\]
Equivalently, all resolvents appearing below are well-defined. In Remark~\ref{rem: fixedContour}, we specify a product domain on which this condition holds. Define the matrix
\begin{equation}\label{eq: Omega}
\Omega(x,z)
:=
R\bigl(z_1;\operatorname{diag}(u)\bigr)
R\bigl(z_2;\operatorname{diag}(v)\bigr)
R\bigl(z_3;\operatorname{diag}(\beta^*)\bigr)
R\bigl(z_4;K\bigr)
\in \mathbb{C}^{d\times d},
\end{equation}
where $R(z;A):=(zI_d-A)^{-1}$ denotes the resolvent of $A$. We then set
\begin{equation}\label{eq: defS}
S(x,z)
:=
\frac{1}{d}W(x)^\top \Omega(x,z)W(x)
\in \mathbb{C}^{3\times 3},
\end{equation}
where $W(x)$ is the stacked matrix defined in \eqref{eqn:stacked_notation}.
In particular, we evaluate these quantities either along SGD, writing $S(x_k,z)$ and $B(x_k)$, or along the diffusion, writing $S(\mathscr{X}_t,z)$ and $B(\mathscr{X}_t)$.

To control the resolvents uniformly over time (and to justify the fixed contour), we work on the high-probability event that the iterates remain uniformly bounded on the time interval of interest. This non-explosiveness condition is standard in high-dimensional/diffusion-limit arguments and is consistent with our numerical experiments, where we do not observe any such explosions. Moreover, we show in Theorem~\ref{thm:hp_decay_main} that this holds automatically for all stepsizes $\gamma$ below a small (explicit) numerical constant.

\begin{assumption}[Non-explosiveness]\upright\label{ass: nonexplosive}
There exists $M>0$ independent of $d$ such that the initialization satisfies $\|x_0\|_\infty\leqslant M$, and such that, for every fixed $T>0$, with overwhelming probability:
\[
\sup_{0\leqslant k\leqslant \lfloor Td\rfloor}\|x_k\|_\infty \leqslant M
\qquad\text{and}\qquad
\sup_{0\leqslant t\leqslant T}\|\mathscr{X}_t\|_\infty \leqslant M.
\]
\end{assumption}

\begin{remark}\upright
Assumption~\ref{ass: nonexplosive} is stated in a general form. In Appendix~\ref{app:entropy_barrier}, we verify that it holds in the isotropic squared-parametrization setting $\beta^*=\mathds{1}_d$, $K=I_d$, and constant stepsize $\gamma$. More precisely, for sufficiently small $\gamma$, we establish a high-probability exponential decay of the risk. We then show that risk integrability yields uniform-in-time bounds on the coordinates of the solution and prevents them from approaching the origin. By the standard continuation criterion for SDEs, these bounds rule out finite-time explosion, thereby verifying Assumption~\ref{ass: nonexplosive} in this setting.
\end{remark}

\begin{remark}\upright\label{rem: fixedContour}
Under Assumption~\ref{ass: nonexplosive}, the spectra of $\operatorname{diag}(u)$ and $\operatorname{diag}(v)$ remain inside the circle
$C_{2M}(0):=\{z\in \mathbb{C}:\ |z|=2M\}$ 
over the time intervals considered. Throughout the paper, we fix the product contour
$\Gamma=\Gamma_1\times\Gamma_2\times\Gamma_3\times\Gamma_4\subset \mathbb{C}^4$ given by $\Gamma_1 = \Gamma_2 = C_{2M}(0)$ and
\begin{equation*}
\Gamma_3 = \{ z_3 : |z_3| = \max\{1, 2\|\beta^*\|_\infty\}\}, \qquad 
\Gamma_4 = \{ z_4 : |z_4| = \max\{1, 2\|K\|_\mathrm{op}\}\}.
\end{equation*}
The contours are chosen in such a way that each $\Gamma_i$ encloses the relevant spectrum with a fixed positive margin. Thus, by the Cauchy integral formula, $B$ can be recovered from $S$ via $B(x)=\frac{1}{(2\pi)^4}\oint_\Gamma z_4\,S(x,z)\,\mathrm{d}z$, and hence
\begin{equation}\label{eq:B_from_S}
B(x_k) = \frac{1}{(2\pi)^4} 
\oint_{\Gamma}
z_4\, S(x_k, z)\,
\mathrm{d}z, \quad \text{and} \quad
B(\mathscr{X}_t) = \frac{1}{(2\pi)^4} 
\oint_{\Gamma}
z_4\, S(\mathscr{X}_t, z)\,
\mathrm{d}z.
\end{equation}
\end{remark}

In order to derive deterministic dynamics, we pass to a  rescaled continuous-time parameter:
\begin{equation*}
k \text{ iterations of SGD } = \lfloor td \rfloor, \quad \text{ where } t\in \R \text{ is the continuous time parameter.}
\end{equation*}
This time change is natural, since when the size of the problem grows, more SGD iterations are needed to solve the
underlying problem and make equivalent progress. This scaling law ensures all training dynamics live on the same space and are comparable across problem sizes. 

\paragraph{Deterministic Limit Dynamics (PDE).}
Next, for each $z\in\Gamma \subset \mathbb{C}^4$, let $\mathrsfs{S}(t,z)\in\mathbb{C}^{3\times 3}$ denote the (local) solution of the partial integro-differential equation 
\begin{equation}\label{eq:Sequation_clean}
    \partial_t \mathrsfs{S}(t,\cdot)=\mathscr{F}\big(\cdot,\mathrsfs{S}(t,\cdot)\big),
    \qquad
    \mathrsfs{S}(0,z)=S(x_0,z),
\end{equation}
where the function $\mathscr{F}$ is defined in \eqref{eq: PDEFormula} and involves contour integrals around $\Gamma$ and derivatives in $z$ of $\mathrsfs{S}$.
We define the associated deterministic quantities
\begin{equation}\label{eq:deterministic_B_R_I}
    \mathrsfs{B}(t)
    :=\frac{1}{(2\pi)^4}\oint_\Gamma z_4\,\mathrsfs{S}(t,z)\,\mathrm{d}z,
    \qquad
    \mathrsfs{R}(t):=h\left(\mathrsfs{B}(t)\right),
    \qquad
    \mathrsfs{I}(t):=I\left(\mathrsfs{B}(t)\right).
\end{equation}
Since \eqref{eq:Sequation_clean} is a fully nonlinear, nonlocal (integro-differential) parabolic PDE, we do not pursue global well-posedness. 
Throughout, $\mathrsfs{S}(t,z)$ denotes any solution defined up to $\mathrsfs{B}(t)$'s first exit time from the domain of validity of the coefficients $\mathcal{U}$ (or blow-up), whenever such a solution exists. In practice, we solve \eqref{eq:Sequation_clean} numerically using standard discretization methods, and these computed solutions form the basis of the numerical simulations presented in this paper.

The next theorem states the concentration result: both sequences $B(x_{\lfloor td\rfloor})$ (SGD) and $B(\mathscr{X}_t)$ (homogenized SGD) track the same deterministic curve $\mathrsfs{B}(t)$ while the dynamics remain in  $\mathcal{U}$.

\begin{theorem}[Learning Curves]\label{thm: learningCurves}
Suppose that Assumptions~\ref{ass: fPseudoLipschitz},~\ref{ass: data},~\ref{ass: diagonalNeuralNetworks},~\ref{ass: parameterScaling},~\ref{ass: riskRepresentation},
\ref{ass: squaredGradientsPseudoLipschitz},~\ref{ass: learningRate}, and~\ref{ass: nonexplosive} hold.
Let $\mathrsfs{S}(t,z)$ solve \eqref{eq:Sequation_clean} and define $\mathrsfs{B}(t)$ by \eqref{eq:deterministic_B_R_I}.

\smallskip
\noindent\textbf{(i) Streaming SGD.}
Let $\vartheta$ be the first time that either $\mathrsfs{B}(t)$ or $B(x_{\lfloor td\rfloor})$ exit $\mathcal U$.
Then there exists an $\varepsilon>0$ such that for any $T>0$, with overwhelming probability, it holds:
\begin{equation}
\sup_{0\leqslant t \leqslant T\wedge \vartheta} \norm{\mathrsfs{B} (t) - B(x_{\lfloor td \rfloor})} \leqslant d^{-\varepsilon}.
\end{equation}

\smallskip
\noindent\textbf{(ii) Homogenized SDE.}
For any $\eta>0$, define the set
$
\mathcal U_\eta:=\Big\{B\in\mathcal U:\ \inf_{\widehat B\notin\mathcal U}\|B-\widehat B\|\geqslant \eta\Big\},
$
and let $\vartheta$ be the first time that either $\mathrsfs{B}(t)$ or $B(\mathscr{X}_t)$ exit $\mathcal U_\eta$.
Then there exists $\varepsilon>0$ such that for any $T>0$, with overwhelming probability, it holds:
\begin{equation}
\sup_{0\leqslant t \leqslant T\wedge \vartheta} \norm{\mathrsfs{B} (t) - B\left(\mathscr{X}_t\right)} \leqslant d^{-\varepsilon}.
\end{equation}
\end{theorem}

\begin{remark}\upright
Theorem~\ref{thm: learningCurves} extends the concentration framework in
\cite{collins-woodfin2024hitting} to a larger class of risks, covering not only models
such as multi-class logistic regression, but also mean-squared error over
diagonal linear networks. The main additional difficulty is the presence of the
inner nonlinear function $\psi$, through which the data enter as
$\langle a,\psi(x)\rangle$ rather than as a linear inner product
$\langle a,x\rangle$. This nonlinear parametrization leads, after applying
It\^o's formula, to products of multiple resolvents involving $u$, $v$,
$\beta^*$, and $K$. This is precisely the step where the finite-dimensional ODE
closure available in least-squares-type settings breaks down. As a result, the deterministic closure is no longer a
finite-dimensional ODE as in \cite{collins-woodfin2024hitting}, but instead takes the
form of a partial integro-differential equation for the matrix-valued statistic
$\mathrsfs S(t,z)$ with $z\in\mathbb C^4$. Thus the extension is not merely to
more general test functions, but requires handling a different closure
mechanism. The commutativity assumption is what allows these products to be reordered into a fixed multi-resolvent statistic; without it, the It\^o expansion would generate noncommuting resolvent words that are not closed by the statistic $S(x,z)$. To retain tractability, we assume a commutativity structure, which in
our setting leads to the restriction that $K$ is diagonal rather than a general
covariance matrix.
\end{remark}

\subsection{Other Statistics}
Summarizing, we have shown that both $B(x_{\lfloor td\rfloor})$ and $B(\mathscr{X}_t)$ concentrate around a deterministic limit. Now we pass to tracking various statistics $\varphi(x)$ along the SGD and homogenized SGD trajectories and showing that they are governed by the same deterministic dynamics in the large dimensional limit.
The statistics that will satisfy this property are the ones satisfying the following assumption:

\begin{assumption}\upright\label{ass: statistic}
The statistic $\varphi\colon \R^{2d} \to \R$ satisfies a composite structure, 
\begin{equation*}
\varphi(x) = g\left(\frac{1}{d} W^\top  q_1(\operatorname{diag}(u)) q_2(\operatorname{diag}(v)) q_4(K) W\right), 
\end{equation*}
where $g \colon \R^{3 \times 3}\to \R$ is $\alpha$-pseudo-Lipschitz on $\mathcal{U}$ and $q_1, q_2, q_4$ are polynomials. 
\end{assumption}

This assumption covers all the typical statistics, including the risk, curvature, squared norms, and estimation errors that do not explicitly involve the covariance $K$
(e.g., $\|\psi(x)\|_2^2$ and $\|\psi(x)-\beta^*\|_2^2$); $\ell^2$--regularized objectives, and higher-order quantities built from derivatives of the risk. In order to describe the deterministic equivalent of $\varphi$ we require its continuous-time analog:
\begin{equation}\label{eq: detEquivVarphi}
\phi(t) := g\left(\frac{1}{(2\pi)^4} 
\oint_{\Gamma}q_1(z_1) q_2(z_2)
q_4(z_4)\, \mathrsfs{S}(t, z)\,
\mathrm{d}z\right).
\end{equation}
The following theorem shows that $\phi(t)$--a deterministic function of a single variable--governs the evolution of $\varphi$ both along the SGD and the homogenized SGD trajectories. 

\begin{theorem}\label{thm: statistic}
Suppose Assumptions~\ref{ass: fPseudoLipschitz},~\ref{ass: data},~\ref{ass: diagonalNeuralNetworks},~\ref{ass: parameterScaling},~\ref{ass: riskRepresentation},
\ref{ass: squaredGradientsPseudoLipschitz},~\ref{ass: learningRate}, and~\ref{ass: nonexplosive} hold. Suppose further that $\mathcal{U} = \R^{3 \times 3}$. For any function $\varphi$, which satisfies Assumption~\ref{ass: statistic}, and for any $T>0$, and $\varepsilon\in (0,1/2)$ there is a constant $C$ (not depending on $d$) so that with overwhelming probability it holds:
\begin{equation}\label{eq: phifinal}
\sup_{0\leqslant t\leqslant T} 
\biggl(
\abs{\varphi(x_{\lfloor td \rfloor}) - \phi(t)} +\abs{\varphi(\mathscr{X}_{t})  - \phi(t)}
\biggr)
\leqslant Cd^{-\varepsilon}.
\end{equation}
\end{theorem}

\paragraph{Intuition of the proof.}
We introduce the statistic $S(x,z)$ in~\eqref{eq: defS} because the resolvents in $\Omega(x,z)$ allow us to represent the polynomial statistics from Assumption~\ref{ass: statistic} by contour integrals, via Cauchy's integral formula. Thus, rather than tracking each statistic separately, it suffices to prove concentration for the single matrix-valued statistic $S$. The concentration is proved before taking contour integrals: for each time, the
control holds simultaneously for all spectral parameters $z$ on the fixed product
contour from Remark~\ref{rem: fixedContour}. This uniform-in-$z$ control is what
allows the Cauchy integral representations to be applied afterward. The main step is to show that both $S(x_{\lfloor td\rfloor},z)$ and $S(\mathscr X_t,z)$ are approximate solutions of the same deterministic equation, whose solution is denoted by $\mathrsfs S(t,z)$. The commutativity assumptions ensure that the terms produced by the Doob/It\^o expansions can be rewritten in terms of this same statistic $S$, yielding the closed PDE system~\eqref{eq: PDESystem}. Stability of this PDE then implies that $S(x_{\lfloor td\rfloor},z)$ and $S(\mathscr X_t,z)$ both remain close to $\mathrsfs S(t,z)$. Finally, $B(x)$ follows from the contour representation in Remark~\ref{rem: fixedContour}, and any admissible statistic $\varphi$ follows by applying the same contour-integral stability argument.

\begin{remark}\upright 
While the PDE \eqref{eq:Sequation_clean} offers the most complete description
of the high-dimensional dynamics, the SDE \eqref{eq: homogenizedSGD} is often
the most effective tool in practice, both because it is cheaper to simulate
numerically (e.g., via Euler--Maruyama time stepping) and because it lets us
quickly identify the deterministic evolution of admissible statistics
$\varphi$. Applying It\^o's formula to $\varphi(\mathscr X_t)$ gives a drift
term, including the second-order It\^o correction, and a martingale term as in
\eqref{eqn:ito}. The key point is that $\mathcal L\varphi(\mathscr X_t)$ can be
rewritten in terms of the same class of admissible statistics (possibly after
enlarging the class to achieve closure), producing a closed deterministic
system. The Brownian term is not discarded at the level of the full trajectory
$\mathscr X_t$; rather, our concentration estimates show that the resulting
martingale term in the evolution of admissible statistics is negligible in the
high-dimensional limit. Thus the deterministic equations for $\phi(t)$ are
obtained by retaining the full drift of the It\^o expansion, including the
It\^o correction, and controlling the martingale term at the level of
statistics.
\end{remark}

As a concrete application for our result, we show that when the stepsize is below a constant threshold, the SDE iterates remain bounded (hence Assumption~\ref{ass: nonexplosive} holds) and the risk decays exponentially fast to zero. Therefore using Theorem~\ref{thm: statistic} we see that the same is true for SGD iterates themselves.  



\begin{theorem}[High-probability Exponential Decay]
\label{thm:hp_decay_main}
Consider homogenized SGD \eqref{eq: sdeFormula} for diagonal linear networks under the squared parametrization \eqref{eq: squaredParam}, initialized at $\mathscr{U}_{0,i}=\mathscr{V}_{0,i}=1$ for $i=1,\dots,d$. Then for sufficiently small stepsize $\gamma>0$ (below an explicit numerical constant) and for any $\delta\in(0,1)$, there exist constants $C,\mu>0$ (not dependent on $t$ or $d$) such that with probability at least $1-\delta$, it holds:
\begin{equation}
\mathcal{R}(\mathscr{X}_t) \leqslant C e^{-\mu t}
\qquad \text{for all } t\geqslant 0.
\end{equation}
In particular, Assumption~\ref{ass: nonexplosive} holds.
\end{theorem}

Our results also enable the analysis of finer statistics of the dynamics.

\paragraph{Potential further applications.}
Beyond convergence, our concentration framework can be used to analyze additional statistics of the dynamics. A promising interesting example is $\varphi(x)=\tfrac{1}{d}\tr(\nabla^2 \mathcal{R}(x))$, which provides a tractable proxy for the \emph{average local curvature} (and hence “sharpness”) of the landscape at the current iterate. This quantity appears in \cite{ding2024flat}, where flatness predicts good generalization, as well as in studies of the Sharpness-Aware Minimization (SAM) algorithm~\cite{wen2023how}. Tracking both $\mathcal{R}(x)$ and $\varphi(x)$ through their deterministic expressions gives a way to study \emph{progressive sharpening}, a phenomenon that has recently garnered some attention \citep{macdonald2023progressive,yoo2025understanding,liu2026minimalist,kalra2026scalable,cohen2021gradient,jastrzębski2019relation}. Empirically, one often observes transient phases where the risk temporarily increases before converging while the curvature decreases, indicating that the trajectory moves toward flatter regions of the landscape (cf.\ Figure~\ref{fig: statistics}). In this sense,
$\frac{1}{d}\tr(\nabla^2\mathcal{R})$ allows one to quantify how these phases depend on the learning
rate. We see such behavior in Figure~\ref{fig: statistics} and numerically can find the different ranges of stepsizes by using the deterministic equations.  

\paragraph{Conclusion.} We introduce a high-dimensional continuous-time description of discrete-time SGD for diagonal network models. Our framework develops an SDE, which we call homogenized SGD, that accurately tracks training dynamics at practical stepsizes. Our main technical contribution is a concentration argument that shows that both SGD and homogenized SGD converge in an appropriate sense to a deterministic PDE for a wide class of statistics, including risk and curvature. Unlike classical diffusion limits, which require vanishing stepsizes, our approximation becomes exact as the parameter dimension $d \to \infty$, even for large stepsizes. 

An interesting next step, which we plan to pursue, is to leverage these deterministic dynamics to study key behaviors of diagonal linear networks, such as implicit bias and progressive sharpening. The framework we developed here provides a precise characterization of the risk and other statistics, enabling a detailed understanding of how problem-dependent quantities (e.g., average eigenvalue of the data covariance matrix) influence the effective stepsize and the convergence or divergence of SGD. In particular, using the deterministic equivalents developed here, we plan to investigate the progressive sharpening phenomenon, widely observed in practice in neural network training. 

\newpage
\section*{Acknowledgments}
Research of B. Garc\'ia Malaxechebarr\'ia was in part funded by NSF DMS 2023166 (NSF TRIPODS II). C. Paquette was supported by a Discovery Grant from the Natural Science and Engineering Research Council (NSERC) of Canada, NSERC CREATE grant Interdisciplinary Math and Artificial Intelligence Program (INTER-MATH-AI), Google x Mila research grant, Fonds de recherche du Quebec Nature et technologies (FRQNT) NOVA Grant, and CIFAR AI Catalyst Grant. Additionally C. Paquette has a 20\% part-time employment at Google DeepMind. The work of M.\ Fazel is supported in part by awards 
NSF TRIPODS II DMS 2023166, NSF CCF 2212261, NSF CCF 2312775, and the Moorthy Family Professorship at UW. She also works part-time at Amazon Inc.\ as an Amazon Scholar. Research of D. Drusvyatskiy was supported by NSF DMS-2306322,
NSF DMS-2023166, and AFOSR FA9550-24-1-0092 awards.

{
\small
\bibliographystyle{plainnat}
\bibliography{references}
}

\newpage
\appendix

\textbf{Outline of the paper.} The remainder of the article is structured as follows:
\begin{enumerate}
    \item Appendix~\ref{app: notation} collects notation and auxiliary tools used throughout the proofs. It fixes our conventions for complex-valued tensor products, coordinate contractions, and tensor norms; records derivative computations for the special functions $\psi$, $\Psi$, and $S$ appearing in the proof of Theorem~\ref{thm: statistic}; and recalls the concentration and pseudo-Lipschitz estimates used in the probabilistic arguments.
    \item Appendix~\ref{app: dynamicalNexus} develops the main dynamical argument. 
    It introduces the partial integro-differential equation \eqref{eq: PDESystem} 
    and the notion of approximate solutions, proves a stability principle for these 
    solutions, and applies it to the resolvent statistic $S$ along SGD and 
    homogenized SGD. This yields Theorem~\ref{thm: mainResultOneStoppingTime}; the 
    result is then transferred to general statistics satisfying 
    Assumption~\ref{ass: statistic}, yielding Theorem~\ref{thm: mainResultStatistic} 
    and its corollaries.
    \item Appendix~\ref{sec: approxSols} proves that the resolvent statistics
    $t\mapsto S(x_{\lfloor td\rfloor},\cdot)$ and
    $t\mapsto S(\mathscr X_t,\cdot)$, associated respectively with SGD and
    homogenized SGD \eqref{eq: homogenizedSGD}, are approximate solutions of the
    partial integro-differential equation \eqref{eq: PDESystem}. The proof uses
    Doob/It\^o decompositions, a net argument over the fixed contour, and
    martingale and Taylor-error bounds.
    \item Appendix~\ref{app:entropy_barrier} studies the homogenized SDE in the
    isotropic squared-parameterization setting. It introduces an empirical entropy adapted to the coordinatewise dynamics, proves an exact entropy SDE and barrier estimates, and uses an exponential supermartingale argument to obtain high-probability global existence and exponential decay of the risk. The section also records consequences such as risk integrability and uniform separation from the saddle.
    \item Appendix~\ref{app: examples} presents key examples illustrating our concentration risk framework.
    \item Appendix~\ref{app: numerics} provides additional details on the numerical simulations used to produce the figures in the main text.
\end{enumerate}

\tableofcontents

\section{Notation and Preliminaries}\label{app: notation}

This appendix collects notation and auxiliary facts used throughout the proofs. 
We first fix basic notation and conventions for tensor products, coordinate 
contractions, and tensor norms for complex-valued tensors. We then record 
derivative identities for $\psi$, $\Psi$, and $S$, which are used 
in the derivation of the limiting dynamics. Finally, we recall the concentration 
and pseudo-Lipschitz estimates needed for the probabilistic bounds.

\paragraph{Basic notation.} Throughout the paper, $e_i$ denotes the $i$-th canonical coordinate vector, with its dimension inferred from context. Additionally, the matrix $E_{ij}$ is defined as the outer product $e_i \cdot e_j^\top$.

For $x=(u,v)$, recall that each coordinate of the inner function $\psi$ is given by
\begin{equation*}
\psi_i(u,v)
=
\begin{pmatrix}
u_i & v_i
\end{pmatrix}
\mathcal{Q}
\begin{pmatrix}
u_i \\
v_i
\end{pmatrix}
+
l^\top
\begin{pmatrix}
u_i \\
v_i
\end{pmatrix}
+c,
\end{equation*}
where $\mathcal{Q}
=
\begin{pmatrix}
q_{11} & q_{12}\\
q_{12} & q_{22}
\end{pmatrix}
\in \R^{2\times 2}$ is symmetric, $l=(l_1,l_2)^\top\in\R^2$, and $c\in\R$.

Equivalently, we have
\[
\psi_i(u,v)
=
q_{11}u_i^2+2q_{12}u_i v_i+q_{22}v_i^2+l_1u_i+l_2v_i+c.
\]

\subsection{Tensor Products and Contractions}

We briefly fix the tensor-contraction conventions used throughout the proofs.
The optimization variables are real, but several quantities, such as $S(x,z)$,
are complex-valued because of the spectral parameter $z$. Consequently, we will
use complex-valued tensors whose derivative directions are indexed by real
coordinate spaces such as $\R^d$.

We order tensor indices so that derivative or ambient coordinates appear first,
and observable coordinates appear last. For example, since
$S(x,z)\in\mathbb{C}^{3\times 3}$, we identify
\[
\nabla_u S(x,z)
\in
\mathbb{C}^{d\times 3\times 3}
\cong
\mathbb{C}^d\otimes \mathbb{C}^{3\times 3},
\]
where the first index corresponds to the derivative with respect to $u$.

We use $\langle\cdot,\cdot\rangle_{\R^d}$ to denote contraction over
the indicated $d$-dimensional coordinate. This contraction is bilinear and is
defined by summing over the contracted coordinate, with no complex conjugation.
Thus, if
\[
A\in\mathbb{C}^{d\times m_1\times\cdots\times m_k},
\qquad
B\in\mathbb{C}^{d\times n_1\times\cdots\times n_\ell},
\]
then the expression
\[
\langle A,B\rangle_{\R^d}
\in
\mathbb{C}^{m_1\times\cdots\times m_k
\times n_1\times\cdots\times n_\ell}
\]
is defined by
\[
\bigl(\langle A,B\rangle_{\R^d}\bigr)_{\alpha_1,\ldots,\alpha_k,
\beta_1,\ldots,\beta_\ell}
:=
\sum_{i=1}^d
A_{i,\alpha_1,\ldots,\alpha_k}
B_{i,\beta_1,\ldots,\beta_\ell}.
\]
In words, the contracted axes are summed over, and the uncontracted axes of the
first tensor are listed before the uncontracted axes of the second tensor.

For instance, if
\[
\nabla_u S(x,z)\in\mathbb{C}^{d\times 3\times 3},
\qquad
\nabla_u \mathcal R(x)\in\R^d,
\]
then the matrix
\[
\left\langle
\nabla_u S(x,z),
\nabla_u \mathcal R(x)
\right\rangle_{\R^d}
\in\mathbb{C}^{3\times 3}
\]
has entries
\[
\left[
\left\langle
\nabla_u S(x,z),
\nabla_u \mathcal R(x)
\right\rangle_{\R^d}
\right]_{ab}
=
\sum_{i=1}^d
\partial_{u_i}S_{ab}(x,z)\,
\partial_{u_i}\mathcal R(x).
\]

Similarly, if
\[
\nabla_u^2 S(x,z)\in\mathbb{C}^{d\times d\times 3\times 3},
\qquad
M\in\R^{d\times d},
\]
then contraction over the two derivative coordinates gives
\[
\langle \nabla_u^2 S(x,z),M\rangle_{\R^{d\times d}}\in\mathbb{C}^{3\times 3},
\]
with entries
\[
\left[
\langle \nabla_u^2 S(x,z),M\rangle_{\R^{d\times d}}
\right]_{ab}
=
\sum_{i,j=1}^d
\partial_{u_i}\partial_{u_j}S_{ab}(x,z)\,M_{ij}.
\]

We reserve the unsubscripted notation $\langle \cdot,\cdot\rangle$ for the
full Hilbert--Schmidt tensor inner product. Thus, for complex matrices or
tensors of the same shape,
\[
\langle A,B\rangle
:=
\sum_{\alpha}\overline{A_\alpha}B_\alpha,
\]
where $\alpha$ ranges over all tensor indices. In particular, for matrices,
\[
\langle A,B\rangle
=
\operatorname{Tr}(A^*B),
\]
where $A^*=\overline{A}^{\,\top}$ denotes the conjugate transpose. The
corresponding unsubscripted norm is
\[
\|A\|:=\sqrt{\langle A,A\rangle}.
\]

In contrast, a subscript on the bracket indicates a partial contraction over
the specified coordinate space. For example,
$\langle A,B\rangle_{\R^d}$ denotes contraction over the
$\R^d$ coordinate. These partial contractions are bilinear coordinate
contractions used in derivative and generator computations, and no complex
conjugation is applied unless explicitly stated. Thus
$\langle\cdot,\cdot\rangle_{\R^d}$ should be read as a contraction
label, not as a Hermitian inner product on $\mathbb{C}^d$.

\subsection{Norms on Tensors}

We now define the tensor norms used throughout the paper. Unless otherwise
specified, $\|\cdot\|$ denotes the Hilbert-space norm: the Euclidean norm for
vectors and the Hilbert--Schmidt/Frobenius norm for matrices and tensors. As
above, $\langle A,B\rangle$ denotes the full Hilbert--Schmidt inner product, and
\[
\|A\|:=\sqrt{\langle A,A\rangle}
\]
denotes the associated Hilbert-space norm. Thus, for matrices, $\|A\|$ is the
Frobenius norm unless another subscript, such as $\|\cdot\|_{\mathrm{op}}$, is
specified. We use subscripts such as $\|\cdot\|_{\mathrm{op}}$,
$\|\cdot\|_{\sigma}$, and $\|\cdot\|_*$ for the operator, injective, and nuclear
norms, respectively.

For a matrix $A\in\mathbb{C}^{d\times d}$, the operator norm admits the
variational representation
\[
\|A\|_{\mathrm{op}}
=
\sup_{\substack{\|y\|_2=1\\ \|z\|_2=1}}
\bigl|y^*Az\bigr|.
\]

The injective tensor norm is the natural higher-order analogue of this formula.
Let
\[
A\in V_1\otimes V_2\otimes \cdots \otimes V_k.
\]
We define
\[
\|A\|_{\sigma}
:=
\sup_{\substack{\|y_i\|_{V_i}=1\\ i=1,\ldots,k}}
\left|
\left\langle
A,\,
y_1\otimes y_2\otimes \cdots \otimes y_k
\right\rangle
\right|.
\]
Equivalently, $\|A\|_{\sigma}$ is the largest correlation of $A$ with a unit
simple tensor. This norm is also known as the \emph{injective tensor norm}. In
the case $k=2$, it reduces to the usual operator norm, up to the standard
identification of matrices with order-two tensors.

By Cauchy--Schwarz, the Hilbert-space norm also has the variational
representation
\[
\|A\|
=
\sup_{\|B\|\leqslant 1}
\left|\langle A,B\rangle\right|,
\]
where the supremum is over tensors $B$ of the same shape as $A$, and
$\|B\|=\sqrt{\langle B,B\rangle}$ is the Hilbert--Schmidt norm.

Finally, we define the nuclear norm as the dual norm of the injective norm:
\[
\|A\|_{*}
:=
\sup_{\|B\|_{\sigma}\leqslant 1}
\left|
\langle A,B\rangle
\right|.
\]
For order-two tensors this agrees with the usual matrix nuclear norm. For
higher-order tensors, this is the tensor nuclear norm, equivalently the
projective tensor norm.

These norms satisfy the chain of inequalities
\begin{equation}\label{eq: tensornorms}
\|A\|_{\sigma}
\leqslant
\|A\|
\leqslant
\|A\|_{*}.
\end{equation}
Indeed, the first inequality follows because unit simple tensors have
Hilbert--Schmidt norm one. The second follows from the variational formula above
and the fact that $\|B\|_{\sigma}\leq \|B\|$ for every tensor $B$.

\subsection{Derivative Identities for Special Statistics}

We record the derivative identities used in the derivation of the limiting
dynamics. All tensor-valued derivatives are interpreted using the contraction
conventions from the previous subsection.

\begin{lemma}[Jacobian of $\psi$ for diagonal linear networks]\label{lem: derivativepsi}
Let $\psi$ be as in Assumption~\ref{ass: diagonalNeuralNetworks}. Then its
Jacobian with respect to $x=(u,v)$ is
\[
\nabla \psi(u,v)
=
\begin{bmatrix}
\nabla_u \psi(u,v) & \nabla_v \psi(u,v)
\end{bmatrix}
\in \mathbb{R}^{d\times 2d},
\]
where
\[
\nabla_u \psi(u,v)
=
2q_{11}\operatorname{diag}(u)
+
2q_{12}\operatorname{diag}(v)
+
l_1 I_d,
\]
and
\[
\nabla_v \psi(u,v)
=
2q_{12}\operatorname{diag}(u)
+
2q_{22}\operatorname{diag}(v)
+
l_2 I_d.
\]
\end{lemma}

\begin{lemma}[Gradient of $\Psi$]\label{lem: derivativePsi}
Let
\[
r
:=
\frac{1}{\sqrt d}
\begin{pmatrix}
\psi(x)\\
\beta^*
\end{pmatrix}^{\!\top}a \in \R^2,
\]
and define $\Psi(x;a)
:=
f (r)$.
Then we hav the expression
\[
\nabla_x\Psi(x;a)
=
\frac{1}{\sqrt d}\,
\nabla_{r_1}f(r)\,
(\nabla\psi(x))^\top a
\in\R^{2d},
\]
where $\nabla_{r_1}f$ denotes the partial derivative of $f$ with respect to its
first coordinate.
\end{lemma}

\begin{lemma}[Derivative identities for $S$]\label{lem: derivativeS}
Fix the product contour $\Gamma \subset \mathbb{C}^4$ from
Remark~\ref{rem: fixedContour}, and let
$z=(z_1,z_2,z_3,z_4)\in\Gamma$. We define
\[
\Omega(x,z)
:=
R\bigl(z_1;\operatorname{diag}(u)\bigr)
R\bigl(z_2;\operatorname{diag}(v)\bigr)
R\bigl(z_3;\operatorname{diag}(\beta^*)\bigr)
R\bigl(z_4;K\bigr)
\in \mathbb{C}^{d\times d},
\]
where
\[
R(z;A):=(zI_d-A)^{-1}
\]
denotes the resolvent of $A$. Since $K$ is diagonal in our setting, all factors
in $\Omega(x,z)$ are diagonal; in particular, $\Omega(x,z)^\top=\Omega(x,z)$.

Let us consider
\[
W(x):=[\,\psi(x)\mid \beta^*\mid \mathds{1}_d\,]
\in\R^{d\times 3},
\]
and set the matrix
\[
S(x,z)
:=
\frac{1}{d}W(x)^\top\Omega(x,z)W(x)
\in\mathbb{C}^{3\times 3}.
\]
Then
\[
\nabla_u S(x,z),\nabla_v S(x,z)\in\mathbb{C}^{d\times 3\times 3},
\]
and
\[
\nabla_u^2 S(x,z),\nabla_{uv}^2 S(x,z),\nabla_v^2 S(x,z)
\in\mathbb{C}^{d\times d\times 3\times 3}.
\]

In the following displays, we write $\Omega$ for $\Omega(x,z)$. Products
involving block matrices are interpreted as tensor contractions over the
$d$-dimensional coordinate: a $3\times 3$ block matrix with vector-valued
entries is identified with an element of $\mathbb{C}^{d\times 3\times 3}$, and
a $3\times 3$ block matrix with matrix-valued entries is identified with an
element of $\mathbb{C}^{d\times d\times 3\times 3}$.

The first derivatives of $S$ have the following form
\allowdisplaybreaks
\begin{align*}
\nabla_u S(x,z)
&=
\frac{1}{d}\nabla_u\psi(x)\Omega
\begin{bmatrix}
\psi(x) & \beta^* & \mathds{1}_d\\
0 & 0 & 0\\
0 & 0 & 0
\end{bmatrix}
+
\frac{1}{d}\nabla_u\psi(x)\Omega
\begin{bmatrix}
\psi(x) & 0 & 0\\
\beta^* & 0 & 0\\
\mathds{1}_d & 0 & 0
\end{bmatrix}
\\[10pt]
&\quad+
\frac{1}{d}\operatorname{diag}(\psi(x))
R\bigl(z_1;\operatorname{diag}(u)\bigr)\Omega
\begin{bmatrix}
\psi(x) & \beta^* & \mathds{1}_d\\
0 & 0 & 0\\
0 & 0 & 0
\end{bmatrix}
\\[10pt]
&\quad+
\frac{1}{d}\operatorname{diag}(\beta^*)
R\bigl(z_1;\operatorname{diag}(u)\bigr)\Omega
\begin{bmatrix}
0 & 0 & 0\\
\psi(x) & \beta^* & \mathds{1}_d\\
0 & 0 & 0
\end{bmatrix}
\\[10pt]
&\quad+
\frac{1}{d}
R\bigl(z_1;\operatorname{diag}(u)\bigr)\Omega
\begin{bmatrix}
0 & 0 & 0\\
0 & 0 & 0\\
\psi(x) & \beta^* & \mathds{1}_d
\end{bmatrix}
\in\mathbb{C}^{d\times 3\times 3},
\\[20pt]
\nabla_v S(x,z)
&=
\frac{1}{d}\nabla_v\psi(x)\Omega
\begin{bmatrix}
\psi(x) & \beta^* & \mathds{1}_d\\
0 & 0 & 0\\
0 & 0 & 0
\end{bmatrix}
+
\frac{1}{d}\nabla_v\psi(x)\Omega
\begin{bmatrix}
\psi(x) & 0 & 0\\
\beta^* & 0 & 0\\
\mathds{1}_d & 0 & 0
\end{bmatrix}
\\[10pt]
&\quad+
\frac{1}{d}\operatorname{diag}(\psi(x))
R\bigl(z_2;\operatorname{diag}(v)\bigr)\Omega
\begin{bmatrix}
\psi(x) & \beta^* & \mathds{1}_d\\
0 & 0 & 0\\
0 & 0 & 0
\end{bmatrix}
\\[10pt]
&\quad+
\frac{1}{d}\operatorname{diag}(\beta^*)
R\bigl(z_2;\operatorname{diag}(v)\bigr)\Omega
\begin{bmatrix}
0 & 0 & 0\\
\psi(x) & \beta^* & \mathds{1}_d\\
0 & 0 & 0
\end{bmatrix}
\\[10pt]
&\quad+
\frac{1}{d}
R\bigl(z_2;\operatorname{diag}(v)\bigr)\Omega
\begin{bmatrix}
0 & 0 & 0\\
0 & 0 & 0\\
\psi(x) & \beta^* & \mathds{1}_d
\end{bmatrix}
\in\mathbb{C}^{d\times 3\times 3}.
\end{align*}

Moreover, the second derivatives of $S$ are
\allowdisplaybreaks
\begin{align*}
\nabla_u^2 S(x,z)
&=
\frac{2q_{11}}{d}\Omega
\begin{bmatrix}
\operatorname{diag}(\psi(x)) & \operatorname{diag}(\beta^*) & I_d\\
0 & 0 & 0\\
0 & 0 & 0
\end{bmatrix}
+
\frac{2q_{11}}{d}\Omega
\begin{bmatrix}
\operatorname{diag}(\psi(x)) & 0 & 0\\
\operatorname{diag}(\beta^*) & 0 & 0\\
I_d & 0 & 0
\end{bmatrix}
\\[10pt]
&\quad+
\frac{2}{d}\Omega
\begin{bmatrix}
(\nabla_u\psi(x))^2 & 0 & 0\\
0 & 0 & 0\\
0 & 0 & 0
\end{bmatrix}
\\[10pt]
&\quad+
\frac{2}{d}\nabla_u\psi(x)
R\bigl(z_1;\operatorname{diag}(u)\bigr)\Omega
\begin{bmatrix}
\operatorname{diag}(\psi(x)) & \operatorname{diag}(\beta^*) & I_d\\
0 & 0 & 0\\
0 & 0 & 0
\end{bmatrix}
\\[10pt]
&\quad+
\frac{2}{d}\nabla_u\psi(x)
R\bigl(z_1;\operatorname{diag}(u)\bigr)\Omega
\begin{bmatrix}
\operatorname{diag}(\psi(x)) & 0 & 0\\
\operatorname{diag}(\beta^*) & 0 & 0\\
I_d & 0 & 0
\end{bmatrix}
\\[10pt]
&\quad+
\frac{2}{d}\operatorname{diag}(\psi(x))
R\bigl(z_1;\operatorname{diag}(u)\bigr)^2\Omega
\begin{bmatrix}
\operatorname{diag}(\psi(x)) & \operatorname{diag}(\beta^*) & I_d\\
0 & 0 & 0\\
0 & 0 & 0
\end{bmatrix}
\\[10pt]
&\quad+
\frac{2}{d}\operatorname{diag}(\beta^*)
R\bigl(z_1;\operatorname{diag}(u)\bigr)^2\Omega
\begin{bmatrix}
0 & 0 & 0\\
\operatorname{diag}(\psi(x)) & \operatorname{diag}(\beta^*) & I_d\\
0 & 0 & 0
\end{bmatrix}
\\[10pt]
&\quad+
\frac{2}{d}
R\bigl(z_1;\operatorname{diag}(u)\bigr)^2\Omega
\begin{bmatrix}
0 & 0 & 0\\
0 & 0 & 0\\
\operatorname{diag}(\psi(x)) & \operatorname{diag}(\beta^*) & I_d
\end{bmatrix}
\in\mathbb{C}^{d\times d\times 3\times 3},
\\[20pt]
\nabla_{uv}^2 S(x,z)
&=
\frac{2q_{12}}{d}\Omega
\begin{bmatrix}
\operatorname{diag}(\psi(x)) & \operatorname{diag}(\beta^*) & I_d\\
0 & 0 & 0\\
0 & 0 & 0
\end{bmatrix}
+
\frac{2q_{12}}{d}\Omega
\begin{bmatrix}
\operatorname{diag}(\psi(x)) & 0 & 0\\
\operatorname{diag}(\beta^*) & 0 & 0\\
I_d & 0 & 0
\end{bmatrix}
\\[10pt]
&\quad+
\frac{2}{d}\Omega
\begin{bmatrix}
(\nabla_u\psi(x))(\nabla_v\psi(x)) & 0 & 0\\
0 & 0 & 0\\
0 & 0 & 0
\end{bmatrix}
\\[10pt]
&\quad+
\frac{1}{d}\nabla_v\psi(x)
R\bigl(z_1;\operatorname{diag}(u)\bigr)\Omega
\begin{bmatrix}
\operatorname{diag}(\psi(x)) & \operatorname{diag}(\beta^*) & I_d\\
0 & 0 & 0\\
0 & 0 & 0
\end{bmatrix}
\\[10pt]
&\quad+
\frac{1}{d}\nabla_v\psi(x)
R\bigl(z_1;\operatorname{diag}(u)\bigr)\Omega
\begin{bmatrix}
\operatorname{diag}(\psi(x)) & 0 & 0\\
\operatorname{diag}(\beta^*) & 0 & 0\\
I_d & 0 & 0
\end{bmatrix}
\\[10pt]
&\quad+
\frac{1}{d}\nabla_u\psi(x)
R\bigl(z_2;\operatorname{diag}(v)\bigr)\Omega
\begin{bmatrix}
\operatorname{diag}(\psi(x)) & \operatorname{diag}(\beta^*) & I_d\\
0 & 0 & 0\\
0 & 0 & 0
\end{bmatrix}
\\[10pt]
&\quad+
\frac{1}{d}\nabla_u\psi(x)
R\bigl(z_2;\operatorname{diag}(v)\bigr)\Omega
\begin{bmatrix}
\operatorname{diag}(\psi(x)) & 0 & 0\\
\operatorname{diag}(\beta^*) & 0 & 0\\
I_d & 0 & 0
\end{bmatrix}
\\[10pt]
&\quad+
\frac{1}{d}\operatorname{diag}(\psi(x))
R\bigl(z_1;\operatorname{diag}(u)\bigr)
R\bigl(z_2;\operatorname{diag}(v)\bigr)\Omega
\begin{bmatrix}
\operatorname{diag}(\psi(x)) & \operatorname{diag}(\beta^*) & I_d\\
0 & 0 & 0\\
0 & 0 & 0
\end{bmatrix}
\\[10pt]
&\quad+
\frac{1}{d}\operatorname{diag}(\beta^*)
R\bigl(z_1;\operatorname{diag}(u)\bigr)
R\bigl(z_2;\operatorname{diag}(v)\bigr)\Omega
\begin{bmatrix}
0 & 0 & 0\\
\operatorname{diag}(\psi(x)) & \operatorname{diag}(\beta^*) & I_d\\
0 & 0 & 0
\end{bmatrix}
\\[10pt]
&\quad+
\frac{1}{d}
R\bigl(z_1;\operatorname{diag}(u)\bigr)
R\bigl(z_2;\operatorname{diag}(v)\bigr)\Omega
\begin{bmatrix}
0 & 0 & 0\\
0 & 0 & 0\\
\operatorname{diag}(\psi(x)) & \operatorname{diag}(\beta^*) & I_d
\end{bmatrix}
\in\mathbb{C}^{d\times d\times 3\times 3},
\\[20pt]
\nabla_v^2 S(x,z)
&=
\frac{2q_{22}}{d}\Omega
\begin{bmatrix}
\operatorname{diag}(\psi(x)) & \operatorname{diag}(\beta^*) & I_d\\
0 & 0 & 0\\
0 & 0 & 0
\end{bmatrix}
+
\frac{2q_{22}}{d}\Omega
\begin{bmatrix}
\operatorname{diag}(\psi(x)) & 0 & 0\\
\operatorname{diag}(\beta^*) & 0 & 0\\
I_d & 0 & 0
\end{bmatrix}
\\[10pt]
&\quad+
\frac{2}{d}\Omega
\begin{bmatrix}
(\nabla_v\psi(x))^2 & 0 & 0\\
0 & 0 & 0\\
0 & 0 & 0
\end{bmatrix}
\\[10pt]
&\quad+
\frac{2}{d}\nabla_v\psi(x)
R\bigl(z_2;\operatorname{diag}(v)\bigr)\Omega
\begin{bmatrix}
\operatorname{diag}(\psi(x)) & \operatorname{diag}(\beta^*) & I_d\\
0 & 0 & 0\\
0 & 0 & 0
\end{bmatrix}
\\[10pt]
&\quad+
\frac{2}{d}\nabla_v\psi(x)
R\bigl(z_2;\operatorname{diag}(v)\bigr)\Omega
\begin{bmatrix}
\operatorname{diag}(\psi(x)) & 0 & 0\\
\operatorname{diag}(\beta^*) & 0 & 0\\
I_d & 0 & 0
\end{bmatrix}
\\[10pt]
&\quad+
\frac{2}{d}\operatorname{diag}(\psi(x))
R\bigl(z_2;\operatorname{diag}(v)\bigr)^2\Omega
\begin{bmatrix}
\operatorname{diag}(\psi(x)) & \operatorname{diag}(\beta^*) & I_d\\
0 & 0 & 0\\
0 & 0 & 0
\end{bmatrix}
\\[10pt]
&\quad+
\frac{2}{d}\operatorname{diag}(\beta^*)
R\bigl(z_2;\operatorname{diag}(v)\bigr)^2\Omega
\begin{bmatrix}
0 & 0 & 0\\
\operatorname{diag}(\psi(x)) & \operatorname{diag}(\beta^*) & I_d\\
0 & 0 & 0
\end{bmatrix}
\\[10pt]
&\quad+
\frac{2}{d}
R\bigl(z_2;\operatorname{diag}(v)\bigr)^2\Omega
\begin{bmatrix}
0 & 0 & 0\\
0 & 0 & 0\\
\operatorname{diag}(\psi(x)) & \operatorname{diag}(\beta^*) & I_d
\end{bmatrix}
\in\mathbb{C}^{d\times d\times 3\times 3}.
\end{align*}
\end{lemma}

\subsection{Concentration and Pseudo-Lipschitz Estimates}

We use the subgaussian norm $\|\cdot\|_{\psi_2}$, which is equivalent up to
universal constants to the optimal variance proxy in a Gaussian tail bound.
Namely, for a real-valued random variable $X$,
\begin{equation}\label{eq: subgaussian_norm}
\norm{X}_{\psi_2}
\asymp
\inf\left\{
V>0:
\PP\bigl(|X|>t\bigr)\leqslant 2e^{-t^2/V^2}
\text{ for all } t>0
\right\}.
\end{equation}
Gaussian random variables are naturally subgaussian. Moreover, Gaussian measures satisfy
the stronger property of dimension-free \emph{Lipschitz concentration}, which gives cocnentration inequalities for nonlinear functions of Gaussian vectors. Specifically,
let $V_0$ be a finite-dimensional Hilbert space, and let $Z$ be a centered
isotropic Gaussian vector in $V_0$. If $g:V_0\to\mathbb{R}$ is Lipschitz with
constant $L(g)$, meaning that
\[
|g(x)-g(y)|
\le
L(g)\|x-y\|_{V_0}
\qquad
\text{for all } x,y\in V_0,
\]
then we have
\[
\norm{g(Z)-\EE g(Z)}_{\psi_2}
\le
C L(f),
\]
where $C>0$ is an absolute universal constant. In particular, the bound does not depend
on the dimension of $V_0$.

\paragraph{Pseudo-Lipschitz functions.}
In our setting, we will also work with functions which are not quite Lipschitz, in that they are locally Lipschitz (Lipschitz on compact sets) and moreover have polynomial growth of their Lipschitz on norm-balls.

\begin{definition}[Constant Pseudo-Lipschitz functions]\upright
A function $f:V_0\to V_1$ is called pseudo-Lipschitz of
order $\alpha$ if there exists a constant $L=L(\alpha,f)$ such that, for all
$x,y\in V_0$, we have
\[
\|f(x)-f(y)\|_{V_1}
\le
L\|x-y\|_{V_0}
\left(1+\|x\|_{V_0}^{\alpha}+\|y\|_{V_0}^{\alpha}\right).
\]
We call $L$ an $\alpha$-pseudo-Lipschitz constant for $f$.
\end{definition}

We will often works with outer functions and statistics whose gradients are $\alpha$-pseudo-Lipschitz. To reduce pseudo-Lipschitz estimates to Lipschitz estimates on bounded sets, we
will use projection onto norm balls. For $\beta>0$, we define the \textit{projection operator onto the ball of radius $\beta$}, denoted $\operatorname{Proj}_{\beta}:V_0\to V_0$, by
\begin{equation*}
\operatorname{Proj}_{\beta}(x)
:=
\argmin_{y\in \beta\mathbb{B}}
\|x-y\|_{V_0}^2
=
\begin{cases}
x, & \text{if } \|x\|_{V_0}\leqslant \beta,\\
\beta \left ( \frac{x}{\norm{x}_{V_0}} \right ), & \text{otherwise}.
\end{cases},
\end{equation*}
where $\mathbb{B}$ denotes the unit ball in $V_0$.

\begin{lemma}\label{lem:pseudoLipProjection}
Suppose $f:V_0\to V_1$ is $\alpha$-pseudo-Lipschitz with constant $L$. Then the composition 
$f\circ\operatorname{Proj}_{\beta}$ is Lipschitz with constant
$L(1+2\beta^\alpha)$.
\end{lemma}

\begin{proof}
Projection onto a closed convex set is $1$-Lipschitz. Therefore, for all
$x,y\in V_0$, it holds
\begin{align*}
\bigl\|
(f\circ\operatorname{Proj}_{\beta})(x)
-
(f\circ\operatorname{Proj}_{\beta})(y)
\bigr\|_{V_1}
&\le
L\,
\bigl\|
\operatorname{Proj}_{\beta}(x)
-
\operatorname{Proj}_{\beta}(y)
\bigr\|_{V_0}
\\
&\quad \times
\left(
1+
\|\operatorname{Proj}_{\beta}(x)\|_{V_0}^{\alpha}
+
\|\operatorname{Proj}_{\beta}(y)\|_{V_0}^{\alpha}
\right)
\\
&\le
L(1+2\beta^\alpha)\|x-y\|_{V_0},
\end{align*}
as we had to show.
\end{proof}

We will use the following growth estimate for moments of the first partial derivative of
the $\alpha$-pseudo-Lipschitz outer function $f$; see \citep{collins-woodfin2024hitting}.

\begin{lemma}[Growth of $\nabla_{r_1} f$]\label{lem:growthGradf}
Suppose $f:\mathbb{R}^2\to\mathbb{R}$ is $\alpha$-pseudo-Lipschitz with
pseudo-Lipschitz constant $L(f)$, as in Assumption~\ref{ass: fPseudoLipschitz}.
Then, for any $p>0$ and $r\in\mathbb{R}^2$, we have
\begin{equation}\label{eq: growthGradf}
\bigl|\nabla_{r_1}f(r)\bigr|^p
\le
C(\alpha,p)\,L(f)^p
\left(1+\|r\|\right)^{\max\{1,\alpha p\}},
\end{equation}
where $\nabla_{r_1}f$ denotes the partial derivative of $f$ with respect to its
first coordinate.

Moreover, set $r = \frac{1}{\sqrt d}
\begin{bmatrix}
\psi(x) & \beta^*
\end{bmatrix}^{\!\top}a
\in\mathbb{R}^2$, and let
\[
W(x):=[\,\psi(x)\mid \beta^*\mid \mathds{1}_d\,]\in\mathbb{R}^{d\times 3}.
\]
Then the following moment bound holds:
\begin{align}\label{eq: expGrowthGradf}
\begin{split}
\EE_a\left[
\bigl|\nabla_{r_1}f(r)\bigr|^p
\right]
&\le
C(\alpha,p)\,L(f)^p
\left(
1+
\frac{1}{\sqrt d}
\|K\|_{\mathrm{op}}^{1/2}\|W(x)\|
\right)^{\max\{1,\alpha p\}},
\\[8pt]
\norm{1+\|r\|}_{\psi_2}
&\le
C\left(
1+
\frac{1}{\sqrt d}
\|K\|_{\mathrm{op}}^{1/2}\|W(x)\|
\right).
\end{split}
\end{align}
\end{lemma}

\section{Dynamics of the Resolvent Statistic}\label{app: dynamicalNexus}

The purpose of this section is to identify the statistic that mediates between
the stochastic dynamics and their deterministic limit. Recall that our goal is
to prove that, for every statistic
$\varphi:\mathbb{R}^{2d}\to\mathbb{R}$ satisfying
Assumption~\ref{ass: statistic}, the processes
$\varphi(x_{\lfloor td\rfloor})$ and $\varphi(\mathscr{X}_t)$ are close in the sense that they converge to the same deterministic limit. 

The central object in this comparison is the matrix-valued statistic
\[
S(x,z)
:=
\frac{1}{d}W(x)^\top \Omega(x,z)W(x)
\in\mathbb{C}^{3\times 3},
\]
where
\[
W(x):=[\,\psi(x)\mid \beta^*\mid \mathds{1}_d\,]\in\mathbb{R}^{d\times 3},
\]
and $\Omega(x,z)$ is the resolvent product defined in \eqref{eq: Omega}. We
study this statistic along both trajectories, namely $S(\mathscr{X}_t,z)$ (homogenized SGD updates) and $S(x_{\lfloor td\rfloor},z)$ (SGD updates), for $z$ on the fixed product contour of Remark~\ref{rem: fixedContour}.

The argument has two main components. First, we show that the discrete and
homogenized dynamics are close when tested against $S$, in the sense that
$S(x_{\lfloor td\rfloor},z)$ and $S(\mathscr{X}_t,z)$ remain close uniformly
over the relevant time interval. Second, we show that $S(\mathscr{X}_t,z)$ is
itself close to a deterministic limit $\mathrsfs{S}(t,z)$, where
$\mathrsfs{S}$ solves the partial integro-differential equation
\eqref{eq: PDESystem}. Combining these two steps yields a deterministic
description of $S(x_{\lfloor td\rfloor},z)$.

This statistic is powerful because it encodes enough spectral information to
recover the deterministic limits of the broader class of statistics
$\varphi$ considered in Assumption~\ref{ass: statistic}. We make this reduction
explicit in Section~\ref{subsec: mainArgumentProof}. Thus, the dynamics of
$S(x,z)$ serve as the dynamical nexus of the proof: once $S$ is controlled, the
limiting behavior of the other admissible statistics follows by the contour
representations developed below. Beyond this, the dynamics of the mapping $S(x,z)$ itself often provide useful insights into analyzing the optimization trajectories of particular optimization problems. Indeed, properties of the solutions to which the algorithms converge can be derived by looking at the mapping $S(x,z)$.

\subsection{Approximate Solutions and Stability}
We begin by recalling the quantities entering the partial integro-differential equation. By Assumptions~\ref{ass: riskRepresentation} and~\ref{ass: squaredGradientsPseudoLipschitz}, we have
\begin{equation*}
\mathcal{R}(x) :=  h\left(B(x)\right) \; \text{ and } \; \EE_a [\nabla_{r_1} f(r)^2] := I\left(B(x)\right) \; \text{ with } \; B(x) = \frac{1}{d} W(x)^\top K W(x),
\end{equation*}
where $h,I:\mathbb{R}^{3\times 3}\to\mathbb{R}$ are differentiable and $\alpha$-pseudo-Lipschitz. It will be useful to isolate the first-column components of the gradient of $h$. We therefore define
\begin{equation*}
H\left(B(x)\right) = \left[
\begin{array}{c|c|c}
\nabla_{11} h\left(B(x)\right) & 0 & 0 \\ \hline
\nabla_{21} h\left(B(x)\right) & 0 & 0\\ \hline
\nabla_{31} h\left(B(x)\right) & 0 & 0
\end{array}\right].
\end{equation*}
With this notation in place, we now introduce the partial integro-differential equation. In what follows, we will work with a function $\mathscr{F}\bigl(z, \mathscr{S}(t,\cdot)\bigr)$. Its explicit formula can be found in the displayed lines after Remark~\ref{rem: matrixEntriesV}, but will not be important for what follows; the only feature that matters is that it can be written as a sum of simple terms described below.
\todo[inline]{\textbf{Partial Integro-differential Equation for $\mathrsfs{S}(t,z)$.} Fix a product contour $\Gamma \subset \mathbb{C}^4$ as in Remark~\ref{rem: fixedContour}. For a multi-index $\zeta=(\zeta_1,\zeta_2,\zeta_3,\zeta_4)$ and a complex vector 
$z=(z_1,z_2,z_3,z_4)\in\Gamma$, write
\[
z^\zeta := \prod_{i=1}^4 z_i^{\zeta_i}.
\]
Let $\mathscr{F}\bigl(z, \mathscr{S}(t,\cdot)\bigr)$ be a finite sum of terms of the form
\begin{align}\label{eq: PDEFormula}
\begin{split}
&- 2 \gamma(t)\, H\bigl(\mathscr{B}(t)\bigr)^\top 
z^\zeta \left(\bigcirc_{i=1}^4 \mathscr{G}_i\right)
\bigl(\mathscr{S}(t,\cdot)\bigr) \\[5pt]
&\quad \text{or} \\[5pt]
& C(Q)\,\gamma(t)^2\, I\bigl(\mathscr{B}(t)\bigr) 
z^\zeta \left(\bigcirc_{i=1}^4 \mathscr{G}_i\right)
\bigl(\mathscr{S}(t,\cdot)\bigr),
\end{split}
\end{align}
where $\bigcirc_{i=1}^4 \mathscr{G}_i$ denotes successive composition,
\[
\left(\bigcirc_{i=1}^4 \mathscr{G}_i\right)
\bigl(\mathscr{S}(t,\cdot)\bigr)
:= \left(\mathscr{G}_1 \circ \mathscr{G}_2 \circ \mathscr{G}_3 \circ \mathscr{G}_4\right)
\bigl(\mathscr{S}(t,\cdot)\bigr).
\]
Here each $\mathscr{G}_i$ acts only on the $z_i$-coordinate and is chosen from the following three operators:
\begin{equation*}
\mathscr{G}_i\bigl(\mathscr{S}(t,\cdot)\bigr)
\in \left\{
\frac{1}{2\pi i} \oint_{\Gamma_i} q(w)\,\mathscr{S}(t,z)\,\mathrm{d}w,\;
- \frac{\mathrm{d}}{\mathrm{d}z_i}\mathscr{S}(t,z),\;
\frac{1}{2}\frac{\mathrm{d}^2}{\mathrm{d}z_i^2}\mathscr{S}(t,z)
\right\}.
\end{equation*}
Finally,
\begin{equation*}
\mathscr{B}(t) := \frac{1}{(2\pi)^4} \oint_{\Gamma} 
z_4\, \mathscr{S}(t,z)\, \mathrm{d}z.
\end{equation*}
We then define $\mathrsfs{S}$ as a solution to the evolution equation
\begin{equation}\label{eq:  PDESystem}
\mathrm{d}\mathrsfs{S}(t,\cdot)
= \mathscr{F}\bigl(z,\mathrsfs{S}(t,\cdot)\bigr)\,\mathrm{d}t,
\end{equation}
with initial condition
\begin{equation}\label{eq:  PDEInitialization}
\mathrsfs{S}(0,z) = S(x_0,z).
\end{equation}}
We next introduce a notion of approximate solution to \eqref{eq: PDESystem}. The point of this definition is that both
$S(\mathscr{X}_t,z)$ and $S(x_{\lfloor td\rfloor},z)$, which are functions
of both homogenized SGD and SGD respectively, will be shown to satisfy the PDE up to a small error.

To measure errors uniformly on the fixed contour, we define a norm, $\norm{\cdot}_\Gamma$, on a continuous function $S\colon \Gamma \subset \mathbb{C}^4 \to \mathbb{C}^{3\times 3}$ by
\begin{equation*}
    \norm{S}_\Gamma = \max_{z\in \Gamma } \norm{S(z)}.
\end{equation*}
The following lemma relates this contour norm to the parameter squared norm along both the homogenized and discrete trajectories.
\begin{lemma}\label{lem: normEquivalence}
There exist constants $0<c<C<\infty$, depending on $\norm{K}_{\mathrm{op}}$,  $\norm{\beta^*}_\infty$, and $\Gamma$, such that for all $t\geqslant 0$,
\begin{equation*}
c \leqslant \frac{\norm{S(\mathscr{X}_t, \cdot)}_\Gamma}{\frac{1}{d}\norm{\mathscr{W}_t}^2}, \frac{\norm{S(x_{\lfloor td\rfloor},\cdot)}_\Gamma}{\frac{1}{d}\norm{W_{\lfloor td \rfloor}}^2} \leqslant C. 
\end{equation*}
\end{lemma}
\begin{proof}
For homogenized SGD, we have
\begin{equation*}
\frac{1}{d} \norm{\mathscr{W}_t}^2 = \frac{1}{(2 \pi)^4} \oint_\Gamma \tr S(\mathscr{X}_t, z) \, \mathrm{d}z \leqslant C \norm{S(\mathscr{X}_t, \cdot)}_\Gamma.
\end{equation*}
On the other hand, by Neumann series and since $\abs{z_i} > \norm{D_i}_{\mathrm{op}}$ on each $\Gamma_i$, we know that
\begin{equation*}
R\left(z_i; D_i \right) = \left(z_i\cdot I_d - D_i\right)^{-1} = \frac{1}{z_i} \left(I_d - \frac{1}{z_i}D_i \right)^{-1} = \frac{1}{z_i} \sum_{j=0}^\infty \left( \frac{1}{z_i}D_i \right)^j,
\end{equation*}
therefore
\begin{equation*}
\norm{R\left(z_i; D_i \right)}_{\mathrm{op}} \leqslant \frac{1}{\abs{z_i}} \sum_{j=0}^\infty \left( \frac{1}{\abs{z_i}}\norm{D_i}_{\mathrm{op}} \right)^j = \frac{1}{\abs{z_i}} \cdot \frac{1}{1 - \frac{1}{\abs{z_i}}\norm{D_i}_{\mathrm{op}}} = \frac{1}{\abs{z_i} - \norm{D_i}_{\mathrm{op}}} \leqslant 2.
\end{equation*}
Thus we conclude 
\begin{equation*}
\norm{S(\mathscr{X}_t, \cdot)}_\Gamma = \max_{z \in \Gamma} \norm{\frac{1}{d} \mathscr{W}_t^\top \Omega \mathscr{W}_t} \leqslant \frac{1}{d} \norm{\mathscr{W}_t}^2 \cdot \Pi_{i=1}^4 \max_{z_i \in \Gamma_i} \norm{R\left(z_i; D_i \right)}_{\mathrm{op}} \leqslant \frac{2^4}{d} \norm{\mathscr{W}_t}^2.
\end{equation*}
The same bounds hold for SGD with obvious changes.
\end{proof}

We will be working with \textit{approximate solutions to the partial integro-differential equation} defined as:
\begin{definition}[$(\varepsilon, M, T)$-approximate solution to the partial integro-differential equation]\upright\label{def: approximateSolution} For constants $M, T, \varepsilon>0$, we say that a continuous function
\begin{equation*}
\mathscr{S}\colon \{t \geqslant 0\}\times \Gamma \subset \{t \geqslant 0\}\times \mathbb{C}^4\to \mathbb{C}^{3 \times 3}\end{equation*}
is an $(\varepsilon,M,T)$-approximate solution to \eqref{eq: PDESystem} if it satisfies the following properties.

\medskip
\noindent
\textbf{Stopping time.}  
Define the stopping time
\begin{equation*}
\hat{\tau}_M(\mathscr{S}) := \inf \biggl\{ t\geqslant 0 : 
\norm{\mathscr{S}(t, \cdot)}_\Gamma> M \text{ or } \mathscr{B}(t) \notin \mathcal{U} \biggr\}.
\end{equation*}
\medskip

\begin{enumerate}[label=(\roman*)]
\item\label{cond:lipschitz} \textbf{Lipschitz regularity of derivatives.}
For any other 
$(\varepsilon,M,T)$–approximate solution $\tilde{\mathscr{S}}$ and for each coordinate $z_i$, the first and second partial derivatives $\frac{\mathrm{d}}{\mathrm{d}z_i}\mathscr{S}$ and $\frac{\mathrm{d}^2}{\mathrm{d}z_i^2}\mathscr{S}$ exist, and they depend Lipschitz–continuously on $\mathscr{S}$ in the sense that for any $s\geqslant 0$ and $z\in \Gamma \subset \mathbb{C}^4$, it holds:
\begin{align*}
\norm{\frac{\mathrm{d}}{\mathrm{d}z_i}\mathscr{S}(s\wedge \hat{\tau}_M^\mathscr{S}, z)  - \frac{\mathrm{d}}{\mathrm{d}z_i} \tilde{\mathscr{S}}(s\wedge \hat{\tau}_M^{\tilde{\mathscr{S}}}, z)}_\Gamma &\leqslant C \norm{ \mathscr{S}(s\wedge \hat{\tau}_M^\mathscr{S}, \cdot) - \tilde{\mathscr{S}}(s\wedge \hat{\tau}_M^{\tilde{\mathscr{S}}}, \cdot)}_\Gamma,\\[10pt]
\norm{ \frac{\mathrm{d}^2}{\mathrm{d}z_i^2}\mathscr{S}(s\wedge \hat{\tau}_M^\mathscr{S}, z)  -  \frac{\mathrm{d}^2}{\mathrm{d}z_i^2} \tilde{\mathscr{S}}(s\wedge \hat{\tau}_M^{\tilde{\mathscr{S}}}, z)}_\Gamma &\leqslant C \norm{ \mathscr{S}(s\wedge \hat{\tau}_M^\mathscr{S}, \cdot) - \tilde{\mathscr{S}}(s\wedge \hat{\tau}_M^{\tilde{\mathscr{S}}}, \cdot)}_\Gamma.
\end{align*}
\item\label{cond:pde} \textbf{Approximate satisfaction of the PDE.} The following integral-form error bound holds:
\begin{align*}
\sup_{0\leqslant t \leqslant (\hat{\tau}_M(\mathscr{S}) \wedge T)} \norm{\mathscr{S}(t,\cdot) - \mathscr{S}(0,\cdot) - \int_0^t \mathscr{F}\left(\cdot, \mathscr{S}(s,\cdot)\right) \mathrm{d}s}_\Gamma &\leqslant \varepsilon,
\end{align*}
where the initial condition is $\mathscr{S}(0,\cdot) = S(x_0, \cdot)$, with $x_0$ the initialization of SGD.
\end{enumerate}
We suppress the $\mathscr{S}$ in the notation for $\hat{\tau}_M$, that is $\hat{\tau}_M = \hat{\tau}_M(\mathscr{S})$, when the function $\mathscr{S}$ is clear from context.
\end{definition}
\begin{remark}\upright\label{rem:Lipschitz}
Consider
\begin{equation*}
S(x,z) = \frac{1}{d} W(x)^\top \Omega(x,z) W(x) \in \mathbb{C}^{3 \times 3}
\end{equation*}
for 
\begin{equation*}
\Omega = R\left(z_1; \operatorname{diag}(u) \right) \cdot R\left(z_2; \operatorname{diag}(v) \right) \cdot R\left(z_3; \operatorname{diag}(\beta^*) \right) \cdot R\left(z_4; K \right) \in \mathbb{C}^{d\times d}.    
\end{equation*}
\vspace{0.1in}
Differentiating the resolvent representation gives
\[
\frac{\mathrm{d}^2}{\mathrm{d} z_i^2} S(x,z)
=
\frac{2}{d}W^\top R(z_i;D_i)^2\Omega(x,z)W,
\qquad
\frac{\mathrm{d}^3}{\mathrm{d} z_i^3} S(x,z)
=
-\frac{6}{d}W^\top R(z_i;D_i)^3\Omega(x,z)W,
\]
so by Lemma~\ref{lem: normEquivalence} and \eqref{eq:resolventBound}, we get 
\begin{equation*}
\norm{\frac{\mathrm{d}^2}{\mathrm{d} z_i^2} S(x,z)}_\Gamma \leqslant C \norm{S(x,\cdot)}_\Gamma \quad \text{and} \quad  \norm{\frac{\mathrm{d}^3}{\mathrm{d} z_i^3} S(x,z)}_\Gamma \leqslant C \norm{S(x,\cdot)}_\Gamma.
\end{equation*}
Consequently, the processes $S(x_{\lfloor td\rfloor},z)$ and $S(\mathscr{X}_t,z)$ satisfy condition~\ref{cond:lipschitz}, since $\norm{\mathscr{S}(s\wedge \hat{\tau}_M, \cdot)}_\Gamma \leqslant M$ for any $s\geqslant 0$. Furthermore, in Section~\ref{sec: approxSols}, we prove that they 
also satisfy condition~\ref{cond:pde} and hence that they are $(\varepsilon, M, T)$-approximate solutions. Note that we must extend the discrete time of
SGD to a continuous time (see Section~\ref{sec: approxSols} for details). Finally, it is clear by definition that the exact solution $\mathrsfs{S}$ of \eqref{eq: PDESystem} is an $(0,M,T)$-approximate solution.
\end{remark}
The first result is a \textit{stability} statement: any two $(\varepsilon,M,T)$-approximate solutions remain uniformly close up to the stopping time.

\begin{proposition}[Stability]\label{prop: stability}
For all $(\varepsilon, M, T)$-approximate solutions $\mathscr{S}_1$ and $\mathscr{S}_2$, there exists a positive constant $C = C(L(h), L(I), \bar{\gamma}, \norm{K}_{\mathrm{op}}, \norm{\beta^*}_\infty, M, \alpha, T)$ such that
\begin{equation*}
\sup_{0\leqslant t \leqslant T} \norm{\mathscr{S}_1(t\wedge \tau_M,\cdot) - \mathscr{S}_2(t\wedge \tau_M, \cdot)}_\Gamma \leqslant C \cdot \varepsilon,
\end{equation*}
where $\tau_M = \min\{ \hat{\tau}_M(\mathscr{S}_1), \hat{\tau}_M(\mathscr{S}_2)\}$.
\end{proposition}
\begin{proof}
First note that $\tau_M \leqslant \hat{\tau}_M(\mathscr{S}_1)$ and $\tau_M \leqslant \hat{\tau}_M(\mathscr{S}_2)$. Thus all estimates below are taken up to the common stopping time $\tau_M$. Write $\mathscr{S}_1$ and $\mathscr{S}_2$ as
\begin{align}
\begin{split}
\label{eq: twoapprox}
\mathscr{S}_1(t,\cdot) &= \mathscr{S}_1(0,\cdot) + \int_0^t \mathscr{F}\left(\cdot, \mathscr{S}_1(s,\cdot)\right) \mathrm{d}s + \varepsilon(\mathscr{S}_1) \\
\mathscr{S}_2(t,\cdot) &= \mathscr{S}_2(0,\cdot) + \int_0^t \mathscr{F}\left(\cdot, \mathscr{S}_2(s,\cdot)\right) \mathrm{d}s + \varepsilon(\mathscr{S}_2),
\end{split}
\end{align}
where $\varepsilon(\mathscr{S}_1)$ and $\varepsilon(\mathscr{S}_2)$ are error terms from the $(\varepsilon, M, T)$-approximate solution inequality and we have for $j=1,2$ the estimate
\begin{equation*}
\sup_{0\leqslant t \leqslant (T\wedge \tau_M)} \norm{\varepsilon(\mathscr{S}_j)}_\Gamma \leqslant \varepsilon.
\end{equation*}
We first prove the stability estimate under the following Lipschitz bound on $\mathscr{F}$. 
Suppose that there exists a constant
$C = C(L(h), L(I), \bar{\gamma}, \norm{K}_{\mathrm{op}}, \norm{\beta^*}_\infty, M, \alpha)$ such that, for all $s$,
\begin{equation}
\label{eq: LipschitzF}
\norm{\mathscr{F}\left(\cdot,\mathscr{S}_1(s\wedge \tau_M, \cdot)\right) - \mathscr{F}\left(\cdot,\mathscr{S}_2(s\wedge \tau_M, \cdot)\right)}_\Gamma \leqslant C  \norm{\mathscr{S}_1(s\wedge \tau_M, \cdot) - \mathscr{S}_2(s\wedge \tau_M, \cdot)}_\Gamma.
\end{equation}
We verify \eqref{eq: LipschitzF} at the end of the proof. However, Equations \eqref{eq: twoapprox} and \eqref{eq: LipschitzF} imply
\begin{align*}
\sup_{0\leqslant t \leqslant T\wedge \tau_M} \norm{\mathscr{S}_1(t, \cdot) - \mathscr{S}_2(t, \cdot)}_\Gamma \leqslant 2\varepsilon + \sup_{0\leqslant t \leqslant T\wedge \tau_M} \int_0^t \norm{\mathscr{F}\left(\cdot,\mathscr{S}_1(s, \cdot)\right) - \mathscr{F}\left(\cdot,\mathscr{S}_2(s, \cdot)\right)}_\Gamma \, &\mathrm{d}s\\[5pt]
= 2\varepsilon + \sup_{0\leqslant t \leqslant T} \int_0^t \norm{\mathscr{F}\left(\cdot,\mathscr{S}_1(s\wedge \tau_M, \cdot)\right) - \mathscr{F}\left(\cdot,\mathscr{S}_2(s\wedge \tau_M, \cdot)\right)}_\Gamma \, &\mathrm{d}s\\[5pt]
\leqslant 2\varepsilon + C \int_0^T  \norm{\mathscr{S}_1(s\wedge \tau_M, \cdot) - \mathscr{S}_2(s\wedge \tau_M, \cdot)}_\Gamma \, &\mathrm{d}s,
\end{align*}
where $C = C(L(h), L(I), \bar{\gamma}, \norm{K}_{\mathrm{op}}, \norm{\beta^*}_\infty, M, \alpha)$ is a positive constant.

Define $Q_T := \sup_{0\leqslant t \leqslant T} \norm{\mathscr{S}_1(t\wedge \tau_M, \cdot) - \mathscr{S}_2(t\wedge \tau_M, \cdot)}_\Gamma$. Then one has that
\begin{align*}
Q_T &= \sup_{0\leqslant t \leqslant T\wedge \tau_M} \norm{\mathscr{S}_1(t, \cdot) - \mathscr{S}_2(t, \cdot)}_\Gamma \leqslant 2\varepsilon + C\int_0^T Q_s \, \mathrm{d}s.
\end{align*}
By an application of Gronwall's inequality, 
\begin{align*}
\sup_{0\leqslant t \leqslant T} \norm{\mathscr{S}_1(t\wedge \tau_M, \cdot) - \mathscr{S}_2(t\wedge \tau_M, \cdot)}_\Gamma &\leqslant 2 \varepsilon e^{C T},
\end{align*}
and the result is shown. 

It remains to verify the Lipschitz estimate \eqref{eq: LipschitzF}. We will do this
in steps. First, define $\mathscr{B}_j(\cdot) = \frac{1}{(2\pi)^4} 
\oint_{\Gamma}
z_4\, \mathscr{S}_j(\cdot, z)\,
\mathrm{d}z$ for $j = 1,2$. We
will use the shorthand $\mathscr{B}_j^{\tau_M}(s) = \mathscr{B}_j(s\wedge \tau_M)$ and $\mathscr{S}_j^{\tau_M}(s,\cdot) = \mathscr{S}_j(s\wedge \tau_M,\cdot)$. Now, by the $\alpha$-pseudo-Lipschitzness of $\nabla h$ (Assumption~\ref{ass: riskRepresentation}), we have
\begin{align}\label{eq: HBound}
\begin{split}
\norm{H(\mathscr{B}_1^{\tau_M}(s))^\top - H(\mathscr{B}_2^{\tau_M}(s))^\top} &\leqslant L(h) \left(1 +  \norm{\mathscr{B}_1^{\tau_M}(s)}^\alpha + \norm{\mathscr{B}_2^{\tau_M}(s)}^\alpha \right) \norm{\mathscr{B}_1^{\tau_M}(s) - \mathscr{B}_2^{\tau_M}(s)}\\[10pt]
&\leqslant C(L(h),  M, \alpha) \norm{\mathscr{B}_1^{\tau_M}(s) - \mathscr{B}_2^{\tau_M}(s)}\\[10pt]
\norm{H(\mathscr{B}_j^{\tau_M}(s))^\top} &\leqslant L(h) \left(1 +  \norm{\mathscr{B}_j^{\tau_M}(s)}^{\alpha+1} \right) \leqslant C(L(h),  M, \alpha)
\end{split}
\end{align}
since we have the expression
\begin{equation*}\label{eq: bBound}
\norm{\mathscr{B}_j^{\tau_M}(s)} = \norm{\frac{1}{(2\pi)^4} 
\oint_{\Gamma}
z_4\, \mathscr{S}_j^{\tau_M}(s, z)\,
\mathrm{d}z} \leqslant C(\abs{\Gamma},\norm{K}_{\mathrm{op}}, \norm{\beta^*}_\infty) \norm{\mathscr{S}_j^{\tau_M}(s, \cdot)}_\Gamma \leqslant C \cdot M.
\end{equation*}
Here we used the stopping time $\tau_M$ explicitly. Similarly,
\begin{align*}\label{eq: bDiffBound}
\begin{split}
\norm{\mathscr{B}_1^{\tau_M}(s) - \mathscr{B}_2^{\tau_M}(s)} &\leqslant C 
\oint_{\Gamma}
\abs{z_4} \norm{ \mathscr{S}_1^{\tau_M}(s, \cdot) - \mathscr{S}_2^{\tau_M}(s, \cdot)}_\Gamma
\mathrm{d}\abs{z} \\[10pt]
&\leqslant C(\abs{\Gamma},\norm{K}_{\mathrm{op}}, \norm{\beta^*}_\infty) \norm{ \mathscr{S}_1^{\tau_M}(s, \cdot) - \mathscr{S}_2^{\tau_M}(s, \cdot)}_\Gamma.
\end{split}
\end{align*}
Consequently, there exists a positive constant (independent of $s$) such that
\begin{equation}\label{eq: HDiffBound}
\norm{H(\mathscr{B}_1^{\tau_M}(s))^\top - H(\mathscr{B}_2^{\tau_M}(s))^\top} \leqslant C(L(h), M, \alpha, \norm{K}_{\mathrm{op}}, \norm{\beta^*}_\infty) \norm{ \mathscr{S}_1^{\tau_M}(s, \cdot) - \mathscr{S}_2^{\tau_M}(s, \cdot)}_\Gamma.
\end{equation}
Analogously,
\begin{align}
\begin{split}
\norm{ \oint_{\Gamma_i} q(z_i) \mathscr{S}_1^{\tau_M}(s, z) \, \mathrm{d}z_i -  \oint_{\Gamma_i} q(z_i) \mathscr{S}_2^{\tau_M}(s, z) \, \mathrm{d}z_i} &\leqslant C(\abs{\Gamma},\norm{K}_{\mathrm{op}}, \norm{\beta^*}_\infty) \cdot\\[5pt]
& \hspace{0.3in} \norm{ \mathscr{S}_1^{\tau_M}(s, \cdot) - \mathscr{S}_2^{\tau_M}(s, \cdot)}_\Gamma\\[10pt]
\norm{ \oint_{\Gamma_i} q(z_i) \mathscr{S}_j^{\tau_M}(s, z) \, \mathrm{d}z_i } &\leqslant C(\abs{\Gamma},\norm{K}_{\mathrm{op}}, \norm{\beta^*}_\infty) \norm{\mathscr{S}_j^{\tau_M}(s, \cdot)}_\Gamma \\[5pt]
&\leqslant C \cdot M.
\end{split}
\end{align}
Lastly, by the definition of approximate solution (\ref{def: approximateSolution}), for every $1\leqslant i \leqslant 4$, we have 
\begin{align}
\begin{split}
\norm{\frac{\mathrm{d}}{\mathrm{d}z_i}\mathscr{S}_1^{\tau_M}(s, z)  - \frac{\mathrm{d}}{\mathrm{d}z_i} \mathscr{S}_2^{\tau_M}(s, z)}_\Gamma &\leqslant C \norm{ \mathscr{S}_1^{\tau_M}(s, \cdot) - \mathscr{S}_2^{\tau_M}(s, \cdot)}_\Gamma\\[10pt]
\norm{\frac{\mathrm{d}}{\mathrm{d}z_i}\mathscr{S}_j^{\tau_M}(s, z)}_\Gamma &\leqslant C \norm{\mathscr{S}_j^{\tau_M}(s, \cdot)}_\Gamma \leqslant C \cdot M,
\end{split}
\end{align}
and 
\begin{align}
\begin{split}
\norm{ \frac{\mathrm{d}^2}{\mathrm{d}z_i^2}\mathscr{S}_1^{\tau_M}(s, z)  -  \frac{\mathrm{d}^2}{\mathrm{d}z_i^2} \mathscr{S}_2^{\tau_M}(s, z)}_\Gamma &\leqslant C \norm{ \mathscr{S}_1^{\tau_M}(s, \cdot) - \mathscr{S}_2^{\tau_M}(s, \cdot)}_\Gamma\\[10pt]
\norm{ \frac{\mathrm{d}^2}{\mathrm{d}z_i^2}\mathscr{S}_j^{\tau_M}(s, z)}_\Gamma &\leqslant C \norm{\mathscr{S}_j^{\tau_M}(s, \cdot)}_\Gamma \leqslant C \cdot M.
\end{split}
\end{align}
Therefore, since the composition of Lipschitz functions is Lipschitz, we have
\begin{align}\label{eq: compGBound}
\begin{split}
\norm{\left(\bigcirc_{i=1}^4 \mathscr{G}_i\right)
( \mathscr{S}_1^{\tau_M}(s, \cdot)) - \left(\bigcirc_{i=1}^4 \mathscr{G}_i\right)
( \mathscr{S}_2^{\tau_M}(s, \cdot))}_\Gamma &\leqslant C \norm{ \mathscr{S}_1^{\tau_M}(s, \cdot) - \mathscr{S}_2^{\tau_M}(s, \cdot)}_\Gamma\\[10pt]
\norm{\left(\bigcirc_{i=1}^4 \mathscr{G}_i\right)
( \mathscr{S}_j^{\tau_M}(s, \cdot)) }_\Gamma &\leqslant C \norm{ \mathscr{S}_j^{\tau_M}(s, \cdot) }_\Gamma \leqslant C \cdot M.
\end{split}
\end{align}
From equations \eqref{eq: HBound}, \eqref{eq: HDiffBound}, and \eqref{eq: compGBound}, it follows that there exists a positive constant $C= C(L(h), M,\\ \alpha, \norm{K}_{\mathrm{op}}, \norm{\beta^*}_\infty, \bar{\gamma})$ such that
\begin{align}\label{eq: firstTermPDE}
\begin{split}
\Big\lVert 2 \gamma(s) H\left( \mathscr{B}_1^{\tau_M}(s)\right)^\top z^\zeta 
\left(\bigcirc_{i=1}^4 \mathscr{G}_i\right)
( \mathscr{S}_1^{\tau_M}(s, \cdot)) &- 2 \gamma(s) H\left( \mathscr{B}_2^{\tau_M}(s)\right)^\top z^\zeta 
\left(\bigcirc_{i=1}^4 \mathscr{G}_i\right)
( \mathscr{S}_1^{\tau_M}(s, \cdot))\Big\rVert_\Gamma \\[10pt]
&\leqslant C \cdot \norm{ \mathscr{S}_1^{\tau_M}(s, \cdot) - \mathscr{S}_2^{\tau_M}(s, \cdot)}_\Gamma.
\end{split}
\end{align}
Next, by $\alpha$-pseudo-Lipschitzness of the squared gradients (Assumption~\ref{ass: squaredGradientsPseudoLipschitz}),
\begin{align}\label{eq: IBound}
\begin{split}
\norm{I(\mathscr{B}_1^{\tau_M}(s)) - I(\mathscr{B}_2^{\tau_M}(s))} &\leqslant C(L(I),  M, \alpha) \norm{\mathscr{B}_1^{\tau_M}(s) - \mathscr{B}_2^{\tau_M}(s)}\\[10pt]
\norm{I(\mathscr{B}_j^{\tau_M}(s))} &\leqslant L(I) \left(1 +  \norm{\mathscr{B}_j^{\tau_M}(s)}^{\alpha+1} \right) \leqslant C(L(I),  M, \alpha).   
\end{split}
\end{align}
Therefore, we deduce that
\begin{align}\label{eq: secondTermPDE}
\begin{split}
\Big \lVert C(Q) \gamma(s)^2 I(\mathscr{B}_1^{\tau_M}(s)) z^\zeta \left(\bigcirc_{i=1}^4 \mathscr{G}_i\right)
( \mathscr{S}_1^{\tau_M}(s, \cdot)) &- C(Q) \gamma(s)^2 I(\mathscr{B}_2^{\tau_M}(s)) z^\zeta \left(\bigcirc_{i=1}^4 \mathscr{G}_i\right)
( \mathscr{S}_2^{\tau_M}(s, \cdot))\Big \rVert_\Gamma \\[10pt]
&\leqslant C \cdot \norm{ \mathscr{S}_1^{\tau_M}(s, \cdot) - \mathscr{S}_2^{\tau_M}(s, \cdot)}_\Gamma,
\end{split}
\end{align}
where $C= C(L(I), M, \alpha, \norm{K}_{\mathrm{op}}, \norm{\beta^*}_\infty, \bar{\gamma})$ is a positive constant. The Lipschitz condition for $\mathscr{F}$ \eqref{eq: LipschitzF} holds after applying expressions \eqref{eq: firstTermPDE} and \eqref{eq: secondTermPDE}.
\end{proof}
We now transfer the stability estimate from the resolvent statistic $S$ to any statistic $\varphi=g\circ Q$ satisfying Assumption~\ref{ass: statistic}. Here we set
\begin{equation*}
Q(x) := \frac{1}{d} W^\top q_1(\operatorname{diag}(u)) q_2(\operatorname{diag}(v)) q_4(K) W,
\end{equation*}
where $q_1, q_2, q_4$ are polynomials. For an approximate solution $\mathscr{S}_i$, define
\begin{equation*}
\mathscr{Q}_i(t) = \frac{1}{(2 \pi)^4} \oint_\Gamma q_1(z_1) q_2(z_2) q_4(z_4) \mathscr{S}_i(t,z) \, \mathrm{d}z.
\end{equation*}
The following proposition shows that given two approximate solutions, $\mathscr{S}_1$ and $\mathscr{S}_2$, the function $g\circ \mathscr{Q}_1(t)$ is close to $g \circ \mathscr{Q}_2(t)$. The pseudo-Lipschitzness of $g$, together with the boundedness imposed by the stopping time, allows us to control the difference between $g(\mathscr Q_1)$ and $g(\mathscr Q_2)$ by the contour distance between $\mathscr S_1$ and $\mathscr S_2$:
\[
\sup_{0\leqslant t\leqslant T}
\left|
(g\circ \mathscr Q_1)(t\wedge\tau_M)
-
(g\circ \mathscr Q_2)(t\wedge\tau_M)
\right|
\leqslant
C
\sup_{0\leqslant t\leqslant T}
\|\mathscr S_1(t\wedge\tau_M,\cdot)-\mathscr S_2(t\wedge\tau_M,\cdot)\|_\Gamma,
\]
and then Proposition~\ref{prop: stability} finishes the result.
\begin{proposition}\label{prop: stabilityStatistics}
Suppose $\varphi \colon \R^{2d} \to \R$ is a statistic satisfying Assumption~\ref{ass: statistic} such that $\varphi(x) = g \circ Q(x)$. Suppose $\mathscr{S}_1$ and $\mathscr{S}_2$ are $(\varepsilon, M, T)$-approximate solutions. Then there exists a
positive constant $C = C(L(h), L(I), \bar{\gamma}, \norm{K}_{\mathrm{op}}, \norm{\beta^*}_\infty, \allowbreak M, \alpha, T)$ such that
\begin{align*}
\sup_{0\leqslant t \leqslant T} \bigg| &g \left( \frac{1}{(2 \pi)^4} \oint_\Gamma q_1(z_1) q_2(z_2) q_4(z_4) \mathscr{S}^{\tau_M}_1(t,z) \, \mathrm{d}z \right)  \\[5pt]
&\hspace{1in}- g \left( \frac{1}{(2 \pi)^4} \oint_\Gamma q_1(z_1) q_2(z_2) q_4(z_4) \mathscr{S}^{\tau_M}_2(t,z) \, \mathrm{d}z \right) \bigg| \leqslant C \cdot \varepsilon
\end{align*}
where $\tau_M:=\hat{\tau}_M(\mathscr{S}_1)\wedge
\hat{\tau}_M(\mathscr{S}_2)$. Here $\mathscr{S}^{\tau_M}_i(t,\cdot) = \mathscr{S}_i(t \wedge \tau_M,\cdot)$.
\end{proposition}
\begin{proof}
Since $\tau_M \leqslant \hat{\tau}_M(\mathscr{S}_1)$ and $\tau_M \leqslant \hat{\tau}_M(\mathscr{S}_2)$, we can always work on the smaller time $\tau_M$. We define $\mathscr{Q}_i(t) = \frac{1}{(2 \pi)^4} \oint_\Gamma q_1(z_1) q_2(z_2) q_4(z_4) \mathscr{S}_i(t,z) \, \mathrm{d}z$ and the stopped process $\mathscr{Q}_i^{\tau_M}(t) = \mathscr{Q}_i(t \wedge \tau_M)$ for $i=1,2$. First, we observe  
\begin{align}
\begin{split}\label{eq: Qbound}
\norm{\mathscr{Q}_i^{\tau_M}(t) } &\leqslant C \oint_\Gamma \abs{q_1(z_1)} \abs{q_2(z_2)} \abs{q_4(z_4)} \norm{\mathscr{S}^{\tau_M}_i(t,z)} \, \mathrm{d}z \\[5pt]
&\leqslant C(\norm{K}_{\mathrm{op}}, \norm{\beta^*}_\infty, \norm{q_1}_{\Gamma_1}, \norm{q_2}_{\Gamma_2}, \norm{q_4}_{\Gamma_4}) \norm{\mathscr{S}^{\tau_M}_i(t,\cdot)}_\Gamma \leqslant C \cdot M.
\end{split}
\end{align}
Moreover, the map $\mathscr S\mapsto \mathscr Q$ is Lipschitz in the contour norm:
\begin{align}\label{eq: Qlipschitz}
\begin{split}
\norm{\mathscr{Q}_1^{\tau_M}(t) - \mathscr{Q}_2^{\tau_M}(t)} &\leqslant C(\norm{q_1}_{\Gamma_1}, \norm{q_2}_{\Gamma_2}, \norm{q_4}_{\Gamma_4}) \oint_\Gamma \norm{\mathscr{S}^{\tau_M}_1(t,z) - \mathscr{S}^{\tau_M}_2(t,z)} \, \mathrm{d}\abs{z}\\
&\leqslant C(\norm{K}_{\mathrm{op}}, \norm{\beta^*}_\infty, \norm{q_1}_{\Gamma_1}, \norm{q_2}_{\Gamma_2}, \norm{q_4}_{\Gamma_4})  \norm{\mathscr{S}^{\tau_M}_1(t,\cdot) - \mathscr{S}^{\tau_M}_2(t,\cdot)}_\Gamma .
\end{split}
\end{align}
Using the $\alpha$-pseudo-Lipschitzness of $g$, together with the boundedness estimate \eqref{eq: Qbound} and the Lipschitz estimate \eqref{eq: Qlipschitz}, we obtain
\begin{align*}
\abs{(g \circ \mathscr{Q}^{\tau_M}_1)(t) - (g \circ \mathscr{Q}^{\tau_M}_2)(t)} &\leqslant L(g) \norm{ \mathscr{Q}^{\tau_M}_1(t) - \mathscr{Q}^{\tau_M}_2(t)}\left(1 + \norm{\mathscr{Q}^{\tau_M}_1(t)}^\alpha + \norm{\mathscr{Q}^{\tau_M}_2(t)}^\alpha\right)\\[10pt]
&\leqslant C \cdot \norm{\mathscr{S}^{\tau_M}_1(t,\cdot) - \mathscr{S}^{\tau_M}_2(t,\cdot)}_\Gamma,
\end{align*}
where $C=C(\norm{K}_{\mathrm{op}}, \norm{\beta^*}_\infty, M, \norm{q_1}_{\Gamma_1}, \norm{q_2}_{\Gamma_2}, \norm{q_4}_{\Gamma_4}, L(g), \alpha)$ is a positive constant. Taking the supremum over all $0\leqslant t \leqslant T$ and applying Proposition~\ref{prop: stability} finishes the result. 
\end{proof}

\subsection{Main Argument of the Proof: Concentration of the Resolvent Statistic \texorpdfstring{$S$}{S}}
\label{subsec: mainArgumentProof}
In this section, we prove concentration of both SGD and homogenized SGD under the resolvent statistic $S$ around the deterministic solution $\mathrsfs{S}(t,z)$ of \eqref{eq: PDESystem}. We first prove a concentration result with a common stopping time, and then strengthen it to a one-sided stopping-time formulation in Theorem~\ref{thm: mainResultOneStoppingTime}. This resolvent-level result implies Theorem~\ref{thm: learningCurves}. The important statistic which will play a pivotal role is 
\begin{equation}
S(x,z) = \frac{1}{d} W(x)^\top \Omega(x,z) W(x) \in \mathbb{C}^{3 \times 3}, 
\end{equation}
as well as the function
\begin{equation}
B(x) = \frac{1}{d} W(x)^\top K W(x) \in \R^{3 \times 3}. 
\end{equation}
Here
\[
W(x):=[\,\psi(x)\mid \beta^*\mid \mathds{1}_d\,]\in\mathbb{R}^{d\times 3}.
\]
We will extend the iterates of SGD,  $\{x_k\}$,  defined on discrete time $k$ to continuous time. This
is so that we can compare SGD and homogenized SGD, $\{\mathscr{X}_t\}$. We relate the $k$-th iterate of SGD
to the continuous time parameter $t$ in homogenized SGD through the relationship $k=\lfloor td \rfloor$. Thus,
when $t = 1$, SGD has done exactly $d$ updates.

We are now ready to state and prove one of our main results.
\begin{theorem}[Concentration under a common stopping time]\label{thm: mainResultTwoStoppingTimes} Suppose the risk function $\mathcal{R}(x)$ \eqref{eq: riskSetting} satisfies Assumptions~\ref{ass: fPseudoLipschitz},~\ref{ass: riskRepresentation}, and~\ref{ass: squaredGradientsPseudoLipschitz}. Suppose the stepsize schedule satisfies~\ref{ass: learningRate}, the iterates $x_k$ and $\mathscr{X}_t$ satisfy~\ref{ass: nonexplosive}, and the hidden parameters $\beta^*$ satisfy Assumption~\ref{ass: parameterScaling}. Moreover the data $a\sim \mathcal{N}(0,K)$ satisfies Assumption~\ref{ass: data}. Write $W_{\lfloor td \rfloor}:=W(x_{\lfloor td \rfloor})$ and $\mathscr W_t:=W(\mathscr X_t)$ initialized with $\mathscr{X}_0 = x_0$. Then there is an $\varepsilon > 0$ so that for any $T, M > 0$ and $d$ sufficiently large, with overwhelming probability,
\begin{align}
\begin{split}
\sup_{0\leqslant t \leqslant T\wedge \tau_M(S(x_{\lfloor td \rfloor},\cdot), \mathrsfs{S}(t,\cdot))} \norm{S(x_{\lfloor td \rfloor},\cdot) - \mathrsfs{S}(t,\cdot)}_\Gamma &\leqslant d^{-\varepsilon},\\
\sup_{0\leqslant t \leqslant T\wedge \tau_M(S(\mathscr{X}_t, \cdot), \mathrsfs{S}(t,\cdot))} \norm{S(\mathscr{X}_t, \cdot) - \mathrsfs{S}(t,\cdot)}_\Gamma &\leqslant d^{-\varepsilon},\\
\sup_{0\leqslant t \leqslant T\wedge \tau_M(S(x_{\lfloor td \rfloor},\cdot), S(\mathscr{X}_t, \cdot))} \norm{S(x_{\lfloor td \rfloor},\cdot) - S(\mathscr{X}_t, \cdot)}_\Gamma &\leqslant d^{-\varepsilon},
\end{split}
\end{align}
where $\mathrsfs{S}$ solves the partial integro-differential equation \eqref{eq: PDESystem}, and
\begin{equation*}
\tau_M \left(\mathscr{S}_1, \mathscr{S}_2 \right) = \min\{\hat{\tau}_M(\mathscr{S}_1), \hat{\tau}_M(\mathscr{S}_2)  \}.
\end{equation*}
\end{theorem}
\begin{proof}
We will consider $\mathscr{S}_1(t,z) = S(x_{\lfloor td \rfloor},z)$ and $\mathscr{S}_2(t,z) = S(\mathscr{X}_t,z)$ and suppress the notation by
setting $\tau_M \left(\mathscr{S}_1, \mathscr{S}_2 \right) = \tau_M$. We also note that the cases when $\mathscr{S}_1(t,z) = S(x_{\lfloor td \rfloor},z)$ and $\mathscr{S}_2(t,z) = \mathrsfs{S}(t,z)$
and $\mathscr{S}_1(t,z) = S(\mathscr{X}_t,z)$ and $\mathscr{S}_2(t,z) = \mathrsfs{S}(t,z)$ follow an analogous proof, so for brevity, we do not
present them.

By Proposition~\ref{prop: HSGDapprox}, for some $\tilde{\varepsilon}>0$, we have that $S(\mathscr{X}_t,z)$ is an $(d^{-\tilde{\varepsilon}}, M, T)$-approximate solution with overwhelming probability. Moreover, by Proposition~\ref{prop: SGDapprox}, $S(x_{\lfloor td \rfloor},z)$ is an $(d^{-\tilde{\varepsilon}}, M, T)$-approximate solution. (For the deterministic function $\mathrsfs{S}(t,z)$, it is an $(0, M+1, T)$-approximate solution by definition.) Applying the stability result, Proposition~\ref{prop: stability}, to the three pairs
\[
(S(x_{\lfloor td\rfloor},\cdot),S(\mathscr X_t,\cdot)),\qquad
(S(x_{\lfloor td\rfloor},\cdot),\mathrsfs S(t,\cdot)),\qquad
(S(\mathscr X_t,\cdot),\mathrsfs S(t,\cdot))
\]
gives the three stated bounds, after possibly decreasing $\varepsilon$.
\end{proof}
In the next theorem, we note that one can remove the condition that \textit{both} processes must remain bounded in $\norm{\cdot}_\Gamma$ and separated from $\mathcal{U}$, and reduce this to show that we only need \textit{one} of the processes to satisfy these properties. In this way, we can show that SGD is well-behaved and then conclude that homogenized SGD
must also be well-behaved.

For any $(\varepsilon, M, T)$-approximate solution $\mathscr{S}(t,\cdot)$, we define the stopping time
\begin{equation*}
\hat{\tau}_{M, \eta} (\mathscr{S}) = \inf \{ t\geqslant 0 : \norm{\mathscr{S}(t,\cdot)}_\Gamma > M \; \text{or} \; \sup_{\hat{\mathscr{B}}\notin \mathcal{U}} \norm{\mathscr{B}(t, \mathscr{S})-\hat{\mathscr{B}}}\leqslant \eta \} 
\end{equation*}
where we set
\begin{equation*}
\mathscr{B}(t, \mathscr{S}) = \frac{1}{(2 \pi)^4}\oint_\Gamma z_4 \mathscr{S}(t, z)\, \mathrm{d}z.  
\end{equation*}

Our main theorem requires that \textit{only one} of the statistics stays bounded, and not, in particular, both. To define this, we introduce a stopping time
\begin{equation}\label{eq: stoppingTimeOnlyOne}
\Theta_{M,\eta}^{\mathscr S_1,\mathscr S_2}
:=
\max_{i=1,2}\hat\tau_{M,\eta}(\mathscr S_i).
\end{equation}
We note that $\hat{\tau}_{M, 0} = \hat{\tau}_M$ with $\hat{\tau}_M$ defined in the $(\varepsilon, M, T)$-approximate solution definition.
\begin{theorem}[Concentration under a one-sided stopping time]\label{thm: mainResultOneStoppingTime}
Suppose the risk function $\mathcal{R}(x)$ \eqref{eq: riskSetting} satisfies Assumptions~\ref{ass: fPseudoLipschitz},~\ref{ass: riskRepresentation}, and~\ref{ass: squaredGradientsPseudoLipschitz}. Suppose the learning
rate schedule satisfies~\ref{ass: learningRate}, the iterates $x_k$ and $\mathscr{X}_t$ satisfy~\ref{ass: nonexplosive}, and the hidden parameters $\beta^*$ satisfy~\ref{ass: parameterScaling}. Moreover the data $a\sim \mathcal{N}(0,K)$ satisfies Assumption~\ref{ass: data}. Let $\Theta_{M,\eta}$ be defined by \eqref{eq: stoppingTimeOnlyOne}. Write $W_{\lfloor td \rfloor}:=W(x_{\lfloor td \rfloor})$ and $\mathscr W_t:=W(\mathscr X_t)$ initialized with $\mathscr{X}_0 = x_0$. Then there is an $\varepsilon > 0$ so that for any
$T, M, \eta > 0$ and $d$ sufficiently large, with overwhelming probability,
\begin{align}
\begin{split}
\sup_{0\leqslant t \leqslant T\wedge \Theta_{M, \eta}^{S(x_{\lfloor td \rfloor}, \cdot), \mathrsfs{S}(t,\cdot)}} \norm{S(x_{\lfloor td \rfloor}, \cdot) - \mathrsfs{S}(t,\cdot)}_\Gamma &\leqslant d^{-\varepsilon}, \\
\sup_{0\leqslant t \leqslant T\wedge \Theta_{M, \eta}^{S(\mathscr{X}_t,\cdot), \mathrsfs{S}(t,\cdot)}} \norm{S(\mathscr{X}_t,\cdot) - \mathrsfs{S}(t,\cdot)}_\Gamma &\leqslant d^{-\varepsilon},\\
\sup_{0\leqslant t \leqslant T\wedge \Theta_{M, \eta}^{S(x_{\lfloor td \rfloor}, \cdot), S(\mathscr{X}_t,\cdot)}} \norm{S(x_{\lfloor td \rfloor}, \cdot) - S(\mathscr{X}_t,\cdot)}_\Gamma &\leqslant d^{-\varepsilon},
\end{split}
\end{align}
where $\mathrsfs{S}$ solves the partial integro-differential equation \eqref{eq:  PDESystem}.
\end{theorem}
\begin{proof}
Fix an $\eta > 0$. For two mappings $\mathscr{S}_1$ and $\mathscr{S}_2$, we define the stopping time
\begin{equation}\label{eq: coupledStoppingTime}
\tau_{M+1, 0}^{\mathscr{S}_1, \mathscr{S}_2} = \min \{ \hat{\tau}_{M+1,0}(\mathscr{S}_1), \hat{\tau}_{M+1,0}(\mathscr{S}_2) \}.
\end{equation}
As in the previous theorem, we will consider $\mathscr{S}_1(t,z) = S(x_{\lfloor td \rfloor},z)$ and $\mathscr{S}_2(t,z) = S(\mathscr{X}_t,z)$ and
suppress the notation by setting $\tau_{M, \eta}^{\mathscr{S}_1, \mathscr{S}_2} = \tau_{M, \eta}$. We also note that the cases when $\mathscr{S}_1(t,z) = S(x_{\lfloor td \rfloor},z)$ and $\mathscr{S}_2(t,z) = \mathrsfs{S}(t,z)$
and $\mathscr{S}_1(t,z) = S(\mathscr{X}_t,z)$ and $\mathscr{S}_2(t,z) = \mathrsfs{S}(t,z)$ follow an analogous proof so for
brevity we do not present them.

By Theorem~\ref{thm: mainResultTwoStoppingTimes}, we have that
\begin{equation}\label{eq: eventStoppingTime}
\sup_{0\leqslant t \leqslant T\wedge \tau_{M+1,0}} \norm{S(x_{\lfloor td \rfloor},\cdot) - S(\mathscr{X}_t,\cdot)}_\Gamma \leqslant d^{-\varepsilon} \quad \text{w.o.p.}
\end{equation}
The remaining component is to replace the stopping time $\tau_{M+1,0}$ which requires \textit{both} statistics
to have $\Gamma$-norm less than $M + 1$ with $\Theta_{M,\eta}$ which only requires \textit{one} of the statistics to remain in the
good set. Denote the event that \eqref{eq: eventStoppingTime} occurs by $A_\varepsilon$ and its complement by $A_\varepsilon^{\mathrm C}$. Then for sufficiently
large $d$, we have
\begin{equation}\label{eq: probEventStoppingTime}
\PP \left[ \Theta_{M,\eta} > \tau_{M+1, 0} \right] \leqslant \PP \left[ A_\varepsilon^{\mathrm C} \right].
\end{equation}
To see this, suppose $\Theta_{M,\eta} > \tau_{M+1,0}$. At time $t=\tau_{M+1,0}$, one of the following exits must occur: $\norm{S(x_{\lfloor td \rfloor},\cdot)}_\Gamma \geqslant M+1$ or $\sup_{\hat{B}\notin \mathcal{U}} \norm{B(x_{\lfloor td \rfloor})-\hat{B}}\leqslant 0$ or $\norm{S(\mathscr{X}_t,\cdot)}_\Gamma \geqslant M+1$ or $\sup_{\hat{B}\notin \mathcal{U}} \norm{B(\mathscr{X}_t)-\hat{B}}\leqslant 0$. On the other hand, since $\tau_{M+1,0} = t < \Theta_{M,\eta}$, then either $\norm{S(x_{\lfloor td \rfloor},\cdot)}_\Gamma \geqslant M$ or $\norm{S(\mathscr{X}_t,\cdot)}_\Gamma \geqslant M$ and then $\sup_{\hat{B}\notin \mathcal{U}} \norm{B(x_{\lfloor td \rfloor})-\hat{B}}> \eta$ or $\sup_{\hat{B}\notin \mathcal{U}} \norm{B(\mathscr{X}_t)-\hat{B}}> \eta$.

Now we consider cases. Suppose $\norm{S(\mathscr{X}_t,\cdot)}_\Gamma \geqslant M+1$. Then $\norm{S(\mathscr{X}_t,\cdot)}_\Gamma$ cannot be less than or
equal to $M$ so it must have been that $\norm{S(x_{\lfloor td \rfloor},\cdot)}_\Gamma \leqslant M$. Since $t = \tau_{M+1,0}$, working on the event that \eqref{eq: eventStoppingTime} occurs, we have that
\begin{equation*}
\norm{S(\mathscr{X}_t,\cdot)}_\Gamma \leqslant  \norm{S(x_{\lfloor td \rfloor},\cdot) - S(\mathscr{X}_t,\cdot)}_\Gamma + \norm{S(x_{\lfloor td \rfloor},\cdot)}_\Gamma \leqslant d^{-\varepsilon} + M.
\end{equation*}
For sufficiently large $d$, then $\norm{S(\mathscr{X}_t,\cdot)}_\Gamma < M + 1$ which is a contradiction.

Suppose $\norm{S(x_{\lfloor td \rfloor},\cdot)}_\Gamma \geqslant M + 1$. Then by reversing the roles of $x_{\lfloor td \rfloor}$ and $\mathscr{X}_t$
in the previous case,
we see that this cannot occur.

Next suppose that $\sup_{\hat{B}\notin \mathcal{U}} \norm{B(\mathscr{X}_t)-\hat{B}}\leqslant 0$. Then $\sup_{\hat{B}\notin \mathcal{U}} \norm{B(\mathscr{X}_t)-\hat{B}}$ cannot be greater than $\eta$. Thus it had to be the case that $\sup_{\hat{B}\notin \mathcal{U}} \norm{B(x_{\lfloor td \rfloor})-\hat{B}}> \eta$. Now working on
the event that \eqref{eq: eventStoppingTime} occurs, we have that
\begin{equation*}
\norm{B(x_{\lfloor td \rfloor})-\hat{B}} \leqslant 
\norm{B(x_{\lfloor td \rfloor})-B(\mathscr{X}_t)} \leqslant C \cdot \sup_{z_4 \in \Gamma_4} \abs{z_4} \cdot
\norm{S(x_{\lfloor td \rfloor},\cdot)-S(\mathscr{X}_t,\cdot)}_\Gamma \leqslant \tilde{C} \cdot d^{-\varepsilon}.
\end{equation*}
where $C$ and $\tilde{C}$ are positive constants. Hence for sufficiently large $d$, we have $\sup_{\hat{B}\notin \mathcal{U}} \norm{B(x_{\lfloor td \rfloor})-\hat{B}} < \eta$, a contradiction.

Lastly suppose $\sup_{\hat{B}\notin \mathcal{U}} \norm{B(x_{\lfloor td \rfloor})-\hat{B}} \leqslant 0$. By reversing the roles of $x_{\lfloor td \rfloor}$ and $\mathscr{X}_t$, we
reach the same conclusion as the previous case.

Hence the inequality \eqref{eq: probEventStoppingTime} holds and thus, $\tau_{M+1,0} \geqslant \Theta_{M,\eta}$ with overwhelming probability. The
result follows.
\end{proof}

We immediately get a corollary which shows that SGD and homogenized SGD concentrate
around the deterministic function $\mathrsfs{S}(t,z)$ which is a solution to the partial integro-differential equation
\eqref{eq: PDESystem} provided Assumption~\ref{ass: nonexplosive} holds.
\begin{corollary}[Non-explosiveness and concentration]\label{cor: boundedNConcentration}
Suppose the assumptions of Theorem~\ref{thm: mainResultOneStoppingTime} hold. Suppose, in addition, for a fixed $T > 0$ and $\eta > 0$ that
\begin{equation}\label{eq: conditionsForTheorem}
\sup_{0\leqslant t \leqslant T} \sup_{\hat{B}\notin \mathcal{U}} \norm{B(x_{\lfloor td \rfloor})-\hat{B}} > \eta \quad \text{w.o.p.}
\end{equation}
Then there is an $\varepsilon > 0$ so that for $d$ sufficiently
large, with overwhelming probability,
\begin{equation}\label{eq: corollarySGD}
\sup_{0\leqslant t \leqslant T} \norm{S(x_{\lfloor td \rfloor},\cdot) - \mathrsfs{S}(t,\cdot)}_\Gamma \leqslant d^{-\varepsilon} \quad \text{and} \quad \sup_{0\leqslant t \leqslant T} \norm{S(x_{\lfloor td \rfloor},\cdot) - S(\mathscr{X}_t,\cdot) }_\Gamma \leqslant d^{-\varepsilon},
\end{equation}
and therefore
\begin{equation}\label{eq: corollaryTriangleIneq}
\sup_{0\leqslant t \leqslant T} \norm{\mathrsfs{S}(t,\cdot) - S(\mathscr{X}_t,\cdot)}_\Gamma \leqslant 2d^{-\varepsilon}.
\end{equation}
\end{corollary}
\begin{proof}
Define the following stopping time similar to $\Theta_{M,\eta}$ in \eqref{eq: stoppingTimeOnlyOne} by
\begin{equation}
\tilde{\Theta}_{M, \eta}^{\mathscr{S}_1, \mathscr{S}_2} = \max_{i=1,2} \inf \{ t\geqslant 0 : \sup_{\hat{\mathscr{B}}\notin \mathcal{U}} \norm{\mathscr{B}(t, \mathscr{S}_i)-\hat{\mathscr{B}}}\leqslant \eta \}.
\end{equation}
Here we think of $\mathscr{S}_1$ as either homogenized SGD or $\mathrsfs{S}$ and $\mathscr{S}_2$ as SGD. The stopping time $\Theta_{M,\eta}^{\mathscr S_1,\mathscr S_2}$ from
\eqref{eq: stoppingTimeOnlyOne} is controlled by the non-explosiveness and
separation assumptions on the SGD trajectory. By Lemma~\ref{lem: normEquivalence} and Assumption~\ref{ass: parameterScaling}, there exists some $C>0$ independent of $d$ such that
\begin{equation*}
\norm{S(x_{\lfloor td \rfloor},\cdot)}_\Gamma \leqslant \frac{2^4}{d} \norm{W_{\lfloor td \rfloor}}^2 \leqslant 2^4\left( \norm{x_{\lfloor td \rfloor}}_\infty^4 + \norm{\beta^*}_\infty^2 + 1\right) \leqslant C\left( \norm{x_{\lfloor td \rfloor}}_\infty^4 + 1\right). 
\end{equation*}
Consequently, this translates into
\begin{equation*}
\{ t\geqslant 0 : \norm{S(x_{\lfloor td \rfloor},\cdot)}_\Gamma > C \left( M^4 + 1\right)\} \subset \{ t\geqslant 0 : \norm{x_{\lfloor td \rfloor}}_\infty > M\},
\end{equation*}
and so the infimum of the right-hand-side is smaller than the infimum of the left-hand-side. Moreover, we have by Assumption~\ref{ass: nonexplosive} that
\begin{equation*}
T \leqslant \inf \{ t\geqslant 0 : \norm{x_{\lfloor td \rfloor}}_\infty > M\} \quad \text{w.o.p.}
\end{equation*}
Similarly we have that
\begin{equation*}
T \leqslant \inf \{ t\geqslant 0 : \sup_{\hat{B}\notin \mathcal{U}} \norm{B(x_{\lfloor td \rfloor})-\hat{B}}\leqslant \eta \} \quad \text{w.o.p.}
\end{equation*}
Thus, we have that
\begin{equation*}
T \leqslant \tilde{\Theta}_{M, \eta}^{\mathscr{S}_1, \mathscr{S}_2} \leqslant \Theta_{C \left( M^4 + 1\right), \eta}^{\mathscr{S}_1, \mathscr{S}_2} \quad \text{w.o.p.,}
\end{equation*}
where $\mathscr{S}_1$ is either $S(\mathscr{X}_t, \cdot)$ or $\mathrsfs{S}(t, \cdot)$. By Theorem~\ref{thm: mainResultOneStoppingTime}, we immediately get the result in \eqref{eq: corollarySGD}.  A
simple triangle inequality gives the result in \eqref{eq: corollaryTriangleIneq}.
\end{proof}
Lastly, we make one final connection to Theorem~\ref{thm: learningCurves},
proving the result below. 
\begin{proof}[Proof of Theorem~\ref{thm: learningCurves}]
The result immediately follows from Theorem~\ref{thm: mainResultOneStoppingTime} and
Corollary~\ref{cor: boundedNConcentration} after noting that 

\begin{equation*}
B(x) = \frac{1}{(2 \pi)^4}\oint_\Gamma z_4 S(x, z)\, \mathrm{d}z \quad \text{and} \quad \mathrsfs{B}(t) = \frac{1}{(2 \pi)^4}\oint_\Gamma z_4 \mathrsfs{S}(t, z)\, \mathrm{d}z
\end{equation*}

and Lipschitzness of the integral, that is,

\begin{equation*}
\norm{ \oint_\Gamma z_4 \mathscr{S}_1(t,\cdot)\, \mathrm{d}z - \oint_\Gamma z_4 \mathscr{S}_2(t,\cdot)\, \mathrm{d}z } \leqslant C \cdot \norm{ \mathscr{S}_1(t,\cdot) - \mathscr{S}_2(t,\cdot) }_\Gamma \quad \text{for some positive} \quad C>0.
\end{equation*}
\end{proof}

\subsection{Concentration of General Statistics}
We now transfer the resolvent-level concentration result of
Theorem~\ref{thm: mainResultOneStoppingTime} to any statistic
$\varphi:\mathbb{R}^{2d}\to\mathbb{R}$ satisfying
Assumption~\ref{ass: statistic}. This gives Theorem~\ref{thm: mainResultStatistic},
which is a reformulation of Theorem~\ref{thm: statistic}. The result applies,
in particular, to the risk and curvature curves,
$\mathcal{R}(x)$ and $\operatorname{tr}(\nabla^2\mathcal{R}(x))$, as well as
to other generalization metrics covered by Assumption~\ref{ass: statistic}.

In this section, the statistics $\varphi\colon \R^{2d} \to \R$ of interest satisfy a composite structure 
\begin{equation*}
\varphi(x) = g\left(\frac{1}{d} W(x)^\top q_1(\operatorname{diag}(u)) q_2(\operatorname{diag}(v)) q_4(K) W(x)\right) 
\end{equation*}
where $g \colon \R^{3 \times 3}\to \R$ is $\alpha$-pseudo-Lipschitz on $\mathcal{U}$ and $q_1, q_2, q_4$ are polynomials (see Assumption~\ref{ass: statistic}). Note that we assume that $g$ is evaluated on the real-valued matrix recovered by the
contour integral, which lies in $\mathcal U$ up to the stopping time.

Define
\[
\mathrsfs{Q}(t)
:=
\frac{1}{(2\pi)^4}
\oint_{\Gamma}
q_1(z_1)q_2(z_2)q_4(z_4)
\mathrsfs{S}(t,z)\,\mathrm{d}z,
\]
where $\mathrsfs{S}$ solves \eqref{eq: PDESystem}. The deterministic equivalent
of $\varphi(x_{\lfloor td \rfloor})$ and $\varphi(\mathscr{X}_t)$ is then
\begin{equation}\label{eq: defPhi}
\phi(t):=g(\mathrsfs{Q}(t)).
\end{equation}
Thus we state our concentration theorem for $\varphi(x_{\lfloor td \rfloor})$ and $\varphi(\mathscr{X}_t)$.
\begin{theorem}[Concentration of general statistics]\label{thm: mainResultStatistic}
Suppose the Assumptions of Theorem~\ref{thm: mainResultOneStoppingTime} hold. Suppose, in addition, the statistic satisfies a composite structure,
\begin{equation*}
\varphi(x) = g\left(\frac{1}{d} W(x)^\top q_1(\operatorname{diag}(u)) q_2(\operatorname{diag}(v)) q_4(K) W(x)\right) 
\end{equation*}
where $g \colon \R^{3 \times 3}\to \R$ is $\alpha$-pseudo-Lipschitz on $\mathcal{U}$ and $q_1, q_2, q_4$ are polynomials (see Assumption~\ref{ass: statistic}).  Then there is an $\varepsilon > 0$ so that for any $T, M , \eta> 0$ and $d$ sufficiently large, with overwhelming
probability,
\begin{align}
\begin{split}
\sup_{0\leqslant t \leqslant T\wedge \Theta_{M, \eta}^{S(x_{\lfloor td \rfloor}, \cdot), \mathrsfs{S}(t,\cdot)}} \abs{\varphi(x_{\lfloor td \rfloor}) - \phi(t)} &\leqslant d^{-\varepsilon},\\
\sup_{0\leqslant t \leqslant T\wedge \Theta_{M, \eta}^{S(\mathscr{X}_t,\cdot), \mathrsfs{S}(t,\cdot)}} \abs{\varphi(\mathscr{X}_t) - \phi(t)} &\leqslant d^{-\varepsilon},\\
\sup_{0\leqslant t \leqslant T\wedge \Theta_{M, \eta}^{S(x_{\lfloor td \rfloor}, \cdot), S(\mathscr{X}_t, \cdot)}} \abs{\varphi(x_{\lfloor td \rfloor}) - \varphi(\mathscr{X}_t)} &\leqslant d^{-\varepsilon},
\end{split}
\end{align}
where $\phi$ is defined in \eqref{eq: defPhi} and the stopping time $\Theta_{M,\eta}^{\mathscr{S}_1,\mathscr{S}_2}$ is defined in \eqref{eq: stoppingTimeOnlyOne}.
\end{theorem}
\begin{proof}
As in the proof of Theorem~\ref{thm: mainResultOneStoppingTime}, we define the stopping time 
$\tau_{M+1, \eta}^{\mathscr{S}_1, \mathscr{S}_2}$ as in \eqref{eq: coupledStoppingTime} and suppress
the notation by setting $\tau_{M+1, \eta}^{\mathscr{S}_1, \mathscr{S}_2} = \tau_{M, \eta}$. We will consider the case when $\mathscr{S}_1(t,\cdot) = S(x_{\lfloor td \rfloor},\cdot)$ and $\mathscr{S}_2(t,\cdot) = S(\mathscr{X}_t,\cdot)$. The other cases will follow by analogous proof.

By Proposition~\ref{prop: HSGDapprox}, we have that $S(\mathscr{X}_t,z)$ is an $(d^{-\tilde{\varepsilon}}, M+1, T)$-approximate solution with
overwhelming probability. Moreover, by Proposition~\ref{prop: SGDapprox}, the function $S(x_{\lfloor td \rfloor},z)$ is an $(d^{-\tilde{\varepsilon}}, M+1, T)$-approximate solution. (For the deterministic function $\mathrsfs{S}$, it is a $(0, M+1, T)$-approximate solution by definition.) We observe that
\begin{align*}
\frac{1}{(2 \pi)^4} \oint_\Gamma
q_1(z_1)q_2(z_2)q_4(z_4)
S(x_{\lfloor td\rfloor},z)\,\mathrm{d}z
&=
\frac{1}{d}
W_{\lfloor td\rfloor}^{\top}
q_1(\operatorname{diag}(u_{\lfloor td\rfloor}))
q_2(\operatorname{diag}(v_{\lfloor td\rfloor}))
q_4(K)
W_{\lfloor td\rfloor},\\[5pt]
\frac{1}{(2 \pi)^4} \oint_\Gamma q_1(z_1) q_2(z_2) q_4(z_4) S(\mathscr{X}_t,z) \, \mathrm{d}z &= \frac{1}{d}
\mathscr W_t^\top
q_1(\operatorname{diag}(\mathscr u_t))
q_2(\operatorname{diag}(\mathscr v_t))
q_4(K)
\mathscr W_t.
\end{align*}
We apply Proposition~\ref{prop: stabilityStatistics} to conclude that  there exists a $\varepsilon > 0$ such that
\begin{equation}
\sup_{0\leqslant t \leqslant T \wedge \tau_{M+1,0}} \abs{\varphi(x_{\lfloor td \rfloor}) - \varphi(\mathscr{X}_t)} \leqslant d^{-\varepsilon} \quad \text{w.o.p.}
\end{equation}
Using the same argument as in Theorem~\ref{thm: mainResultOneStoppingTime}, 
we can remove the stopping time $\tau_{M+1,0}$ and replace
it with $\Theta_{M,0}$ for sufficiently large $d$.
\end{proof}
Lastly we formulate an immediate corollary which follows directly from the proofs of Theorem~\ref{thm: mainResultStatistic} and Corollary~\ref{cor: boundedNConcentration}.

\begin{corollary}\label{cor: generalStatResult}
Suppose the Assumptions of Theorem~\ref{thm: mainResultStatistic} and Corollary~\ref{cor: boundedNConcentration} hold. For any fixed $T>0$, there exists $\varepsilon>0$ such that, for $d$
sufficiently large, with overwhelming probability,
\begin{equation}
\sup_{0\leqslant t \leqslant T} \abs{\varphi(x_{\lfloor td \rfloor}) - \phi(t)} \leqslant d^{-\varepsilon}\quad \text{and} \quad \sup_{0\leqslant t \leqslant T} \abs{\varphi(x_{\lfloor td \rfloor}) - \varphi(\mathscr{X}_t)} \leqslant d^{-\varepsilon},   
\end{equation}
and therefore
\begin{equation}
\sup_{0\leqslant t \leqslant T} \abs{\varphi(\mathscr{X}_t) - \phi(t)} \leqslant 2d^{-\varepsilon}.
\end{equation}
\end{corollary}
Theorem~\ref{thm: statistic} follows by applying Corollary~\ref{cor: generalStatResult} to the statistic $\varphi$ appearing in the theorem statement.

\section{SGD and Homogenized SGD are Approximate Solutions}\label{sec: approxSols}
In this section, we show that the resolvent statistic
$t\mapsto S(x_{\lfloor td\rfloor},z)$ associated with SGD and the corresponding
statistic $t\mapsto S(\mathscr X_t,z)$ associated with homogenized SGD satisfy the partial integro-differential equation
\eqref{eq: PDESystem} up to a small error. The argument uses a martingale method
in the spirit of diffusion approximation \citep{ethier1986markov}: after applying
a Doob decomposition to the statistic $S$, we show that the martingale terms are
negligible and that the drift terms match the operator $\mathscr F$.

Recall the matrix-valued statistics
\[
S(x,z)
=
\frac1d W(x)^\top\Omega(x,z)W(x)
\in\mathbb C^{3\times 3},
\qquad
B(x)
=
\frac1d W(x)^\top K W(x)
\in\mathbb R^{3\times 3},
\]
where $W(x)=[\,\psi(x)\mid\beta^*\mid\mathds 1_d\,]$ and
$\Omega(x,z)$ is defined in \eqref{eq: Omega}.

We first show that both homogenized SGD and SGD on $S(\cdot, z)$ are $(\varepsilon, M, T)$-approximate solutions as defined in Definition~\ref{def: approximateSolution}. Then by Proposition~\ref{prop: stability},  it is immediately implied that both homogenized SGD and SGD on $S(\cdot, z)$ are uniformly close. Finally, Proposition~\ref{prop: stabilityStatistics} establishes that the same hold for
any statistic $\varphi(x)$ satisfying Assumption~\ref{ass: statistic}. In order to show that both homogenized SGD and
SGD on $S(\cdot,z)$ are $(\varepsilon, M, T)$-approximate solutions, we perform a Doob decomposition for both
homogenized SGD and SGD and then show that both martingale terms are small.

To compare homogenized SGD with SGD, we rescale time by setting $k=\lfloor td \rfloor$. 
Thus, one unit of continuous time corresponds to $d$ SGD updates, and we write
\[
x_{td} = x_{\lfloor td \rfloor} \quad \text{(SGD)} \qquad \text{and} \qquad \mathscr{X}_t \quad \text{(HSGD)}.
\]

Throughout this section, expressions such as $S(x_{\lfloor td\rfloor},z)$ are understood under this identification, and scalar statistics are real-valued $\varphi:\mathbb{R}^{2d}\to\mathbb{R}$.

When applying the scalar It\^o, Taylor, Doob decomposition, or martingale
estimates to the complex-valued matrix statistic
$S(\cdot,z)\in\mathbb{C}^{3\times 3}$, we do so coordinatewise. Namely, for
$a,b\in\{1,2,3\}$ and fixed $z\in\Gamma$, define the real-valued functions
\[
S_{ab}^{\operatorname{Re},z}(x):=\operatorname{Re}S_{ab}(x,z),
\qquad
S_{ab}^{\operatorname{Im},z}(x):=\operatorname{Im}S_{ab}(x,z).
\]
All scalar identities and estimates below are applied to these real-valued
functions. The corresponding complex matrix-valued identities are then obtained
by recombining real and imaginary parts. For example,
\[
\bigl(\mathcal M_t(S)(z)\bigr)_{ab}
:=
\mathcal M_t\bigl(S_{ab}^{\operatorname{Re},z}\bigr)
+
i\,\mathcal M_t\bigl(S_{ab}^{\operatorname{Im},z}\bigr).
\]
Since there are only finitely many entries, passing from scalar estimates for
the real and imaginary parts to matrix estimates for $S$ only changes constants.

Our first argument is a net argument showing that we do not need to work with every $z\in \Gamma \subset \mathbb{C}^4$, but
only polynomially many in $d$. For this, recall $\Gamma$ defined in Remark~\ref{rem: fixedContour}. For a fixed $\delta > 0$, we say that $\Gamma_i^\delta$ is a $d^{-\delta}$-mesh of $\Gamma_i$ if $\Gamma_i^\delta \subset \Gamma_i$ and for every $z_i\in \Gamma_i$ there exists a $\bar{z}_i \in \Gamma_i^\delta$ such that $\abs{z_i - \bar{z}_i}<d^{-\delta}$. 
\begin{lemma}[Net argument]\label{lem: net_argument}
Fix $T,M>0$ and let $\delta>0$. For each $i=1,\ldots,4$, let
$\Gamma_i^\delta$ be a $d^{-\delta}$-mesh of $\Gamma_i$, chosen so that
\[
|\Gamma_i^\delta|\leqslant C_i d^\delta .
\]
Set
\[
\Gamma^\delta
:=
\Gamma_1^\delta\times\Gamma_2^\delta\times\Gamma_3^\delta\times\Gamma_4^\delta .
\]
Then
\[
|\Gamma^\delta|
=
\prod_{i=1}^4 |\Gamma_i^\delta|
\leqslant C_\Gamma d^{4\delta},
\]
where $C_\Gamma:=\prod_{i=1}^4 C_i$ depends only on the fixed contour $\Gamma$.

Let $S(t,z)$ denote either $S(x_{\lfloor td\rfloor},z)$ or
$S(\mathscr X_t,z)$, and suppose that
\begin{equation}\label{eq:net_argument_gamma_delta}
\sup_{0\leqslant t\leqslant \hat\tau_M(S)\wedge T}
\left\|
S(t,\cdot)-S(0,\cdot)
-
\int_0^t
\mathscr F\bigl(\cdot,S(s,\cdot)\bigr)\,\mathrm ds
\right\|_{\Gamma^\delta}
\leqslant \varepsilon,
\end{equation}
where
\[
\hat\tau_M(S)
:=
\inf\left\{
t\ge0:
\|S(t,\cdot)\|_\Gamma>M
\ \text{or}\
B(t,S)\notin\mathcal U
\right\},
\qquad
B(t,S)
:=
\frac{1}{(2\pi)^4}
\oint_\Gamma z_4 S(t,z)\,\mathrm dz .
\]
Then the approximate PDE residual bound ~\ref{cond:lipschitz} in
Definition~\ref{def: approximateSolution} holds on the full contour $\Gamma$
with error $\varepsilon+C d^{-\delta}$, that is,
\[
\sup_{0\leqslant t\leqslant \hat\tau_M(S)\wedge T}
\left\|
S(t,\cdot)-S(0,\cdot)
-
\int_0^t
\mathscr F\bigl(\cdot,S(s,\cdot)\bigr)\,\mathrm ds
\right\|_{\Gamma}
\le
\varepsilon+C d^{-\delta}.
\]
Here $C=C(M,T,\Gamma,\|K\|_{\mathrm{op}},\|\beta^*\|_\infty,
\bar\gamma,L(I),L(h))$ is positive and independent of $d$. Consequently, if $S$ also satisfies the derivative
regularity condition~\ref{cond:lipschitz}, then $S$ is an
$(\varepsilon+C d^{-\delta},M,T)$-approximate solution.
\end{lemma}
\begin{proof} We consider only $S(t,z) = S(\mathscr{X}_t, z)$ as the same argument will also hold for SGD. We also will always work with the stopped process, that is, $S(t \wedge \hat{\tau}_M, z)$, where $\hat{\tau}_M = \inf \{ t \geqslant 0 \, : \, \norm{S(t,\cdot)}_{\Gamma} \geqslant M\}$. To simplify the notation, we suppress the $\hat{\tau}_M$ and use $S(t,z)$. 

First, we state some resolvent identities. One such resolvent identity gives
\begin{equation} \label{eq:resolvent_identity}
\norm{R(z_i; D_i) - R(\bar{z}_i; D_i)}_{\mathrm{op}} \leqslant \abs{z_i-\bar{z}_i} \norm{R(z_i; D_i) R(\bar{z}_i; D_i)}_{\mathrm{op}} \quad \text{for any $z_i, \bar{z}_i \in \Gamma_i$.}
\end{equation}

Furthermore, by Neumann series and since $\abs{z_i} > \norm{D_i}_{\mathrm{op}}$ on each $\Gamma_i$, we know that
\begin{equation*}
R\left(z_i; D_i \right) = \left(z_i\cdot I_d - D_i\right)^{-1} = \frac{1}{z_i} \left(I_d - \frac{1}{z_i}D_i \right)^{-1} = \frac{1}{z_i} \sum_{j=0}^\infty \left( \frac{1}{z_i}D_i \right)^j,
\end{equation*}
so then
\begin{equation*}
\norm{R\left(z_i; D_i \right)}_{\mathrm{op}} \leqslant \frac{1}{\abs{z_i}} \sum_{j=0}^\infty \left( \frac{1}{\abs{z_i}}\norm{D_i}_{\mathrm{op}} \right)^j = \frac{1}{\abs{z_i}} \cdot \frac{1}{1 - \frac{1}{\abs{z_i}}\norm{D_i}_{\mathrm{op}}} = \frac{1}{\abs{z_i} - \norm{D_i}_{\mathrm{op}}} \leqslant C
\end{equation*}
and we immediately get 
\begin{equation}\label{eq:resolventBound}
\sup_{z_i \in \Gamma_i} \norm{R(z_i;D_i)}_{\mathrm{op}} \leqslant C(M,\|K\|_{\mathrm{op}},\|\beta^*\|_\infty).
\end{equation}

These bounds will be useful later in the proof. 

Next, with these bounds, we can get estimates on quantities involving $S(t, \cdot)$ where $t$ is fixed and $z$ varies. Fix $z \in \Gamma$ and let $\bar{z} \in \Gamma_{\delta}$ be such that $\abs{z_i-\bar{z}_i} < d^{-\delta}$ for each $i=1,\cdots,4$. Then, using Lemma~\ref{lem: normEquivalence} (and the stopping time $\hat{\tau}_M$), we obtain
\begin{align}
\begin{split} \label{eq:S_3}
\norm{S(t,z) - S(t,\bar{z})} 
&\leqslant \frac{C}{d} \norm{\mathscr{W}_t}^2  \sum_{i=1}^4
\abs{z_i-\bar{z}_i}  \norm{R(z_i; D_i)}_{\mathrm{op}} \norm{R(\bar{z}_i; D_i)}_{\mathrm{op}} \\
& \leqslant
C \norm{S(\mathscr{X}_t, \cdot)}_\Gamma d^{-\delta}\\
& \leqslant 
C \cdot M \cdot d^{-\delta},
\end{split}
\end{align}
where we used the boundedness of the product contour $\Gamma$ in the last inequality.

Similarly, for any $z \in \Gamma$, we have
\[
\norm{S(t,z)} \leqslant \frac{1}{d} \norm{\mathscr{W}_t}^2 \Pi_{i=1}^4 \norm{R(z_i; D_i)}_{\mathrm{op}} \leqslant C \norm{S(\mathscr{X}_t, \cdot)}_\Gamma \leqslant C \cdot M.  
\]

Since $S(t,z)$ is a product of resolvents and $t\leqslant \hat\tau_M$, all finitely many
$z$-derivatives appearing in the terms defining $\mathscr F$ are uniformly
bounded on $\Gamma$. Hence each term
\[
z^\zeta\left(\bigcirc_i\mathscr G_i\right)S(t,z)
\]
is Lipschitz in $z$ on $\Gamma$, with Lipschitz constant depending only on the
stopped bounds and the fixed contour. Thus, since $z, \bar{z} \in \Gamma$ and the contour $\Gamma$ is bounded,
\begin{equation}
\label{eq:S_2}
\norm{z^\zeta 
\left(\bigcirc_{i=1}^4 \mathscr{G}_i\right)
( S(s,z)) - \bar{z}^\zeta 
\left(\bigcirc_{i=1}^4 \mathscr{G}_i\right)
( S(s,\bar{z}))} \leqslant C \cdot M \cdot d^{-\delta},
\end{equation}
where we used equations \eqref{eq: compGBound} and \eqref{eq:S_3}.

Now we are ready to prove the main result of the proposition. For a fixed $t \leqslant \hat{\tau}_{M}$ and $z \in \Gamma$ with $\bar{z} \in \Gamma_{\delta}$ such that $|z_i-\bar{z}_i| \leqslant d^{-\delta}$ for each $i=1,\cdots, 4$,
\begin{equation}
\begin{aligned} \label{eq: S_1}
&\norm{ S(t,z) - 
 S(0,z) - \int_0^t \mathscr{F}(z, S(s, \cdot) ) \, \mathrm{d} s  } \\[10pt]
&\hspace{0.3in}
\leqslant
\norm{ S(t,z)- S(t, \bar{z}) } + \norm{ S(0, z) - S(0, \bar{z}) } + \int_0^t \norm{ \mathscr{F}(z, S(s,\cdot)) - \mathscr{F}(\bar{z}, S(s, \cdot)) } \, \mathrm{d} s 
\\[10pt]
&\hspace{0.4in} + \norm{ S(t,\bar{z}) - 
 S(0,\bar{z}) - \int_0^t \mathscr{F}(\bar{z}, S(s, \cdot) ) \, \mathrm{d} s  }
\\[10pt]
&\hspace{0.3in}
\leqslant C M d^{-\delta} \\[10pt]
&\hspace{0.4in}+ 2 \bar{\gamma} \int_0^t  \norm{H(B(s))} \cdot \norm{z^\zeta 
\left(\bigcirc_{i=1}^4 \mathscr{G}_i\right)
( S(s,z)) - \bar{z}^\zeta 
\left(\bigcirc_{i=1}^4 \mathscr{G}_i\right)
( S(s,\bar{z}))}  \, \mathrm{d} s
\\[10pt]
&\hspace{0.4in}+ C(Q) \bar{\gamma}^2 \int_0^t \abs{ I(B(s)) } \cdot \norm{z^\zeta 
\left(\bigcirc_{i=1}^4 \mathscr{G}_i\right)
( S(s,z)) - \bar{z}^\zeta 
\left(\bigcirc_{i=1}^4 \mathscr{G}_i\right)
( S(s,\bar{z}))}  \, \mathrm{d} s  \\[10pt]
&\hspace{0.4in} +\varepsilon, 
\end{aligned}
\end{equation}
where the omitted terms are bounded in the same way, since $\mathscr F$ is a finite
sum of terms of the two forms in \eqref{eq: PDEFormula}.

Here we used \eqref{eq:S_3} to bound the first two terms in the first inequality and $\varepsilon$ for the last term by the assumption \eqref{eq:net_argument_gamma_delta} in the statement. In the difference
$\mathscr F(z,S(s,\cdot))-\mathscr F(\bar z,S(s,\cdot))$, the coefficients
depending only on $s$ and on $\mathscr B(s)$ are identical. Thus the difference
is controlled by the variation in the monomial and coordinate-operator factor
as $z$ is replaced by $\bar z$.


As we have already shown that $z^\zeta 
\left(\bigcirc_{i=1}^4 \mathscr{G}_i\right)
( S(s,z))$ is Lipschitz in $z$, we only need to bound $\abs{I(B(s))}$ and $\norm{H(B(s))}$. We have already shown a uniform bound on $\abs{I(B(s))}$ and $\norm{H(B(s))}$ in the proof of Proposition~\ref{prop: stability}. Notably, we showed that for $s \leqslant \hat{\tau}_M$, we have that $\abs{I(B(s))} \leqslant C(L(I),M,\alpha)$ and $\norm{H(B(s))} \leqslant C(L(h),M,\alpha)$. 

Last, by taking the supremum over $z \in \Gamma$ and then the supremum over $0 \leqslant t \leqslant (\hat{\tau}_M \wedge T)$ on the left-hand-side of \eqref{eq: S_1} and then using the bounds \eqref{eq:S_3} and \eqref{eq:S_2}, yields the result. 
\end{proof}

It remains to verify the approximate-solution property for the two processes
\[
t\mapsto S(\mathscr X_t,z)
\qquad\text{and}\qquad
t\mapsto S(x_{\lfloor td\rfloor},z).
\]
We do this by applying a Doob decomposition to each process and showing that the resulting martingale terms are negligible.

To do so, it will be convenient to work directly with the stopped process $x_{t \wedge \hat{\tau}_M}$ on the iterates. Since $\hat{\tau}_M$ is a time based on $S$-values, it is often difficult to apply to iterates of SGD and homogenized SGD, so we introduce equivalent stopping times
\begin{align}
\begin{split}
\vartheta_M &= \inf \{t \geqslant 0 : \frac{1}{d}\norm{W_{\lfloor td \rfloor}}^2 > M \text{ or } B\left(x_{\lfloor td\rfloor}\right) \notin \mathcal{U}\} \quad \text{or}\\
\vartheta_M &= \inf \{t \geqslant 0 : \frac{1}{d}\norm{\mathscr{W}_t}^2 > M \text{ or } B\left(\mathscr{X}_t\right) \notin \mathcal{U} \}.
\end{split}
\end{align} 
We overload the notation $\vartheta_M$ to be either applied to SGD iterates, $x_{\lfloor td \rfloor}$ or homogenized SGD iterates, $\mathscr{X}_t$, for which it will be made clear in the context which criterion is used. These stopping times are equivalent to $\hat{\tau}_M$ (see Lemma~\ref{lem: normEquivalence}). Moreover, we often drop the $M$ so that $\vartheta = \vartheta_M$. It will be convenient to work with the stopped processes, $x_{\lfloor td \rfloor}^{\vartheta} = x_{\lfloor d(t\wedge \vartheta)\rfloor}$ and $\mathscr{X}_t^{\vartheta} = \mathscr{X}_{t \wedge \vartheta}$. 

\subsection{Homogenized SGD Under the Resolvent Statistic \texorpdfstring{$S$}{S}}
We first verify the approximate-solution property for the homogenized process. 
Specifically, we show that the resolvent statistic
$t\mapsto S(\mathscr X_t,z)$ satisfies the partial integro-differential equation
\eqref{eq: PDESystem} up to a martingale error, and that this martingale error is negligible uniformly on the fixed contour.


With this, we recall \textit{homogenized SGD}
\begin{equation*}
\mathrm{d} \mathscr{X}_t = - \gamma(t) d \nabla \mathcal{R}(\mathscr{X}_t) \mathrm{d}t + \gamma(t) \sqrt{I\left(B(\mathscr{X}_t) \right)} \, \left(\nabla \psi(\mathscr{X}_t)\right)^\top  \sqrt{K}\, \mathrm{d}\mathfrak{B}_t,
\end{equation*}
where $\mathscr{X}_t$ is a stochastic process taking values in $\R^{2d}$ with initial conditions $\mathscr{X}_0 = x_0$, and $\mathrm{d}\mathfrak{B}_t$ is the differential of a standard Brownian motion in $\R^d$.

Along the homogenized trajectory, write
\begin{equation*}
\mathscr{W}_t := [\, \psi \left(\mathscr{X}_t\right) \, |  \,  \beta^* \, | \, \mathds{1}_d \, ]\in \R^{d\times 3}, \quad \text{and} \quad \rho_t := \frac{1}{\sqrt{d}} \begin{pmatrix}
\psi(\mathscr{X}_t) \\
\beta^*
\end{pmatrix}^{\!\top} a
\in \R^2.
\end{equation*}
We study the homogenized process through the resolvent statistic
\begin{equation*}
x \in \R^{2d} \mapsto S(x,z) = \frac{1}{d} W(x)^\top \Omega(x,z) W(x) \in \mathbb{C}^{3 \times 3} \quad \text{for} \quad z \in \Gamma \subset \mathbb{C}^4.
\end{equation*}
We will show that $S\left(\mathscr{X}_t, z\right)$ is an approximate solution \eqref{def: approximateSolution} to the partial integro-differential equation \eqref{eq: PDESystem} which we state below.

\begin{proposition}[Homogenized SGD is an approximate solution]\label{prop: HSGDapprox}
Fix $T, M > 0$ and $0 < \delta < 1/2$. Then $S\left(\mathscr{X}_t, z\right)$ is a $(d^{-\delta}, M, T)$-approximate solution w.o.p., that is, 
\begin{align*}
\sup_{0\leqslant t \leqslant (\hat{\tau}_M \wedge T)} \norm{S(\mathscr{X}_t, \cdot) - S(x_0, \cdot) - \int_0^t \mathscr{F}\left(\cdot, S(\mathscr{X}_s, \cdot)\right) \mathrm{d}s}_\Gamma &\leqslant d^{-\delta} \quad \text{w.o.p.}
\end{align*}
\end{proposition}
The proof of this Proposition is deferred to Section~\ref{subsubsec: HSGDproof}.

\subsubsection{Doob Decomposition for Homogenized SGD} 
We begin by applying It\^o calculus to homogenized SGD under smooth statistics
$\varphi:\mathbb{R}^{2d}\to\mathbb{R}$. Later, when applying the resulting
identity to $S(\cdot,z)$, we apply it separately to
$S_{ab}^{\operatorname{Re},z}$ and $S_{ab}^{\operatorname{Im},z}$ for each
matrix entry.


Applying It\^o's lemma, we deduce that 
\begin{align}
\begin{split}
\mathrm{d} \varphi(\mathscr{X}_t) 
&= \inner{\nabla \varphi(\mathscr{X}_t), \mathrm{d} \mathscr{X}_t} + \frac{1}{2}\inner{\nabla^2 \varphi(\mathscr{X}_t), (\mathrm{d} \mathscr{X}_t)^{\otimes 2} } \\
&= - \gamma(t) d \inner{ \nabla \varphi(\mathscr{X}_t), \nabla \mathcal{R}(\mathscr{X}_t)}\, \mathrm{d}t + \gamma(t) \sqrt{I\left(B(\mathscr{X}_t) \right)}\inner{\nabla \varphi(\mathscr{X}_t), \left(\nabla \psi(\mathscr{X}_t)\right)^\top  \sqrt{K}\, \mathrm{d}\mathfrak{B}_t} \\
&+ \frac{\gamma(t)^2}{2}I\left(B(\mathscr{X}_t) \right)\inner{\nabla^2 \varphi(\mathscr{X}_t), \left( \left(\nabla \psi(\mathscr{X}_t)\right)^\top  \sqrt{K}\, \mathrm{d}\mathfrak{B}_t\right)^{\otimes 2} }.
\end{split}
\label{eq:HSGD_2}
\end{align}
We seek to simplify some of the terms in \eqref{eq:HSGD_2}. For this, we flatten the second term in sum:
\begin{equation}
\inner{\nabla \varphi(\mathscr{X}_t), \left(\nabla \psi(\mathscr{X}_t)\right)^\top  \sqrt{K}\, \mathrm{d}\mathfrak{B}_t} = \inner{\left(\nabla \psi(\mathscr{X}_t)\right)^\top  \sqrt{K}, \nabla \varphi(\mathscr{X}_t) \left(\mathrm{d}\mathfrak{B}_t\right)^\top}_{\R^{2d\times d}}.
\label{eq:martingale_HSGD_1}
\end{equation}
Next, we look at the second derivative term of $\varphi$,
\begin{equation}
\label{eq:Hessian_HSGD_1}
\inner{\nabla^2 \varphi(\mathscr{X}_t), \left( \left(\nabla \psi(\mathscr{X}_t)\right)^\top  \sqrt{K}\, \mathrm{d}\mathfrak{B}_t\right)^{\otimes 2} } = \inner{\nabla^2\varphi(\mathscr X_t),
(\nabla\psi(\mathscr X_t))^\top K\nabla\psi(\mathscr X_t)}
\,\mathrm dt,
\end{equation}
where we used the symmetry of $\sqrt{K}$. With this, we can now identify the martingale increment for homogenized SGD, 
\begin{align*}
\mathrm{d} \varphi(\mathscr{X}_t) = &- \gamma(t) d \inner{ \nabla \varphi(\mathscr{X}_t), \nabla \mathcal{R}(\mathscr{X}_t)}\, \mathrm{d}t \\
&+ \frac{\gamma(t)^2}{2}I\left(B(\mathscr{X}_t) \right)\inner{\nabla^2\varphi(\mathscr X_t),
(\nabla\psi(\mathscr X_t))^\top K\nabla\psi(\mathscr X_t)}\, \mathrm{d}t + \mathrm{d}\mathcal{M}_t^\mathrm{HSGD}(\varphi),
\end{align*}
where $\mathrm{d}\mathcal{M}_t^\mathrm{HSGD}(\varphi) := \gamma(t) \sqrt{I\left(B(\mathscr{X}_t) \right)}\inner{\left(\nabla \psi(\mathscr{X}_t)\right)^\top  \sqrt{K}, \nabla \varphi (\mathscr{X}_t) \left(\mathrm{d}\mathfrak{B}_t\right)^\top}_{\R^{2d\times d}}$. By integrating, we derive the Doob decomposition for $\varphi(\mathscr{X}_t)$
\begin{align}
\begin{split}
\varphi(\mathscr{X}_t) 
&
= \varphi(x_0) - \int_0^t \gamma(s)d \inner{ \nabla \varphi(\mathscr{X}_s), \nabla \mathcal{R}(\mathscr{X}_s)}\, \mathrm{d}s
\\
&+\frac{1}{2} \int_0^t \gamma(s)^2 I\left(B(\mathscr{X}_s) \right) \inner{\nabla^2\varphi(\mathscr X_s),
(\nabla\psi(\mathscr X_s))^\top K\nabla\psi(\mathscr X_s)}\, \mathrm{d}s + \int_0^t \mathrm{d} \mathcal{M}_s^{\mathrm{HSGD}}(\varphi). 
\end{split}
\end{align}
\subsubsection{\texorpdfstring{$S(\mathscr{X}_t, z)$}{S(Xt, z)} is an Approximate Solution, Proof of Proposition~\ref{prop: HSGDapprox}}\label{subsubsec: HSGDproof}
The goal in this section is to prove Proposition~\ref{prop: HSGDapprox}, that is, show that
\begin{equation*}
S(\mathscr{X}_t, z) = \frac{1}{d} \mathscr W_t^\top \Omega(x,z) \mathscr W_t \in \mathbb{C}^{3 \times 3}
\end{equation*}
is an approximate solution \eqref{def: approximateSolution} to the partial integro-differential equation in \eqref{eq:  PDESystem}.

The first step is to derive a closed equation for \(S(\mathscr X_t,z)\) using Itô calculus.

\textbf{It\^o calculus applied to $S(\mathscr{X}_t, z)$.}
Recall the expected risk $\mathcal{R}$ can be expressed as a composition, $\mathcal{R}(\mathscr{X}_t) = h \left( B(\mathscr{X}_t)\right)$, for some function $h \colon \R^{3 \times 3}\to \R$ and 
\begin{equation*}
B(\mathscr{X}_t) = \frac{1}{d} \mathscr W_t^\top K \mathscr W_t \in \R^{3\times 3}.    
\end{equation*}
A straightforward application of the chain rule shows that
\begin{align*}
\R^{2d} &\ni \nabla \mathcal{R}(\mathscr{X}_t) = \inner{\nabla B(\mathscr{X}_t), \nabla h\left(B(\mathscr{X}_t)\right)}_{\R^{3 \times 3}} \\[10pt]
&= \frac{2}{d} \left( \nabla h_{11} \cdot \left(\nabla \psi(\mathscr{X}_t)\right)^\top K \psi(\mathscr{X}_t) + \nabla h_{21} \cdot \left(\nabla \psi(\mathscr{X}_t)\right)^\top K \beta^* + \nabla h_{31} \cdot \left(\nabla \psi(\mathscr{X}_t)\right)^\top K \mathds{1}_d \right).
\end{align*}
Using the product rule for It\^o
derivatives, we obtain
\begin{align}\label{eq: ItoS}
\begin{split}
\mathrm{d} S = 
&- \gamma(t) d \inner{ \nabla S(\mathscr{X}_t, z), \nabla \mathcal{R}(\mathscr{X}_t)}\, \mathrm{d}t + \mathrm{d}\mathcal{M}_t^\mathrm{HSGD}(S)\\[8pt]
&+ \frac{\gamma(t)^2}{2}I\left(B(\mathscr{X}_t) \right)\inner{\nabla^2S(\mathscr X_t,z),
(\nabla\psi(\mathscr X_t))^\top K\nabla\psi(\mathscr X_t)}\, \mathrm{d}t \\[8pt]
=&- \gamma(t) d\left(\inner{ \nabla_u S(\mathscr{X}_t, z), \nabla_u \mathcal{R}(\mathscr{X}_t)}+ \inner{ \nabla_v S(\mathscr{X}_t, z), \nabla_v \mathcal{R}(\mathscr{X}_t)}\right) \mathrm{d}t + \mathrm{d}\mathcal{M}_t^\mathrm{HSGD}(S)\\[8pt]
&+ \frac{\gamma(t)^2}{2} I\left(B(\mathscr{X}_t)\right) \Big[\inner{\nabla_u^2 S(\mathscr{X}_t, z), K \left( \nabla_u \psi(\mathscr X_t) \right)^2} \\[8pt]
&+ 2 \inner{\nabla_{uv}^2 S(\mathscr{X}_t, z), \left( \nabla_u \psi(\mathscr X_t) \right)^\top K \nabla_v \psi(\mathscr X_t)} +  \inner{\nabla_v^2 S(\mathscr{X}_t, z), K \left( \nabla_v \psi(\mathscr X_t) \right)^2}\Big] \mathrm{d}t. \\[8pt]
\end{split}
\end{align}

\begin{remark}[Applying scalar formulas to $S$]\upright\label{rem: matrixEntriesV}
The statistic $S(x,z)$ is complex-valued:
\[
S(x,z)\in\mathbb{C}^{3\times 3}.
\]
However, all scalar statistic formulas above were stated for real-valued
functions $\varphi:\mathbb{R}^{2d}\to\mathbb{R}$. Therefore, whenever we apply
these formulas to $S$, we apply them to the real-valued coordinate functions
\[
S_{ab}^{\operatorname{Re},z}(x)
:=
\operatorname{Re}S_{ab}(x,z),
\qquad
S_{ab}^{\operatorname{Im},z}(x)
:=
\operatorname{Im}S_{ab}(x,z),
\]
for each $a,b\in\{1,2,3\}$ and fixed $z\in\Gamma$.

For example, we define the complex-valued HSGD martingale associated with $S$
entrywise by
\[
\bigl(\mathrm d\mathcal M_t^{\mathrm{HSGD}}(S)(z)\bigr)_{ab}
:=
\mathrm d\mathcal M_t^{\mathrm{HSGD}}
\bigl(S_{ab}^{\operatorname{Re},z}\bigr)
+
i\,\mathrm d\mathcal M_t^{\mathrm{HSGD}}
\bigl(S_{ab}^{\operatorname{Im},z}\bigr)
\]
for
\[
\mathrm{d}\mathcal{M}_t^\mathrm{HSGD}(S_{ab}^{\operatorname{Re},z}) :=  \gamma(t) \sqrt{I\left(B(\mathscr{X}_t) \right)}\inner{\left(\nabla \psi(\mathscr{X}_t)\right)^\top \sqrt{K}, \nabla S_{ab}^{\operatorname{Re},z} (\mathscr{X}_t, z) \left(\mathrm{d}\mathfrak{B}_t\right)^\top}_{\R^{2d\times d}},
\]
and analogously for $S_{ab}^{\operatorname{Im},z}$.

Thus matrix-valued stochastic identities involving $S$ are shorthand for the
collection of real-valued identities for the real and imaginary parts of its
entries.

Similarly, we define
\begin{equation*}
\mathcal{M}_t^\mathrm{HSGD}(S) = \int_0^t \mathrm{d} \mathcal{M}_s^\mathrm{HSGD}(S).
\end{equation*}
\end{remark}

We consider the first term in the summation above, and plugging in $\nabla \mathcal{R}$ and using Lemma~\ref{lem: derivativeS}, we have
\allowdisplaybreaks
\begin{align*}
\mathbb{C}^{3 \times 3} \ni \inner{ \nabla_u S(\mathscr{X}_t, z), \nabla_u \mathcal{R}(\mathscr{X}_t)}_{\R^{d}} &= \frac{2}{d^2} ( \nabla h_{11} \cdot \psi(\mathscr{X}_t) + \nabla h_{21} \cdot \beta^* + \nabla h_{31} \cdot \mathds{1}_d ) ^\top \cdot \\
&\hspace{1.0in}\left(\nabla_u \psi(\mathscr{X}_t)\right)^2 K \Omega \begin{bmatrix}
 \psi(\mathscr{X}_t) &  \beta^* & \mathds{1}_d\\
0 & 0 & 0\\
0 & 0 & 0
\end{bmatrix}\\
&+ \frac{2}{d^2} ( \nabla h_{11} \cdot \psi(\mathscr{X}_t) + \nabla h_{21} \cdot \beta^* + \nabla h_{31} \cdot \mathds{1}_d ) ^\top \cdot\\
&\hspace{1.0in}\left(\nabla_u \psi(\mathscr{X}_t)\right)^2 K \Omega \begin{bmatrix}
 \psi(\mathscr{X}_t) &  0 & 0\\
\beta^* & 0 & 0\\
\mathds{1}_d & 0 & 0
\end{bmatrix}\\
&+ \frac{2}{d^2} ( \nabla h_{11} \cdot \psi(\mathscr{X}_t) + \nabla h_{21} \cdot \beta^* + \nabla h_{31} \cdot \mathds{1}_d ) ^\top \cdot \\
&\hspace{-0.5in}\nabla_u \psi(\mathscr{X}_t) \operatorname{diag}(\psi(\mathscr{X}_t)) K R\left(z_1; \operatorname{diag}(\mathscr{U}_t) \right) \Omega \begin{bmatrix}
 \psi(\mathscr{X}_t) &  \beta^* & \mathds{1}_d  \\
 0 & 0 & 0\\
 0 & 0 & 0
\end{bmatrix}\\
&+ \frac{2}{d^2} ( \nabla h_{11} \cdot \psi(\mathscr{X}_t) + \nabla h_{21} \cdot \beta^* + \nabla h_{31} \cdot \mathds{1}_d ) ^\top \cdot \\
&\hspace{-0.3in}\nabla_u \psi(\mathscr{X}_t) \operatorname{diag}(\beta^*) K R\left(z_1; \operatorname{diag}(\mathscr{U}_t) \right) \Omega \begin{bmatrix}
0 & 0 & 0\\
 \psi(\mathscr{X}_t) &  \beta^* & \mathds{1}_d\\
 0 & 0 & 0
\end{bmatrix}\\
&+ \frac{2}{d^2} ( \nabla h_{11} \cdot \psi(\mathscr{X}_t) + \nabla h_{21} \cdot \beta^* + \nabla h_{31} \cdot \mathds{1}_d ) ^\top \cdot \\
&\hspace{0.2in}\nabla_u \psi(\mathscr{X}_t) K R\left(z_1; \operatorname{diag}(\mathscr{U}_t) \right) \Omega \begin{bmatrix}
0 & 0 & 0\\
 0 & 0 & 0\\
 \psi(\mathscr{X}_t) &  \beta^* & \mathds{1}_d
 \end{bmatrix}\\
&= \frac{2}{d} H^\top \cdot \frac{1}{d}\mathscr{W}_t^\top \left(\nabla_u \psi(\mathscr{X}_t)\right)^2 K \Omega \mathscr{W}_t\\
&+ \frac{2}{d^2} \mathscr{W}_t^\top \left(\nabla_u \psi(\mathscr{X}_t)\right)^2 K \Omega \mathscr{W}_t \cdot H\\
&\hspace{-0.6in}+ \frac{2}{d} H^\top \cdot \frac{1}{d}\mathscr{W}_t^\top \nabla_u \psi(\mathscr{X}_t) \operatorname{diag}(\psi(\mathscr{X}_t)) K R\left(z_1; \operatorname{diag}(\mathscr{U}_t) \right) \Omega \mathscr{W}_t\\
&\hspace{-0.6in}+ \frac{2}{d} E_{21} \cdot H^\top \cdot \frac{1}{d}\mathscr{W}_t^\top \nabla_u \psi(\mathscr{X}_t) \operatorname{diag}(\beta^*) K R\left(z_1; \operatorname{diag}(\mathscr{U}_t) \right) \Omega \mathscr{W}_t\\
&\hspace{-0.6in}+ \frac{2}{d} E_{31} \cdot H^\top \cdot \frac{1}{d}\mathscr{W}_t^\top \nabla_u \psi(\mathscr{X}_t) K R\left(z_1; \operatorname{diag}(\mathscr{U}_t) \right) \Omega \mathscr{W}_t,\\[20pt]
\mathbb{C}^{3 \times 3} \ni \inner{ \nabla_v S(\mathscr{X}_t, z), \nabla_v \mathcal{R}(\mathscr{X}_t)}_{\R^{d}} &= \frac{2}{d} H^\top \cdot \frac{1}{d}\mathscr{W}_t^\top \left(\nabla_v \psi(\mathscr{X}_t)\right)^2 K \Omega \mathscr{W}_t\\
&+ \frac{2}{d^2} \mathscr{W}_t^\top \left(\nabla_v \psi(\mathscr{X}_t)\right)^2 K \Omega \mathscr{W}_t \cdot H\\
&\hspace{-0.6in}+ \frac{2}{d} H^\top \cdot \frac{1}{d}\mathscr{W}_t^\top \nabla_v \psi(\mathscr{X}_t) \operatorname{diag}(\psi(\mathscr{X}_t)) K R\left(z_2; \operatorname{diag}(\mathscr{V}_t) \right) \Omega \mathscr{W}_t\\
&\hspace{-0.6in}+ \frac{2}{d} E_{21} \cdot H^\top \cdot \frac{1}{d}\mathscr{W}_t^\top \nabla_v \psi(\mathscr{X}_t) \operatorname{diag}(\beta^*) K R\left(z_2; \operatorname{diag}(\mathscr{V}_t) \right) \Omega \mathscr{W}_t\\
&\hspace{-0.6in}+ \frac{2}{d} E_{31} \cdot H^\top \cdot \frac{1}{d}\mathscr{W}_t^\top \nabla_v \psi(\mathscr{X}_t) K R\left(z_2; \operatorname{diag}(\mathscr{V}_t) \right) \Omega \mathscr{W}_t,
\end{align*}
for 
\begin{equation*}
H\left(B(\mathscr{X}_t)\right) = \left[
\begin{array}{c|c|c}
\nabla_{11} h & 0 & 0 \\ \hline
\nabla_{21} h & 0 & 0 \\ \hline
\nabla_{31} h & 0 & 0
\end{array}\right].
\end{equation*}
Similarly for the second term we have
\allowdisplaybreaks
\begin{align*}
\mathbb{C}^{3\times 3} \ni \inner{\nabla_u^2 S(\mathscr{X}_t, z), K \left( \nabla_u \psi(x) \right)^2}_{\R^{d\times d}} &= \frac{2 q_{11}}{d} \mathds{1}_d^\top \cdot  \left( \nabla_u \psi(x) \right)^2 K \Omega \begin{bmatrix}
\psi(\mathscr{X}_t) &  \beta^* & \mathds{1}_d\\
0 & 0 & 0\\
0 & 0 & 0
\end{bmatrix}\\
&\hspace{0.1in}+ \frac{2 q_{11}}{d} \mathds{1}_d^\top \cdot \left( \nabla_u \psi(x) \right)^2 K \Omega \begin{bmatrix}
\psi(\mathscr{X}_t) &  0 & 0\\
\beta^* & 0 & 0\\
\mathds{1}_d & 0 & 0
\end{bmatrix}\\
&\hspace{-0.75in}+ \frac{2}{d} \mathds{1}_d^\top \cdot \left( \nabla_u \psi(x) \right)^3 K R\left(z_1; \operatorname{diag}(\mathscr{U}_t) \right) \Omega \begin{bmatrix}
\psi(\mathscr{X}_t) & \beta^* & \mathds{1}_d\\
0 & 0 & 0\\
0 & 0 & 0
\end{bmatrix} \\
&\hspace{-0.65in}+ \frac{2}{d} \mathds{1}_d^\top \cdot \left( \nabla_u \psi(x) \right)^3 K R\left(z_1; \operatorname{diag}(\mathscr{U}_t) \right) \Omega \begin{bmatrix}
\psi(\mathscr{X}_t) & 0 & 0\\
\beta^* & 0 & 0\\
\mathds{1}_d & 0 & 0
\end{bmatrix} \\
&\hspace{0.1in}+ \frac{2}{d} \begin{bmatrix}
\tr\left( \left( \nabla_u \psi(x) \right)^4 K \Omega\right) &  0 & 0\\
0 & 0 & 0\\
0 & 0 & 0
\end{bmatrix}\\
&\hspace{-1.2in}+ \frac{2}{d} ( \psi(\mathscr{X}_t))^\top \cdot \left( \nabla_u \psi(x) \right)^2 K
R\left(z_1; \operatorname{diag}(\mathscr{U}_t) \right)^2 \Omega \begin{bmatrix}
\psi(\mathscr{X}_t) &  \beta^* & \mathds{1}_d \\
 0 & 0 & 0\\
 0 & 0 & 0
\end{bmatrix}\\
&\hspace{-1.0in}+ \frac{2}{d} (\beta^*)^\top \cdot \left( \nabla_u \psi(x) \right)^2 K
R\left(z_1; \operatorname{diag}(\mathscr{U}_t) \right)^2 \Omega \begin{bmatrix}
 0 & 0  & 0\\
 \psi(\mathscr{X}_t)  &  \beta^* & \mathds{1}_d\\
 0 & 0  & 0
\end{bmatrix}\\
&\hspace{-0.8in}+ \frac{2}{d} \mathds{1}_d^\top \cdot \left( \nabla_u \psi(x) \right)^2 K
R\left(z_1; \operatorname{diag}(\mathscr{U}_t) \right)^2 \Omega \begin{bmatrix}
 0 & 0  & 0\\
 0 & 0  & 0\\
 \psi(\mathscr{X}_t)  &  \beta^* & \mathds{1}_d
\end{bmatrix}\\[15pt]
&= 2 q_{11} E_{13} \cdot \frac{1}{d}\mathscr{W}_t^\top \left(\nabla_u \psi(\mathscr{X}_t)\right)^2 K\Omega \mathscr{W}_t\\
&+ 2 q_{11} \frac{1}{d}\mathscr{W}_t^\top \left(\nabla_u \psi(\mathscr{X}_t)\right)^2 K\Omega \mathscr{W}_t \cdot E_{31}\\
&+ 2 E_{13} \cdot \frac{1}{d}\mathscr{W}_t^\top \left(\nabla_u \psi(\mathscr{X}_t)\right)^3 K R\left(z_1; \operatorname{diag}(\mathscr{U}_t) \right) \Omega \mathscr{W}_t\\
&+ 2 \frac{1}{d}\mathscr{W}_t^\top \left(\nabla_u \psi(\mathscr{X}_t)\right)^3 K R\left(z_1; \operatorname{diag}(\mathscr{U}_t) \right) \Omega \mathscr{W}_t \cdot E_{31}\\
&+ 2 E_{13} \cdot \frac{1}{d}\mathscr{W}_t^\top \left(\nabla_u \psi(\mathscr{X}_t)\right)^4 K\Omega \mathscr{W}_t \cdot E_{31}\\
&+ 2 \frac{1}{d}\mathscr{W}_t^\top \left( \nabla_u \psi(x) \right)^2 K
R\left(z_1; \operatorname{diag}(\mathscr{U}_t) \right)^2 \Omega \mathscr{W}_t,
\end{align*}
\begin{align*}
\mathbb{C}^{3\times 3} \ni \inner{\nabla_{uv}^2 S(\mathscr{X}_t, z), \left( \nabla_u \psi(x) \right)^\top K \nabla_v \psi(x)}_{\R^{d\times d}} &=\\
&\hspace{-0.2in}2 q_{12} E_{13} \cdot \frac{1}{d}\mathscr{W}_t^\top 
\left(\nabla_u \psi(\mathscr{X}_t)\right) \left(\nabla_v \psi(\mathscr{X}_t)\right) K\Omega \mathscr{W}_t\\
&\hspace{-0.4in}+ 2 q_{12} \frac{1}{d}\mathscr{W}_t^\top \left(\nabla_u \psi(\mathscr{X}_t)\right) \left(\nabla_v \psi(\mathscr{X}_t)\right) K\Omega \mathscr{W}_t \cdot E_{31}\\
&\hspace{-1.1in}+ E_{13} \cdot \frac{1}{d}\mathscr{W}_t^\top \left(\nabla_u \psi(\mathscr{X}_t)\right) \left(\nabla_v \psi(\mathscr{X}_t)\right)^2 K R\left(z_1; \operatorname{diag}(\mathscr{U}_t) \right) \Omega \mathscr{W}_t\\
&\hspace{-1.1in}+ \frac{1}{d}\mathscr{W}_t^\top \left(\nabla_u \psi(\mathscr{X}_t)\right) \left(\nabla_v \psi(\mathscr{X}_t)\right)^2 K R\left(z_1; \operatorname{diag}(\mathscr{U}_t) \right) \Omega \mathscr{W}_t \cdot E_{31}\\
&\hspace{-1.1in}+ E_{13} \cdot \frac{1}{d}\mathscr{W}_t^\top \left(\nabla_u \psi(\mathscr{X}_t)\right)^2 \left(\nabla_v \psi(\mathscr{X}_t)\right) K R\left(z_2; \operatorname{diag}(\mathscr{V}_t) \right) \Omega \mathscr{W}_t\\
&\hspace{-1.1in}+ \frac{1}{d}\mathscr{W}_t^\top \left(\nabla_u \psi(\mathscr{X}_t)\right)^2 \left(\nabla_v \psi(\mathscr{X}_t)\right) K R\left(z_2; \operatorname{diag}(\mathscr{V}_t) \right) \Omega \mathscr{W}_t \cdot E_{31}\\
&\hspace{-0.6in}+ 2 E_{13} \cdot \frac{1}{d}\mathscr{W}_t^\top \left(\nabla_u \psi(\mathscr{X}_t)\right)^2 \left(\nabla_v \psi(\mathscr{X}_t)\right)^2 K\Omega \mathscr{W}_t \cdot E_{31}\\
&\hspace{-1.7in}+ \frac{1}{d}\mathscr{W}_t^\top \left(\nabla_u \psi(\mathscr{X}_t)\right) \left(\nabla_v \psi(\mathscr{X}_t)\right) K
R\left(z_1; \operatorname{diag}(\mathscr{U}_t) \right) R\left(z_2; \operatorname{diag}(\mathscr{V}_t) \right) \Omega \mathscr{W}_t,
\end{align*}
\begin{align*}
\mathbb{C}^{3\times 3} \ni \inner{\nabla_v^2 S(\mathscr{X}_t, z), K \left( \nabla_v \psi(x) \right)^2}_{\R^{d\times d}} &= 2 q_{22} E_{13} \cdot \frac{1}{d}\mathscr{W}_t^\top \left(\nabla_v \psi(\mathscr{X}_t)\right)^2 K\Omega \mathscr{W}_t\\
&+ 2 q_{22} \frac{1}{d}\mathscr{W}_t^\top \left(\nabla_v \psi(\mathscr{X}_t)\right)^2 K\Omega \mathscr{W}_t \cdot E_{31}\\
&+ 2 E_{13} \cdot \frac{1}{d}\mathscr{W}_t^\top \left(\nabla_v \psi(\mathscr{X}_t)\right)^3 K R\left(z_2; \operatorname{diag}(\mathscr{V}_t) \right) \Omega \mathscr{W}_t\\
&+ 2 \frac{1}{d}\mathscr{W}_t^\top \left(\nabla_v \psi(\mathscr{X}_t)\right)^3 K R\left(z_2; \operatorname{diag}(\mathscr{V}_t) \right) \Omega \mathscr{W}_t \cdot E_{31}\\
&+ 2 E_{13} \cdot \frac{1}{d}\mathscr{W}_t^\top \left(\nabla_v \psi(\mathscr{X}_t)\right)^4 K\Omega \mathscr{W}_t \cdot E_{31}\\
&+ 2 \frac{1}{d}\mathscr{W}_t^\top \left( \nabla_v \psi(x) \right)^2 K
R\left(z_2; \operatorname{diag}(\mathscr{V}_t) \right)^2 \Omega \mathscr{W}_t.
\end{align*}
Now recall that 
\begin{equation*}
\nabla \psi(\mathscr{X}_t) = \left[
    2q_{11}\operatorname{diag}(\mathscr{U}_t)  + 2q_{12}\operatorname{diag}(\mathscr{V}_t) + l_1 I_d \; 
    2q_{12}\operatorname{diag}(\mathscr{U}_t) + 2q_{22}\operatorname{diag}(\mathscr{V}_t) + l_2 I_d
\right] \in \R^{d\times 2d}.
\end{equation*}

Accordingly, each term in \eqref{eq: ItoS} takes a form such as
\begin{align*}
&-2 \gamma(t) H^\top \cdot \frac{1}{d} \mathscr{W}_t^\top 
\operatorname{diag}(\mathscr{U}_t)^{m_1} 
R\big(z_1; \operatorname{diag}(\mathscr{U}_t)\big)^{\rho_1} \\
&\quad \cdot \operatorname{diag}(\mathscr{V}_t)^{m_2} 
R\big(z_2; \operatorname{diag}(\mathscr{V}_t)\big)^{\rho_2} 
\operatorname{diag}(\beta^*)^{m_3} 
R\big(z_3; \operatorname{diag}(\beta^*)\big) K  
R\big(z_4; K \big) \mathscr{W}_t,\\[2mm]
&\text{or} \quad C(Q) \gamma(t)^2 I\big(B(\mathscr{X}_t)\big) \cdot \frac{1}{d} \mathscr{W}_t^\top
\operatorname{diag}(\mathscr{U}_t)^{m_1} 
R\big(z_1; \operatorname{diag}(\mathscr{U}_t)\big)^{\rho_1} \\
&\quad \cdot \operatorname{diag}(\mathscr{V}_t)^{m_2} 
R\big(z_2; \operatorname{diag}(\mathscr{V}_t)\big)^{\rho_2} 
R\big(z_3; \operatorname{diag}(\beta^*)\big) K  
R\big(z_4; K \big) \mathscr{W}_t.
\end{align*}
Expanding these terms and applying the Cauchy integral formula along each contour $\Gamma_i$, we write
\begin{equation*}
q(D_i) = \frac{1}{2\pi i} \oint_{\Gamma_i} q(z_i) R\left(z_i; D_i \right) \, \mathrm{d}z_i,   
\end{equation*}
together with the standard resolvent identities,
\begin{equation*}
R\left(z_i; D_i \right)^2 = - \frac{\mathrm{d}}{\mathrm{d}z_i} R\left(z_i; D_i \right) \quad \text{and} \quad R\left(z_i; D_i \right)^3 = \frac{1}{2} \cdot \frac{\mathrm{d}^2}{\mathrm{d}z_i^2} R\left(z_i; D_i \right).
\end{equation*}
It then follows that
\begin{equation}\label{eq: FinalItoS}
\mathrm{d} S(\mathscr{X}_t, z) = \mathscr{F}\left(z, S(\mathscr{X}_t, z)\right)\mathrm{d}t + \mathrm{d}\mathcal{M}_t^\mathrm{HSGD}(S), 
\end{equation}
with $\mathscr{X}_0 = x_0$, where the summands of $\mathscr{F}\left(z, S(\mathscr{X}_t, z)\right)$ are of the form
\begin{align}\label{eq: sampleTermPDE}
\begin{split}
- 2 \gamma(t) H^\top \cdot \frac{1}{d} \mathscr{W}_t^\top z^\zeta  
\left(\bigcirc_{i=1}^4 \mathscr{G}_i\right)
( R\left(z_i; D_i \right))  \mathscr{W}_t \quad \text{or} \\[5pt]
C(Q) \gamma(t)^2 I(B(\mathscr{X}_t)) \cdot \frac{1}{d} \mathscr{W}_t^\top z^\zeta  
\left(\bigcirc_{i=1}^4 \mathscr{G}_i\right)
( R\left(z_i; D_i \right))  \mathscr{W}_t,
\end{split}
\end{align}
for 
\begin{equation}\label{eq: sampleSubtermPDE}
\mathscr{G}_i( R\left(z_i; D_i \right)) = \frac{1}{2\pi i} \oint_{\Gamma_i} q(z_i) R\left(z_i; D_i \right) \, \mathrm{d}z_i, \quad - \frac{\mathrm{d}}{\mathrm{d}z_i} R\left(z_i; D_i \right), \quad \text{or} \quad  \frac{1}{2} \cdot \frac{\mathrm{d}^2}{\mathrm{d}z_i^2} R\left(z_i; D_i \right).
\end{equation}
We now prove Proposition~\ref{prop: HSGDapprox}.

\begin{proof}[Proof of Proposition~\ref{prop: HSGDapprox}]
By It\^o’s Lemma, we have seen that
\begin{align*}
S(\mathscr{X}_t, \cdot) &= S(x_0, \cdot) + \int_0^t \mathscr{F}\left(\cdot, S(\mathscr{X}_s, \cdot)\right)\mathrm{d}s + \int_0^t \mathrm{d}\mathcal{M}_s^\mathrm{HSGD}(S(\mathscr{X}_s, \cdot)).
\end{align*}

Thus to show that $S(\mathscr{X}_t, \cdot)$ is an approximate solution of the partial integro-differential equation \eqref{eq:  PDESystem} it amounts to bounding the martingale term where $C$ is a positive constant independent of $d$. For all $z\in \Gamma$, we note that for some constants $C, c > 0$ we have $\vartheta_{c\cdot M} \leqslant \hat{\tau}_M \leqslant \vartheta_{C\cdot M}$ (see Lemma~\ref{lem: normEquivalence}). Consequently, we can work with the stopped process $\mathscr{X}_t^\vartheta = \mathscr{X}_{t\wedge \vartheta}$ instead of using $\hat{\tau}_{M}$. We thus have that for all $z\in \Gamma$ 
\begin{equation*}
\sup_{0\leqslant t \leqslant (\hat{\tau}_{M} \wedge T)} \norm{S(\mathscr{X}_t, z) - S(x_0,z) - \int_0^t \mathscr{F}\left(z, S(\mathscr{X}_s,z)\right)\, \mathrm{d}s} \leqslant \sup_{0\leqslant t \leqslant (\vartheta_{C\cdot M} \wedge T)} \norm{\mathcal{M}_t^\mathrm{HSGD}(S(\cdot, z))}.
\end{equation*}
Fix a constant $\delta>0$. Let $\Gamma^\delta = \Gamma_1^\delta \times \Gamma_2^\delta \times \Gamma_3^\delta \times \Gamma_4^\delta$ where each  $\Gamma_i^{\delta}$ is a $d^{-\delta}$-mesh of $\Gamma_i$ with $|\Gamma^{\delta}| \leqslant C_\Gamma d^{4\delta}$ for positive $C_\Gamma > 0$ depending on $\norm{K}_{\mathrm{op}}$ and $\norm{\beta^*}_\infty$.

By the martingale error proposition, Proposition~\ref{prop: homSGDmartingale}, which we have deferred the proof to Section~\ref{subsubsec: homSGDmartingale}, we have that for any $ \hat{\delta}> 0$
\begin{equation*}
\sup_{0\leqslant t \leqslant T} \norm{\mathcal{M}_{t\wedge \vartheta_{C\cdot M}}^\mathrm{HSGD}(S(\cdot, z))} \leqslant C L(f) d^{\hat{\delta}/2 -1/2}   \quad \text{w.o.p.} 
\end{equation*}
By a union bound over $z\in\Gamma^\delta$, using
$|\Gamma^\delta|\leqslant C_\Gamma d^{4\delta}$, the martingale bound holds uniformly
on $\Gamma^\delta$ with overwhelming probability:
\begin{equation*}
\sup_{z\in \Gamma^\delta}\sup_{0\leqslant t \leqslant T} \norm{\mathcal{M}_{t\wedge \vartheta_{C\cdot M}}^\mathrm{HSGD}(S(\cdot, z))} \leqslant C L(f) d^{\hat{\delta}/2 -1/2}   \quad \text{w.o.p.} 
\end{equation*}
Consequently, we deduce that
\begin{align*}
&\sup_{0\leqslant t \leqslant (\hat{\tau}_{M} \wedge T)} \norm{S(\mathscr{X}_t, z) - S(x_0,z) - \int_0^t \mathscr{F}\left(z, S(\mathscr{X}_s,z)\right)\, \mathrm{d}s}_{\Gamma^\delta} \\[10pt]
&\hspace{3in}\leqslant\sup_{0\leqslant t \leqslant (\vartheta_{C\cdot M} \wedge T)} \norm{\mathcal{M}_t^\mathrm{HSGD}(S(\cdot, z))}_{\Gamma^\delta} \\[10pt] 
&\hspace{3in}\leqslant C L(f) d^{\hat{\delta}/2 - 1/2} \quad \text{w.o.p.}
\end{align*}
An application of the net argument, Lemma~\ref{lem: net_argument}, finishes the proof after setting $\hat{\delta}= 1-2\delta$.

The derivative regularity condition~\ref{cond:lipschitz} is automatic for
$t\mapsto S(\mathscr X_t,\cdot)$: for each fixed trajectory, the dependence on
$z$ is through a finite product of resolvents on the fixed contour $\Gamma$, and
the resolvent derivative bounds from Remark~\ref{rem:Lipschitz} give the required Lipschitz control
up to $\hat\tau_M$. Thus the estimate above verifies condition~\ref{cond:pde},
and hence $S(\mathscr X_t,\cdot)$ is an
$(d^{-\varepsilon},M,T)$-approximate solution.
\end{proof}
\subsection{SGD Under the Resolvent Statistic \texorpdfstring{$S$}{S}}
In this section, we show that $S\left(x_{\lfloor td\rfloor}, z\right)$ is an approximate solution \eqref{def: approximateSolution} to the partial integro-differential
equation \eqref{eq: PDESystem} which we state below.

\begin{proposition}[SGD is an approximate solution]\label{prop: SGDapprox}
Fix a $T, M > 0$ and $0 < \delta < 1/2$. Then $S\left(x_{\lfloor td\rfloor},z\right)$ is a $(d^{-\delta}, M, T)$-approximate solution w.o.p., that is, 
\begin{align*}
\sup_{0\leqslant t \leqslant (\hat{\tau}_M \wedge T)} \norm{S(x_{\lfloor td\rfloor},z) - S(x_0,z) - \int_0^t \mathscr{F}\left(z, S(x_{\lfloor sd\rfloor},z)\right)\,  \mathrm{d}s} &\leqslant d^{-\delta} \quad \text{w.o.p.}
\end{align*}
\end{proposition}
The proof of this Proposition is deferred to Section~\ref{subsubsec: SGDapprox}.

Recall our iterates satisfy the recurrence 
\begin{equation}
x_{k+1} = x_k - \gamma_k \, \nabla_x \Psi (x_k; a_{k+1}) =   x_k - \frac{\gamma_k}{\sqrt{d}} \,   \nabla_{r_1} f\left(  r_k \right) \cdot (\nabla \psi(x_k))^\top \cdot a_{k+1} 
\end{equation}
with $r_k := \frac{1}{\sqrt{d}} \begin{pmatrix}
\psi(x_k) \\
\beta^*
\end{pmatrix}^{\!\top} a_{k+1}$.

The derivation follows the martingale-decomposition strategy used in \citep{collins-woodfin2024hitting}. We first derive a one-step martingale decomposition for a smooth scalar statistic $\varphi:\mathbb{R}^{2d}\to\mathbb{R}$. We later apply this decomposition to the real-valued coordinate functions
$S_{ab}^{\operatorname{Re},z}$ and $S_{ab}^{\operatorname{Im},z}$ and then
recombine the resulting identities to obtain the complex matrix-valued
decomposition for $S(x_{\lfloor td\rfloor},z)$.

Applying Taylor's expansion, we have
\begin{align}\label{eq: taylorExpansion}
\begin{split}
\varphi(x_{k+1}) &= \varphi(x_k) - \gamma_k \inner{\nabla \varphi(x_k), \nabla_x \Psi(x_k; a_{k+1})} + \frac{\gamma_k^2}{2} \inner{\nabla^2 \varphi(x_k), \left(\nabla_x \Psi(x_k; a_{k+1}) \right)^{\otimes 2}}  + \dots\\
&= \varphi(x_k) - \frac{\gamma_k}{\sqrt{d}} \, \inner{\nabla \varphi(x_k), \nabla_{r_1} f\left(  r_k \right) \cdot (\nabla \psi(x_k))^\top \cdot a_{k+1}} \\
&+ \frac{\gamma_k^2}{2d} \,  \inner{\nabla^2 \varphi(x_k), \left(\nabla_{r_1} f\left(  r_k \right) \cdot (\nabla \psi(x_k))^\top \cdot a_{k+1}\right)^{\otimes 2}} +  \mathcal{E}_k^{\mathrm{High}}(\varphi).
\end{split}
\end{align}

We will show in Section~\ref{subsubsec: highOrderTerm} that the third and higher-order terms, $\mathcal{E}_k^{\mathrm{High}}(\varphi)$, 
vanish when $d\to \infty$.

\subsubsection{Doob Decomposition for SGD}

Recall that, for a Hessian $A$ and vector $v$,
\[
\langle A,v^{\otimes 2}\rangle = v^\top A v.
\]

To write the Doob decomposition, the idea is to condition on $r_k = \frac{1}{\sqrt{d}} \begin{pmatrix}
\psi(x_k) \\
\beta^*
\end{pmatrix}^{\!\top} a_{k+1}$ and $W_k = [\, \psi(x_k) \, |  \,  \beta^* \, | \, \mathds{1}_d \, ]\in \R^{d\times 3}$. For this, we will introduce some notation. Define the $\sigma$-algebras
\begin{equation*}
    \mathscr{F}_k := \sigma \left( \{ W_i \}_{i=0}^k \right) \subset \mathcal{G}_k := \sigma \left( \{ r_i \}_{i=0}^k , \, \{ W_i \}_{i=0}^k  \right). 
\end{equation*}

\textbf{Gradient term in Taylor expansion.} First, consider the conditional expectation with respect to $\mathscr{F}_k$ of the gradient term in equation \eqref{eq: taylorExpansion}. Note we can safely assume the interchange of the gradient and the expectation in the risk formula $\mathcal{R}$ since $a$ follows a Gaussian distribution, and generally, this assumption holds true. Therefore, we have 
\begin{equation*}
\frac{1}{\sqrt{d}} \EE_{a_{k+1}} \left[  \inner{\nabla \varphi(x_k), \nabla_{r_1} f\left(  r_k \right) \cdot (\nabla \psi(x_k))^\top \cdot a_{k+1}} \middle| \mathscr{F}_k \right] = \inner{\nabla \varphi (x_k), \nabla \mathcal{R}(x_k)}.
\end{equation*}
By applying Doob's decomposition in this context, the gradient term takes the form
\begin{equation}\label{eq: gradientTerm}
\frac{\gamma_k}{\sqrt{d}} \, \inner{\nabla \varphi(x_k), \nabla_{r_1} f\left(  r_k \right) \cdot (\nabla \psi(x_k))^\top \cdot a_{k+1}} = \gamma_k \, \inner{\nabla \varphi (x_k), \nabla \mathcal{R}(x_k)} - \Delta \mathcal{M}_k^{\mathrm{Grad}}(\varphi)
\end{equation}
where $\Delta \mathcal{M}_k^{\mathrm{Grad}}(\varphi)$ represents a martingale increment that tends to zero as $d$ tends to infinity.

\textbf{Hessian term in the Taylor expansion.} Proceeding to the Hessian term in  \eqref{eq: taylorExpansion}, we initiate the analysis by conditioning first on $\mathcal{G}_k$ and subsequently on $\mathscr{F}_k$, using the tower property. A simple computation leads to
\begin{equation*}
    \inner{\nabla^2 \varphi(x_k), \left(\nabla_{r_1} f\left(  r_k \right) \cdot (\nabla \psi(x_k))^\top \cdot a_{k+1}\right)^{\otimes 2}} = \inner{\mathcal{B}, a_{k+1} ^{\otimes 2}},
\end{equation*}
where $\mathcal{B}:= \nabla_{r_1} f\left(  r_k \right)^2 \nabla \psi(x_k)  \nabla^2 \varphi(x_k)\, (\nabla \psi(x_k))^\top$. To proceed, we must evaluate 

\begin{align}
\begin{split}
    \EE_{a_{k+1}} \left[\inner{\mathcal{B}, a_{k+1}^{\otimes 2}}\middle| \mathcal{G}_k \right] &= \inner{\mathcal{B}, \EE_{a_{k+1}} \left[  a_{k+1}^{\otimes 2} \middle| \mathcal{G}_k \right]}\\ 
    &= \inner{\mathcal{B}, \EE_{a_{k+1}} \left[ \left(
        a_{k+1} - \EE_{a_{k+1}} \left[ 
        a_{k+1}  \middle| \mathcal{G}_k \right] \right)^{\otimes 2} \middle| \mathcal{G}_k \right]}\\
    &+ \inner{\mathcal{B}, \left( \EE_{a_{k+1}} \left[ 
        a_{k+1} \middle| \mathcal{G}_k \right]\right)^{\otimes 2}}.
\end{split}
\label{eq: term}
\end{align}

To establish the conditional distribution of $a_{k+1}$ given $r_k$, we invoke a standard conditioning lemma, see e.g. \citep{collins-woodfin2024hitting}. 
\begin{lemma}[Gaussian conditioning]\label{lem: gaussianConditioning}
Let $v\sim\mathcal N(0,I_d)$ and let $Q\in\mathbb R^{d\times 2}$ have
orthonormal columns. Then, conditional on $Q^\top v$, we have
\[
v \stackrel{\mathcal{D}}{=} (I_d-QQ^\top)g+QQ^\top v
\]
where $g\sim\mathcal N(0,I_d)$ is independent of $Q^\top v$. In particular,
\[
\mathbb E[v\mid Q^\top v]=QQ^\top v,
\qquad
\operatorname{Cov}(v\mid Q^\top v)=I_d-QQ^\top.
\]
\end{lemma}

By Assumption~\ref{ass: data}, we write $a_{k+1} = \sqrt{K} v_k$ with $v_k\sim \mathcal{N}(0,I_d)$, and express $\sqrt{K}
\begin{pmatrix}
\psi(x_k) \\
\beta^*
\end{pmatrix}$ as $Q_k R_k$, where $Q_k \in \R^{d\times 2}$ is orthogonal and $R_k\in \R^{2\times 2}$ is upper triangular and invertible (obtained through QR decomposition, assuming $2 \ll d$). Observe here $\Pi_k := Q_k Q_k^\top$ has rank $2$.

Consequently, we can simplify $\begin{pmatrix}
\psi(x_k) \\
\beta^*
\end{pmatrix}^{\!\top} a_{k+1}$ as follows
\begin{equation*}
    \begin{pmatrix}
\psi(x_k) \\
\beta^*
\end{pmatrix}^{\!\top} a_{k+1} = \begin{pmatrix}
\psi(x_k) \\
\beta^*
\end{pmatrix}^{\!\top} \sqrt{K} v_k = R_k^\top Q_k^\top v_k.
\end{equation*}
Applying the conditioning lemma, we derive
\begin{align*}
    a_{k+1}\Big|\begin{pmatrix}
\psi(x_k) \\
\beta^*
\end{pmatrix}^{\!\top} a_{k+1} &\stackrel{\mathcal{D}}{=} \sqrt{K} v_k|Q_k^\top v_k \\
&\stackrel{\mathcal{D}}{=} \sqrt{K} \left(
(I_d - Q_k Q_k^\top)g_k + Q_k Q_k^\top v_k \right) \\
&= \sqrt{K}\left(g_k -\Pi_k g_k\right) + \sqrt{K} \Pi_k v_k,
\end{align*}
where $g_k \sim \mathcal{N}(0, I_d)$ is independent of $Q_k^\top v_k$. From this, we have that 
\begin{equation}
    \EE_{a_{k+1}} \left[ 
        a_{k+1} \middle| \mathcal{G}_k \right] = 
        \sqrt{K} \Pi_k v_k \quad \text{where } v_k \sim \mathcal{N}(0,I_d).
\end{equation}
Moreover, the conditional covariance of $a_{k+1}$ is precisely 
\begin{equation}
    \EE_{a_{k+1}} \left[ \left(
        a_{k+1} - \EE_{a_{k+1}} \left[ 
        a_{k+1} \middle| \mathcal{G}_k \right] \right)^{\otimes 2} \middle| \mathcal{G}_k \right] = 
        \sqrt{K}\left(I_d -\Pi_k \right) \sqrt{K}.
\end{equation}
Thus, from \eqref{eq: term} we have 
\begin{equation}\label{eq: hessianDiagCovSimplified}
    \EE_{a_{k+1}} \left[\inner{\mathcal{B},  a_{k+1}^{\otimes 2}} \middle| \mathcal{G}_k \right] = \inner{\mathcal{B}, K} - \inner{\mathcal{B}, \sqrt{K}\Pi_k \sqrt{K}} + \inner{\mathcal{B}, \left( \sqrt{K} \Pi_k v_k \right)^{\otimes 2}}.
\end{equation}

We will later see, in Section~\ref{subsubsec: hessianErrorTerm}, that the term
\begin{equation*}
\mathcal{E}_k^{\mathrm{Hess}}(\varphi) := -\frac{\gamma_k^2}{2 d} \inner{\mathcal{B}, \sqrt{K}\Pi_k \sqrt{K}} + \frac{\gamma_k^2}{2 d} \inner{\mathcal{B}, \left( \sqrt{K} \Pi_k v_k \right)^{\otimes 2}}
\end{equation*}
is of lower order in expectation and will disappear as $d\to \infty$.  So we may write
\begin{equation*}
\frac{\gamma_k^2}{2 d} \EE_{a_{k+1}} \left[ \inner{\mathcal{B},  a_{k+1}^{\otimes 2}} \middle| \mathscr{F}_k \right] = \frac{\gamma_k^2}{2 d} \EE_{a_{k+1}} \left[ \inner{\mathcal{B}, K}\middle| \mathscr{F}_k \right] + \EE_{a_{k+1}} \left[ \mathcal{E}_k^{\mathrm{Hess}}(\varphi)\middle| \mathscr{F}_k \right].
\end{equation*}

Moreover, observe that
\begin{align}
\begin{split}
\frac{\gamma_k^2}{2 d} \EE_{a_{k+1}} \left[ \inner{\mathcal{B}, K}\middle| \mathscr{F}_k \right] &=  \frac{\gamma_k^2}{2 d} \EE_{a_{k+1}} \left[\nabla_{r_1} f\left(  r_k \right)^2 \middle | \mathscr{F}_k \right]   \inner{\nabla \psi(x_k)  \nabla^2 \varphi(x_k)\, (\nabla \psi(x_k))^\top, K} \\
&= \frac{\gamma_k^2}{2 d} \, I(B(x_k))
\, \inner{\nabla \psi(x_k)  \nabla^2 \varphi(x_k)\, (\nabla \psi(x_k))^\top, K}\\ 
&= \frac{\gamma_k^2}{2 d} \, I(B(x_k))
\, \inner{\nabla^2 \varphi(x_k), \inner{K, \left( \nabla \psi(x_k) \right)^{\otimes 2}}}.
\end{split}
\end{align}
By applying Doob's decomposition in this context, the Hessian term takes the form
\begin{align}\label{eq: hessianTerm}
\begin{split}
    &\frac{\gamma_k^2}{2d} \inner{\nabla^2 \varphi(x_k), \left(\nabla_{r_1} f\left(  r_k \right) \cdot (\nabla \psi(x_k))^\top \cdot a_{k+1}\right)^{\otimes 2}} =\\ &\frac{\gamma_k^2}{2d}  \, I(B(x_k))
    \, \inner{\nabla^2 \varphi(x_k), \inner{K, \left( \nabla \psi(x_k) \right)^{\otimes 2}}} + \Delta \mathcal{M}_k^{\mathrm{Hess}}(\varphi) + \EE_{a_{k+1}} \left[ \mathcal{E}_k^{\mathrm{Hess}}(\varphi)\middle| \mathscr{F}_k \right]
\end{split}
\end{align}
where $\Delta \mathcal{M}_k^{\mathrm{Hess}}(\varphi)$ represents a martingale increment that tends to zero as $d$ tends to infinity.

We have successfully identified the martingale increments of a \textit{single} update of SGD, that is,
by \eqref{eq: gradientTerm} and \eqref{eq: hessianTerm} in the Taylor expansion \eqref{eq: taylorExpansion}
\begin{align}
\begin{split}
    \varphi(x_{k+1}) &= \varphi(x_k) - \gamma_k \, \inner{\nabla \varphi (x_k), \nabla \mathcal{R}(x_k)} + \Delta \mathcal{M}_k^{\mathrm{Grad}}(\varphi) \\
    &+ \frac{\gamma_k^2}{2d}  \, I(B(x_k))
    \, \inner{\nabla^2 \varphi(x_k), \inner{K, \left( \nabla \psi(x_k) \right)^{\otimes 2}}} + \Delta \mathcal{M}_k^{\mathrm{Hess}}(\varphi) \\
    &+ \EE_{a_{k+1}} \left[ \mathcal{E}_k^{\mathrm{Hess}}(\varphi)\middle| \mathscr{F}_k \right]
    + \EE_{a_{k+1}} \left[ \mathcal{E}_k^{\mathrm{High}}(\varphi)\middle| \mathscr{F}_k \right],
\end{split}
\end{align}
where the error terms look like
\begin{align*}
\Delta \mathcal{M}_k^{\mathrm{Grad}}(\varphi) &= \gamma_k \, \inner{\nabla \varphi (x_k), \nabla \mathcal{R}(x_k)} - 
\frac{\gamma_k}{\sqrt{d}} \, \inner{\nabla \varphi(x_k), \nabla_{r_1} f\left(  r_k \right) \cdot (\nabla \psi(x_k))^\top \cdot a_{k+1}}\\
\Delta \mathcal{M}_k^{\mathrm{Hess}}(\varphi) &= \frac{\gamma_k^2}{2d} \inner{\nabla^2 \varphi(x_k), \left(\nabla_{r_1} f\left(  r_k \right) \cdot (\nabla \psi(x_k))^\top \cdot a_{k+1}\right)^{\otimes 2}}\\
&- \frac{\gamma_k^2}{2d} \EE_{a_{k+1}} \left[ \inner{\nabla^2 \varphi(x_k), \left(\nabla_{r_1} f\left(  r_k \right) \cdot (\nabla \psi(x_k))^\top \cdot a_{k+1}\right)^{\otimes 2}} \middle| \mathscr{F}_k \right]\\
\mathcal{E}_k^{\mathrm{Hess}}(\varphi) &= -\frac{\gamma_k^2}{2 d} \inner{\nabla_{r_1} f\left(  r_k \right)^2 \nabla \psi(x_k)  \nabla^2 \varphi(x_k)\, (\nabla \psi(x_k))^\top, \sqrt{K}\Pi_k \sqrt{K}} \\
&+ \frac{\gamma_k^2}{2 d} \inner{\nabla_{r_1} f\left(  r_k \right)^2 \nabla \psi(x_k)  \nabla^2 \varphi(x_k)\, (\nabla \psi(x_k))^\top, \left( \sqrt{K} \Pi_k v_k \right)^{\otimes 2}}.
\end{align*}
We now pass from the one-step decomposition to the continuous-time embedding.
For $t\geqslant 0$, write
\[
n_t:=\lfloor td\rfloor,
\qquad
t_d:=\frac{n_t}{d}.
\]
Thus $0\leqslant t-t_d < 1/d$ and $x_{\lfloor td\rfloor}=x_{n_t}$.

Define the predictable drift integrand
\[
A_j(\varphi)
:=
-d\,\gamma_j
\left\langle
\nabla\varphi(x_j),\nabla\mathcal R(x_j)
\right\rangle
+
\frac{\gamma_j^2}{2}
I(B(x_j))
\left\langle
\nabla^2\varphi(x_j),
\left\langle
K,(\nabla\psi(x_j))^{\otimes 2}
\right\rangle
\right\rangle.
\]
Equivalently, the predictable part of one SGD step is $(1/d)A_j(\varphi)$.
Let
\[
A_\varphi^{(d)}(s):=A_{\lfloor sd\rfloor}(\varphi)
\]
be its piecewise-constant interpolation. Then
\[
\sum_{j=0}^{n_t-1}\frac{1}{d}A_j(\varphi)
=
\int_0^{t_d}A_\varphi^{(d)}(s)\,\mathrm ds.
\]
Therefore,
\[
\sum_{j=0}^{n_t-1}\frac{1}{d}A_j(\varphi)
=
\int_0^{t}A_\varphi^{(d)}(s)\,\mathrm ds
+
\xi_t^{\mathrm{mesh}}(\varphi),
\]
where the mesh error is
\[
\xi_t^{\mathrm{mesh}}(\varphi)
:=
-\int_{t_d}^{t}A_\varphi^{(d)}(s)\,\mathrm ds.
\]
Since $t-t_d\leqslant 1/d$, we have
\[
\left|\xi_t^{\mathrm{mesh}}(\varphi)\right|
\leqslant
\frac{1}{d}
\sup_{0\leqslant s\leqslant T}
\left|A_\varphi^{(d)}(s)\right|.
\]

Using the one-step decomposition above, we obtain
\begin{align}
\varphi(x_{\lfloor td\rfloor})
&=
\varphi(x_0)
+
\int_0^t A_\varphi^{(d)}(s)\,\mathrm ds
+
\mathcal M_{\lfloor td\rfloor}^{\mathrm{Grad}}(\varphi)
+
\mathcal M_{\lfloor td\rfloor}^{\mathrm{Hess}}(\varphi)
\nonumber\\
&\quad
+
\sum_{j=0}^{\lfloor td\rfloor-1}
\mathbb E_{a_{j+1}}
\left[
\mathcal E_j^{\mathrm{Hess}}(\varphi)
\,\middle|\,
\mathscr F_j
\right]
+
\sum_{j=0}^{\lfloor td\rfloor-1}
\mathbb E_{a_{j+1}}
\left[
\mathcal E_j^{\mathrm{High}}(\varphi)
\,\middle|\,
\mathscr F_j
\right]
+
\xi_t^{\mathrm{mesh}}(\varphi).
\label{eq: SGDDoobContinuous}
\end{align}
where
\[
\mathcal M_{\lfloor td\rfloor}^{\mathrm{Grad}}(\varphi)
:=
\sum_{j=0}^{\lfloor td\rfloor-1}
\Delta\mathcal M_j^{\mathrm{Grad}}(\varphi),
\qquad
\mathcal M_{\lfloor td\rfloor}^{\mathrm{Hess}}(\varphi)
:=
\sum_{j=0}^{\lfloor td\rfloor-1}
\Delta\mathcal M_j^{\mathrm{Hess}}(\varphi).
\]

In Section~\ref{subsec: errorBounds}, we prove that the error terms in \eqref{eq: SGDDoobContinuous} are negligible as $d\to \infty$. The other two terms in $\int_0^t A_\varphi^{(d)}(s)\,\mathrm ds$ survive the limit. Next, we show that SGD on $S$ is an $(\varepsilon, M, T)$-approximate solution.

\subsubsection{\texorpdfstring{$S\left(x_{\lfloor td\rfloor},z\right)$}{B(xtd)} is an Approximate Solution, Proof of Proposition~\ref{prop: SGDapprox}}\label{subsubsec: SGDapprox}
The goal in this section is to prove Proposition~\ref{prop: SGDapprox}, that is, show that 
\begin{equation*}
S\left(x_{\lfloor td\rfloor},z\right) = \frac{1}{d} W_{\lfloor td\rfloor}^\top \Omega\left(x_{\lfloor td\rfloor},z\right) W_{\lfloor td\rfloor} \in \mathbb{C}^{3 \times 3}
\end{equation*}
is an approximate solution to the partial integro-differential equation \eqref{eq:  PDESystem}.
\begin{proof}[Proof of Proposition~\ref{prop: SGDapprox}]
Applying the preceding real-valued identity to the real and imaginary parts of
each matrix entry of $S(\cdot,z)$, and then recombining, yields
\begin{align*}
S\left(x_{\lfloor td\rfloor}, z\right) &= S\left(x_0,z\right) +
\int_0^t \mathscr{F}\left(z, S(x_{\lfloor sd\rfloor},z)\right) \mathrm{d}s +
\mathcal M_{\lfloor td\rfloor}^{\mathrm{Grad}}(S)
+
\mathcal M_{\lfloor td\rfloor}^{\mathrm{Hess}}(S)\\
&\quad
+
\sum_{j=0}^{\lfloor td\rfloor-1}
\mathbb E_{a_{j+1}}
\left[
\mathcal E_j^{\mathrm{Hess}}(S)
\,\middle|\,
\mathscr F_j
\right]
+
\sum_{j=0}^{\lfloor td\rfloor-1}
\mathbb E_{a_{j+1}}
\left[
\mathcal E_j^{\mathrm{High}}(S)
\,\middle|\,
\mathscr F_j
\right]
+
\xi_t^{\mathrm{mesh}}(S). 
\end{align*}
Thus to show that $S\left(x_{\lfloor td\rfloor}, z\right)$ is an approximate solution of the partial integro-differential equation \eqref{eq:  PDESystem} it
suffices to bound the martingales and error terms where $C$ is a positive constant independent
of $d$. We thus have for all $z\in \Gamma$, the estimate
\begin{align*}
\sup_{0\leqslant t \leqslant (\hat{\tau}_{M} \wedge T) } & \norm{S\left(x_{\lfloor td\rfloor}, z\right) - S\left(x_0,z\right) -
\int_0^t \mathscr{F}\left(z, S(x_{\lfloor sd\rfloor},z)\right) \mathrm{d}s} \\[10pt]
&\leqslant \sup_{0\leqslant t \leqslant (\hat{\tau}_{M} \wedge T) } \norm{\mathcal{M}_{\lfloor td\rfloor}^{\mathrm{Grad}}(S)} + \sup_{0\leqslant t \leqslant (\hat{\tau}_{M} \wedge T) } \norm{ \mathcal{M}_{\lfloor td\rfloor}^{\mathrm{Hess}}(S)} \\[10pt]
+ \sup_{0\leqslant t \leqslant (\hat{\tau}_{M} \wedge T) } &\norm{\sum_{j=0}^{\lfloor td\rfloor - 1} \EE_{a_{j+1}} \left[ \mathcal{E}_j^{\mathrm{Hess}}(S)\middle| \mathscr{F}_j \right]} + \sup_{0\leqslant t \leqslant (\hat{\tau}_{M} \wedge T) } \norm{ \sum_{j=0}^{\lfloor td\rfloor - 1} \EE_{a_{j+1}} \left[ \mathcal{E}_j^{\mathrm{High}}(S)\middle| \mathscr{F}_j \right]}\\[10pt]
+ \sup_{0\leqslant t \leqslant (\hat{\tau}_{M} \wedge T) } &\norm{\xi_t^{\mathrm{mesh}}(S)} .
\end{align*}
Next, fix a constant $\delta>0$. Let $\Gamma^\delta = \Gamma_1^\delta \times \Gamma_2^\delta \times \Gamma_3^\delta \times \Gamma_4^\delta$ where each  $\Gamma_i^{\delta}$ is a $d^{-\delta}$-mesh of $\Gamma_i$ with $|\Gamma^{\delta}| \leqslant C_\Gamma d^{4\delta}$ for positive $C_\Gamma > 0$ depending on $\norm{K}_{\mathrm{op}}$ and $\norm{\beta^*}_\infty$. For all $z\in \Gamma$, we note that for some constants $C,c>0$ we have $\vartheta_{c\cdot M} \leqslant \hat{\tau}_M \leqslant \vartheta_{C\cdot M}$ (see Lemma~\ref{lem: normEquivalence}). Consequently, we evaluate
the error with the stopped process  $x_{td}^\vartheta = x_{td\wedge \vartheta}$ instead of using $\hat{\tau}_{M}$. By the martingale error
propositions, Proposition~\ref{prop: gradientMartingale} and~\ref{prop: hessianMartingale}, the proofs of which are deferred to Section~\ref{subsubsec: boundsMartingales}, we have for any $\hat{\delta}>0$ the estimates
\begin{align*}
\sup_{z\in \Gamma^\delta}\sup_{0\leqslant t \leqslant T}  \norm{ \mathcal{M}_{\lfloor (t\wedge \vartheta_{C\cdot M})d\rfloor}^{\mathrm{Grad}}(S(\cdot, z))}  &< d^{-\frac{2+\alpha}{2} + \hat{\delta}} \quad \text{w.o.p.},\\[10pt] 
\text{and} \quad \sup_{z\in \Gamma^\delta} \sup_{0\leqslant t \leqslant T} \norm{ \mathcal{M}_{\lfloor (t\wedge \vartheta_{C\cdot M})d\rfloor}^{\mathrm{Hess}}(S(\cdot, z))}  &< d^{-(2+\alpha) + \hat{\delta}} \quad \text{w.o.p.}
\end{align*}

In addition, for the Hessian and higher order terms errors, by Propositions~\ref{prop: highOrderTerm} and~\ref{prop: hessianErrorTerm}, the proofs of which are deferred to Sections~\ref{subsubsec: highOrderTerm} and~\ref{subsubsec: hessianErrorTerm}, we have
\begin{align*}
\sup_{z\in \Gamma^\delta}\sup_{0\leqslant t \leqslant T}  \sum_{j=0}^{\lfloor (t\wedge \vartheta_{C\cdot M})d\rfloor -1} \norm{ \EE_{a_{k+1}} \left[ \mathcal{E}_j^{\mathrm{Hess}}(S)\middle| \mathscr{F}_j \right]}   &\leqslant C d^{-1} \quad \text{w.o.p.},\\[10pt]
\text{and} \quad \sup_{z\in \Gamma^\delta}\sup_{0\leqslant t \leqslant T}  \sum_{j=0}^{\lfloor (t\wedge \vartheta_{C\cdot M})d\rfloor -1} \norm{ \EE_{a_{k+1}} \left[ \mathcal{E}_j^{\mathrm{High}}(S)\middle| \mathscr{F}_j \right]}   &\leqslant C d^{-\frac{1}{2}} \quad \text{w.o.p.}
\end{align*}
The mesh error is the contribution of the last fractional time interval:
\[
\xi_t^{\mathrm{mesh}}(S)(z)
=
-\int_{t_d}^{t}
\mathscr F\bigl(z,S(x_{\lfloor sd\rfloor},\cdot)\bigr)\,\mathrm ds,
\qquad
t_d=\frac{\lfloor td\rfloor}{d}.
\]
Hence, since $0\leqslant t-t_d\leqslant 1/d$,
\[
\sup_{0\leqslant t\leqslant (\hat\tau_M\wedge T)}
\left\|
\xi_t^{\mathrm{mesh}}(S)(\cdot)
\right\|_{\Gamma}
\leqslant
\frac{1}{d}
\sup_{0\leqslant s\leqslant (\hat\tau_M\wedge T)}
\left\|
\mathscr F\bigl(\cdot,S(x_{\lfloor sd\rfloor},\cdot)\bigr)
\right\|_{\Gamma}.
\]
Last we show that in the stopped interval, the right-hand side is bounded by $Cd^{-1}$. For any $0\leqslant s\leqslant (\hat\tau_M\wedge T)$, we have
\begin{align*}
\norm{\mathscr{F}(z, S(x_{\lfloor sd\rfloor},z))}&\leqslant \bar{\gamma} C \norm{H\left( B(x_{\lfloor sd\rfloor})\right) } \cdot \norm{z^\zeta 
\left(\bigcirc_{i=1}^4 \mathscr{G}_i\right)
( S(x_{\lfloor sd\rfloor},z))} \\[10pt]
&+ \bar{\gamma}^2 C \norm{I\left( B(x_{\lfloor sd\rfloor})\right) } \cdot \norm{z^\zeta 
\left(\bigcirc_{i=1}^4 \mathscr{G}_i\right)
( S(x_{\lfloor sd\rfloor},z))} + \cdots
\end{align*}
such that $B(x_{\lfloor sd\rfloor}) = \frac{1}{(2\pi)^4} 
\oint_{\Gamma}
z_4\, S(x_{\lfloor sd\rfloor}, z)\,
\mathrm{d}z$. Next, 
plugging equations \eqref{eq: HBound}, \eqref{eq: compGBound}, \eqref{eq: IBound}, and 
\begin{equation*}
\norm{z^\zeta 
\left(\bigcirc_{i=1}^4 \mathscr{G}_i\right)
( S(x_{\lfloor sd\rfloor},z))} \leqslant C(\abs{\Gamma}) \cdot \norm{\left(\bigcirc_{i=1}^4 \mathscr{G}_i\right)
( S(x_{\lfloor sd\rfloor},z))},    
\end{equation*}
we know there is a positive constant $C=C(L(h), L(I), \bar{\gamma}, \norm{K}_{\mathrm{op}}, \norm{\beta^*}_\infty, M, \alpha)$, such that $\left\|
\mathscr F\bigl(\cdot,S(x_{\lfloor sd\rfloor},\cdot)\bigr)
\right\|_{\Gamma}
\leqslant C$. Therefore,
\begin{equation}
\sup_{0\leqslant t \leqslant (\hat{\tau}_{M} \wedge T) } \norm{\xi_t^{\mathrm{mesh}}(S)(\cdot)}_\Gamma \leqslant C d^{-1}.
\end{equation}
Consequently, combining all the errors, we deduce that for some $C > 0$, which does not depend on $d$,
\begin{align*}
\sup_{0\leqslant t \leqslant (\hat{\tau}_{M} \wedge T) }  \norm{S\left(x_{\lfloor td\rfloor}, z\right) - S\left(x_0,z\right) -
\int_0^t \mathscr{F}\left(z, S(x_{\lfloor sd\rfloor},z)\right) \mathrm{d}s}_{\Gamma^\delta} &\leqslant C d^{\hat{\delta}/2 - 1/2} \quad \text{w.o.p.}
\end{align*}
An application of the net argument, Lemma~\ref{lem: net_argument}, finishes the proof after setting $\hat{\delta}= 1-2\delta$ for $\delta \in (0,1/2)$.

The derivative regularity condition~\ref{cond:lipschitz} is automatic for the
continuous-time embedding $t\mapsto S(x_{\lfloor td\rfloor},\cdot)$: for each
fixed trajectory, the dependence on $z$ is through a finite product of resolvents
on the fixed contour $\Gamma$, and the resolvent derivative bounds from Remark~\ref{rem:Lipschitz} give the required
Lipschitz control up to $\hat\tau_M$. Thus the estimate above verifies
condition~\ref{cond:pde}, and hence $S(x_{\lfloor td\rfloor},\cdot)$ is an
$(d^{-\varepsilon},M,T)$-approximate solution.
\end{proof}

\subsection{Error Bounds}\label{subsec: errorBounds}
Recall that we are interested in the statistic
\[
S(x,z):=\frac1d W(x)^\top\Omega(x,z)W(x)\in\mathbb C^{3\times 3},
\]
where
\[
W(x):=[\,\psi(x)\mid \beta^*\mid \mathds 1_d\,] \in \R^{d\times 3}
\]
and $\Omega(x,z)$ is defined in \eqref{eq: Omega}. Throughout this section, the contour $\Gamma$ is as in Remark~\ref{rem: fixedContour}. 
We control the error terms arising in the comparison between SGD and homogenized SGD under $S$, with $\mathscr{F}$ given by the partial integro-differential equation~\eqref{eq: PDESystem}. 
All estimates are taken on the stopped event where $\|x_{\lfloor td\rfloor}\|_\infty$ and $\|\mathscr X_t\|_\infty$ remain uniformly bounded for $0\leqslant t\leqslant T$ (see Assumption~\ref{ass: nonexplosive}). 
Constants may depend on this bound, but not on $d$.

Before proceeding, we present some bounds on the derivatives of $S$.

\begin{lemma}\label{lem: derivativeSbound}
There exist constants $c, C > 0$ such that, for large enough $d$, the following hold:
\begin{align*}
\frac{c}{d}\norm{W}^2 \leqslant \norm{S(x, \cdot)}_\Gamma &\leqslant \frac{C}{d}\norm{W}^2,\\[10pt]
\norm{\nabla_x S(x, \cdot)}_\Gamma &\leqslant \frac{C(\norm{x}_\infty)}{d}\norm{W},\\[10pt]
\text{and} \quad \norm{\nabla^2_x S(x, \cdot)}_\Gamma &\leqslant \frac{C(\norm{x}_\infty)}{d}.
\end{align*} 
\end{lemma}
\begin{proof}
The first result follows from the proof of Lemma~\ref{lem: normEquivalence}.

For the derivative, observe that 
\begin{equation}\label{eq: opNormBoundpsi}
\norm{\operatorname{diag}(\psi(x))}_{\mathrm{op}} = \norm{\psi(x)}_\infty 
= \max_{i} \abs{\psi_i(x)}
\leqslant 2 \norm{Q} \norm{x}_\infty^2 + \norm{l}_1 \norm{x}_\infty + \abs{c}.
\end{equation}

Moreover, by Lemma~\ref{lem: derivativepsi}, we have
\begin{align}\label{eq: opNormBoundNablapsi}
\begin{split}
\norm{\nabla_u \psi(x)}_{\mathrm{op}} &= \norm{2q_{11}\operatorname{diag}(u) + 2q_{12}\operatorname{diag}(v)+ l_1 I_d}_{\mathrm{op}} \leqslant 2\sqrt{2} \norm{x}_\infty + \abs{l_1},\\[10pt]
\norm{\nabla_v \psi(x)}_{\mathrm{op}} &= \norm{2q_{12}\operatorname{diag}(u) + 2q_{22}\operatorname{diag}(v)+ l_2 I_d}_{\mathrm{op}} \leqslant 2\sqrt{2} \norm{x}_\infty + \abs{l_2}.
\end{split}
\end{align}
Taking norms on Lemma~\ref{lem: derivativeS} and using that $\sup_{z_i \in \Gamma_i}\norm{R\left(z_i; D_i \right)}_{\mathrm{op}} \leqslant 2$, the second result follows.

Analogously, for the Hessian, the bound immediately follows.
\end{proof}
Consequently, the same bounds hold for the real-valued coordinate functions
$S_{ab}^{\operatorname{Re},z}$ and $S_{ab}^{\operatorname{Im},z}$, uniformly in
$a,b\in\{1,2,3\}$ and $z\in\Gamma$.

To control the errors, we will need to make an \textit{a priori} estimate that effectively shows that the iterates of homogenized SGD and SGD remain bounded. Thus, recall our definition, for fixed $M>0$, of the stopping times
\begin{align}
\begin{split}
\vartheta_M &= \inf \{t \geqslant 0 : \frac{1}{d} \norm{W_{\lfloor td \rfloor}}^2 > M \text{ or } B\left(x_{\lfloor td\rfloor}\right) \notin \mathcal{U}\} \quad \text{or}\\
\vartheta_M &= \inf \{t \geqslant 0 : \frac{1}{d}\norm{\mathscr{W}_t}^2 > M \text{ or } B\left(\mathscr{X}_t\right) \notin \mathcal{U} \},
\end{split}
\end{align}
depending on whether we are working with SGD iterates or homogenized SGD iterates. We often drop the $M$ so that $\vartheta := \vartheta_M$. It will be convenient to work with the stopped processes, $W_{\lfloor td \rfloor}^\vartheta := W_{\lfloor (t\wedge \vartheta)d \rfloor}$ and $\mathscr{W}_t^\vartheta := \mathscr{W}_{t\wedge \vartheta}$.
\begin{remark}\upright
The stopping time $\hat{\tau}_M = \inf \{ t\geqslant 0 : 
\norm{S(x_{\lfloor td \rfloor}, z)}_\Gamma> M \text{ or } B(x_{\lfloor td \rfloor}) \notin \mathcal{U} \}$ and $\hat{\tau}_M = \inf \{ t\geqslant 0 : 
\norm{S(\mathscr{X}_t, z)}_\Gamma> M \text{ or } B(\mathscr{X}_t) \notin \mathcal{U} \}$ are related to $\vartheta_M$ by positive constants $c,C>0$, $\vartheta_{c \cdot M} \leqslant \hat{\tau}_M \leqslant \vartheta_{C \cdot M}$ (see Lemma~\ref{lem: derivativeSbound}).
\end{remark}

In the remainder of this section, we establish a series of propositions that provide bounds on the martingale terms arising from both homogenized SGD and SGD. 
Throughout the proofs below, $C$ denotes a positive constant independent of $d$.
It may depend on $L(h)$, $L(I)$, $\bar\gamma$, $T$, $\|K\|_{\mathrm{op}}$,
$\|\beta^*\|_\infty$, $M$, and $\alpha$, and may change from line to line and not necessarily be the same as the constant $C$ in Lemma~\ref{lem: derivativeSbound}.

\subsubsection{Homogenized SGD Martingale Error}\label{subsubsec: homSGDmartingale}
In this section we control the martingale that arises in homogenized SGD, that is, for a test function $\varphi \colon \R^{2d} \to \R$, define
\begin{equation}
\mathcal{M}_t^\mathrm{HSGD}(\varphi) := \int_0^t \mathrm{d}\mathcal{M}_s^\mathrm{HSGD}(\varphi) = \int_0^t \gamma(s) \sqrt{I\left(B(\mathscr{X}_s) \right)}\inner{\left(\nabla \psi(\mathscr{X}_s)\right)^\top  \sqrt{K}, \nabla \varphi (\mathscr{X}_s) \left(\mathrm{d}\mathfrak{B}_s\right)^\top}.
\end{equation}
As introduced in Remark~\ref{rem: matrixEntriesV}, we are interested in controlling $\mathcal{M}_t^\mathrm{HSGD}(S(\cdot, z))$.

To control the fluctuations of this martingale, we need to control its quadratic variation, defined as follows. Consider a partition of time for $[0,t]$, that is, $0=t_0 < t_1 < \cdots < t_n = t$ such that the
size of the partition $\Delta t = \max_i \{ t_i - t_{i-1}\} \to 0$. We define for the continuous process $Y$,
\begin{equation*}
\left[Y_t(n) \right] = \sum_{k=1}^n \left(Y_{t_k} - Y_{t_{k-1}}\right)^2.
\end{equation*}
If, for every partition of time  $[0,t]$ such that $\Delta t \to 0$, the process $\left[Y_t(n)\right]$ converges in probability
to a process $\left[Y_t\right]$ as $n\to \infty$, we call $\left[Y_t\right]$ the \textit{quadratic variation} of $Y$. Using the quadratic variation of $\mathcal{M}_t^\mathrm{HSGD}$, we will show that the martingale arising from
homogenized SGD is small.
\begin{proposition}[Homogenized SGD martingale small]\label{prop: homSGDmartingale}
Suppose $f\colon \R^2 \to \R$ is $\alpha$-pseudo-\\ Lipschitz with constant $L(f)$ (see Assumption~\ref{ass: fPseudoLipschitz}). Let the statistic $S\colon \R^{2d} \times \Gamma \subset \R^{2d}\times \mathbb{C}^4 \to \mathbb{C}^{3 \times 3}$ be defined as in \eqref{eq: defS}. Then for each fixed $z\in\Gamma$ and each $T,\zeta>0$, there is some constant $C$ such that, with overwhelming probability,
\begin{equation}
\sup_{0\leqslant t \leqslant T} \norm{\mathcal{M}_{t\wedge \vartheta}^\mathrm{HSGD}(S(\cdot, z))} \leqslant C L(f) d^{\zeta/2 - 1/2}.
\end{equation}
\end{proposition}
\begin{proof}
Fix indices $i,j\in\{1,2,3\}$. Let $\varphi$ denote either
\[
\varphi(x)=S_{ij}^{\operatorname{Re},z}(x)
=
\operatorname{Re}S_{ij}(x,z)
\]
or
\[
\varphi(x)=S_{ij}^{\operatorname{Im},z}(x)
=
\operatorname{Im}S_{ij}(x,z).
\]
We prove the scalar estimate for this real-valued $\varphi$. The corresponding
matrix-valued estimate for $S(\cdot,z)$ follows by applying the same bound to
all real and imaginary parts and taking a finite union bound over the
$3\times 3$ entries. First, we rewrite the martingale increment, $\mathrm{d} \mathcal{M}_t^\mathrm{HSGD}$, as
\begin{equation}
\mathrm{d} \mathcal{M}_t^\mathrm{HSGD} \left(\varphi\right) = \gamma(t) \sqrt{I\left(B(\mathscr{X}_t) \right)}\inner{\left(\nabla \psi(\mathscr{X}_t)\right)^\top  \sqrt{K}, \nabla_x \varphi(\mathscr{X}_t, z)  \left(\mathrm{d}\mathfrak{B}_t\right)^\top}_{\R^{2d\times d}}.
\end{equation}
The quadratic variation of $\mathcal{M}_t^\mathrm{HSGD}$ is 
\begin{equation}\label{eq: quadraticVar}
\left[\mathcal{M}_t^\mathrm{HSGD}\left(\varphi\right)\right] = \int_0^t \gamma(s)^2 \, \abs{I\left(B(\mathscr{X}_s) \right)} \norm{\inner{\nabla_x \varphi(\mathscr{X}_s, z), \nabla \psi(\mathscr{X}_s)^\top  \sqrt{K}  }_{\R^{2d}}}^2 \, \mathrm{d}s.
\end{equation}
We need to compute $\sup_{0\leqslant t\leqslant T} \left[\mathcal{M}_t^\mathrm{HSGD}\left(\varphi\right)\right]$ and show that this quantity is small. In particular, we only need to show that the integrand is small. For this, we see that
\begin{align*}
\norm{\inner{\nabla_x \varphi(\mathscr{X}_s, z), \nabla \psi(\mathscr{X}_s)^\top  \sqrt{K}  }_{\R^{2d}}}^2 &\leqslant  \norm{K}_{\mathrm{op}} \norm{\nabla \psi(\mathscr{X}_s)}_{\mathrm{op}}^2 \norm{\nabla_x \varphi(\mathscr{X}_s, z)}^2.
\end{align*}
Using \eqref{eq: opNormBoundNablapsi}, we have that $\norm{\nabla \psi(\mathscr{X}_s)}_{\mathrm{op}} \leqslant C(\norm{\mathscr{X}_s}_\infty)$. By Lemma~\ref{lem: derivativeSbound}, we have a bound on $\norm{\nabla_x \varphi(\mathscr{X}_s, z)} \leqslant \norm{\nabla_x S(\mathscr{X}_s, \cdot)}_\Gamma \leqslant \frac{C(\norm{\mathscr{X}_s}_\infty)}{d}\norm{\mathscr{W}_s}$. From Lemma~\ref{lem:growthGradf},
the growth condition on $\abs{I\left(B(\mathscr{X}_s) \right)} = \EE_a[\abs{\nabla_{r_1} f(\rho_s) }^2]$ yields
\begin{align}\label{eq: integrandBound}
\begin{split}
&\abs{I\left(B(\mathscr{X}_s) \right)} \norm{\inner{\nabla_x \varphi(\mathscr{X}_s, z), \nabla \psi(\mathscr{X}_s)^\top  \sqrt{K}  }_{\R^{2d}}}^2 \\
& \hspace{0.2in}\leqslant \frac{C(\norm{\mathscr{X}_s}_\infty)}{d^2}(L(f))^2 \norm{\mathscr{W}_s}^2 \left(1 + \frac{1}{\sqrt{d}}\norm{K}^{1/2}_{\mathrm{op}} \norm{\mathscr{W}_s}\right)^{\max \{1, 2\alpha \}}\\
& \hspace{0.2in}\leqslant \frac{C(\norm{\mathscr{X}_s}_\infty)}{d} (L(f))^2 M\left(1 + \sqrt{M}\right)^{\max \{1, 2\alpha \}}.
\end{split}
\end{align}
Thus, \eqref{eq: quadraticVar} and \eqref{eq: integrandBound}, together
\begin{equation}
\sup_{0\leqslant t\leqslant T} \left[\mathcal{M}_{t\wedge \vartheta}^\mathrm{HSGD}\left(\varphi\right)\right] \leqslant C(\norm{\mathscr{X}_s}_\infty) (L(f))^2 \cdot \bar{\gamma}^2 \cdot d^{-1}.
\end{equation}
Using the fact that if $\sup_{0\leqslant t\leqslant T} \left[\mathcal{M}_{t\wedge \vartheta}^\mathrm{HSGD}\left(\varphi\right)\right] \leqslant b$ a.s. then $\PP\left[\sup_{0\leqslant t\leqslant T} \Big|\mathcal{M}_{t\wedge \vartheta}^\mathrm{HSGD}\left(\varphi\right)\Big|>p\right] \leqslant \exp(-p^2/2b)$, for 
any $\zeta>0$ and $p=\sqrt{C} L(f) d^{\zeta/2 - 1/2}$, 
\begin{equation*}
\PP\left[\sup_{0\leqslant t\leqslant T} \Big|\mathcal{M}_{t\wedge \vartheta}^\mathrm{HSGD}\left(\varphi\right)\Big|>p\right] \leqslant C\exp(-d^\zeta).    
\end{equation*}
Applying the scalar bound to the real and imaginary parts of all nine entries of
$S(\cdot,z)$ and taking a finite union bound gives the displayed matrix norm
bound. Since the matrix dimension is fixed, this only changes the constant.
\end{proof}

\subsubsection{\texorpdfstring{Bounds on the Martingales $\mathcal{M}_k^{\mathrm{Grad}}$ and $\mathcal{M}_k^{\mathrm{Hess}}$}{Bounds on the martingales MkGrad and MkHess}}\label{subsubsec: boundsMartingales}
Now we move on to the martingale increments coming from SGD applied to test functions $\varphi$. Recall, the expressions for the martingale increments for any quartic statistics $\varphi$
\begin{align*}
\Delta \mathcal{M}_k^{\mathrm{Grad}}(\varphi) &= \gamma_k \, \inner{\nabla \varphi (x_k), \nabla \mathcal{R}(x_k)} - 
\frac{\gamma_k}{\sqrt{d}} \, \inner{\nabla \varphi(x_k), \nabla_{r_1} f\left(  r_k \right) \cdot (\nabla \psi(x_k))^\top \cdot a_{k+1}}\\
\Delta \mathcal{M}_k^{\mathrm{Hess}}(\varphi) &= \frac{\gamma_k^2}{2d} \inner{\nabla^2 \varphi(x_k), \left(\nabla_{r_1} f\left(  r_k \right) \cdot (\nabla \psi(x_k))^\top \cdot a_{k+1}\right)^{\otimes 2}}\\
&- \frac{\gamma_k^2}{2d} \EE_{a_{k+1}} \left[ \inner{\nabla^2 \varphi(x_k), \left(\nabla_{r_1} f\left(  r_k \right) \cdot (\nabla \psi(x_k))^\top \cdot a_{k+1}\right)^{\otimes 2}} \middle| \mathscr{F}_k \right]
\end{align*}
with
\begin{equation*}
\mathcal{M}_k(\varphi) = \sum_{j=0}^{k - 1} \Delta \mathcal{M}_j(\varphi).
\end{equation*}
\begin{proposition}[Gradient martingale]\label{prop: gradientMartingale}
Suppose $f\colon \R^2 \to \R$ is $\alpha$-pseudo-Lipschitz with constant $L(f)$ (see Assumption~\ref{ass: fPseudoLipschitz}). Let the statistic $S\colon \R^{2d} \times \Gamma \subset \R^{2d}\times \mathbb{C}^4 \to \mathbb{C}^{3 \times 3}$ be defined as in \eqref{eq: defS}. Then for each fixed $z\in\Gamma$ and each $T, \zeta>0$, with overwhelming probability,
\begin{equation}
\sup_{0\leqslant t \leqslant T} \norm{ \mathcal{M}_{\lfloor d(t\wedge \vartheta) \rfloor}^\mathrm{Grad}(S(\cdot, z))} \leqslant  d^{-\frac{2+\alpha}{2} + \zeta}.
\end{equation}
\end{proposition}
\begin{proof}
Fix indices $i,j\in\{1,2,3\}$. Let $\varphi$ denote either
\[
\varphi(x)=S_{ij}^{\operatorname{Re},z}(x)
=
\operatorname{Re}S_{ij}(x,z)
\]
or
\[
\varphi(x)=S_{ij}^{\operatorname{Im},z}(x)
=
\operatorname{Im}S_{ij}(x,z).
\]
We prove the scalar estimate for this real-valued $\varphi$. The corresponding
matrix-valued estimate for $S(\cdot,z)$ follows by applying the same bound to
all real and imaginary parts and taking a finite union bound over the
$3\times 3$ entries. Throughout the proof of this proposition,
we will be working on the stopped version of the martingale, $\mathcal{M}_{\lfloor d(t\wedge \vartheta) \rfloor}^\mathrm{Grad}$. However, to lighten the
notation, we will suppress the $\vartheta$ dependence in the subscript as well as the $\varphi$ and simply write $\mathcal{M}_{\lfloor td \rfloor}^\mathrm{Grad} := \mathcal{M}_{\lfloor d(t\wedge \vartheta) \rfloor}^\mathrm{Grad}(\varphi)$. We have the martingale increments
\begin{align*}
\Delta \mathcal{M}_k^{\mathrm{Grad}} &= \gamma_k \, \inner{\nabla \varphi (x_k), \nabla \mathcal{R}(x_k)} - 
\frac{\gamma_k}{\sqrt{d}} \, \inner{\nabla \varphi(x_k), \nabla_{r_1} f\left(  r_k \right) \cdot (\nabla \psi(x_k))^\top \cdot a_{k+1}}\\
&= \frac{\gamma_k}{\sqrt{d}} \EE_{a_{k+1}} \left[ \inner{\nabla \varphi(x_k), \nabla_{r_1} f\left(  r_k \right) \cdot (\nabla \psi(x_k))^\top \cdot a_{k+1}} \middle| \mathscr{F}_k \right] \\
&- 
\frac{\gamma_k}{\sqrt{d}} \, \inner{\nabla \varphi(x_k), \nabla_{r_1} f\left(  r_k \right) \cdot (\nabla \psi(x_k))^\top \cdot a_{k+1}}\\
&= \frac{\gamma_k}{\sqrt{d}} \EE_{a_{k+1}} \left[  \inner{\nabla \varphi(x_k),  (\nabla \psi(x_k))^\top \cdot a_{k+1}} \nabla_{r_1} f\left(  r_k \right) \middle| \mathscr{F}_k \right] \\
&- \frac{\gamma_k}{\sqrt{d}} \,  \inner{\nabla \varphi(x_k),  (\nabla \psi(x_k))^\top \cdot a_{k+1}} \nabla_{r_1} f\left(  r_k \right).
\end{align*}
We define $\mathcal{M}_k^{\mathrm{Grad,}\beta}$ to be a new martingale with increments
\begin{align*}
\Delta \mathcal{M}_k^{\mathrm{Grad,}\beta} &= \frac{\gamma_k}{\sqrt{d}} \EE_{a_{k+1}} \left[ \operatorname{Proj}_{d^{-\frac{1}{2}}\beta} \inner{\nabla \varphi(x_k),  (\nabla \psi(x_k))^\top \cdot a_{k+1}} \nabla_{r_1} f\circ \operatorname{Proj}_{d^{-\frac{1}{2}}\beta} \left(  r_k \right)  \middle| \mathscr{F}_k \right]\\
&- \frac{\gamma_k}{\sqrt{d}} \, \operatorname{Proj}_{d^{-\frac{1}{2}}\beta} \inner{\nabla \varphi(x_k),  (\nabla \psi(x_k))^\top \cdot a_{k+1}} \nabla_{r_1} f\circ \operatorname{Proj}_{d^{-\frac{1}{2}}\beta} \left(  r_k \right),
\end{align*}
where we note that there are two projections and the projection of $r_k$ is in both coordinates
of $\R^2$, even though the gradient $\nabla_{r_1} f$ is only with respect to the first coordinate. We take $\beta = d^\zeta$ in the projection radius for some $\zeta>0$ to be determined later. We will bound $\mathcal{M}_k^{\mathrm{Grad,}\beta}$ first, and then bound the difference between $\mathcal{M}_k^\mathrm{Grad}$ and $\mathcal{M}_k^{\mathrm{Grad,}\beta}$.

We begin by computing subgaussian bounds on the quantities that are going to be projected, namely $r_k$ and $\inner{\nabla \varphi(x_k),  (\nabla \psi(x_k))^\top \cdot a_{k+1}}$. For the purposes of this section, when we refer to a vector as \textit{subgaussian}, we mean that its entries individually satisfy the stated subgaussian concentration bound. We can rewrite the quantities $r_k$ and $\inner{\nabla \varphi(x_k),  (\nabla \psi(x_k))^\top \cdot a_{k+1}}$ as
\begin{align}
\begin{split}
r_k = \frac{1}{\sqrt{d}}\begin{pmatrix}
\psi(x_k) \\
\beta^*
\end{pmatrix}^{\!\top} a_{k+1} &= \frac{1}{\sqrt{d}} \begin{pmatrix}
\psi(x_k) \\
\beta^*
\end{pmatrix}^{\!\top} \sqrt{K}v_k  \quad \text{and}\\[5pt]
\inner{\nabla \varphi(x_k),  (\nabla \psi(x_k))^\top \cdot a_{k+1}} & = \inner{\sqrt{K} \cdot \nabla \psi(x_k) \cdot \nabla \varphi(x_k),  v_k},
\end{split}
\end{align}
so $r_k$ is $\frac{1}{\sqrt{d}} \norm{\sqrt{K}}_{\mathrm{op}} \cdot \norm{\begin{pmatrix}
\psi(x_k) \\
\beta^*
\end{pmatrix}}_{\mathrm{op}}$ - subgaussian and $\inner{\nabla \varphi(x_k),  (\nabla \psi(x_k))^\top \cdot a_{k+1}}$ is $\norm{\sqrt{K} }_{\mathrm{op}} \cdot \norm{\nabla \psi(x_k)}_{\mathrm{op}} \cdot \norm{\nabla \varphi(x_k)}_{\mathrm{op}}$ - subgaussian where 
\begin{equation*}
\norm{\nabla \varphi(x_k)}_{\mathrm{op}} = \norm{\nabla_x S_{ij}(x_k,z)}_{\mathrm{op}} \leqslant \norm{\nabla_x S(x_k,\cdot)}_\Gamma \leqslant \frac{C(\norm{x_k}_\infty)}{d}\norm{W_k}.
\end{equation*}
by Lemma~\ref{lem: derivativeSbound}, and $\norm{\begin{pmatrix}
\psi(x_k) \\
\beta^*
\end{pmatrix}}_{\mathrm{op}}, \norm{\nabla \psi(x_k)}_{\mathrm{op}} \leqslant C(\norm{x_k}_\infty)$ by \eqref{eq: opNormBoundpsi} and \eqref{eq: opNormBoundNablapsi}. Thus, since we are working on the stopped processes, 
\begin{align}
\norm{r_k}_{\psi_2} = d^{-\frac{1}{2}}C(\norm{x_k}_\infty) \quad \text{and} \quad \norm{\inner{\nabla \varphi(x_k),  (\nabla \psi(x_k))^\top \cdot a_{k+1}}}_{\psi_2} = d^{-\frac{1}{2}}C(\norm{x_k}_\infty).
\end{align}
These subgaussian bounds will be used to bound the difference between $\mathcal{M}_k^{\mathrm{Grad}}$ and $\mathcal{M}_k^{\mathrm{Grad,}\beta}$.

Furthermore, leveraging the projections and the fact that $f$ is $\alpha$-pseudo-Lipschitz, which implies the growth bound on $\nabla_{r_1} f(r)$ in Lemma~\ref{lem:growthGradf}, we can derive the norm bounds
\begin{align}
\abs{\nabla_{r_1} f \circ \operatorname{Proj}_{d^{-\frac{1}{2}}\beta} \left(r_k \right)} &\leqslant L(f) C (1+d^{-\frac{1}{2}}\beta)^{\max \{1,\alpha\}}, \\
\abs{\operatorname{Proj}_{d^{-\frac{1}{2}}\beta} \inner{\nabla \varphi(x_k),  (\nabla \psi(x_k))^\top \cdot a_{k+1}}} &\leqslant  d^{-\frac{1}{2}}\beta.
\end{align}
This gives us the bound 
\begin{align}
\abs{\operatorname{Proj}_{d^{-\frac{1}{2}}\beta} \inner{\nabla \varphi(x_k),  (\nabla \psi(x_k))^\top \cdot a_{k+1}}  \nabla_{r_1} f \circ \operatorname{Proj}_{d^{-\frac{1}{2}}\beta} \left(r_k \right) }\leqslant L(f) C d^{-\frac{2+\alpha}{2}}\beta^{2 + \alpha}
\end{align}
and, since this is an almost sure bound, it holds for the expectation as well, and we get
\begin{align*}
    \abs{\Delta \mathcal{M}_k^{\mathrm{Grad,}\beta}} \leqslant 2\gamma_k L(f) C  d^{-\frac{3+\alpha}{2}}\beta^{2 + \alpha}.
\end{align*}
Applying Azuma's inequality with the assumption $n=O(d)$, we obtain
\begin{equation*}
    \sup_{1\leqslant k \leqslant n} \PP\left[\abs{\mathcal{M}_k^{\mathrm{Grad,}\beta}}> t\right] < 2\exp \left( \frac{-t^2}{2n \cdot (C d^{-\frac{3+\alpha}{2}}\beta^{2+\alpha})^2} \right) \leqslant 2\exp \left( \frac{-t^2}{C'd^{-(2+\alpha)} \beta^{2(2+\alpha)}} \right).
\end{equation*}
Thus, with overwhelming probability, 
\begin{equation}
    \sup_{1\leqslant k \leqslant n} \abs{\mathcal{M}_k^{\mathrm{Grad,}\beta}} < d^{-\frac{2+\alpha}{2}}\beta^{3+\alpha}.
\end{equation}
Finally, we bound the difference between $\{\mathcal{M}_k^{\mathrm{Grad}} \}^n_{k=1}$ and $\{\mathcal{M}_k^{\mathrm{Grad,}\beta} \}^n_{k=1}$. For ease of notation, we write
\begin{align*}
    G_k &:= \frac{\gamma_k}{\sqrt{d}} \, \inner{\nabla \varphi(x_k),  (\nabla \psi(x_k))^\top \cdot a_{k+1}} \nabla_{r_1} f\left(  r_k \right)\\
    G_{k, \beta} &:= \frac{\gamma_k}{\sqrt{d}} \, \operatorname{Proj}_{d^{-\frac{1}{2}}\beta} \inner{\nabla \varphi(x_k),  (\nabla \psi(x_k))^\top \cdot a_{k+1}} \nabla_{r_1} f\circ \operatorname{Proj}_{d^{-\frac{1}{2}} \beta} \left( r_k \right).
\end{align*}
The quantity we are trying to bound is
\begin{align*}
    \abs{\Delta \mathcal{M}_k^{\mathrm{Grad}} - \Delta \mathcal{M}_k^{\mathrm{Grad,}\beta}} &= \abs{\left(G_k - \EE_{a_{k+1}} G_k\right) - \left(G_{k,\beta} - \EE_{a_{k+1}} G_{k,\beta} \right)} \\[10pt]
    &\leqslant \abs{G_k - G_{k,\beta}} + \abs{\EE_{a_{k+1}}\left(G_k - G_{k,\beta} \right)}.
\end{align*}
First, we will show that $G_k - G_{k,\beta} = 0$ with overwhelming probability. Using the subgaussian bounds on $r_k$ and $\inner{\nabla \varphi(x_k),  (\nabla \psi(x_k))^\top \cdot a_{k+1}}$, we have
\begin{align*}
    \PP\left[G_k \neq G_{k,\beta}\right] &\leqslant \PP\left[\norm{r_k} > d^{-\frac{1}{2}}\beta \right] + \PP\left[\abs{\inner{\nabla \varphi(x_k),  (\nabla \psi(x_k))^\top \cdot a_{k+1}}}>d^{-\frac{1}{2}}\beta \right] \\[10pt]
    &< 4 \exp \left( -\frac{\beta^2}{2C}\right).
\end{align*}
Since $\beta = d^\zeta$ for some $\zeta>0$, the probability bounds above imply that $G_k - G_{k,\beta} = 0$ with overwhelming probability, and it remains to bound the difference in their expectations. For this, we have
\begin{align*}
    \abs{\EE_{a_{k+1}}\left[G_k - G_{k,\beta} \right]} &= \abs{\EE_{a_{k+1}}\left[\left(G_k - G_{k,\beta} \right) \cdot \mathds{1}\{G_k \neq G_{k,\beta} \}\right]} \\[10pt]
    &\leqslant \abs{\EE_{a_{k+1}}\left[G_k  \cdot \mathds{1}\{G_k \neq G_{k,\beta} \}\right]} + \abs{\EE_{a_{k+1}}\left[ G_{k,\beta}  \cdot \mathds{1}\{G_k \neq G_{k,\beta} \}\right]}.
\end{align*}
For $\abs{\EE_{a_{k+1}}\left[ G_{k,\beta}  \cdot \mathds{1}\{G_k \neq G_{k,\beta} \}\right]}$, we have
\begin{align*}
    \abs{\EE_{a_{k+1}}\left[ G_{k,\beta}  \cdot \mathds{1}\{G_k \neq G_{k,\beta} \}\right]} &\leqslant \max \abs{G_{k,\beta}} \cdot \PP\left[G_k \neq G_{k,\beta}\right]\\
    &\leqslant L(f) C d^{-\frac{2+\alpha}{2}}\beta^{3+\alpha} \cdot 4 \exp \left( -\frac{\beta^2}{2C}\right).
\end{align*}
For $\abs{\EE_{a_{k+1}}\left[ G_k  \cdot \mathds{1}\{G_k \neq G_{k,\beta} \}\right]}$, we have
\begin{align*}
    \abs{\EE_{a_{k+1}}\left[ G_k  \cdot \mathds{1}\{G_k \neq G_{k,\beta} \}\right]} &\leqslant \EE_{a_{k+1}}  \abs{G_k \cdot \mathds{1}\{ E_1 \}} \\[2pt]
    &+ \EE_{a_{k+1}} \abs{G_k \cdot \mathds{1}\{ E_2 \}} \\[2pt]
    &+ \EE_{a_{k+1}}  \abs{G_k \cdot \mathds{1}\{ E_3 \}} ,\\[10pt]
    \text{where }  E_1 &:= \{\norm{r_k} \leqslant d^{-\frac{1}{2}}\beta \} \cap \{\abs{\inner{\nabla \varphi(x_k),  (\nabla \psi(x_k))^\top \cdot a_{k+1}}}>d^{-\frac{1}{2}}\beta \},\\
    E_2 &:= \{\norm{r_k} > d^{-\frac{1}{2}} \beta \} \cap \{\abs{\inner{\nabla \varphi(x_k),  (\nabla \psi(x_k))^\top \cdot a_{k+1}}}\leqslant d^{-\frac{1}{2}}\beta \},\\
    E_3 &:= \{\norm{r_k} > d^{-\frac{1}{2}}\beta \} \cap \{\abs{\inner{\nabla \varphi(x_k),  (\nabla \psi(x_k))^\top \cdot a_{k+1}}} > d^{-\frac{1}{2}}\beta \}.
\end{align*}
The term $\EE_{a_{k+1}}  \abs{G_k \cdot \mathds{1}\{ E_1 \}}$ can be bounded as
\begin{align*}
    &\EE_{a_{k+1}}  \abs{G_k \cdot \mathds{1}\{ E_1 \}} \leqslant L(f) C  (d^{-\frac{1}{2}}\beta)^{\max \{1,\alpha\}} \cdot\\[10pt]
    &\hspace{0.6in}\EE_{a_{k+1}}\left[  \abs{\inner{\nabla \varphi(x_k),  (\nabla \psi(x_k))^\top \cdot a_{k+1}}} \cdot \mathds{1}\{\abs{\inner{\nabla \varphi(x_k),  (\nabla \psi(x_k))^\top \cdot a_{k+1}}} > d^{-\frac{1}{2}}\beta \}  \right], 
\end{align*}
where the expectation on the right-hand side is exponentially small due to being a tail of a sub-gaussian first moment (where $\beta^2$ is larger than the sub-gaussian variance and grows with $d$). By similar reasoning, $\EE_{a_{k+1}}  \abs{G_k \cdot \mathds{1}\{ E_2 \}}$ is also exponentially small (using the growth bound on $\nabla_{r_1} f$). For $\EE_{a_{k+1}}  \abs{G_k \cdot \mathds{1}\{ E_3 \}}$, we have
\begin{align*}
    \EE_{a_{k+1}}  \abs{G_k \cdot \mathds{1}\{ E_3 \}} &\leqslant \frac{\gamma_k}{\sqrt{d}} \EE_{a_{k+1}}\Big[  \abs{\nabla_{r_1} f\left(r_k \right) \cdot \mathds{1}\{\norm{r_k} > d^{-\frac{1}{2}}\beta \}  } \cdot \\
    &\hspace{-0.2in}\abs{\inner{\nabla \varphi(x_k),  (\nabla \psi(x_k))^\top \cdot a_{k+1}} \cdot \mathds{1}\{\abs{\inner{\nabla \varphi(x_k),  (\nabla \psi(x_k))^\top \cdot a_{k+1}}} > d^{-\frac{1}{2}}\beta \}  }\Big]\\
    &\leqslant \frac{\gamma_k}{\sqrt{d}}\EE_{a_{k+1}} \abs{\nabla_{r_1} f\left(r_k \right) \cdot \mathds{1}\{\norm{r_k} > d^{-\frac{1}{2}}\beta \}  }^2 \cdot \\
    & \hspace{-0.5in}\EE_{a_{k+1}} \abs{\inner{\nabla \varphi(x_k),  (\nabla \psi(x_k))^\top \cdot a_{k+1}} \cdot \mathds{1}\{\abs{\inner{\nabla \varphi(x_k),  (\nabla \psi(x_k))^\top \cdot a_{k+1}}} > d^{-\frac{1}{2}}\beta \}  }^2.
\end{align*}
This is a product of tails of Gaussian moments, which is again exponentially small. Thus, we conclude that, with overwhelming probability, $\sup_{1\leqslant k \leqslant n} \abs{\Delta \mathcal{M}_k^{\mathrm{Grad}} - \Delta \mathcal{M}_k^{\mathrm{Grad,}\beta}}$ is exponentially small and thus, taking $\beta = d^\zeta$, we conclude that, with overwhelming probability, 
\begin{equation*}
    \sup_{1\leqslant k \leqslant n} \abs{\mathcal{M}_k^{\mathrm{Grad}}} <  d^{- \frac{2+\alpha}{2} +\zeta(3+\alpha)}.
\end{equation*}
Adjusting the value of $\zeta$, recalling that all of this has been proved on the stopped process, and applying the scalar bound to the real and imaginary parts of all nine entries of
$S(\cdot,z)$ and taking a finite union bound gives the displayed matrix norm
bound. Since the matrix dimension is fixed, this only changes the constant.
\end{proof}
\begin{proposition}[Hessian martingale]\label{prop: hessianMartingale}
Suppose $f\colon \R^2 \to \R$ is $\alpha$-pseudo-Lipschitz with constant $L(f)$ (see Assumption~\ref{ass: fPseudoLipschitz}). Let the statistic $S\colon \R^{2d} \times \Gamma \subset \R^{2d}\times \mathbb{C}^4 \to \mathbb{C}^{3 \times 3}$ be defined as in \eqref{eq: defS}. Then for each fixed $z\in\Gamma$ and each $T, \zeta>0$, with overwhelming probability,
\begin{equation}
\sup_{0\leqslant t \leqslant T} \big|\mathcal{M}_{\lfloor d(t\wedge \vartheta) \rfloor}^\mathrm{Hess}(S(\cdot,z))\big| \leqslant  d^{-(2+\alpha)+ \zeta}.
\end{equation}
\end{proposition}
\begin{proof}
As in the proof of the previous proposition, we will work on the stopped version of the
martingale but will suppress the $\vartheta$ dependence in the subscript in order to lighten the notation.
We also, as before, let $\varphi$ denote either
\[
\varphi(x)=S_{ij}^{\operatorname{Re},z}(x)
=
\operatorname{Re}S_{ij}(x,z)
\]
or
\[
\varphi(x)=S_{ij}^{\operatorname{Im},z}(x)
=
\operatorname{Im}S_{ij}(x,z). 
\]
We have the martingale increment
\begin{align*}
\Delta \mathcal{M}_k^{\mathrm{Hess}} &= \frac{\gamma_k^2}{2d} \inner{\nabla^2 \varphi(x_k), \left(\nabla_{r_1} f\left(  r_k \right) \cdot (\nabla \psi(x_k))^\top \cdot a_{k+1}\right)^{\otimes 2}}\\
&- \frac{\gamma_k^2}{2d} \EE_{a_{k+1}} \left[ \inner{\nabla^2 \varphi(x_k), \left(\nabla_{r_1} f\left(  r_k \right) \cdot (\nabla \psi(x_k))^\top \cdot a_{k+1}\right)^{\otimes 2}} \middle| \mathscr{F}_k \right]\\
&= \frac{\gamma_k^2}{2d} \, \inner{
\inner{\nabla^2 \varphi(x_k), \left[(\nabla \psi(x_k))^\top \right]^{\otimes 2}},  a_{k+1}^{\otimes 2}} \left(\nabla_{r_1} f\left(  r_k \right)\right)^2 \\
&- \frac{\gamma_k^2}{2d} \EE_{a_{k+1}} \left[ \inner{
\inner{\nabla^2 \varphi(x_k), \left[(\nabla \psi(x_k))^\top \right]^{\otimes 2}},  a_{k+1}^{\otimes 2}} \left(\nabla_{r_1} f\left(  r_k \right)\right)^2 \middle| \mathscr{F}_k \right].
\end{align*}
We define $\mathcal{M}_k^{\mathrm{Hess,} \beta}$ to be a new martingale with increments
\begin{align*}
&\Delta \mathcal{M}_k^{\mathrm{Hess,}\beta} = \frac{\gamma_k^2}{2d} \, \operatorname{Proj}_{d^{-\frac{1}{2}} \beta} \inner{
\inner{\nabla^2 \varphi(x_k), \left[(\nabla \psi(x_k))^\top \right]^{\otimes 2}},  a_{k+1}^{\otimes 2}}  \left( \nabla_{r_1} f \circ \operatorname{Proj}_{d^{-\frac{1}{2}} \beta} \left(r_k \right)\right)^2 \\
&\hspace{0.3in}- \frac{\gamma_k^2}{2d} \EE_{a_{k+1}} \left[ \operatorname{Proj}_{d^{-\frac{1}{2}} \beta} \inner{
\inner{\nabla^2 \varphi(x_k), \left[(\nabla \psi(x_k))^\top \right]^{\otimes 2}},  a_{k+1}^{\otimes 2}}  \left( \nabla_{r_1} f \circ \operatorname{Proj}_{d^{-\frac{1}{2}}\beta} \left(r_k \right)\right)^2   \middle| \mathscr{F}_k \right].
\end{align*}
The approach here is similar to the procedure for bounding  $\mathcal{M}_k^{\mathrm{Grad}}$. 
As we saw in the proof of the previous Proposition, $r_k$ is $\frac{1}{\sqrt{d}}\norm{\sqrt{K}}_{\mathrm{op}} \cdot \norm{\begin{pmatrix}
\psi(x_k) \\
\beta^*
\end{pmatrix}}_{\mathrm{op}}$--subgaussian in each entry. To obtain a concentration bound for $\inner{
\inner{\nabla^2 \varphi(x_k), \left[(\nabla \psi(x_k))^\top \right]^{\otimes 2}},  a_{k+1}^{\otimes 2}}$, we rewrite it as
\begin{align*}
    \inner{
    \inner{\nabla^2 \varphi(x_k), \left[(\nabla \psi(x_k))^\top \right]^{\otimes 2}},  a_{k+1}^{\otimes 2}} &= \inner{
    \inner{\nabla^2 \varphi(x_k), \left[(\nabla \psi(x_k))^\top \right]^{\otimes 2}},  \left(\sqrt{K}v_{k+1}\right)^{\otimes 2}} \\
    &= \inner{
    \inner{\nabla^2 \varphi(x_k), \left[(\nabla \psi(x_k))^\top \sqrt{K}\right]^{\otimes 2}},  v_{k+1}^{\otimes 2}},
\end{align*}
where the vector $v_k$ has independent standard gaussian entries. Since $\nabla^2 \varphi(x_k)\in \R^{2d \times 2d}$ and $\left[(\nabla \psi(x_k))^\top \sqrt{K}\right]^{\otimes 2} \in (\R^{2d \times d})^{\otimes 2}$, we get $\inner{\nabla^2 \varphi(x_k), \left[(\nabla \psi(x_k))^\top \sqrt{K}\right]^{\otimes 2}} \in \R^{d \times d}$.

Moreover, by Lemma~\ref{lem: derivativeSbound} and \eqref{eq: opNormBoundNablapsi}, we have
\begin{align*}
    \norm{\inner{\nabla^2 \varphi(x_k), \left[(\nabla \psi(x_k))^\top \sqrt{K}\right]^{\otimes 2}}}_{\mathrm{op}} \leqslant \, \norm{K}_{\mathrm{op}} \cdot \norm{\nabla \psi(x_k)}^2_{\mathrm{op}} \cdot \norm{\nabla^2 \varphi(x_k)}_{\mathrm{op}}  \leqslant \frac{C(\norm{x_k}_\infty)}{d},
\end{align*}
so by Hanson-Wright inequality we obtain that, for large enough $t$,
{\small
\begin{align*}
    &\PP\left[ \inner{
    \inner{\nabla^2 \varphi(x_k), \left[(\nabla \psi(x_k))^\top \sqrt{K}\right]^{\otimes 2}},  v_{k+1}^{\otimes 2}}  >t   \right]\\ 
    &\hspace{0.1in}< 2 \exp \left(  -C \min \biggl\{ \frac{t^2}{\norm{\inner{\nabla^2 \varphi(x_k), \left[(\nabla \psi(x_k))^\top \sqrt{K}\right]^{\otimes 2}}}^2}, \frac{t}{\norm{\inner{\nabla^2 \varphi(x_k), \left[(\nabla \psi(x_k))^\top \sqrt{K}\right]^{\otimes 2}}}_{\mathrm{op}}} \biggr\} \right).
\end{align*}
}
Using the fact that 
\begin{equation*}
\norm{\inner{\nabla^2 \varphi(x_k), \left[(\nabla \psi(x_k))^\top \sqrt{K}\right]^{\otimes 2}}}^2 \leqslant d \norm{\inner{\nabla^2 \varphi(x_k), \left[(\nabla \psi(x_k))^\top \sqrt{K}\right]^{\otimes 2}}}_{\mathrm{op}}^2 \leqslant \frac{C(\norm{x_k}_\infty)}{d}, 
\end{equation*}
we conclude that
\begin{align*}
    \PP\left[ \abs{\inner{
    \inner{\nabla^2 \varphi(x_k), \left[(\nabla \psi(x_k))^\top \sqrt{K}\right]^{\otimes 2}},  v_{k+1}^{\otimes 2}}} >t   \right] < 2 \exp \left(  -\frac{ d \cdot \min \{t^2, t \}}{C} \right).
\end{align*}
In particular, this tells us that, for any $\zeta>0$, 
\begin{equation*}
\abs{\inner{
    \inner{\nabla^2 \varphi(x_k), \left[(\nabla \psi(x_k))^\top \right]^{\otimes 2}},  a_{k+1}^{\otimes 2}}} \leqslant d^{-\frac{1}{2}+\zeta}    
\end{equation*}
with overwhelming probability. 

Having obtained concentration bounds for $r_k$ and $\inner{
\inner{\nabla^2 \varphi(x_k), \left[(\nabla \psi(x_k))^\top \right]^{\otimes 2}},  a_{k+1}^{\otimes 2}}$, we proceed to bound $\mathcal{M}_k^{\mathrm{Hess,} \beta}$ and show that it is close to $\mathcal{M}_k^{\mathrm{Hess}}$.  From the projections and the growth bound on $\nabla_{r_1} f$ in Lemma~\ref{lem:growthGradf}, we can derive the norm bounds
\begin{align}
    \big | \nabla_{r_1} f \circ \operatorname{Proj}_{d^{-\frac{1}{2}}\beta} \left(r_k \right)\big |^2 &\leqslant \left( L(f) C  (1+d^{-\frac{1}{2}}\beta)^{\max \{1, \alpha\}} \right)^2\\[5pt]
    \abs{\operatorname{Proj}_{d^{-\frac{1}{2}}\beta} \inner{
    \inner{\nabla^2 \varphi(x_k), \left[(\nabla \psi(x_k))^\top \right]^{\otimes 2}},  a_{k+1}^{\otimes 2}}} &\leqslant d^{-\frac{1}{2}}\beta,
\end{align}
and thus
\begin{align*}
&\abs{\operatorname{Proj}_{d^{-\frac{1}{2}}\beta} \inner{
\inner{\nabla^2 \varphi(x_k), \left[(\nabla \psi(x_k))^\top \right]^{\otimes 2}},  a_{k+1}^{\otimes 2}}  \left( \nabla_{r_1} f \circ \operatorname{Proj}_{d^{-\frac{1}{2}}\beta} \left(r_k \right)\right)^2} \\
&\hspace{4in}\leqslant \left( L(f) C\right)^2 d^{-\frac{3+2\alpha}{2}} \beta^{3+2\alpha}.
\end{align*}
Since this is an almost sure bound, it holds for the expectation as well and we get 
\begin{equation*}
    \abs{\Delta \mathcal{M}_k^{\mathrm{Hess,}\beta}} \leqslant \gamma_k^2 \left( L(f) C\right)^2 d^{-\frac{5+2\alpha}{2}} \beta^{3+2\alpha}.
\end{equation*}
Applying Azuma's inequality with $n=O(d)$, we obtain
\begin{equation*}
    \sup_{1\leqslant k \leqslant n} \PP \left[\big|\mathcal{M}_k^{\mathrm{Hess,}\beta}\big|> t \right] < 2\exp \left( \frac{-t^2}{2n \cdot \left(C d^{-\frac{5+2\alpha}{2}} \beta^{3+2\alpha}\right)^2} \right) \leqslant 2\exp \left( \frac{-t^2}{C' d^{-2(2+\alpha)}\beta^{2(3+2\alpha)}} \right)
\end{equation*}
so, with overwhelming probability, 
\begin{equation}
    \sup_{1\leqslant k \leqslant n} \abs{\mathcal{M}_k^{\mathrm{Hess,}\beta}} < d^{-(2+\alpha)} \beta^{4+2\alpha}.
\label{eq:  hessianBound}
\end{equation}
It remains only to bound the difference between $\{\mathcal{M}_k^{\mathrm{Hess}} \}^n_{k=1}$ and $\{\mathcal{M}_k^{\mathrm{Hess,}\beta} \}^n_{k=1}$. This follows a very similar argument to what was in the proof of Proposition~\ref{prop: gradientMartingale}, we write
\begin{align*}
    G_k &:= \frac{\gamma_k^2}{2d} \, \inner{
    \inner{\nabla^2 \varphi(x_k), \left[(\nabla \psi(x_k))^\top \right]^{\otimes 2}},  a_{k+1}^{\otimes 2}} \left(\nabla_{r_1} f\left(  r_k \right)\right)^2\\
    G_{k, \beta} &:= \frac{\gamma_k^2}{2d} \, \operatorname{Proj}_{d^{-\frac{1}{2}}\beta} \inner{
    \inner{\nabla^2 \varphi(x_k), \left[(\nabla \psi(x_k))^\top \right]^{\otimes 2}},  a_{k+1}^{\otimes 2}}  \left( \nabla_{r_1} f \circ \operatorname{Proj}_{d^{-\frac{1}{2}} \beta} \left(r_k \right)\right)^2.
\end{align*}
The quantity we are trying to bound is

\begin{align*}
    \abs{\Delta \mathcal{M}_k^{\mathrm{Hess}}- \Delta \mathcal{M}_k^{\mathrm{Hess,}\beta}} &= \abs{\left(G_k - \EE_{a_{k+1}} G_k\right) - \left(G_{k,\beta} - \EE_{a_{k+1}} G_{k,\beta} \right)} \\[10pt]
    &\leqslant \abs{G_k - G_{k,\beta}} + \abs{\EE_{a_{k+1}}\left(G_k - G_{k,\beta} \right)}.
\end{align*}

As in the proof of Proposition 
\ref{prop: gradientMartingale}, the first of the terms on the right-hand side is $0$ with overwhelming probability, while the second is exponentially small. Computing the bound for $\abs{\EE_{a_{k+1}}\left(G_k - G_{k,\beta} \right)}$ is similar to what was done in the previous proof and is not repeated here. To see that $\abs{G_k - G_{k,\beta}} = 0$ with overwhelming probability, we write
\begin{align*}
    \PP\left[G_k \neq G_{k,\beta}\right] &\leqslant \PP[\norm{r_k} > d^{-\frac{1}{2}}\beta] + \PP[\abs{\inner{
    \inner{\nabla^2 \varphi(x_k), \left[(\nabla \psi(x_k))^\top \right]^{\otimes 2}},  a_{k+1}^{\otimes 2}}}>d^{-\frac{1}{2}}\beta] \\[10pt]
    &< 2 \exp \left( -\frac{\beta^2}{2C}\right) + 2 \exp \left(  -\frac{\min\{\beta^2, d^{\frac{1}{2}} \beta\}}{2C} \right).
\end{align*}
Thus, $\abs{ \mathcal{M}_k^{\mathrm{Hess}} -  \mathcal{M}_k^{\mathrm{Hess,}\beta}}$ is exponentially small with overwhelming probability. Using \eqref{eq:  hessianBound} and setting $\beta$ to be an arbitrarily small power of $d$, we obtain that, with overwhelming probability, 
\begin{equation*}
    \sup_{1\leqslant k \leqslant n} \abs{\mathcal{M}_k^{\mathrm{Hess}}} < d^{-(2+\alpha)+\zeta(4+2\alpha)}.
\end{equation*}
Adjusting the value of $\zeta$, recalling that all of this has been proved on the stopped process, and applying the scalar bound to the real and imaginary parts of all nine entries of
$S(\cdot,z)$ and taking a finite union bound gives the displayed matrix norm
bound. Since the matrix dimension is fixed, this only changes the constant.
\end{proof}

\subsubsection{\texorpdfstring{Bounds on the Higher Order Error Term in the Taylor Expansion, $\mathcal{E}_t^{\mathrm{High}}$}{Bounds on the higher-order error term in the Taylor expansion, EtHigh}}
\label{subsubsec: highOrderTerm}
This section is devoted to showing that the high-order terms (third and higher-order terms derived in the Taylor expansion of the statistic $\varphi$) are negligible. Recall the third-order derivative in the Taylor expansion \eqref{eq: taylorExpansion} is
\begin{equation}\label{eq: thirdOrderTaylor}
-\frac{\gamma_k^3}{2} \int_0^1 (1-s)^2 \inner{\nabla^3 \varphi(x_k -\gamma_k\,  s \, \nabla_x \Psi(x_k; a_{k+1})), \left(\nabla_x \Psi(x_k;a_{k+1})\right)^{\otimes 3}} \, \mathrm{d}s.
\end{equation}
At first glance, these error terms look, in terms of $d$, large due to the tensor $\left(\nabla_x \Psi(x_k;a_{k+1})\right)^{\otimes 3}$, but the scaling will be sufficient to control this term.

\begin{proposition}[Higher-order error term]\label{prop: highOrderTerm}
Suppose $f\colon \R^2 \to \R$ is $\alpha$-pseudo-Lipschitz with constant $L(f)$ (see Assumption~\ref{ass: fPseudoLipschitz}). Let the statistic $S\colon \R^{2d} \times \Gamma \subset \R^{2d}\times \mathbb{C}^4 \to \mathbb{C}^{3 \times 3}$ be defined as in \eqref{eq: defS}. Then for each fixed $z\in\Gamma$ and each $T>0$, with
overwhelming probability,
\begin{equation}
\sup_{0\leqslant t \leqslant T }  \sum_{k=0}^{\lfloor (t\wedge \vartheta)d\rfloor -1} \Big| \EE_{a_{k+1}} \left[ \mathcal{E}_k^{\mathrm{High}}(S(\cdot,z))\big| \mathscr{F}_k \right] \Big| \leqslant C d^{-1/2}. 
\end{equation}
\end{proposition}
\begin{proof}
Fix indices $i,j\in\{1,2,3\}$. Let $\varphi$ denote either
\[
\varphi(x)=S_{ij}^{\operatorname{Re},z}(x)
=
\operatorname{Re}S_{ij}(x,z)
\]
or
\[
\varphi(x)=S_{ij}^{\operatorname{Im},z}(x)
=
\operatorname{Im}S_{ij}(x,z).
\]
We prove the scalar estimate for this real-valued $\varphi$. The corresponding
matrix-valued estimate for $S(\cdot,z)$ follows by applying the same bound to
all real and imaginary parts and taking a finite union bound over the
$3\times 3$ entries. First, because $\psi_i$ and the diagonal resolvents depend only on the coordinate pair
$(u_i,v_i)$, the third derivative of each entry of $S$ is supported only on
triples of derivative indices belonging to the same coordinate pair, so the inner product in \eqref{eq: thirdOrderTaylor} can be simplified as
\begin{align*}
& \inner{\nabla^3 \varphi(x_k -\gamma_k\,  s \, \nabla_x \Psi(x_k; a_{k+1})), \left(\nabla_x \Psi(x_k;a_{k+1})\right)^{\otimes 3}}  \\[5pt]
&\hspace{0.1in}= \sum_{h,i,j=1}^{2d} \left(\nabla^3 \varphi(x_k -\gamma_k\,  s \, \nabla_x \Psi(x_k; a_{k+1}))\right)_{hij} \left(\nabla_x \Psi(x_k)\right)_h \left(\nabla_x \Psi(x_k)\right)_i \left(\nabla_x \Psi(x_k)\right)_j\\
& \hspace{0.1in} = \sum_{\substack{h=1 \\ i,j\in \{h, h+d\}}}^{d} \left(\nabla^3 \varphi(x_k -\gamma_k\,  s \, \nabla_x \Psi(x_k; a_{k+1}))\right)_{hij} \left(\nabla_x \Psi(x_k)\right)_h \left(\nabla_x \Psi(x_k)\right)_i \left(\nabla_x \Psi(x_k)\right)_j \\
&\hspace{0.1in}+  \sum_{\substack{h=d+1 \\ i,j\in \{h, h-d\}}}^{2d} \left(\nabla^3 \varphi(x_k -\gamma_k\,  s \, \nabla_x \Psi(x_k; a_{k+1}))\right)_{hij} \left(\nabla_x \Psi(x_k)\right)_h \left(\nabla_x \Psi(x_k)\right)_i \left(\nabla_x \Psi(x_k)\right)_j.
\end{align*}
Moreover, a simple computation shows that, for any $0 < s < 1$, we have
\begin{align*}
\norm{x_k -\gamma_k\,  s \, \nabla_x \Psi(x_k; a_{k+1})}_\infty &= \norm{x_k -s \frac{\gamma_k}{\sqrt{d}} \nabla_{r_1} f\left(  r_k \right) (\nabla \psi(x_k))^\top a_{k+1}}_\infty \\[5pt]
&\leqslant \norm{x_k}_\infty + \frac{\bar{\gamma}}{\sqrt{d}} \abs{\nabla_{r_1}f(r_k)} \norm{\left(\nabla \psi(x_k)\right)^\top a_{k+1}}_\infty.
\end{align*}
If we write $a_{k+1} = \sqrt{K}v_k$ with $v_k \sim \mathcal{N}(0,I_d)$, then by \eqref{eq: opNormBoundNablapsi} we have that, with overwhelming probability, 
\begin{align*}
\norm{\left(\nabla \psi(x_k)\right)^\top a_{k+1}}_\infty &= \norm{\left(\nabla \psi(x_k)\right)^\top  \sqrt{K}v_k}_\infty \leqslant 
\norm{\nabla \psi(x_k)}_{\mathrm{op}}
\norm{\sqrt{K}}_{\mathrm{op}} \norm{v_k} \leqslant C(\norm{x_k}_\infty)  \sqrt{d}.
\end{align*}

Additionally, by Lemma~\ref{lem:growthGradf},
\begin{equation*}
\abs{\nabla_{r_1}f(r_k)} \leqslant C(\alpha) L(f) \left(1 + \frac{1}{\sqrt{d}}\norm{K}^{1/2}_{\mathrm{op}} \norm{W_k}\right)^{\max \{1, \alpha\}} \quad \text{w.o.p.}    
\end{equation*}
Thus, for any $k \leqslant (t \wedge \vartheta)d$, we get
\begin{equation*}
\norm{x_k -\gamma_k\,  s \, \nabla_x \Psi(x_k; a_{k+1})}_\infty \leqslant C(\norm{x_k}_\infty),
\end{equation*}
where $C$ is a constant depending on $\bar{\gamma}$, $\norm{K}_{\mathrm{op}}$, $L(f)$, and $\alpha$, but independent of $d$.

Next, we know
\begin{equation*}
\abs{\left(\nabla^3 \varphi(x_k -\gamma_k\,  s \, \nabla_x \Psi(x_k; a_{k+1}))\right)_{hij}}  \leqslant \frac{1}{d} C(\norm{x_k -\gamma_k\,  s \, \nabla_x \Psi(x_k; a_{k+1})}_\infty) \leqslant \frac{C(\norm{x_k}_\infty)}{d},
\end{equation*}
and so
\begin{align*}
&\abs{\inner{\nabla^3 \varphi(x_k -\gamma_k\,  s \, \nabla_x \Psi(x_k; a_{k+1})), \left(\nabla_x \Psi(x_k;a_{k+1})\right)^{\otimes 3}}}\\[0.5pt]
&\hspace{0.5in}\leqslant \frac{C(\norm{x_k}_\infty)}{d} \sum_{\substack{h=1 \\ i,j\in \{h, h+d\}}}^{d} \Big|\left(\nabla_x \Psi(x_k;a_{k+1})\right)_h \left(\nabla_x \Psi(x_k;a_{k+1})\right)_i \left(\nabla_x \Psi(x_k;a_{k+1})\right)_j\Big| \\[0.5pt]
&\hspace{0.5in}+  \frac{C(\norm{x_k}_\infty)}{d}\sum_{\substack{h=d+1 \\ i,j\in \{h, h-d\}}}^{2d} \Big|\left(\nabla_x \Psi(x_k;a_{k+1})\right)_h \left(\nabla_x \Psi(x_k;a_{k+1})\right)_i \left(\nabla_x \Psi(x_k;a_{k+1})\right)_j\Big|
\end{align*}
with overwhelming probability.

Furthermore, we can bound 
\begin{align*}
\EE_{a_{k+1}} \biggl[ &\Big|\left(\nabla_x \Psi(x_k;a_{k+1})\right)_h \left(\nabla_x \Psi(x_k;a_{k+1})\right)_i \left(\nabla_x \Psi(x_k;a_{k+1})\right)_j\Big| \big| \mathscr{F}_k \biggr] \\
&\leqslant \frac{1}{d\sqrt{d}} \norm{\nabla \psi(x_k)}_{\mathrm{op}}^3 \EE_{a_{k+1}} \left[\Big| \left(\nabla_{r_1} f\left(  r_k \right)\right)^3  \left(a_{k+1}\right)_h  \left(a_{k+1}\right)_i  \left(a_{k+1}\right)_j \Big| \big| \mathscr{F}_k \right],
\end{align*}
where
\begin{align*}
&\EE_{a_{k+1}} \left[\Big| \left(\nabla_{r_1} f\left(  r_k \right)\right)^3  \left(a_{k+1}\right)_h  \left(a_{k+1}\right)_i  \left(a_{k+1}\right)_j \Big| \big| \mathscr{F}_k \right]\\[10pt]
&\hspace{0.5in}\leqslant \sqrt{\EE_{a_{k+1}} \left[ \left(\nabla_{r_1} f\left(  r_k \right)\right)^6  \big| \mathscr{F}_k \right]} \cdot \sqrt{\EE_{a_{k+1}} \left[ \left(a_{k+1}\right)_h^2  \left(a_{k+1}\right)_i^2  \left(a_{k+1}\right)_j^2 \right]}\\[10pt]
&\hspace{0.5in}\leqslant \sqrt{C(\alpha) (L(f))^6 \left(1 + \frac{1}{\sqrt{d}}\norm{K}^{1/2}_{\mathrm{op}} \norm{W_k}\right)^{\max \{1, 6\alpha \}}} \cdot \norm{K}^{3/2}_{\mathrm{op}}.
\end{align*}

Therefore, noting that $k\leqslant (t\wedge \vartheta)d$ and using \eqref{eq: opNormBoundNablapsi}, yields
\begin{equation*}
\EE_{a_{k+1}} \biggl[\Big| \left(\nabla_x \Psi(x_k;a_{k+1})\right)_h \left(\nabla_x \Psi(x_k;a_{k+1})\right)_i \left(\nabla_x \Psi(x_k;a_{k+1})\right)_j\Big| \big| \mathscr{F}_k \biggr] \leqslant \frac{C(\norm{x_k}_\infty)}{d \sqrt{d}}.
\end{equation*}

Consequently, for any $k\leqslant (t\wedge \vartheta)d$, the quantity
\begin{align*}
\EE_{a_{k+1}}\left[\abs{-\frac{\gamma_k^3}{2} \int_0^1 (1-s)^2 \inner{\nabla^3 \varphi(x_k -\gamma_k\,  s \, \nabla_x \Psi(x_k; a_{k+1})), \left(\nabla_x \Psi(x_k;a_{k+1})\right)^{\otimes 3}} \, \mathrm{d}s } \big| \mathscr{F}_k \right]
\end{align*}
is bounded by $C(\norm{x_k}_\infty)/(d \sqrt{d})$. Summing over at most $Td$ indices gives
\[
\sup_{0\leqslant t\leqslant T}
\sum_{k=0}^{\lfloor (t\wedge\vartheta)d\rfloor-1} \Big| \EE_{a_{k+1}} \left[ \mathcal{E}_k^{\mathrm{High}}(\varphi)\big|
\mathscr F_k \right]
\Big|
\leqslant
C d^{-1/2}.
\]
Applying the scalar bound to the real and imaginary parts of all nine entries of
$S(\cdot,z)$ and taking a finite union bound gives the displayed matrix norm
bound. Since the matrix dimension is fixed, this only changes the constant. Thus the third-order Taylor remainder in \eqref{eq: taylorExpansion}
vanish when $d$ grows large after summing up over $k$.
\end{proof}
\subsubsection{\texorpdfstring{Bounds on the Lower Order Terms in the Hessian, $\mathcal{E}_t^{\mathrm{Hess}}$}{Bounds on the lower order terms in the Hessian, EtHess}}
\label{subsubsec: hessianErrorTerm}
We now bound the error term, $\sup_{0\leqslant t \leqslant T }  \sum_{k=0}^{\lfloor (t\wedge \vartheta)d\rfloor -1} \Big|\EE_{a_{k+1}} \left[ \mathcal{E}_k^{\mathrm{Hess}}\middle| \mathscr{F}_k \right]\Big|$, in \eqref{eq: SGDDoobContinuous}.  For this, we utilize the
operator norm and its dual norm, the nuclear norm.
\begin{proposition}[Hessian error term]\label{prop: hessianErrorTerm}
Suppose $f\colon \R^2 \to \R$ is $\alpha$-pseudo-Lipschitz with constant $L(f)$ (see Assumption~\ref{ass: fPseudoLipschitz}). Let the statistic $S\colon \R^{2d} \times \Gamma \subset \R^{2d}\times \mathbb{C}^4 \to \mathbb{C}^{3 \times 3}$ be defined as in \eqref{eq: defS}. Then for each fixed $z\in\Gamma$ and each $T>0$, with
overwhelming probability,
\begin{equation}
\sup_{0\leqslant t \leqslant T }  \sum_{k=0}^{\lfloor (t\wedge \vartheta)d\rfloor -1} \abs{\EE_{a_{k+1}} \left[ \mathcal{E}_k^{\mathrm{Hess}}(S(\cdot, z))\middle| \mathscr{F}_k \right]} \leqslant C (L(f))^4 d^{-1}. 
\end{equation}
\end{proposition}
\begin{proof}
Fix indices $i,j\in\{1,2,3\}$. Let $\varphi$ denote either
\[
\varphi(x)=S_{ij}^{\operatorname{Re},z}(x)
=
\operatorname{Re}S_{ij}(x,z)
\]
or
\[
\varphi(x)=S_{ij}^{\operatorname{Im},z}(x)
=
\operatorname{Im}S_{ij}(x,z).
\]
We prove the scalar estimate for this real-valued $\varphi$. The corresponding
matrix-valued estimate for $S(\cdot,z)$ follows by applying the same bound to
all real and imaginary parts and taking a finite union bound over the
$3\times 3$ entries. Define $\Pi_k := Q_k Q_k^\top$ and note that $\norm{\Pi_k}^2 = \text{rank}(\Pi_k) =2$. Recall

\begin{equation*}
\mathcal{E}_k^{\mathrm{Hess}}(\varphi) := -\frac{\gamma_k^2}{2 d} \inner{\mathcal{B}, \sqrt{K}\Pi_k \sqrt{K}} + \frac{\gamma_k^2}{2 d} \inner{\mathcal{B}, \left( \sqrt{K} \Pi_k v_k \right)^{\otimes 2}}
\end{equation*}

where $\mathcal{B}:= \nabla_{r_1} f\left(  r_k \right)^2 \nabla \psi(x_k)  \nabla^2 \varphi(x_k)\, (\nabla \psi(x_k))^\top$. First, we consider the following term

\begin{align*}
    \abs{\inner{\mathcal{B}, \sqrt{K}\Pi_k \sqrt{K}}} &=
    \abs{\inner{\nabla_{r_1} f\left(  r_k \right)^2 \nabla \psi(x_k)  \nabla^2 \varphi(x_k)\, (\nabla \psi(x_k))^\top, \sqrt{K}\Pi_k \sqrt{K}}}\\
    & \leqslant \norm{\sqrt{K}\Pi_k \sqrt{K}}_* \norm{\nabla_{r_1} f\left(  r_k \right)^2 \nabla \psi(x_k)  \nabla^2 \varphi(x_k)\, (\nabla \psi(x_k))^\top }_{\mathrm{op}}\\
    &\leqslant \norm{\sqrt{K}\Pi_k \sqrt{K}}_* \nabla_{r_1} f\left(  r_k \right)^2 \norm{\nabla \psi(x_k)}_{\mathrm{op}}^2 \norm{\nabla^2 \varphi(x_k)}_{\mathrm{op}} \\
    &\leqslant \norm{K}_{\mathrm{op}} \norm{\Pi_k}_* \nabla_{r_1} f\left(  r_k \right)^2 \norm{\nabla \psi(x_k)}_{\mathrm{op}}^2 \norm{\nabla^2 \varphi(x_k)}_{\mathrm{op}}.
\end{align*}

From Lemma~\ref{lem:growthGradf}, we have 
\begin{equation*}
\EE_{a_{k+1}}[\abs{\nabla_{r_1} f(r_k) }^2\big| \mathscr{F}_k] \leqslant C (L(f))^2 \left(1 + \frac{1}{\sqrt{d}}\norm{K}^{1/2}_{\mathrm{op}} \norm{W_k}\right)^{\max \{1, 2\alpha \}}.
\end{equation*}
Using \eqref{eq: opNormBoundNablapsi}, we have that $\norm{\nabla \psi(x_k)}_{\mathrm{op}} \leqslant C(\norm{x_k}_\infty)$. Moreover,  we also, by Lemma~\ref{lem: derivativeSbound}, have $\norm{\nabla^2 \varphi(x_k)}_{\mathrm{op}} = \norm{\nabla_x^2 S_{ij}(x_k,z)}_{\mathrm{op}} \leqslant \norm{\nabla_x^2 S(x_k,\cdot)}_\Gamma \leqslant \frac{C(\norm{x_k}_\infty)}{d}$. Since $k\leqslant (t\wedge \vartheta)d$, we get
\begin{equation}\label{eq: hessianFirstBound}
\EE_{a_{k+1}} \left[\abs{\inner{\mathcal{B}, \sqrt{K}\Pi_k \sqrt{K}}} \big| \mathscr{F}_k\right] \leqslant \frac{C(\norm{x_k}_\infty)}{d} (L(f))^2 \left(1 + \frac{1}{\sqrt{d}}\norm{K}^{1/2}_{\mathrm{op}} \norm{W_k}\right)^{\max \{1, 2\alpha \}}.
\end{equation}
Similarly we get that
\begin{align*}
\big| \inner{\mathcal{B}, \left( \sqrt{K} \Pi_k v_k \right)^{\otimes 2}}\big| &= \abs{\inner{\nabla_{r_1} f\left(  r_k \right)^2 \nabla \psi(x_k)  \nabla^2 \varphi(x_k)\, (\nabla \psi(x_k))^\top, \left( \sqrt{K} \Pi_k v_k \right)^{\otimes 2}}}\\
&\leqslant \norm{\left( \sqrt{K} \Pi_k v_k \right)^{\otimes 2}}_*  \norm{\nabla_{r_1} f\left(  r_k \right)^2 \nabla \psi(x_k)  \nabla^2 \varphi(x_k)\, (\nabla \psi(x_k))^\top }_{\mathrm{op}}\\
&\leqslant \norm{\sqrt{K} \Pi_k v_k }^2  \nabla_{r_1} f\left(  r_k \right)^2 \norm{\nabla \psi(x_k)}_{\mathrm{op}}^2 \norm{\nabla^2 \varphi(x_k)}_{\mathrm{op}}\\
&\leqslant \norm{K}_{\mathrm{op}}\norm{ \Pi_k v_k }^2  \nabla_{r_1} f\left(  r_k \right)^2 \norm{\nabla \psi(x_k)}_{\mathrm{op}}^2 \norm{\nabla^2 \varphi(x_k)}_{\mathrm{op}}.
\end{align*}

Upon taking expectations, using Cauchy--Schwarz we have
\[
\EE_{a_{k+1}}[\|\Pi_kv_k\|^2 \nabla_{r_1} f(r_k)^2  \big| \mathscr{F}_k] \leqslant \left(
\mathbb E[\|\Pi_kv_k\|^4 \big| \mathscr{F}_k]
\right)^{1/2} \left(
\mathbb E[|\nabla_{r_1}f(r_k)|^4 \big| \mathscr{F}_k]
\right)^{1/2}.
\]
Now by Lemma~\ref{lem:growthGradf} we know 
\begin{equation*}
\EE_{a_{k+1}}[\abs{\nabla_{r_1} f(r_k) }^4 \big| \mathscr{F}_k] \leqslant C (L(f))^4 \left(1 + \frac{1}{\sqrt{d}}\norm{K}^{1/2}_{\mathrm{op}} \norm{W_k}\right)^{\max \{1, 4\alpha \}},    
\end{equation*}
and $\EE_{a_{k+1}} \left[\norm{\Pi_k v_k}^4 \middle| \mathscr{F}_k \right] = 8$, as $\Pi_k$ is a projection. Using Lemma~\ref{lem: derivativeSbound} on the growth of $\varphi$, yields
\begin{equation}\label{eq: hessianSecondBound}
   \EE_{a_{k+1}} \left[\Big|\inner{\mathcal{B}, \left( \sqrt{K} \Pi_k v_k \right)^{\otimes 2}}\Big| \big| \mathscr{F}_k \right] \leqslant \frac{C(\norm{x_k}_\infty)}{d} (L(f))^4 \left(1 + \frac{1}{\sqrt{d}}\norm{K}^{1/2}_{\mathrm{op}} \norm{W_k}\right)^{\max \{1, 4\alpha \}}.
\end{equation}
As $k\leqslant (t\wedge \vartheta)d$, then $\frac{1}{\sqrt{d}}\norm{W_k} \leqslant \sqrt{M}$. The result then immediately follows by combining \eqref{eq: hessianFirstBound} and \eqref{eq: hessianSecondBound} and summing up with the extra factor $\gamma_k^2 / d$. Applying the scalar bound to the real and imaginary parts of all nine entries of
$S(\cdot,z)$ and taking a finite union bound gives the displayed matrix norm
bound. Since the matrix dimension is fixed, this only changes the constant.
\end{proof}

Note that since $|\Gamma^\delta|\leqslant C_\Gamma d^{4\delta}$, all of the same estimates hold
uniformly over $z\in\Gamma^\delta$ by a union bound.

\section{Entropy Barrier and High-probability Exponential Decay}
\label{app:entropy_barrier}
This section analyzes the homogenized dynamics in the isotropic squared
parameterization setting. We prove that, for sufficiently small constant
stepsize, the risk decays exponentially with high probability and the dynamics
exist globally. The argument is based on an entropy barrier tailored to the
coordinatewise structure of the SDE. First, we derive an exact product identity
for each coordinate pair $(\mathscr U_{t,i},\mathscr V_{t,i})$ and introduce
logarithmic coordinates in which the residual
$\mathscr U_{t,i}^2-\mathscr V_{t,i}^2-1$ has an entropy representation. We then
derive an exact SDE for the empirical entropy and prove deterministic barrier
estimates comparing the entropy to both the coordinate size and the risk. These
estimates yield an exponential supermartingale argument, giving explicit
confidence bounds and a high-probability exponential decay theorem. Finally, we
record dynamical consequences of this decay, including integrability of the risk
and uniform separation from the saddle at the origin.

In this section, we specialize to the isotropic setting
$\beta^*=\mathds{1}_d$ and $a\sim\mathcal{N}(0,I_d)$, so that the risk in
\eqref{eq: squaredParam} becomes
\begin{equation*}
\mathcal{R}(x)
=
\frac{1}{4d}
\sum_{i=1}^d
\left(u_i^2-v_i^2-1\right)^2.
\end{equation*}
We write $\mathscr X_t=(\mathscr U_t,\mathscr V_t)$ for the corresponding
homogenized process. For constant stepsize $\gamma>0$, the coordinatewise SDE
in \eqref{eq: sdeFormula} takes the form
\begin{align}
\mathrm{d} \mathscr{U}_{t,i}
&=
-\gamma \mathscr{U}_{t,i}
\left(\mathscr{U}_{t,i}^2-\mathscr{V}_{t,i}^2-1\right)\,\mathrm{d} t
+
2\gamma \sqrt{\mathcal{R}(\mathscr{X}_t)}\,
\mathscr{U}_{t,i}\,\mathrm{d} \mathfrak{B}_{t,i},
\label{eq:U_sde_appendix}
\\
\mathrm{d} \mathscr{V}_{t,i}
&=
\gamma \mathscr{V}_{t,i}
\left(\mathscr{U}_{t,i}^2-\mathscr{V}_{t,i}^2-1\right)\,\mathrm{d} t
-
2\gamma \sqrt{\mathcal{R}(\mathscr{X}_t)}\,
\mathscr{V}_{t,i}\,\mathrm{d} \mathfrak{B}_{t,i},
\label{eq:V_sde_appendix}
\end{align}
for $i=1,\ldots,d$, where $\mathfrak B_t$ is a standard Brownian motion in
$\mathbb{R}^d$. We initialize the dynamics at
\begin{equation*}
\mathscr{U}_{0,i}
=
\mathscr{V}_{0,i}
=
1,
\qquad i=1,\ldots,d.
\end{equation*}

Since the coefficients are locally Lipschitz, the SDE admits a unique maximal
local strong solution on a random interval $[0,\zeta)$, where
$\zeta\in(0,\infty]$ denotes the explosion time.

\subsection{Exact Product Identity and Logarithmic Coordinates}

We first record an exact product identity for each coordinate pair.

\begin{lemma}[Exact product identity]
\label{lem:product_identity}
For each $i=1,\dots,d$ and all $0\leqslant t<\zeta$, we have
\begin{equation*}
\mathscr{U}_{t,i}^2 \mathscr{V}_{t,i}^2
=
\exp\!\Big(-8\gamma^2 I_t\Big), \quad \text{where}
\quad
I_t:=\int_0^t \mathcal{R}(\mathscr{X}_s)\,\mathrm{d} s.
\end{equation*}
In particular, $\mathscr{U}_{t,i}^2>0$ and $\mathscr{V}_{t,i}^2>0$ for all $0\leqslant t<\zeta$.
\end{lemma}

\begin{proof}
By It\^o's formula,
\begin{align*}
\mathrm{d} \mathscr{U}_{t,i}^2
&=
\mathscr{U}_{t,i}^2\Big(-2\gamma(\mathscr{U}_{t,i}^2-\mathscr{V}_{t,i}^2-1)
+4\gamma^2\mathcal{R}(\mathscr{X}_t)\Big)\,\mathrm{d} t
+4\gamma\sqrt{\mathcal{R}(\mathscr{X}_t)}\,\mathscr{U}_{t,i}^2\,\mathrm{d}\mathfrak{B}_{t,i},
\\
\mathrm{d} \mathscr{V}_{t,i}^2
&=
\mathscr{V}_{t,i}^2\Big(+2\gamma(\mathscr{U}_{t,i}^2-\mathscr{V}_{t,i}^2-1)
+4\gamma^2\mathcal{R}(\mathscr{X}_t)\Big)\,\mathrm{d} t
-4\gamma\sqrt{\mathcal{R}(\mathscr{X}_t)}\,\mathscr{V}_{t,i}^2\,\mathrm{d}\mathfrak{B}_{t,i}.
\end{align*}
Applying It\^o's product rule to $\mathscr{U}_{t,i}^2\mathscr{V}_{t,i}^2$, the martingale terms cancel and the quadratic covariation contributes
\[
\mathrm{d} \langle \mathscr{U}_{\cdot,i}^2,\mathscr{V}_{\cdot,i}^2\rangle_t
=
-16\gamma^2 \mathcal{R}(\mathscr{X}_t) \mathscr{U}_{t,i}^2\mathscr{V}_{t,i}^2\,\mathrm{d} t,
\]
so that
\begin{equation*}
\mathrm{d}(\mathscr{U}_{t,i}^2\mathscr{V}_{t,i}^2)
=
-8\gamma^2 \mathcal{R}(\mathscr{X}_t) \mathscr{U}_{t,i}^2\mathscr{V}_{t,i}^2\,\mathrm{d} t.
\end{equation*}
Solving this scalar ODE and using
$\mathscr U_{0,i}^2\mathscr V_{0,i}^2=1$ gives
\[
\mathscr U_{t,i}^2\mathscr V_{t,i}^2
=
\exp\!\Big(-8\gamma^2 I_t\Big)>0.
\]
Because both factors are nonnegative, each must be strictly positive.
\end{proof}

Define
\begin{equation*}
\rho(I_t):=\sqrt{1+4e^{-8\gamma^2 I_t}}.
\end{equation*}
Then
\begin{equation*}
\frac{\rho(I_t)+1}{2}>0,
\qquad
\frac{\rho(I_t)-1}{2}>0,
\qquad
\frac{\rho(I_t)+1}{2}\cdot \frac{\rho(I_t)-1}{2}
=
e^{-8\gamma^2 I_t}.
\end{equation*}
By Lemma~\ref{lem:product_identity}, the product of the two normalizing factors
matches $\mathscr U_{t,i}^2\mathscr V_{t,i}^2$. Hence, for each $i$, we may define a unique real-valued process
\begin{equation*}
\mathscr{Z}_{t,i}
:=
\log\!\Bigg(\frac{\mathscr{U}_{t,i}^2}{(\rho(I_t)+1)/2}\Bigg)
=
-\log\!\Bigg(\frac{\mathscr{V}_{t,i}^2}{(\rho(I_t)-1)/2}\Bigg),
\end{equation*}
such that
\begin{equation*}
\mathscr{U}_{t,i}^2
=
\frac{\rho(I_t)+1}{2}e^{\mathscr{Z}_{t,i}},
\qquad
\mathscr{V}_{t,i}^2
=
\frac{\rho(I_t)-1}{2}e^{-\mathscr{Z}_{t,i}}.
\end{equation*}

For $c>0$ and $x>0$, define
\[
f_c(x):=x-c-c\log(x/c),
\]
and for $\rho\geqslant 1$ define
\[
\Phi_\rho(z):=\rho(\cosh z-1)+\sinh z-z.
\]

\begin{lemma}[Logarithmic-coordinate identities]
\label{lem:log_coord_identities}
For every $0\leqslant t<\zeta$ and every $i=1,\dots,d$,
\begin{align}
\mathscr{U}_{t,i}^2-\mathscr{V}_{t,i}^2-1
&=
\rho(I_t)\sinh(\mathscr{Z}_{t,i})+\cosh(\mathscr{Z}_{t,i})-1,
\label{eq:diff_identity_appendix}\\
\mathscr{U}_{t,i}^2+\mathscr{V}_{t,i}^2
&=
\rho(I_t)\cosh(\mathscr{Z}_{t,i})+\sinh(\mathscr{Z}_{t,i}),
\label{eq:sum_identity_appendix}\\
\Phi_{\rho(I_t)}(\mathscr{Z}_{t,i})
&=
f_{\frac{\rho(I_t)+1}{2}}(\mathscr{U}_{t,i}^2)
+
f_{\frac{\rho(I_t)-1}{2}}(\mathscr{V}_{t,i}^2).
\label{eq:Phi_entropy_identity_appendix}
\end{align}
Moreover, $\mathscr{U}_{t,i}>0$ and $\mathscr{V}_{t,i}>0$ for all $0\leqslant t<\zeta$, and hence
\begin{align}
\mathrm{d} \log \mathscr{U}_{t,i}
&=
\Big(-\gamma(\mathscr{U}_{t,i}^2-\mathscr{V}_{t,i}^2-1)-2\gamma^2\mathcal{R}(\mathscr{X}_t)\Big)\,\mathrm{d} t
+2\gamma\sqrt{\mathcal{R}(\mathscr{X}_t)}\,\mathrm{d} \mathfrak{B}_{t,i},
\label{eq:logU_appendix}\\
\mathrm{d} \log \mathscr{V}_{t,i}
&=
\Big(+\gamma(\mathscr{U}_{t,i}^2-\mathscr{V}_{t,i}^2-1)-2\gamma^2\mathcal{R}(\mathscr{X}_t)\Big)\,\mathrm{d} t
-2\gamma\sqrt{\mathcal{R}(\mathscr{X}_t)}\,\mathrm{d} \mathfrak{B}_{t,i}.
\label{eq:logV_appendix}
\end{align}
\end{lemma}

\begin{proof}
The identities \eqref{eq:diff_identity_appendix}--\eqref{eq:Phi_entropy_identity_appendix} follow directly from the parametrization
\[
\mathscr{U}_{t,i}^2=\frac{\rho(I_t)+1}{2}e^{\mathscr{Z}_{t,i}},
\qquad
\mathscr{V}_{t,i}^2=\frac{\rho(I_t)-1}{2}e^{-\mathscr{Z}_{t,i}}.
\]
Since $\mathscr{U}_{0,i}=\mathscr{V}_{0,i}=1$ and neither process can hit $0$ before $\zeta$ by Lemma~\ref{lem:product_identity}, continuity gives strict positivity. The logarithmic SDEs then follow from It\^o's formula applied to $\log \mathscr{U}_{t,i}$ and $\log \mathscr{V}_{t,i}$.
\end{proof}

For each coordinate, define the entropy density
\begin{equation}
h_{t,i}
:=
\Phi_{\rho(I_t)}(\mathscr Z_{t,i})
=
f_{\frac{\rho(I_t)+1}{2}}(\mathscr U_{t,i}^2)
+
f_{\frac{\rho(I_t)-1}{2}}(\mathscr V_{t,i}^2),
\qquad 0\leqslant t<\zeta.
\label{eq:coordinate_entropy_density}
\end{equation}
The empirical entropy is then the average of these coordinate densities:
\begin{equation}
H_t
:=
\frac1d\sum_{i=1}^d h_{t,i}.
\label{eq:empirical_entropy}
\end{equation}
Equivalently,
\begin{equation*}
H_t
=
\frac{1}{d}\sum_{i=1}^d
\Phi_{\rho(I_t)}(\mathscr Z_{t,i})
=
\frac{1}{d}\sum_{i=1}^d
\Bigg[
f_{\frac{\rho(I_t)+1}{2}}(\mathscr U_{t,i}^2)
+
f_{\frac{\rho(I_t)-1}{2}}(\mathscr V_{t,i}^2)
\Bigg].
\end{equation*}

At time $t=0$,
\begin{equation*}
\rho(I_0)=\sqrt5,
\qquad
\mathscr Z_{0,i}
=
-\log\!\Big(\tfrac{\sqrt5+1}{2}\Big),
\end{equation*}
and hence
\begin{equation*}
h_{0,i}
=
2-\sqrt5+\log\!\Big(\tfrac{\sqrt5+1}{2}\Big)
\qquad
\text{for every } i=1,\ldots,d.
\end{equation*}
Therefore,
\begin{equation*}
H_0
=
2-\sqrt5+\log\!\Big(\tfrac{\sqrt5+1}{2}\Big),
\end{equation*}
which is independent of $d$.

\subsection{Exact Dynamics of the Empirical Entropy}
Recall that
\[
H_t=\frac1d\sum_{i=1}^d h_{t,i},
\qquad
h_{t,i}:=\Phi_{\rho(I_t)}(\mathscr Z_{t,i}).
\]
We next derive the SDE satisfied by the empirical entropy $H_t$.

\begin{proposition}[Empirical entropy SDE]
\label{prop:entropy_sde}
For $0\leqslant t<\zeta$,
\begin{equation}
\mathrm{d} H_t
=
\Bigg(
-8\gamma \mathcal{R}(\mathscr{X}_t)
+
4\gamma^2 \mathcal{R}(\mathscr{X}_t)
\Bigg[
\frac1d\sum_{i=1}^d (\mathscr{U}_{t,i}^2+\mathscr{V}_{t,i}^2)
+\rho(I_t)
\Bigg]
\Bigg)\,\mathrm{d} t
+\mathrm{d} M_t,
\label{eq:Ht_sde_appendix}
\end{equation}
where
\begin{equation*}
M_t
=
4\gamma \frac1d\sum_{i=1}^d
\int_0^t \mathds{1}_{\{s<\zeta\}}
\sqrt{\mathcal{R}(\mathscr{X}_s)}\,
(\mathscr{U}_{s,i}^2-\mathscr{V}_{s,i}^2-1)\,\mathrm{d} \mathfrak{B}_{s,i},
\end{equation*}
and
\begin{equation}
\mathrm{d} \langle M\rangle_t
= \mathds{1}_{\{t<\zeta\}}
\frac{64\gamma^2}{d}\mathcal{R}(\mathscr{X}_t)^2\,\mathrm{d} t.
\label{eq:Mt_qv_appendix}
\end{equation}
\end{proposition}

\begin{proof}
Since
\[
\partial_x f_c(x)=1-\frac{c}{x},
\qquad
\partial_{xx} f_c(x)=\frac{c}{x^2},
\qquad
\partial_c f_c(x)=\log(c/x),
\]
It\^o's formula gives
\begin{align*}
\mathrm{d} f_{\frac{\rho(I_t)+1}{2}}(\mathscr{U}_{t,i}^2)
&=
\Bigg[
\Big(\mathscr{U}_{t,i}^2-\tfrac{\rho(I_t)+1}{2}\Big)
\Big(-2\gamma(\mathscr{U}_{t,i}^2-\mathscr{V}_{t,i}^2-1)+4\gamma^2\mathcal{R}(\mathscr{X}_t)\Big)
+8\gamma^2\mathcal{R}(\mathscr{X}_t)\tfrac{\rho(I_t)+1}{2}
\Bigg]\mathrm{d} t
\nonumber\\
&\quad
+4\gamma\sqrt{\mathcal{R}(\mathscr{X}_t)}
\Big(\mathscr{U}_{t,i}^2-\tfrac{\rho(I_t)+1}{2}\Big)\mathrm{d} \mathfrak{B}_{t,i}
+\log\!\Bigg(\frac{(\rho(I_t)+1)/2}{\mathscr{U}_{t,i}^2}\Bigg)\mathrm{d}\Big(\tfrac{\rho(I_t)+1}{2}\Big),
\end{align*}
and similarly
\begin{align*}
\mathrm{d} f_{\frac{\rho(I_t)-1}{2}}(\mathscr{V}_{t,i}^2)
&=
\Bigg[
\Big(\mathscr{V}_{t,i}^2-\tfrac{\rho(I_t)-1}{2}\Big)
\Big(+2\gamma(\mathscr{U}_{t,i}^2-\mathscr{V}_{t,i}^2-1)+4\gamma^2\mathcal{R}(\mathscr{X}_t)\Big)
+8\gamma^2\mathcal{R}(\mathscr{X}_t)\tfrac{\rho(I_t)-1}{2}
\Bigg]\mathrm{d} t
\nonumber\\
&\quad
-4\gamma\sqrt{\mathcal{R}(\mathscr{X}_t)}
\Big(\mathscr{V}_{t,i}^2-\tfrac{\rho(I_t)-1}{2}\Big)\mathrm{d} \mathfrak{B}_{t,i}
+\log\!\Bigg(\frac{(\rho(I_t)-1)/2}{\mathscr{V}_{t,i}^2}\Bigg)\mathrm{d}\Big(\tfrac{\rho(I_t)-1}{2}\Big).
\end{align*}
Now we compute
\begin{align*}
\Big(\mathscr{U}_{t,i}^2-\tfrac{\rho(I_t)+1}{2}\Big)
-
\Big(\mathscr{V}_{t,i}^2-\tfrac{\rho(I_t)-1}{2}\Big)
&=
\mathscr{U}_{t,i}^2-\mathscr{V}_{t,i}^2-1,
\\
\Big(\mathscr{U}_{t,i}^2-\tfrac{\rho(I_t)+1}{2}\Big)
+
\Big(\mathscr{V}_{t,i}^2-\tfrac{\rho(I_t)-1}{2}\Big)
&=
\mathscr{U}_{t,i}^2+\mathscr{V}_{t,i}^2-\rho(I_t).
\end{align*}

Moreover, the two moving-target terms have the same finite-variation factor
\[
\mathrm d\left(\frac{\rho(I_t)+1}{2}\right)
=
\mathrm d\left(\frac{\rho(I_t)-1}{2}\right)
=
\frac12\,\mathrm d\rho(I_t),
\]
and also we have
\[
\mathscr{U}_{t,i}^2\mathscr{V}_{t,i}^2
=
\frac{\rho(I_t)+1}{2}\cdot \frac{\rho(I_t)-1}{2}.
\]
Therefore the two moving-target terms cancel:
\begin{equation*}
\log\!\Bigg(\frac{(\rho(I_t)+1)/2}{\mathscr{U}_{t,i}^2}\Bigg)
+
\log\!\Bigg(\frac{(\rho(I_t)-1)/2}{\mathscr{V}_{t,i}^2}\Bigg)
=
0.
\end{equation*}
Summing the two It\^o identities therefore yields
\begin{align*}
\mathrm{d} \Phi_{\rho(I_t)}(\mathscr{Z}_{t,i})
&=
\Big(
-2\gamma(\mathscr{U}_{t,i}^2-\mathscr{V}_{t,i}^2-1)^2
+
4\gamma^2\mathcal{R}(\mathscr{X}_t)(\mathscr{U}_{t,i}^2+\mathscr{V}_{t,i}^2+\rho(I_t))
\Big)\mathrm{d} t
\nonumber\\
&\quad
+
4\gamma\sqrt{\mathcal{R}(\mathscr{X}_t)}
(\mathscr{U}_{t,i}^2-\mathscr{V}_{t,i}^2-1)\mathrm{d} \mathfrak{B}_{t,i}.
\end{align*}
Averaging over $i$ and using
\[
\frac1d\sum_{i=1}^d (\mathscr{U}_{t,i}^2-\mathscr{V}_{t,i}^2-1)^2
=
4\mathcal{R}(\mathscr{X}_t)
\]
gives \eqref{eq:Ht_sde_appendix}. The bracket identity \eqref{eq:Mt_qv_appendix} follows from independence of the Brownian motions:
\[
\mathrm{d}\langle M\rangle_t
=
\frac{16\gamma^2}{d^2}\sum_{i=1}^d
\mathcal{R}(\mathscr{X}_t)(\mathscr{U}_{t,i}^2-\mathscr{V}_{t,i}^2-1)^2\,\mathrm{d} t
=
\frac{64\gamma^2}{d}\mathcal{R}(\mathscr{X}_t)^2\,\mathrm{d} t.
\]
\end{proof}

\subsection{Barrier Estimates}
We next derive deterministic barrier estimates. The first shows that the
empirical entropy controls the average size of
$\mathscr U_{t,i}^2+\mathscr V_{t,i}^2$. The second gives a dimension-free
comparison between the entropy density  $\Phi_{\rho(I_t)}(\mathscr Z_{t,i})$ and the squared residual, 
\[
\left(\rho(I_t)\sinh(\mathscr Z_{t,i})+\cosh(\mathscr Z_{t,i})-1\right)^2,
\]
provided we
work below a coordinatewise entropy barrier.

\begin{lemma}[Pointwise control of $\mathscr{U}^2+\mathscr{V}^2$ by the entropy]
\label{lem:UV_by_entropy}
For every $\rho\geqslant 1$ and every $z\in\R$,
\begin{equation}
\rho\cosh z+\sinh z
\leqslant
2\Phi_\rho(z)+2\rho\log 2.
\label{eq:pointwise_UV_entropy}
\end{equation}
Consequently, for all $0\leqslant t<\zeta$,
\begin{equation}
\frac1d\sum_{i=1}^d
(\mathscr{U}_{t,i}^2+\mathscr{V}_{t,i}^2)
\leqslant
2H_t+2\rho(I_t)\log 2
\leqslant
2H_t+2\sqrt5\log 2.
\label{eq:average_UV_entropy}
\end{equation}
\end{lemma}

\begin{proof}
It suffices to prove the claim for $\rho>1$, since both sides of
\eqref{eq:pointwise_UV_entropy} are continuous in $\rho$ and the case
$\rho=1$ follows by taking the limit $\rho\downarrow 1$.

Fix $c>0$ and define
\[
g_c(x):=f_c(x)-\frac{x}{2}
=
\frac{x}{2}-c-c\log(x/c),
\qquad x>0.
\]
Then $g_c'(x)=\frac12-\frac{c}{x}$, so $g_c$ is minimized at $x=2c$, where
$g_c(2c)=-c\log 2$. Hence
\begin{equation*}
f_c(x)\geqslant \frac{x}{2}-c\log 2
\qquad \text{for all }x>0.
\end{equation*}
Applying this with
\[
(c,x)=\Big(\tfrac{\rho+1}{2},\,\tfrac{\rho+1}{2}e^z\Big),
\qquad
(c,x)=\Big(\tfrac{\rho-1}{2},\,\tfrac{\rho-1}{2}e^{-z}\Big),
\]
and using \eqref{eq:Phi_entropy_identity_appendix}, we obtain
\begin{align*}
\Phi_\rho(z)
&\geqslant
\frac12\Big(\tfrac{\rho+1}{2}e^z+\tfrac{\rho-1}{2}e^{-z}\Big)
-\Big(\tfrac{\rho+1}{2}+\tfrac{\rho-1}{2}\Big)\log 2
\\
&=
\frac12\big(\rho\cosh z+\sinh z\big)-\rho\log 2.
\end{align*}
Rearranging gives \eqref{eq:pointwise_UV_entropy}. Averaging with
$(\rho,z)=(\rho(I_t),\mathscr{Z}_{t,i})$ and using
$\rho(I_t)\leqslant \rho(0)=\sqrt5$ yields \eqref{eq:average_UV_entropy}.
\end{proof}

\begin{lemma}[Coercivity under a coordinatewise entropy barrier]
\label{lem:coercivity_barrier}
Fix $L_\ast>H_0$ and define
\begin{equation*}
K_{L_\ast}
:=
\left\{
(\rho,z):
1\leqslant \rho\leqslant \sqrt5,\;
\Phi_\rho(z)\leqslant L_\ast
\right\}.
\end{equation*}
Then $K_{L_\ast}$ is compact. Moreover, there exist constants
\begin{equation*}
0<m_{L_\ast}\leqslant M_{L_\ast}<\infty
\end{equation*}
depending only on $L_\ast$ such that, for every $(\rho,z)\in K_{L_\ast}$,
\begin{equation}
m_{L_\ast}\Phi_\rho(z)
\leqslant
\big(\rho\sinh z+\cosh z-1\big)^2
\leqslant
M_{L_\ast}\Phi_\rho(z).
\label{eq:coercivity_compact}
\end{equation}
\end{lemma}

\begin{proof}
Since $\rho\in[1,\sqrt5]$ and $\Phi_\rho(z)\to\infty$ as $|z|\to\infty$
uniformly over that interval, $K_{L_\ast}$ is compact by Heine--Borel. Consider
\begin{equation*}
Q(\rho,z)
:=
\frac{(\rho\sinh z+\cosh z-1)^2}{\Phi_\rho(z)}
\qquad (z\neq 0).
\end{equation*}
As $z\to0$, we have
\begin{align*}
\rho\sinh z+\cosh z-1
&=
\rho z+\frac{z^2}{2}+O(z^3),\\
\Phi_\rho(z)
=
\rho(\cosh z-1)+\sinh z-z
&=
\frac{\rho}{2}z^2+O(z^3),
\end{align*}
and therefore
\begin{equation*}
Q(\rho,z)\to 2\rho
\qquad \text{as } z\to0.
\end{equation*}
Thus $Q$ extends continuously across $\{z=0\}$ by setting
$Q(\rho,0):=2\rho$. Next,
\begin{equation*}
\partial_z\big(\rho\sinh z+\cosh z-1\big)
=
\rho\cosh z+\sinh z>0
\qquad (\rho\geqslant 1),
\end{equation*}
so $\rho\sinh z+\cosh z-1$ has the unique zero $z=0$. Also,
\[
\Phi_\rho'(z)=\rho\sinh z+\cosh z-1,
\qquad
\Phi_\rho''(z)=\rho\cosh z+\sinh z,
\]
with $\Phi_\rho(0)=\Phi_\rho'(0)=0$ and $\Phi_\rho''(0)=\rho>0$.
Hence $\Phi_\rho(z)>0$ for all $z\neq0$. Therefore the continuous extension of $Q$ is strictly positive on the compact
set $K_{L_\ast}$, so it attains a positive minimum and finite maximum:
\[
0<m_{L_\ast}:=\min_{K_{L_\ast}}Q
\leqslant
\max_{K_{L_\ast}}Q=:M_{L_\ast}<\infty.
\]
This is exactly \eqref{eq:coercivity_compact}.
\end{proof}

\begin{lemma}[Risk--entropy comparison below the coordinate barrier]
\label{lem:risk_entropy_comparison}
Fix $L_\ast>H_0$. On any time interval on which
\[
\max_{1\leqslant i\leqslant d}h_{t,i}<L_\ast,
\]
we have
\begin{equation}
\frac{m_{L_\ast}}{4}H_t
\leqslant
\mathcal R(\mathscr X_t)
\leqslant
\frac{M_{L_\ast}}{4}H_t.
\label{eq:R_vs_H_barrier}
\end{equation}
\end{lemma}

\begin{proof}
If $\max_i h_{t,i}<L_\ast$, then
\[
(\rho(I_t),\mathscr Z_{t,i})\in K_{L_\ast}
\qquad \text{for every } i.
\]
By Lemma~\ref{lem:coercivity_barrier} and
\eqref{eq:diff_identity_appendix},
\[
m_{L_\ast}h_{t,i}
\leqslant
(\mathscr U_{t,i}^2-\mathscr V_{t,i}^2-1)^2
\leqslant
M_{L_\ast}h_{t,i}.
\]
Averaging over $i$ and using
\[
\mathcal R(\mathscr X_t)
=
\frac{1}{4d}\sum_{i=1}^d
(\mathscr U_{t,i}^2-\mathscr V_{t,i}^2-1)^2
\]
gives \eqref{eq:R_vs_H_barrier}.
\end{proof}

\subsection{Exponential Decay with Explicit Confidence}

We now prove exponential decay of the empirical entropy, and hence of the risk,
using a coordinatewise entropy barrier. 

Fix $H_\ast>H_0$ and $L_\ast>H_0$, and define the stopping time
\begin{equation}
\tau_{H,L}
:=
\inf\left\{
t\in[0,\zeta):
H_t\geqslant H_\ast
\;\text{or}\;
\max_{1\leqslant i\leqslant d}h_{t,i}\geqslant L_\ast
\right\},
\label{eq:tau_HL_def}
\end{equation}
with the convention $\inf\emptyset=\infty$.

\begin{lemma}[The entropy barrier is hit before explosion]
\label{lem:tau_HL_before_explosion}
Almost surely,
\begin{equation*}
\tau_{H,L}\leqslant \zeta.
\end{equation*}
\end{lemma}

\begin{proof}
Suppose, toward a contradiction, that $\zeta<\tau_{H,L}$ on some sample path.
Then $H_t<H_\ast$ and $h_{t,i}<L_\ast$ for every $i$ and all $t<\zeta$.
Hence
\[
(\rho(I_t),\mathscr Z_{t,i})\in K_{L_\ast}
\qquad \text{for all }i\text{ and all }t<\zeta.
\]
Since $K_{L_\ast}$ is compact, the logarithmic coordinates
$\mathscr Z_{t,i}$ remain bounded on $[0,\zeta)$. By the parametrization
\[
\mathscr{U}_{t,i}^2
=
\frac{\rho(I_t)+1}{2}e^{\mathscr{Z}_{t,i}},
\qquad
\mathscr{V}_{t,i}^2
=
\frac{\rho(I_t)-1}{2}e^{-\mathscr{Z}_{t,i}},
\]
the coordinates $\mathscr U_{t,i}$ and $\mathscr V_{t,i}$ remain bounded on
every finite interval $[0,T]\subset[0,\zeta)$. Therefore the coefficients in
\eqref{eq:U_sde_appendix}--\eqref{eq:V_sde_appendix} remain bounded and locally
Lipschitz on a neighborhood of the attained path. The standard continuation
criterion for SDEs then extends the solution beyond $\zeta$, contradicting
maximality. Thus $\tau_{H,L}\leqslant \zeta$ almost surely.
\end{proof}

On $[0,\tau_{H,L})$, Lemma~\ref{lem:risk_entropy_comparison} gives
\[
\frac{m_{L_\ast}}{4}H_t
\leqslant
\mathcal R(\mathscr X_t)
\leqslant
\frac{M_{L_\ast}}{4}H_t,
\]
and Lemma~\ref{lem:UV_by_entropy} gives
\begin{equation}
\frac1d\sum_{i=1}^d
(\mathscr U_{t,i}^2+\mathscr V_{t,i}^2)
\leqslant
2H_\ast+2\sqrt5\log2,
\qquad t<\tau_{H,L}.
\label{eq:UV_bound_before_tau_HL}
\end{equation}

Define
\begin{equation}
\lambda_{H,L}
:=
\gamma
\Big(
8-4\gamma(2H_\ast+2\sqrt5\log 2+\sqrt5)
\Big)
\frac{m_{L_\ast}}{4},
\label{eq:lambda_HL_def}
\end{equation}
and
\begin{equation}
\sigma_{H,L}^2
:=
\frac{4\gamma^2M_{L_\ast}^2}{d}.
\label{eq:sigma_HL_def}
\end{equation}

\begin{lemma}[Drift and quadratic-variation bounds before the barrier]
\label{lem:drift_qv_bounds}
Assume
\begin{equation}
0<\gamma<
\frac{2}{2H_\ast+2\sqrt5\log 2+\sqrt5},
\label{eq:small_gamma_condition_first}
\end{equation}
so that $\lambda_{H,L}>0$. Then, for $t<\tau_{H,L}$,
\begin{equation}
\mathrm{d}H_t
\leqslant
-\lambda_{H,L}H_t\,\mathrm{d}t+\mathrm{d}M_t,
\label{eq:H_supermartingale_drift}
\end{equation}
and
\begin{equation}
\mathrm{d}\langle M\rangle_t
\leqslant
\sigma_{H,L}^2H_t^2\,\mathrm{d}t.
\label{eq:M_qv_bound}
\end{equation}
\end{lemma}

\begin{proof}
Starting from Proposition~\ref{prop:entropy_sde} and using
\eqref{eq:UV_bound_before_tau_HL},
\begin{align*}
\mathrm{d}H_t
&\leqslant
\Big(
-8\gamma\mathcal R(\mathscr X_t)
+
4\gamma^2\mathcal R(\mathscr X_t)
(2H_\ast+2\sqrt5\log2+\sqrt5)
\Big)\mathrm{d}t
+\mathrm{d}M_t
\\
&=
-\gamma
\Big(
8-4\gamma(2H_\ast+2\sqrt5\log2+\sqrt5)
\Big)
\mathcal R(\mathscr X_t)\,\mathrm{d}t
+\mathrm{d}M_t.
\end{align*}
Using the lower bound in \eqref{eq:R_vs_H_barrier} gives
\eqref{eq:H_supermartingale_drift}. Similarly, from
\eqref{eq:Mt_qv_appendix} and the upper bound in
\eqref{eq:R_vs_H_barrier},
\[
\mathrm{d}\langle M\rangle_t
=
\frac{64\gamma^2}{d}\mathcal R(\mathscr X_t)^2\,\mathrm{d}t
\leqslant
\frac{64\gamma^2}{d}
\left(\frac{M_{L_\ast}}{4}H_t\right)^2
\mathrm{d}t
=
\frac{4\gamma^2M_{L_\ast}^2}{d}H_t^2\,\mathrm{d}t.
\]
This is \eqref{eq:M_qv_bound}.
\end{proof}

The next result gives exponential decay up to the coordinatewise barrier.

\begin{theorem}[Stopped exponential decay with explicit confidence]
\label{thm:stopped_explicit_confidence_decay}
Let $H_\ast, L_\ast>H_0$, and define
$\lambda_{H,L}$ and $\sigma_{H,L}^2$ by
\eqref{eq:lambda_HL_def}--\eqref{eq:sigma_HL_def}. Assume
\[
0<\gamma<
\frac{2}{2H_\ast+2\sqrt5\log 2+\sqrt5}.
\]
Fix $\theta>0$ such that
\[
0<\theta<\frac{2\lambda_{H,L}}{\sigma_{H,L}^2},
\]
and define
\[
\mu_{H,L}(\theta)
:=
\lambda_{H,L}
-
\frac{\theta}{2}\sigma_{H,L}^2
>0.
\]
Then
\begin{equation}
\PP\!\left(
\sup_{0\leqslant s\leqslant \tau_{H,L}}
e^{\mu_{H,L}(\theta)s}H_s
\geqslant H_\ast
\right)
\leqslant
\left(\frac{H_0}{H_\ast}\right)^\theta.
\label{eq:stopped_exp_decay}
\end{equation}
\end{theorem}

\begin{proof}
Fix $0<\varepsilon<H_0$ and define
\[
\tau_\varepsilon
:=
\inf\{t\in[0,\zeta):H_t\leqslant\varepsilon\},
\qquad
\tau:=\tau_{H,L}\wedge\tau_\varepsilon.
\]
On $[0,\tau]$, we have $H_t\in[\varepsilon,H_\ast]$, so It\^o's formula gives
\[
\log H_{t\wedge\tau}
=
\log H_0
+
\int_0^{t\wedge\tau}\frac{1}{H_s}\,\mathrm{d}H_s
-
\frac12
\int_0^{t\wedge\tau}\frac{1}{H_s^2}\,\mathrm{d}\langle M\rangle_s.
\]
By Lemma~\ref{lem:drift_qv_bounds},
\[
\log H_{t\wedge\tau}
\leqslant
\log H_0
-\lambda_{H,L}(t\wedge\tau)
-\frac12
\int_0^{t\wedge\tau}\frac{1}{H_s^2}\,\mathrm{d}\langle M\rangle_s
+
N_t,
\]
where
\[
N_t
:=
\int_0^{t\wedge\tau}\frac{1}{H_s}\,\mathrm{d}M_s.
\]
Moreover,
\[
\langle N\rangle_t
=
\int_0^{t\wedge\tau}
\frac{1}{H_s^2}\,\mathrm{d}\langle M\rangle_s
\leqslant
\sigma_{H,L}^2(t\wedge\tau).
\]
Thus, for every $\theta>0$,
\[
\mathcal E_t
:=
\exp\!\left(
\theta N_t-\frac{\theta^2}{2}\langle N\rangle_t
\right)
\]
is a positive local martingale, hence a supermartingale. Since
\[
\mu_{H,L}(\theta)=\lambda_{H,L}-\frac{\theta}{2}\sigma_{H,L}^2,
\]
we obtain
\begin{align*}
e^{\theta\mu_{H,L}(\theta)(t\wedge\tau)}
H_{t\wedge\tau}^{\theta}
&\leqslant
H_0^\theta
\exp\!\Big(
\theta N_t-\theta(\lambda_{H,L}-\mu_{H,L}(\theta))(t\wedge \tau)
\Big)
\\
&=
H_0^\theta
\exp\!\Big(
\theta N_t-\frac{\theta^2}{2}\sigma_{H,L}^2(t\wedge \tau)
\Big)
\\
&\leqslant
H_0^\theta\mathcal E_t.
\end{align*}
Ville's maximal inequality gives
\[
\PP\!\left(
\sup_{0\leqslant s\leqslant \tau}
e^{\mu_{H,L}(\theta)s}H_s
\geqslant H_\ast
\right)
\leqslant
\left(\frac{H_0}{H_\ast}\right)^\theta.
\]
It remains to remove the auxiliary lower stopping time $\tau_\varepsilon$.
Define
\[
\tau_0
:=
\inf\{t\in[0,\zeta): H_t=0\},
\qquad
S_\varepsilon
:=
\tau_{H,L}\wedge\tau_\varepsilon,
\qquad
S_0
:=
\tau_{H,L}\wedge\tau_0.
\]
Since $H_t$ is continuous and nonnegative, $S_\varepsilon\uparrow S_0$
pathwise as $\varepsilon\downarrow0$. Therefore,
\[
\left\{
\sup_{0\leqslant s\leqslant S_\varepsilon}
e^{\mu_{H,L}(\theta)s}H_s
\geqslant H_\ast
\right\}
\uparrow
\left\{
\sup_{0\leqslant s\leqslant S_0}
e^{\mu_{H,L}(\theta)s}H_s
\geqslant H_\ast
\right\}.
\]
By continuity from below of probability, the estimate obtained above gives
\[
\PP\!\left(
\sup_{0\leqslant s\leqslant S_0}
e^{\mu_{H,L}(\theta)s}H_s
\geqslant H_\ast
\right)
\leqslant
\left(\frac{H_0}{H_\ast}\right)^\theta.
\]

Finally, $H_t=0$ is absorbing. Indeed, if $H_t=0$, then every coordinate entropy
density $h_{t,i}$ vanishes, hence $\mathscr Z_{t,i}=0$ for every $i$. Therefore
\[
\mathscr U_{t,i}^2-\mathscr V_{t,i}^2-1=0
\qquad
\text{for all }i,
\]
so $\mathcal R(\mathscr X_t)=0$. The drift and diffusion coefficients in
\eqref{eq:U_sde_appendix}--\eqref{eq:V_sde_appendix} then vanish, and by
pathwise uniqueness the solution remains constant thereafter. Thus
\[
\sup_{0\leqslant s\leqslant S_0}
e^{\mu_{H,L}(\theta)s}H_s
=
\sup_{0\leqslant s\leqslant \tau_{H,L}}
e^{\mu_{H,L}(\theta)s}H_s.
\]
This proves \eqref{eq:stopped_exp_decay}.
\end{proof}

We now give a pathwise criterion for removing the coordinatewise barrier. Let
\[
V_{H,L}(\theta)
:=
\frac{M_{L_\ast}}{4}
\frac{H_\ast}{\mu_{H,L}(\theta)}.
\]
For $\eta\in(0,1)$, define
\[
R_{H,L}(\theta,\eta)
:=
\sqrt{
2V_{H,L}(\theta)\log\!\left(\frac{4d}{\eta}\right)
},
\]
and
\[
p_{H,L}(\theta)
:=
\exp\!\left(-4\gamma^2V_{H,L}(\theta)\right).
\]
Set
\[
A_{H,L}(\theta)
:=
\operatorname{arcsinh}
\!\left(\frac{1}{2p_{H,L}(\theta)}\right),
\qquad
B_{H,L}(\theta,\eta)
:=
2\cosh\!\left(A_{H,L}(\theta)+8\gamma R_{H,L}(\theta,\eta)\right).
\]
Finally, define
\[
c_{H,L}(\theta,\eta)
:=
\frac{p_{H,L}(\theta)^2}{B_{H,L}(\theta,\eta)},
\qquad
\underline c_{H,L}(\theta)
:=
\frac{\sqrt{1+4p_{H,L}(\theta)^2}-1}{2},
\qquad
\overline c:=\frac{\sqrt5+1}{2},
\]
and
\begin{equation}
\mathfrak L_{H,L}(\theta,\eta)
:=
2
\sup_{\substack{
c\in[\underline c_{H,L}(\theta),\overline c]\\
x\in[c_{H,L}(\theta,\eta),B_{H,L}(\theta,\eta)]
}}
f_c(x).
\label{eq:self_consistency_L}
\end{equation}

\begin{lemma}[Removal of the coordinatewise entropy barrier]
\label{lem:remove_coordinate_barrier}
Assume the hypotheses of Theorem~\ref{thm:stopped_explicit_confidence_decay}.
Fix $\eta\in(0,1)$ and suppose that
\begin{equation}
L_\ast>\mathfrak L_{H,L}(\theta,\eta).
\label{eq:Lstar_self_consistency}
\end{equation}
Then
\begin{equation}
\PP\!\left(
\zeta=\infty
\ \text{and}\
H_t\leqslant H_\ast e^{-\mu_{H,L}(\theta)t}
\ \text{for all }t\geqslant0
\right)
\geqslant
1-
\left(\frac{H_0}{H_\ast}\right)^\theta
-\eta.
\label{eq:explicit_confidence_entropy}
\end{equation}
Moreover,
\begin{equation}
\PP\!\left(
\zeta=\infty
\ \text{and}\
\mathcal R(\mathscr X_t)
\leqslant
\frac{M_{L_\ast}}{4}H_\ast e^{-\mu_{H,L}(\theta)t}
\ \text{for all }t\geqslant0
\right)
\geqslant
1-
\left(\frac{H_0}{H_\ast}\right)^\theta
-\eta.
\label{eq:explicit_confidence_risk}
\end{equation}
\end{lemma}

\begin{proof}
Let
\[
E_H
:=
\left\{
\sup_{0\leqslant s\leqslant \tau_{H,L}}
e^{\mu_{H,L}(\theta)s}H_s
<
H_\ast
\right\}.
\]
By Theorem~\ref{thm:stopped_explicit_confidence_decay},
\[
\PP(E_H)
\geqslant
1-
\left(\frac{H_0}{H_\ast}\right)^\theta.
\]
On $E_H$, for all $t<\tau_{H,L}$,
\[
H_t\leqslant H_\ast e^{-\mu_{H,L}(\theta)t}.
\]
Using Lemma~\ref{lem:risk_entropy_comparison}, we get
\[
\int_0^{t\wedge\tau_{H,L}}
\mathcal R(\mathscr X_s)\,\mathrm ds
\leqslant
V_{H,L}(\theta).
\]

For each coordinate define the martingale
\[
N_{t,i}
:=
\int_0^{t\wedge\tau_{H,L}}
\sqrt{\mathcal R(\mathscr X_s)}\,\mathrm d\mathfrak B_{s,i}.
\]
On $E_H$, its quadratic variation is bounded by $V_{H,L}(\theta)$. Therefore, by the reflection principle and a union bound over $i=1,\ldots,d$,
\[
\PP\!\left(
E_H
\cap
\left\{
\max_{1\leqslant i\leqslant d}
\sup_{t\geqslant0}|N_{t,i}|
>
R_{H,L}(\theta,\eta)
\right\}
\right)
\leqslant \eta.
\]
Let $E_N$ denote the complementary martingale event. On $E_H\cap E_N$, define
\[
\widetilde{\mathscr Z}_{t,i}
:=
\log\!\left(\frac{\mathscr U_{t,i}}{\mathscr V_{t,i}}\right).
\]
Since
\[
\mathscr U_{t,i}\mathscr V_{t,i}
=
\exp\!\left(-4\gamma^2 I_t\right),
\]
we have, for $t<\tau_{H,L}$,
\[
\mathscr U_{t,i}\mathscr V_{t,i}
\geqslant
p_{H,L}(\theta).
\]
Moreover, It\^o's formula gives
\[
\mathrm d\widetilde{\mathscr Z}_{t,i}
=
\left(
-4\gamma\mathscr U_{t,i}\mathscr V_{t,i}
\sinh(\widetilde{\mathscr Z}_{t,i})
+
2\gamma
\right)\mathrm dt
+
4\gamma\sqrt{\mathcal R(\mathscr X_t)}\,\mathrm d\mathfrak B_{t,i}.
\]
Thus the noise-compensated process
\[
\mathscr Y_{t,i}
:=
\widetilde{\mathscr Z}_{t,i}
-
4\gamma N_{t,i}
\]
satisfies
\[
\frac{\mathrm d}{\mathrm dt}\mathscr Y_{t,i}
=
-4\gamma\mathscr U_{t,i}\mathscr V_{t,i}
\sinh(\widetilde{\mathscr Z}_{t,i})
+
2\gamma.
\]
A deterministic barrier argument gives, for every $t<\tau_{H,L}$,
\[
|\widetilde{\mathscr Z}_{t,i}|
\leqslant
A_{H,L}(\theta)+8\gamma R_{H,L}(\theta,\eta).
\]
Consequently,
\[
\mathscr U_{t,i}^2+\mathscr V_{t,i}^2
=
2\mathscr U_{t,i}\mathscr V_{t,i}
\cosh(\widetilde{\mathscr Z}_{t,i})
\leqslant
B_{H,L}(\theta,\eta),
\]
and, since
\[
\mathscr U_{t,i}\mathscr V_{t,i}
\geqslant p_{H,L}(\theta),
\]
we also have
\[
\mathscr U_{t,i}^2,\mathscr V_{t,i}^2
\geqslant
c_{H,L}(\theta,\eta).
\]
Moreover,
\[
\frac{\rho(I_t)-1}{2}
\geqslant
\underline c_{H,L}(\theta),
\qquad
\frac{\rho(I_t)+1}{2}
\leqslant
\overline c.
\]
Therefore, by the definition of $\mathfrak L_{H,L}(\theta,\eta)$, we have
\[
h_{t,i}
=
f_{\frac{\rho(I_t)+1}{2}}(\mathscr U_{t,i}^2)
+
f_{\frac{\rho(I_t)-1}{2}}(\mathscr V_{t,i}^2)
\leqslant
\mathfrak L_{H,L}(\theta,\eta)
<
L_\ast.
\]
Thus the coordinatewise entropy barrier cannot be hit on $E_H\cap E_N$.
The entropy barrier $H_t=H_\ast$ cannot be hit either, because on $E_H$ we have
$H_t<H_\ast e^{-\mu_{H,L}(\theta)t}<H_\ast$ for all $t<\tau_{H,L}$.
Hence $\tau_{H,L}=\infty$ on $E_H\cap E_N$. By
Lemma~\ref{lem:tau_HL_before_explosion}, this also implies $\zeta=\infty$.

The entropy decay estimate follows from the definition of $E_H$, and the risk
decay follows from the upper bound in \eqref{eq:R_vs_H_barrier}. The probability
bound follows from
\[
\PP(E_H\cap E_N)
\geqslant
1-
\left(\frac{H_0}{H_\ast}\right)^\theta
-\eta.
\]
\end{proof}

\subsection{High-probability Decay at Prescribed Confidence}

The previous theorem gives an explicit confidence estimate. For fixed
$d$, $H_\ast$, $L_\ast$, and $\gamma$, the admissible range of $\theta$ is
finite. By choosing the stepsize sufficiently small, we can achieve a prescribed
confidence level.

\begin{corollary}[High-probability decay for sufficiently small stepsize]
\label{cor:delta_high_probability}
Fix $H_\ast, L_\ast>H_0$, and $\delta\in(0,1)$. Define
\begin{equation*}
\theta_\delta
:=
\frac{\log(2/\delta)}{\log(H_\ast/H_0)}
\end{equation*}
and
\begin{equation*}
\bar\gamma_{H,L,\delta}
:=
\frac{2}{
2H_\ast+2\sqrt5\log 2+\sqrt5
+
\frac{2M_{L_\ast}^2}{d\,m_{L_\ast}}\theta_\delta
}.
\end{equation*}
Assume
\begin{equation*}
0<\gamma<\bar\gamma_{H,L,\delta}.
\end{equation*}
Then
\begin{equation*}
\mu_{H,L}(\delta)
:=
m_{L_\ast}\gamma
\Big(2-\gamma(2H_\ast+2\sqrt5\log 2+\sqrt5)\Big)
-
\frac{2\gamma^2M_{L_\ast}^2}{d}\theta_\delta
>0.
\end{equation*}
Suppose further that the coordinatewise barrier is chosen so that
\begin{equation}
L_\ast>
\mathfrak L_{H,L}\!\left(\theta_\delta,\frac{\delta}{2}\right),
\label{eq:Lstar_delta_self_consistency}
\end{equation}
where $\mathfrak L_{H,L}$ is defined in \eqref{eq:self_consistency_L}. Then
\begin{equation}
\PP\!\left(
\zeta=\infty
\ \text{and}\
\mathcal{R}(\mathscr{X}_t)
\leqslant
\frac{M_{L_\ast}}{4}H_\ast e^{-\mu_{H,L}(\delta)t}
\ \text{for all } t\geqslant 0
\right)
\geqslant
1-\delta.
\label{eq:delta_high_prob_result}
\end{equation}
\end{corollary}

\begin{proof}
By construction,
\[
\left(\frac{H_0}{H_\ast}\right)^{\theta_\delta}
=
\frac{\delta}{2}.
\]
The condition $\gamma<\bar\gamma_{H,L,\delta}$ is exactly equivalent to
\[
\theta_\delta
<
\frac{2\lambda_{H,L}}{\sigma_{H,L}^2},
\]
where $\lambda_{H,L}$ and $\sigma_{H,L}^2$ are defined in
\eqref{eq:lambda_HL_def}--\eqref{eq:sigma_HL_def}. Moreover,
\[
\mu_{H,L}(\delta)
=
\lambda_{H,L}
-
\frac{\theta_\delta}{2}\sigma_{H,L}^2.
\]
Thus Lemma~\ref{lem:remove_coordinate_barrier} applies with
$\theta=\theta_\delta$ and $\eta=\delta/2$. Substituting these values into
\eqref{eq:explicit_confidence_risk} gives
\eqref{eq:delta_high_prob_result}.
\end{proof}

\begin{figure}[t!]
    \centering
    \includegraphics[width=0.9\textwidth]{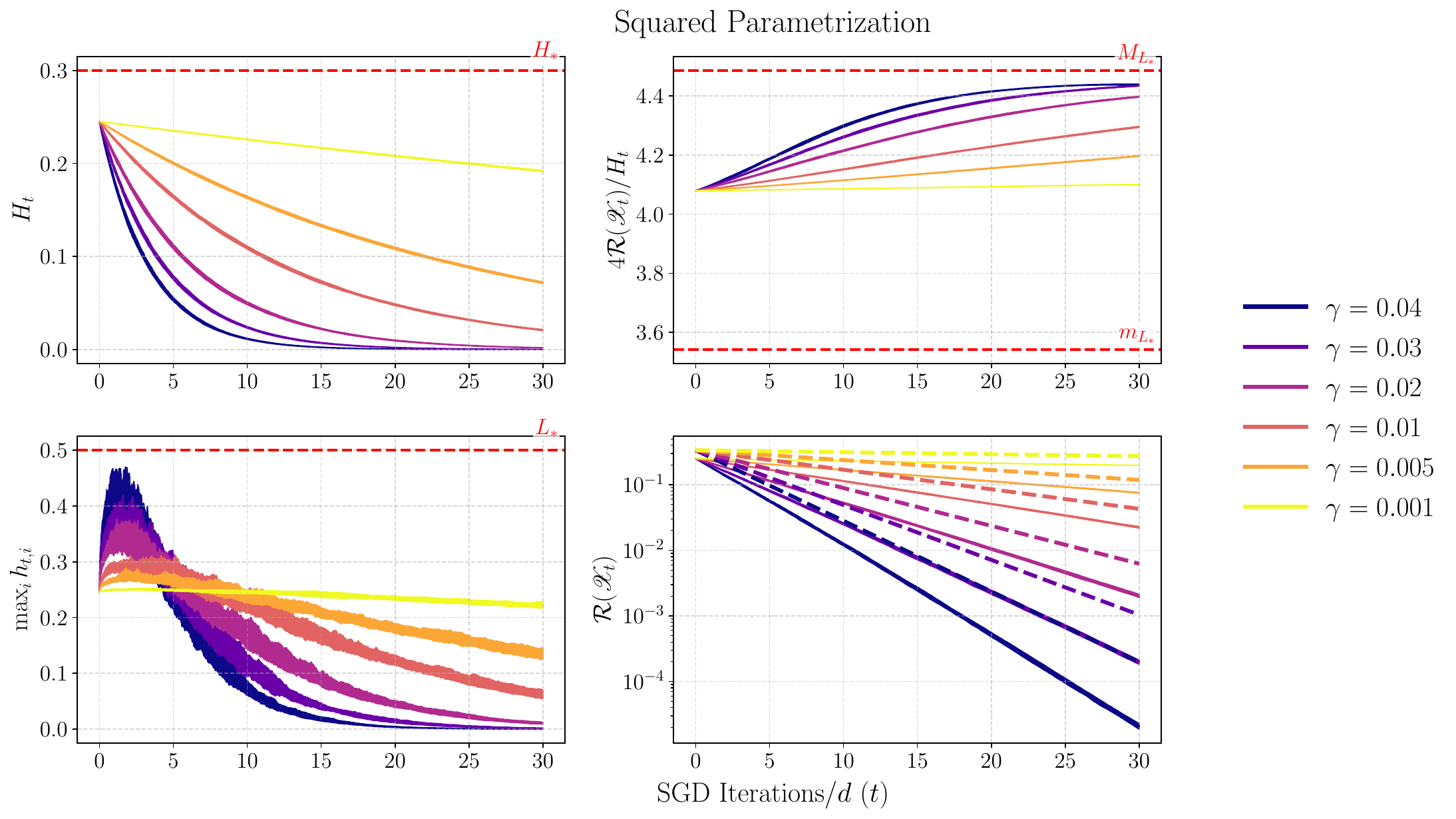}
    
    \caption{\textbf{Coordinatewise entropy barriers and exponential risk decay on a diagonal linear network.}
The figure illustrates the entropy-barrier mechanism from Appendix~\ref{app:entropy_barrier}. The top-left panel shows the empirical entropy $H_t$, while the bottom-left panel shows the largest coordinatewise entropy density $\max_i h_{t,i}$; the red dashed lines mark the barriers $H_\ast$ and $L_\ast$. The top-right panel plots the risk--entropy ratio $4\mathcal R(\mathscr X_t)/H_t$, with red dashed lines marking empirical coordinatewise coercivity constants $m_{L_\ast}$ and $M_{L_\ast}$ estimated from the coordinate quotients $q_{t,i}$. The bottom-right panel shows the risk $\mathcal R(\mathscr X_t)$ on a logarithmic scale, together with exponential envelopes of the form $(M_{L_\ast}/4)H_\ast e^{-\mu t}$. We plot these envelopes for all tested stepsizes, including those beyond the theorem's certified small-stepsize regime, illustrating that the predicted exponential decay persists empirically beyond the sufficient condition.}
\label{fig: risk_exponential_decay}
\end{figure}

\subsection{Risk Integrability and Dynamical Consequences}

We now record two consequences of the high-probability exponential decay:
integrability of the risk and uniform separation from the saddle.

\begin{lemma}[Finite integral of the risk]
\label{lem:finite_risk_integral}
On the event in Corollary~\ref{cor:delta_high_probability},
\begin{equation}
\int_0^\infty \mathcal R(\mathscr X_s)\,\mathrm ds
\leqslant
\frac{M_{L_\ast}}{4}
\frac{H_\ast}{\mu_{H,L}(\delta)}
<\infty.
\label{eq:risk_integral_bound}
\end{equation}
\end{lemma}

\begin{proof}
On the event in Corollary~\ref{cor:delta_high_probability},
\[
\mathcal R(\mathscr X_t)
\leqslant
\frac{M_{L_\ast}}{4}H_\ast e^{-\mu_{H,L}(\delta)t}
\qquad \text{for all } t\geqslant0.
\]
Therefore
\[
\int_0^\infty \mathcal R(\mathscr X_s)\,\mathrm ds
\leqslant
\frac{M_{L_\ast}}{4}H_\ast
\int_0^\infty e^{-\mu_{H,L}(\delta)s}\,\mathrm ds
=
\frac{M_{L_\ast}}{4}
\frac{H_\ast}{\mu_{H,L}(\delta)}.
\]
\end{proof}

\subsubsection{Non-explosion via Risk Integrability}

Although Corollary~\ref{cor:delta_high_probability} already includes
$\zeta=\infty$, the following pathwise argument explains why risk
integrability prevents finite-time explosion.

For each coordinate define
\[
\widetilde{\mathscr Z}_{t,i}
:=
\log\!\left(\frac{\mathscr U_{t,i}}{\mathscr V_{t,i}}\right),
\qquad 0\leqslant t<\zeta.
\]
By Lemma~\ref{lem:product_identity},
\[
\mathscr U_{t,i}\mathscr V_{t,i}
=
\exp\!\left(-4\gamma^2 I_t\right).
\]
Applying It\^o's formula to
$\widetilde{\mathscr Z}_{t,i}=\log\mathscr U_{t,i}-\log\mathscr V_{t,i}$ gives
\begin{equation}
\mathrm d\widetilde{\mathscr Z}_{t,i}
=
\left(
-4\gamma \mathscr U_{t,i}\mathscr V_{t,i}
\sinh(\widetilde{\mathscr Z}_{t,i})
+
2\gamma
\right)\mathrm dt
+
4\gamma\sqrt{\mathcal R(\mathscr X_t)}\,\mathrm d\mathfrak B_{t,i}.
\label{eq:log_ratio_sde}
\end{equation}

\begin{lemma}[Pathwise coordinate bound from risk integrability]
\label{lem:pathwise_coordinate_bound}
Assume
\[
\int_0^\infty \mathcal R(\mathscr X_s)\,\mathrm ds<\infty.
\]
Then, for every coordinate $i$,
\[
\sup_{0\leqslant t<\zeta}
(\mathscr U_{t,i}^2+\mathscr V_{t,i}^2)<\infty
\qquad \text{almost surely}.
\]
Consequently, finite risk integral is incompatible with finite-time explosion.
\end{lemma}

\begin{proof}
Define the continuous martingale
\[
N_{t,i}
:=
\int_0^t
\sqrt{\mathcal R(\mathscr X_s)}\,\mathrm d\mathfrak B_{s,i}.
\]
Its quadratic variation satisfies
\[
\langle N_i\rangle_\infty
=
\int_0^\infty \mathcal R(\mathscr X_s)\,\mathrm ds<\infty.
\]
Hence $N_{t,i}$ converges almost surely as $t\to\infty$, and therefore
\[
N_i^\ast:=\sup_{t\geqslant0}|N_{t,i}|<\infty
\qquad\text{almost surely}.
\]

Let
\[
P_t:=\mathscr U_{t,i}\mathscr V_{t,i}
=
\exp(-4\gamma^2 I_t).
\]
Since $I_\infty<\infty$, there exists
\[
p_\infty:=\exp(-4\gamma^2 I_\infty)>0
\]
such that $P_t\geqslant p_\infty$ for all $t$. Define
\[
\mathscr Y_{t,i}
:=
\widetilde{\mathscr Z}_{t,i}-4\gamma N_{t,i}.
\]
By \eqref{eq:log_ratio_sde},
\[
\frac{\mathrm d}{\mathrm dt}\mathscr Y_{t,i}
=
-4\gamma P_t\sinh(\widetilde{\mathscr Z}_{t,i})+2\gamma.
\]
A deterministic barrier argument gives
\[
|\widetilde{\mathscr Z}_{t,i}|
\leqslant
8\gamma N_i^\ast
+
\operatorname{arcsinh}\!\left(\frac{1}{2p_\infty}\right)
\qquad \text{for all }t.
\]
Therefore,
\[
\mathscr U_{t,i}^2+\mathscr V_{t,i}^2
=
2P_t\cosh(\widetilde{\mathscr Z}_{t,i})
\leqslant
2\cosh\!\left(
8\gamma N_i^\ast
+
\operatorname{arcsinh}\!\left(\frac{1}{2p_\infty}\right)
\right)
<\infty.
\]
Thus all coordinates remain bounded on finite time intervals. Since the SDE
coefficients are locally Lipschitz, the standard continuation criterion rules
out finite-time explosion.
\end{proof}

\subsubsection{Uniform Separation from the Saddle via Risk Integrability}

We now show that risk integrability also prevents the dynamics from approaching
the saddle at the origin.

\begin{lemma}[Uniform separation from the saddle]
\label{lem:away_from_saddle}
If
\[
\int_0^\infty \mathcal R(\mathscr X_s)\,\mathrm ds<\infty,
\]
then, for every coordinate $i$ and all $t\geqslant0$,
\begin{equation*}
\mathscr U_{t,i}^2+\mathscr V_{t,i}^2
\geqslant
m,
\end{equation*}
where
\begin{equation*}
m
:=
2\exp\!\left(
-4\gamma^2
\int_0^\infty \mathcal R(\mathscr X_s)\,\mathrm ds
\right)
>0.
\end{equation*}
\end{lemma}

\begin{proof}
By the arithmetic--geometric mean inequality,
\[
\mathscr U_{t,i}^2+\mathscr V_{t,i}^2
\geqslant
2\mathscr U_{t,i}\mathscr V_{t,i}.
\]
Using Lemma~\ref{lem:product_identity},
\[
\mathscr U_{t,i}\mathscr V_{t,i}
=
\exp\!\left(
-4\gamma^2
\int_0^t \mathcal R(\mathscr X_s)\,\mathrm ds
\right).
\]
Therefore
\[
\mathscr U_{t,i}^2+\mathscr V_{t,i}^2
\geqslant
2\exp\!\left(
-4\gamma^2
\int_0^t \mathcal R(\mathscr X_s)\,\mathrm ds
\right)
\geqslant
2\exp\!\left(
-4\gamma^2
\int_0^\infty \mathcal R(\mathscr X_s)\,\mathrm ds
\right)
=:m.
\]
Since the integral is finite, $m>0$.
\end{proof}

Combining Corollary~\ref{cor:delta_high_probability} with
Lemma~\ref{lem:away_from_saddle}, we obtain that, with probability at least
$1-\delta$, the dynamics exist globally, the risk decays exponentially, and the
coordinates remain uniformly separated from the saddle:
\[
\mathscr U_{t,i}^2+\mathscr V_{t,i}^2
\geqslant
2\exp\!\left(
-\gamma^2
\frac{M_{L_\ast}H_\ast}{\mu_{H,L}(\delta)}
\right)
\qquad
\text{for all }t\geqslant0,\ i=1,\ldots,d.
\]

\section{Examples}\label{app: examples}
In this section, we provide further key learning problems within our family of models.

\subsection{Mean-squared Error}\label{subsec: meanSquaredSetting}
One canonical example of \eqref{eq: riskSetting} is the outer function $f\colon \R^2 \to \R$ given by $f(r) = \frac{1}{2} (r_1-r_2)^2$. 

In this case, the problem is
\begin{equation}
\min_{x\in \R^{2d}} \biggl\{ \mathcal{R}(x) = \frac{1}{2d} \EE_a \inner{\psi(x)- \beta^*, a}^2 = \frac{1}{2d} (\psi(x)- \beta^*)^\top K (\psi(x)- \beta^*) \biggr\},
\end{equation}
and satisfies
\begin{equation*}
I\left(B(x)\right) =
\EE_a [\nabla_{r_1} f(r)^2] =  2 \mathcal{R}(x) = 2 h\left(B(x)\right) = B(x)_{11} - B(x)_{12} - B(x)_{21} + B(x)_{22}.
\end{equation*}

Moreover, 
\begin{align*}
\frac{1}{d} \tr \left(\nabla^2 \mathcal{R}(x) \right) &= \frac{1}{d} \sum_{i=1}^d K_{ii} \left[ \left( 2q_{11} u_i + 2q_{12} v_i + l_1\right)^2 + \left( 2q_{12} u_i + 2q_{22} v_i + l_2\right)^2\right]\\[5pt]
&+ \frac{2}{d}(q_{11} + q_{22}) \sum_{i=1}^d K_{ii} \left(\psi(x) -\beta^* \right)_i.
\end{align*}

\begin{remark}\upright
Note that both $\mathcal{R}(x)$ and $\frac{1}{d}\tr \left(\nabla^2 \mathcal{R}(x) \right)$ satisfy Assumption~\ref{ass: statistic}. Consequently, their SGD dynamics concentrate and can be analyzed via Theorem~\ref{thm: statistic}. 
\end{remark}
\begin{example}\upright
The simplest parametrization is to consider $\psi(x) = u$ which corresponds to the classical leas-squares framework,
\begin{equation}
    \min_{u\in \R^d} \frac{1}{2d} \EE_a \inner{u - \beta^*, a}^2 .
\end{equation}
Here, $\psi(u,v) = u$ for $\mathcal{Q} = \begin{bmatrix}
    0 & 0\\
    0 & 0
\end{bmatrix}$, $l = \begin{pmatrix}
    1\\
    0
\end{pmatrix}$, and $c=0$.
\end{example}

\begin{example}[Diagonal Linear Network, Standard Parametrization]\upright\label{ex: hadamardParam}
This formulation arises in the problem of minimizing the expected squared risk over a standard depth-two diagonal linear neural network, represented by
\begin{equation}
    \min_{u,v\in \R^d} \frac{1}{2d} \EE_a \inner{u\odot v - \beta^*, a}^2.
\end{equation}
Here, $\psi(u,v) = u\odot v$ for $\mathcal{Q} = \begin{pmatrix}
    0 & \frac{1}{2}\\
    \frac{1}{2} & 0
\end{pmatrix}$, $l=\begin{pmatrix}
    0\\
    0
\end{pmatrix}$, and $c=0$.
\end{example}
\begin{example}[Diagonal Linear Network, Squared Parametrization]\upright\label{ex: squaredParam}
The most relevant example is the squared parametrization of the linear regression model. The optimization problem associated with this parametrization is formulated as
\begin{equation}
    \min_{u,v\in \R^d} \frac{1}{4d}\EE_a  \inner{u^2 - v^2 - \beta^*, a}^2 .
\end{equation}
Here, $\psi(u,v) = u^2 -v^2$ for $\mathcal{Q} = \begin{pmatrix}
    1 & 0\\
    0 & -1
\end{pmatrix}$, $l=\begin{pmatrix}
    0\\
    0
\end{pmatrix}$, and $c=0$.
\end{example}

\subsection{Binary Logistic Regression}
We consider the setting of \eqref{eq: riskSetting} with outer function $f\colon \R^2 \to \R$ given by 
\begin{equation*}
f(r) = - r_1 \cdot \frac{\exp(r_2)}{\exp(r_2) + 1} + \log (\exp(r_1) + 1).
\end{equation*}
In this case, the problem takes the form
\begin{equation*}
\min_{x\in \R^{2d}} \biggl\{ \mathcal{R}(x) = - \EE_a \inner{\psi(x), a} / \sqrt{d} \cdot \frac{\exp(\inner{\beta^*, a}/\sqrt{d})}{\exp(\inner{\beta^*, a}/\sqrt{d}) + 1} + \log (\exp(\inner{\psi(x), a}/\sqrt{d}) + 1)\biggr\},
\end{equation*}
and, using a multivariate normal distribution as in
\cite{collins-woodfin2024hitting}, satisfies
\begin{align*}
\mathcal{R}(x) &= h\left(B(x)\right) = \frac{1}{\pi} \int_{\R^2} \left(- \tilde{u} \cdot \frac{\exp(\tilde{v})}{\exp(\tilde{v}) + 1} + \log (\exp(\tilde{u}) + 1) \right) \exp \left( - \begin{pmatrix}
    u \\ 
    v
\end{pmatrix}^\top  \begin{pmatrix}
    u \\ 
    v
\end{pmatrix} \right) \mathrm{d}u \mathrm{d}v
\end{align*}
where $\begin{pmatrix}
    \tilde{u} \\ 
    \tilde{v}
\end{pmatrix} = \sqrt{2} L \begin{pmatrix}
    u \\ 
    v
\end{pmatrix}$ and $B(x) = L L^\top$.

Analogously, since
\begin{equation*}
\nabla_{r_1} f(r) = - \frac{\exp(r_2)}{\exp(r_2) + 1} + \frac{\exp(r_1)}{\exp(r_1) + 1},
\end{equation*}
we obtain
\begin{equation*}
I\left(B(x)\right) =
\EE_a [\nabla_{r_1} f(r)^2] =  \frac{1}{\pi} \int_{\R^2} \left(-  \frac{\exp(\tilde{v})}{\exp(\tilde{v}) + 1} + \frac{\exp(\tilde{u})}{\exp(\tilde{u}) + 1} \right)^2 \exp \left( - \begin{pmatrix}
    u \\ 
    v
\end{pmatrix}^\top  \begin{pmatrix}
    u \\ 
    v
\end{pmatrix} \right) \mathrm{d}u \mathrm{d}v.
\end{equation*}

\section{Numerical Simulation Details}\label{app: numerics}

Here we provide more details for the figures that appear in the main paper.

\paragraph{Figure~\ref{fig: risk}:} 
\textbf{Three views of empirical risk dynamics for SGD on a diagonal linear network:}
concentration, spectral effects, and regime transitions.
\emph{Left:} Identity covariance $K=I_d$. As the dimension $d$ increases, the risk trajectory of SGD concentrates around a deterministic limit (shown in red), as characterized in Theorem~\ref{thm: statistic}.
\emph{Middle:} Power-law covariance spectrum. Our homogenized SGD prediction (transparent) from Theorem~\ref{thm: statistic} closely tracks SGD (opaque) over a range of power-law exponents $\beta$, at fixed dimension $d=10^3$.
\emph{Right:} Identity covariance $K=I_d$ at fixed dimension $d=10^3$. Varying the stepsize $\gamma$ reveals distinct convergence/divergence regimes; the homogenized prediction remains accurate even for stepsizes above the convergence threshold.

The left and right panels use the squared parametrization (Example~\ref{ex: squaredParam}), whereas the middle panel uses the standard parametrization (Example~\ref{ex: hadamardParam}).
For the power-law spectrum, eigenvalues are generated via inverse-CDF sampling: draw $u_i\sim\mathcal{U}(0,1)$ and set $\lambda_i=u_i^{1/(1-\beta)}$, yielding i.i.d.\ $\lambda_i\in[0,1]$ with density $p(\lambda)=(1-\beta)\lambda^{-\beta}$.
We set $\beta^*=\mathds{1}_d$ and initialize SGD at $u_0=v_0=\alpha\,\mathds{1}_d$ with $\alpha=0.6$ (left and right), and at $u_0=2\cdot \mathds{1}_d$, $v_0=0$ (middle).
Curves are aggregated over $30$ independent stochastic runs. The stepsize is fixed to $\gamma=0.1$ (left) and $\gamma=1.1$ (middle); in the right panel, $\gamma$ is swept across the displayed range at fixed $d=10^3$.

\paragraph{Numerical construction of the theoretical curve:} The contour-integral evaluation used to recover the theory curve is very sensitive to the choice of contour radius.
Under symmetric initialization ($u_0=v_0$), a single contour $C_M(0)$ suffices; we empirically verify that $\|x_k\|_\infty<M$ along the trajectory and use $M=1.3$ with $N=100$ contour discretization points.
If $M$ is too small, the condition $\|x_k\|_\infty<M$ is violated, leading to incorrect contour integrals and a mismatch (notably at initialization) between SGD and the PDE prediction; if $M$ is too large, the PDE solver may overflow.
Accordingly, $M$ must be calibrated for each initialization.

\paragraph{Figure~\ref{fig: statistics}:} \textbf{Curvature dynamics for SGD on a diagonal linear network.} \emph{Left:} The evolution of the curvature measured by the scaled trace of the Hessian $\frac{1}{d}\tr(\nabla^2\mathcal{R})$ is shown alongside the empirical risk $\mathcal{R}$, illustrating ``flat'' progress in which the risk increases sharply accompanied by a marked drop in curvature as we vary the stepsize $\gamma$.
    \emph{Right:} As the dimension $d$ increases, the curvature dynamics of SGD concentrate around a deterministic limit (shown in red), as proven in Theorem~\ref{thm: statistic}.
    The left panel corresponds to the squared parametrization (Example~\ref{ex: squaredParam}), while the right one corresponds to the standard parametrization (Example~\ref{ex: hadamardParam}).
    
    Note that in the standard parametrization, the scaled Hessian trace simplifies to
    $\tr(\nabla^2 \mathcal{R}(x))=\frac{1}{d}\sum_{i=1}^d K_{ii}(u_i^2+v_i^2)$.
    We set $\beta^*=\mathds{1}_d$ and $K=I_d$, and initialize SGD at $u_0=v_0=\alpha\,\mathds{1}_d$ with $\alpha=1.0$ (left) and $\alpha=0.6$ (right).
    Curves are aggregated over $30$ independent stochastic runs.
    \emph{Left:} fixed $d=10^3$ with a sweep over $\gamma$.
    \emph{Right:} fixed $\gamma=0.1$ with varying $d$. We use $M=1.0$ and $N=300$ for the contour construction.

\paragraph{Figure~\ref{fig: nondiagK}:}
\textbf{Risk concentration of SGD and the homogenized SDE under non-diagonal covariance on a diagonal linear network.}
As the dimension $d$ increases, the risk trajectories of SGD (opaque) concentrate around the prediction of the non-diagonal homogenized SDE \eqref{eq:nonDiagonalSquareLossSDE} (transparent), suggesting that the same high-dimensional concentration phenomenon persists beyond the diagonal covariance setting. The covariance matrix $K$ is sampled from a Marchenko--Pastur ensemble. 

The left panel uses the standard parametrization (Example~\ref{ex: hadamardParam}) and is initialized at $u_0=2\cdot\mathds{1}_d$, while the right panel uses the squared parametrization (Example~\ref{ex: squaredParam}) with initialization $u_0=v_0=\alpha\,\mathds{1}_d$ and $\alpha=0.6$. In both panels, we set $\beta^\star=\mathds{1}_d$ and fix the stepsize to $\gamma=0.2$. Curves are aggregated over $30$ independent stochastic runs. The covariance matrix $K$ is generated by drawing a Gaussian random matrix $X\in\mathbb R^{d\times d}$ with i.i.d.\ standard normal entries and setting $K=\frac{1}{d}XX^\top$.

\paragraph{Figure~\ref{fig: sdes}:} 
\textbf{Risk discrepancy between SGD and its continuous-time approximations on a diagonal linear network.}
    For each stepsize $\gamma$, we report the absolute difference between the empirical risk of SGD after $T\!\cdot\! d$ iterations (with $T=20$) and two approximations: (i) homogenized SGD (HSGD) \eqref{eq: homogenizedSGD} (blue), and (ii) stochastic gradient flow (SGF) \eqref{eq: homogenizedSGF} (pink). As $\gamma$ increases, the proposed HSGD---which captures large stepsize effects---provides a progressively more accurate approximation of SGD, whereas SGF is derived under a vanishing stepsize regime and thus degrades for larger $\gamma$.
    Initialization scale $\alpha$ controls proximity to the saddle point $x=0$: smaller $\alpha$ corresponds to a longer transient before learning accelerates. 
    
    We consider the squared parametrization (Example~\ref{ex: squaredParam}) with mean-squared error, in the sparse recovery setting of~\cite{pesme2021implicit}. The ground-truth signal satisfies $\beta^*_i=1$ for $1\leqslant i\leqslant 5$ and $\beta^*_i=0$ for $i>5$, with covariance $K=I_d$ and dimension $d=100$.
    SGD is initialized at $u_0=v_0=\alpha\,\mathds{1}_d$, with $\alpha=0.01$ (left) and $\alpha=0.1$ (right), and results are aggregated over $30$ independent stochastic runs.
    Both HSGD and SGF trajectories are simulated using the Euler--Maruyama method with time step $\mathrm{d}t=2^{-8}$.

\paragraph{Figure~\ref{fig: risk_exponential_decay}:}
\textbf{Coordinatewise entropy barriers and exponential risk decay on a diagonal linear network.}
The figure illustrates the entropy-barrier mechanism developed in Appendix~\ref{app:entropy_barrier}. 
The left panels track the two quantities that define the controlled regime of the proof: the empirical entropy $H_t$ and the largest coordinatewise entropy density $\max_i h_{t,i}$. The red dashed lines mark the corresponding barriers $H_\ast$ and $L_\ast$. Across the tested stepsizes, the trajectories remain below these barriers, indicating that the dynamics stay in the regime where the coordinatewise entropy-barrier argument applies.

The top-right panel displays the risk--entropy ratio $4\mathcal R(\mathscr X_t)/H_t$. The red dashed lines mark empirical estimates of the coordinatewise coercivity constants $m_{L_\ast}$ and $M_{L_\ast}$. These constants are estimated from the coordinate quotients
\[
q_{t,i}
=
\frac{(\mathscr U_{t,i}^2-\mathscr V_{t,i}^2-1)^2}{h_{t,i}},
\]
rather than directly from the averaged ratio $4\mathcal R(\mathscr X_t)/H_t$. Specifically, we pool the lower and upper coordinate-quotient summaries across all stepsizes and stochastic runs, restricted to times satisfying $H_{\min}\leqslant H_t\leqslant H_\ast$ and $\max_i h_{t,i}\leqslant L_\ast$, and then take robust quantiles to estimate $m_{L_\ast}$ and $M_{L_\ast}$. The panel visualizes the theorem’s risk--entropy comparison: while the coordinatewise barrier holds, the averaged ratio $4\mathcal R(\mathscr X_t)/H_t$ should remain controlled by the same coercivity constants.

The bottom-right panel shows the risk $\mathcal R(\mathscr X_t)$ on a logarithmic scale. The dashed curves are the exponential envelopes suggested by the theorem, of the form
\[
\frac{M_{L_\ast}}{4}H_\ast e^{-\mu t},
\]
using the empirical constants $m_{L_\ast}$ and $M_{L_\ast}$ from the coordinatewise quotient estimates. In the theorem, this bound is guaranteed under a sufficient small-stepsize condition ensuring a positive decay rate. In the experiment, we plot the same exponential envelope for all tested stepsizes, rather than suppressing curves outside the certified small-stepsize range. This highlights that the predicted exponential decay behavior appears to persist empirically even for stepsizes larger than those directly covered by the sufficient condition.

Shaded regions denote the $10$th--$90$th percentile range over independent runs, and solid curves denote the median trajectory. We simulate the isotropic squared-parametrization setting with $\beta^*=\mathds{1}_d$, covariance $K=I_d$, Gaussian covariates $a\sim\mathcal N(0,I_d)$, and initialization $u_0=v_0=\mathds{1}_d$. The dimension is $d=10^3$, the time horizon is $T=30$, and results are aggregated over $30$ independent stochastic runs for each constant stepsize $\gamma\in\{0.04, 0.03, 0.02, 0.01, 0.005, 0.001\}$. The horizontal axis is the rescaled time $t=k/d$, i.e., SGD iterations divided by the dimension.

\end{document}